\documentclass[12pt]{amsart}%
\usepackage{amsfonts}
\usepackage{amsmath}
\usepackage{amssymb}
\usepackage{graphicx}%
\providecommand{\U}[1]{\protect\rule{.1in}{.1in}}
\newtheorem{theorem}{Theorem}
\theoremstyle{plain}

\newtheorem{corollary}{Corollary}

\newtheorem{definition}{Definition}

\newtheorem{lemma}{Lemma}

\newtheorem{proposition}{Proposition}
\newtheorem{remark}{Remark}

\numberwithin{equation}{section}
\numberwithin{equation}{section}
\numberwithin{theorem}{section}
\numberwithin{lemma}{section}
\numberwithin{remark}{section}
\numberwithin{example}{section}
\numberwithin{proposition}{section}
\numberwithin{definition}{section}
\numberwithin{corollary}{section}
\begin{document}
\title[Deformation quantization]{Deformation quantization of projective schemes and differential operators}
\author{Anar Dosi}
\address{Middle East Technical University NCC, Guzelyurt, KKTC, Mersin 10, Turkey}
\email{(dosiev@yahoo.com), (dosiev@metu.edu.tr)}
\date{November 14, 2022}

\begin{abstract}
The paper is devoted to noncommutative projective schemes within Kapranov's
framework of noncommutative algebraic geometry. We classify all noncommutative
projective schemes obtained from the differential chains in the universal
enveloping algebra $\mathcal{U}\left(  \mathfrak{g}_{q}\left(  \mathbf{x}%
\right)  \right)  $ of the free nilpotent Lie algebra $\mathfrak{g}_{q}\left(
\mathbf{x}\right)  $ of index $q$ generated by $\mathbf{x=}\left(
x_{0},\ldots,x_{n}\right)  $. The construction proposed allows us to provide a
new method of deformation quantization of the commutative projective schemes
within the considered framework.

\end{abstract}
\maketitle

\section{\textbf{Introduction\label{Sec1}}}

The algebraic analogue of celebrated Kontsevich's result \cite{Ko1}
establishes a link (modulo the gauge equivalence) between all formal Poisson
structures on a smooth algebraic variety $\left(  X,\mathcal{O}_{X}\right)  $
(over $\mathbb{C}$) and the deformation quantizations of $\mathcal{O}_{X}$
whenever $X$ is $\mathcal{D}_{X}$-affine, that is, $H^{i}\left(
X,\mathcal{M}\right)  =0$, $i>0$ for all quasi-coherent $\mathcal{D}_{X}%
$-modules $\mathcal{M}$, where $\mathcal{D}_{X}$ is the sheaf of differential
operators on $X$. If $X$ is affine this correspondence is a bijection. The
result is due to A. Yekutieli \cite{Yak}. Recall that a deformation
quantization of $\mathcal{O}_{X}$ is given by the star product $f\ast
g=fg+\sum_{j}\beta_{j}\left(  f,g\right)  \hbar^{j}$ on the sheaf
$\mathcal{O}_{X}\left[  \left[  \hbar\right]  \right]  $ of all formal power
series in an indeterminate $\hbar$ (Planck constant) over $\mathcal{O}_{X}$,
where $\beta_{j}:\mathcal{O}_{X}\times\mathcal{O}_{X}\rightarrow
\mathcal{O}_{X}$ are bi-differential operators acting on the local sections
$f,g$ of $\mathcal{O}_{X}$. The quantization problem of a smooth symplectic
variety based on the algebraic version of Fedosov quantization was
investigated by R. Bezrukavnikov and D. Kaledin in \cite{BK}. As the main tool
they used Harish-Chandra torsors and their extensions. The quantized schemes
$\left(  X,\mathcal{O}_{X}\left[  \left[  \hbar\right]  \right]  \right)  $
are basic examples of the noncommutative schemes obtained by deformation of
the commutative schemes. Notice that $\hbar\mathcal{O}_{X}\left[  \left[
\hbar\right]  \right]  $ is the commutant of $\mathcal{O}_{X}\left[  \left[
\hbar\right]  \right]  $ and its commutativization $\mathcal{O}_{X}\left[
\left[  \hbar\right]  \right]  _{c}=\mathcal{O}_{X}\left[  \left[
\hbar\right]  \right]  /\hbar\mathcal{O}_{X}\left[  \left[  \hbar\right]
\right]  $ is reduced to $\mathcal{O}_{X}$ up to an isomorphism. A similar
structure of noncommutative schemes is observed in M. Kapranov's approach
\cite{Kap} to noncommutative algebraic geometry. They are nilpotent
thickenings of the commutative schemes, which is based on the commutator
filtration of noncommutative rings. An associative ring $R$ can be equipped
with the NC-topology defined by means of the commutator filtration $\left\{
F_{m}\left(  R\right)  \right\}  $ (see below Subsection \ref{subsecLieNCfil}%
). The ring $R$ is called an NC-complete ring if it is Hausdorff and complete
with respect to its NC-topology. One of the key properties of NC-complete
rings is that they admit topological (Ore) localizations (see \cite{Kap},
\cite{DComA}), which are in turn NC-complete rings. The surjective
homomorphism $R\rightarrow R_{c}$ of $R$ onto its commutativization
$R_{c}=R/I\left(  \left[  R,R\right]  \right)  $ allows to quantize the
commutative geometric object $X=\operatorname{Spec}\left(  R_{c}\right)  $
into a noncommutative affine scheme $\left(  X,\mathcal{O}_{R}\right)  $. It
is a ringed space with the same underlying topological space $X$ equipped with
a noncommutative structure sheaf $\mathcal{O}_{R}$ of NC-complete rings such
that $\mathcal{O}_{R,c}=\mathcal{O}_{X}$. The structure sheaf $\mathcal{O}%
_{R}$ is defined as the sheaf of all continuous sections of the covering space
over $X$ defined by the noncommutative topological (Ore) localizations of $R$
(see Subsection \ref{subsecNCdefsch}).

The affine NC-space $\mathbb{A}_{k,q}^{n}$ over a field $k$ with the thickness
of $q$ is defined to be the formal scheme $\operatorname{Spf}\left(
A_{q}\right)  $ of the universal enveloping algebra $A_{q}=\mathcal{U}\left(
\mathfrak{g}_{q}\left(  \mathbf{x}\right)  \right)  $ (equipped with the
NC-topology) of the free nilpotent Lie $k$-algebra $\mathfrak{g}_{q}\left(
\mathbf{x}\right)  $ of index $q$ in the independent variables $\mathbf{x=}%
\left(  x_{1},\ldots,x_{n}\right)  $. For the structure sheaf $\mathcal{O}%
_{A_{q}}$ of the NC-space $\mathbb{A}_{k,q}^{n}$ we obtain that $\mathcal{O}%
_{A_{q}}=\mathcal{O}\left[  \left[  \mathbf{y}\right]  \right]  $ up to an
isomorphism, where $\mathcal{O}$ is the structure sheaf of the commutative
affine space $\mathbb{A}_{k}^{n}$ (or $\operatorname{Spec}k\left[
\mathbf{x}\right]  $) over $k$, and $\mathbf{y}$ is a basis of commutators in
$\mathbf{x}$ (see \cite{Kap}, \cite{DIZV}). Thus $\mathbb{A}_{k,q}^{n}$ is a
quantization of the affine space $\mathbb{A}_{k}^{n}$ which depends on $q$. A
similar description of $\mathbb{A}_{k,q}^{n}$ for $q=\infty$ was obtained in
\cite{Kap} (see also \cite{Dproj1}).

The projective spaces of Kapranov's framework were considered in \cite{Kap},
\cite{Dproj1}. Recall that the basic principle of noncommutative projective
geometry is to assign a geometric object to a graded noncommutative ring,
generalizing $\operatorname{Proj}\left(  S\right)  $ for the commutative
graded rings $S$. In the fundamental paper \cite{AZ} by M. Artin and J. J.
Zhang a formalism of noncommutative projective schemes is suggested based on a
noncommutative setting of Serre's theorem \cite[Proposition 2.5.15]{Harts},
which establishes an equivalence of the categories of the coherent sheaves on
projective schemes and finitely generated graded algebras. The construction of
a twisted homogeneous coordinate ring $B=B\left(  X,\sigma,\mathcal{L}\right)
$ from a projective scheme $X$ equipped with an automorphism $\sigma$ and a
$\sigma$-ample invertible sheaf $\mathcal{L}$ on $X$ was proposed by M. Artin
and M. Van den Bergh in \cite{AV1}. That is a way how to generate
noncommutative graded algebras from the commutative projective schemes. So is
the quantum plane $k\left\langle x,y\right\rangle /\left\{  xy-qyx\right\}  $
obtained from the projective line $\mathbb{P}_{k}^{1}$ as $B\left(
\mathbb{P}_{k}^{1},\sigma,\mathcal{O}\left(  1\right)  \right)  $ (see
\cite{SV1}), where $\sigma\left(  a_{0}:a_{1}\right)  =\left(  a_{0}%
:qa_{1}\right)  $. Similarly, the Sklyanin algebra is obtained from an
elliptic curve \cite{ATV} (see also \cite{SV1}). The noncommutative scheme
whose homogenous coordinate ring is reduced to a free algebra was constructed
by S. P. Smith in \cite{Sp1}.

The universal enveloping algebra $S_{q}=\mathcal{U}\left(  \mathfrak{g}%
_{q}\left(  \mathbf{x}\right)  \right)  $ of the free nilpotent Lie
$k$-algebra $\mathfrak{g}_{q}\left(  \mathbf{x}\right)  $ generated by a
finite set of generators $\mathbf{x}=\left(  x_{0},\ldots,x_{n}\right)  $
admits a natural grading with the property $\deg\left(  x_{i}\right)  =1$ for
all $i$. As above the tuple $\mathbf{x}$ can be extended up to a Hall basis
$\mathbf{z}=\mathbf{x}$\textbf{$\sqcup$}$\mathbf{y=}\left(  z_{0},\ldots
,z_{v}\right)  $ for $\mathfrak{g}_{q}\left(  \mathbf{x}\right)  $ by adding
up the commutators $\mathbf{y}$ in $\mathbf{x}$. The degrees $e_{i}%
=\deg\left(  z_{i}\right)  $, $n<i\leq v$ of the commutators are running
within $2$ and $q$, and let $R_{q}^{m}\left\langle \mathbf{y}\right\rangle $
be a subspace of $S_{q}$ generated by all ordered monomials $\mathbf{y}%
^{\alpha}$ of degree $m$. Then $S_{q}=\bigoplus\limits_{d\geq0}S_{q}^{d}$ with
$S_{q}^{d}=\bigoplus\limits_{m}S^{d-m}\otimes R_{q}^{m}\left\langle
\mathbf{y}\right\rangle $, where $S=k\left[  \mathbf{x}\right]  =\bigoplus
\limits_{e}S^{e}$ is the\ original grading. Note that $S_{q}=S\otimes
R_{q}\left\langle \mathbf{y}\right\rangle $, where $R_{q}\left\langle
\mathbf{y}\right\rangle =\oplus_{m}R_{q}^{m}\left\langle \mathbf{y}%
\right\rangle $ is the subalgebra in $S_{q}$ generated by $\mathbf{y}$. The
graded algebra $S_{q}$ possesses many key properties to be included into
noncommutative projective geometry (see \cite{DK}). The geometric object that
stands for this algebra within Kapranov's framework is the projective
$q$-space $\mathbb{P}_{k,q}^{n}$. The scheme $\mathbb{P}_{k,q}^{n}$ is defined
to be the set $\operatorname{Proj}\left(  S_{q}\right)  $ of all open graded
prime (two-sided) ideals of $S_{q}$ which do not contain $\bigoplus
\limits_{d>0}S_{q}^{d}$. As a topological space $\mathbb{P}_{k,q}^{n}$ is
identified with the projective space $\mathbb{P}_{k}^{n}=\operatorname{Proj}%
\left(  S\right)  $ up to a homeomorphism. The structure sheaf $\mathcal{O}%
_{q}$ of $\mathbb{P}_{k,q}^{n}$ is defined by means of the zero degree term of
the topological graded (Ore) localizations $S_{q,h\left(  \mathbf{x}\right)
}$, $h\in S$. The projective NC-space $\mathbb{P}_{k,q}^{n}$ is described in
terms of the twisted sheaves $\mathcal{O}\left(  d\right)  $, $d\in\mathbb{Z}$
of $\mathbb{P}_{k}^{n}$ in the following way \cite[Section 7]{Dproj1}
\[
\mathcal{O}_{q}=\prod_{m=0}^{\infty}\mathcal{O}\left(  -m\right)  \otimes
R_{q}^{m}\left\langle \mathbf{y}\right\rangle .
\]
Note that $\mathcal{O}_{q}$ possesses the filtration $\left\{  \mathcal{J}%
_{m}\right\}  $ with $\mathcal{J}_{m}=\prod_{s=m}^{\infty}\mathcal{O}\left(
-s\right)  \otimes R_{q}^{s}\left\langle \mathbf{y}\right\rangle $. It is a
subsheaf of two-sided ideals in $\mathcal{O}_{q}$ such that $\mathcal{O}%
_{q}/\mathcal{J}_{m}$ is the coherent sheaf $\bigoplus\limits_{s<m}%
\mathcal{O}\left(  -s\right)  \otimes R_{q}^{s}\left\langle \mathbf{y}%
\right\rangle $ of the NC-nilpotent $k$-algebras. Thus $\mathcal{O}_{q}$ being
a filtered sheaf is the inverse limit $\underleftarrow{\lim}\left\{
\mathcal{O}_{q}/\mathcal{J}_{m}\right\}  $ of the coherent $\mathcal{O}%
$-modules of NC-nilpotent $k$-algebras.

The present paper is devoted to the deformation quantization (or
NC-deformation) problem of the projective schemes within $\mathbb{P}_{k,q}%
^{n}$. Actually, we classify noncommutative projective schemes based on the
differential chains in $S_{q}$. The proposed construction allows us to provide
a new method of the geometric quantization of the commutative projective
schemes in $\mathbb{P}_{k,q}^{n}$. A general approach to NC-deformations of
the commutative schemes is presented in Section \ref{SecNCdef}. Let $\left(
X,\mathcal{O}_{X}\right)  $ be a commutative scheme and let $\left(
X,\mathcal{F}_{X}\right)  $ be an NC-complete scheme whose noncommutative
structure sheaf $\mathcal{F}_{X}$ is a sheaf of $\mathcal{O}_{X}$-modules on
$X$ equipped with a filter base $\left\{  \mathcal{J}_{X,m}\right\}  $ of
two-sided ideal subsheaves of $\mathcal{O}_{X}$-modules in $\mathcal{F}_{X}$
such that all $\mathcal{F}_{X}/\mathcal{J}_{X,m}$ are quasi-coherent
$\mathcal{O}_{X}$-modules of NC-nilpotent rings, $\mathcal{C}_{X}%
=\mathcal{J}_{X,1}$, $\mathcal{F}_{X}/\mathcal{J}_{X,1}=\mathcal{O}_{X}$ and
$\mathcal{F}_{X}=\underleftarrow{\lim}\left\{  \mathcal{F}_{X}/\mathcal{J}%
_{X,m}\right\}  $, where $\mathcal{C}_{X}$ is the two-sided ideal sheaf in
$\mathcal{F}_{X}$ generated by the commutator $\left[  \mathcal{F}%
_{X},\mathcal{F}_{X}\right]  $. Thus $\left(  X,\mathcal{F}_{X}\right)  $ is
the inverse limit of the NC-nilpotent schemes $\left(  X,\mathcal{F}%
_{X}/\mathcal{J}_{X,m}\right)  $ and $\mathcal{F}_{X,c}=\mathcal{O}_{X}$. If
the $\mathcal{F}_{X}$-bimodule structure of $\operatorname{gr}\left(
\mathcal{F}_{X}\right)  =\bigoplus\limits_{m}\mathcal{J}_{X,m}/\mathcal{J}%
_{X,m+1}$ can be lifted from its $\mathcal{O}_{X}$-module one through the
morphism $\mathcal{F}_{X}\rightarrow\mathcal{O}_{X}$, then we say that
$\left(  X,\mathcal{F}_{X}\right)  $ is \textit{an NC-deformation of} $\left(
X,\mathcal{O}_{X}\right)  $ (see also \cite[Definition 1.3]{BK},
\cite[Definition 1.2]{Yak}). The projective NC-space $\mathbb{P}_{k,q}^{n}$ is
an NC-deformation of $\mathbb{P}_{k}^{n}$, all $\left(  \mathbb{P}_{k,q}%
^{n},\mathcal{O}_{q}/\mathcal{J}_{m}\right)  $ are projective NC-nilpotent
deformations of $\mathbb{P}_{k}^{n}$, and the affine NC-space $\mathbb{A}%
_{k,q}^{n}$ is an affine NC-deformation of $\mathbb{A}_{k}^{n}$. The first key
result (see Theorem \ref{thNCDaff}) of the paper asserts that an
NC-deformation of an affine scheme turns out to be an affine NC-scheme.

By an NC-morphism $\left(  X,\mathcal{F}_{X}\right)  \rightarrow\left(
Y,\mathcal{F}_{Y}\right)  $ of the NC-deformations, we mean a couple $\left(
f,f^{+}\right)  $ of a continuous mapping $f:X\rightarrow Y$ and a filtered
sheaf morphism $f^{+}:\mathcal{F}_{Y}\rightarrow f_{\ast}\mathcal{F}_{X}$ of
the filtered $\mathcal{O}_{Y}$-modules such that the canonical mapping
$f_{x}^{+}:\mathcal{F}_{Y,f\left(  x\right)  }\rightarrow\mathcal{F}_{X,x}$,
$x\in X$ over stalks is a local homomorphism of noncommutative rings.
Actually, an NC-morphism $\left(  f,f^{+}\right)  $ defines the morphism
$\left(  f,f^{\times}\right)  :\left(  X,\mathcal{O}_{X}\right)
\rightarrow\left(  Y,\mathcal{O}_{Y}\right)  $ of the supported commutative
schemes compatible with $\left(  f,f^{+}\right)  $ (see Subsection
\ref{subsecMNC}). If $f$ is a homeomorphism onto its range and $f^{+}%
:\mathcal{F}_{Y}\rightarrow f_{\ast}\mathcal{F}_{X}$ is a filtered surjective
morphism, then $\left(  f,f^{+}\right)  $ is called \textit{a closed
NC-immersion }(see Subsection \ref{subsecCIPS}). In this case, $\left(
f,f^{\times}\right)  $ turns out to be a closed immersion of the related
commutative schemes. An NC-deformation $\left(  Y,\mathcal{F}_{Y}\right)  $
over $k$ is called a \textit{projective NC-deformation over} $k$ or $\left(
Y,\mathcal{F}_{Y}\right)  \rightarrow\operatorname{Spec}\left(  k\right)  $ is
a\textit{ projective morphism} if it factors into a closed NC-immersion
$\left(  \iota,\iota^{+}\right)  :\left(  Y,\mathcal{F}_{Y}\right)
\rightarrow\mathbb{P}_{k,q}^{n}$ followed by the projection $\mathbb{P}%
_{k,q}^{n}\rightarrow\operatorname{Spec}\left(  k\right)  $.

The noncommutative algebraic structure of $S_{q}$ can be described in term of
the differential operators on $S_{q}$. To formulate our main results let us
introduce some technical details. The differential operators on a
noncommutative algebra can be defined recursively based on operators which
commute with the left (resp., right) multiplication operators \cite{IU},
\cite{Haz}. The differential operators in its general noncommutative geometry
setting were investigated in \cite{Gin}. To introduce our differential
operators on $S_{q}$ we fix as above a Hall basis $\mathbf{z}$ for
$\mathfrak{g}_{q}\left(  \mathbf{x}\right)  $. Since $S_{q}=k\left[
\mathbf{z}\right]  $ as the vector $k$-spaces, every element $z_{j}$ of the
basis $\mathbf{z}$ defines the commutative multiplication operator $z_{j}%
\in\mathcal{L}_{k}\left(  S_{q}\right)  $, $z_{j}\left(  p\left(
\mathbf{z}\right)  \right)  =\left(  z_{j}p\right)  \left(  \mathbf{z}\right)
$, and the partial differential operators $\partial^{\mathbf{i}}%
,\overline{\partial}^{\mathbf{i}}\in\mathcal{L}_{k}\left(  S_{q}\right)  $,
$\mathbf{i}\in\mathbb{Z}_{+}^{v+1}$ (see (\ref{DforM})). Using the right
regular (anti-)representation $R$ of $S_{q}$, we introduce the following
differential operators
\[
\Delta_{j}=\sum_{\mathbf{i}\in\mathbb{Z}_{+}^{v+1},\left\vert \mathbf{i}%
\right\vert _{j}>0}R\left(  \operatorname{ad}\left(  \overline{\mathbf{z}%
}\right)  ^{\mathbf{i}}\left(  z_{j}\right)  \right)  \overline{\partial
}^{\mathbf{i}},\text{\quad}\nabla_{j}=-\sum_{\mathbf{i}\in\mathbb{Z}_{+}%
^{v+1},\left\vert \mathbf{i}\right\vert _{j}=0}R\left(  \operatorname{ad}%
\left(  \overline{\mathbf{z}}\right)  ^{\mathbf{i}}\left(  z_{j}\right)
\right)  \overline{\partial}^{\mathbf{i}}%
\]
on $S_{q}$, where $0\leq j\leq v$, $\operatorname{ad}\left(  \overline
{\mathbf{z}}\right)  ^{\mathbf{i}}\left(  z_{j}\right)  =\operatorname{ad}%
\left(  z_{v}\right)  ^{i_{v}}\cdots\operatorname{ad}\left(  z_{0}\right)
^{i_{0}}\left(  z_{j}\right)  $ (see Subsection \ref{subsecDOSQ}). Note that
$z_{j}\left(  S_{q}^{d}\right)  +\Delta_{j}\left(  S_{q}^{d}\right)
+\nabla_{j}\left(  S_{q}^{d}\right)  \subseteq S_{q}^{d+e_{j}}$ (see below
Lemma \ref{lemOp0}). We also put $D_{il}=\operatorname{ad}\left(
x_{i}\right)  ^{l}\left(  \Delta_{i}\right)  $, $D_{il}\left(  S_{q}%
^{d}\right)  \subseteq S_{q}^{d+l+1}$, where $0\leq i\leq n$ and $0\leq l<q$.
These differential operators can be lifted to $\mathbb{P}_{k,q}^{n}$ so that
$D_{il}:\mathcal{O}_{q}\rightarrow\mathcal{O}_{q}\left(  l+1\right)  $ and
$\nabla_{i}:\mathcal{O}_{q}\rightarrow\mathcal{O}_{q}\left(  1\right)  $.

Now let $I=\oplus_{d}I^{d}$ be a two-sided graded ideal of $S_{q}$. We say
that $I$ is \textit{an NC-graded ideal} if $I^{d}=\bigoplus_{m=0}^{d}I^{d}%
\cap\left(  S^{d-m}\otimes R_{q}^{m}\left\langle \mathbf{y}\right\rangle
\right)  $ for every $d$. We prove (see Proposition \ref{propSG1}) that a
graded two-sided ideal of $S_{q}$ is an NC-graded ideal if and only if it is
the sum of a certain differential chain in $S_{q}$. A differential chain in
$S_{q}$ is defined in terms of the differential operators introduced above.
Namely, let $I_{m}\subseteq S\otimes R_{q}^{m}\left\langle \mathbf{y}%
\right\rangle $ be a graded $S$-submodule (note that $S\otimes R_{q}%
^{m}\left\langle \mathbf{y}\right\rangle $ is an $S$-module along the diagonal
map $S\rightarrow S\otimes R_{q}^{m}\left\langle \mathbf{y}\right\rangle $). A
family $\left\{  I_{m}\right\}  $ of $S$-submodules is said to be \textit{a
chain }if $z_{j}\left(  I_{m}\right)  \subseteq I_{m+e_{j}}$ in $S_{q}$ for
all $m$ and $j>n$. We define \textit{the sum of a chain }$\left\{
I_{m}\right\}  $ to be the subspace $I=\oplus_{m}I_{m}$ in $S_{q}$. A key
property of the chains asserts that $N_{m}=\left(  S\otimes R_{q}%
^{m}\left\langle \mathbf{y}\right\rangle \right)  /I_{m}$ turns into an
$S/I_{0}$-module and there is a compatible morphism $S/I_{0}\rightarrow N_{m}$
of $S/I_{0}$-modules extending the diagonal one. A chain $\left\{
I_{m}\right\}  $ of submodules in $S_{q}$ is said to be \textit{a differential
chain} if the differential operators $\Delta_{j}$ and $\nabla_{j}$ leave
invariant its sum $I$.

The geometry behind of a differential chain given by (the commutative) ideals
$\left\{  I_{m}\right\}  $ is described in terms of the closed subschemes of
the disjoint unions $\operatorname{Proj}\left(  S\otimes R_{q}^{m}\left\langle
\mathbf{y}\right\rangle \right)  =\bigsqcup\limits^{r_{m}}\mathbb{P}_{k}^{n}$
of several copies of $\mathbb{P}_{k}^{n}$, where $r_{m}=\dim\left(  R_{q}%
^{m}\left\langle \mathbf{y}\right\rangle \right)  $. In this case,
$Y_{m}=\operatorname{Proj}\left(  N_{m}\right)  $ are closed subschemes of
$\bigsqcup\limits^{r_{m}}\mathbb{P}_{k}^{n}$, and $Y=Y_{0}=\operatorname{Proj}%
\left(  S/I_{0}\right)  $ is a closed subscheme of $\mathbb{P}_{k}^{n}$. The
diagonal homomorphism defines the canonical scheme morphisms $\sigma
_{m}:\bigsqcup\limits^{r_{m}}\mathbb{P}_{k}^{n}\rightarrow\mathbb{P}_{k}^{n}$
and $\rho_{m}:Y_{m}\rightarrow Y$ which commute the diagram%
\[%
\begin{array}
[c]{ccc}%
\mathbb{P}_{k}^{n} & \overset{\sigma_{m}}{\longleftarrow} & \bigsqcup
\limits^{r_{m}}\mathbb{P}_{k}^{n}\\
\uparrow &  & \uparrow\\
Y & \overset{\rho_{m}}{\longleftarrow} & Y_{m}%
\end{array}
\]
for every $m$, whose vertical arrows are closed immersions. We define a new
sheaf
\[
\mathcal{O}_{q,Y}=\prod\limits_{m=0}^{\infty}\rho_{m,\ast}\mathcal{O}_{Y_{m}%
}\left(  -m\right)
\]
on $Y$ of $\mathcal{O}_{Y}$-modules based on the schemes $Y_{m}$ and their
morphisms into $Y$.

\bigskip

\textbf{Theorem 1. }\textit{The sheaf }$\mathcal{O}_{q,Y}$ \textit{is a sheaf
of NC-complete }$k$-\textit{algebras on }$Y$\textit{ such that }$\left(
Y,\mathcal{O}_{q,Y}\right)  $ \textit{is a projective NC-deformation of
}$\left(  Y,\mathcal{O}_{Y}\right)  $ \textit{with its sheaf of ideals}
$\mathcal{I}_{q,Y}=\prod\limits_{m=0}^{\infty}\widetilde{I_{m}}\left(
-m\right)  $. \textit{Moreover, there are operators} $D_{il,Y}:\mathcal{O}%
_{q,Y}\rightarrow\mathcal{O}_{q,Y}\left(  l+1\right)  $ \textit{and }%
$\nabla_{i,Y}:\mathcal{O}_{q,Y}\rightarrow\mathcal{O}_{q,Y}\left(  1\right)  $
\textit{lifting the related differential operators through a closed
NC-immersion} $\left(  Y,\mathcal{F}_{Y}\right)  \rightarrow\mathbb{P}%
_{k,q}^{n}$\textit{. }

\bigskip

The assertion takes place for arbitrary differential chains (see below Theorem
\ref{propDCS1}). A projective NC-deformation $\left(  Y,\mathcal{F}%
_{Y}\right)  $ over $k$ which possesses the operators $D_{il,Y}:\mathcal{F}%
_{Y}\rightarrow\mathcal{F}_{q}\left(  l+1\right)  $ and $\nabla_{i,Y}%
:\mathcal{F}_{Y}\rightarrow\mathcal{F}_{Y}\left(  1\right)  $ lifting the
related differential operators through a closed NC-immersion $\left(
Y,\mathcal{F}_{Y}\right)  \rightarrow\mathbb{P}_{k,q}^{n}$ is called \textit{a
projective }$q$\textit{-scheme over }$k$\textit{. }So are all projective
NC-deformations obtained from the differential chains in $S_{q}$. The second
central result describes all projective $q$-schemes over $k$.

\bigskip

\textbf{Theorem 2. }\textit{The projective }$q$\textit{-schemes over }$k$
\textit{are only projective NC-deformations over }$k$ \textit{obtained from
the differential chains in }$S_{q}$\textit{.}

\bigskip

The proposed framework provides a new method of quantization of the projective
schemes over $k$. One needs to differentiate the ideal $I_{0}$ of the given
projective scheme $Y$ by the differential operators listed above. That results
in a sequence of the graded $S$-modules $I_{m}\subseteq S\otimes R_{q}%
^{m}\left\langle \mathbf{y}\right\rangle $ and the related sheaves (or schemes
$Y_{m}$), which form a differential chain in $S_{q}$. The latter in turn
generates a projective NC-deformation of $Y$. So are the projective Lie-spaces
$\mathbb{P}_{\mathfrak{lie},k,q}^{n}$ which stand for the free Lie-nilpotent
$k$-algebras. Other examples of the geometric quantizations of the projective
curves and hypersurfaces, and their cohomology are considered in Section
\ref{Sec5}. Some necessary material of the projective geometry is provided in
Appendix Section \ref{SecApp}.

Finally, I wish to thank O. Yu. Aristov and A. Yu. Pirkovskii for their
interest to the paper and valuable remarks made.

\section{Preliminaries\label{SecPrel}}

In this section we provide the paper with some preliminaries. Everywhere below
we fix an algebraically closed field $k$ of characteristic zero, all
considered algebras are supposed to be unital $k$-algebras, and they are
noncommutative if the latter is not specified. The algebra of all $k$-linear
transformations on a $k$-vector space $Y$ is denoted by $\mathcal{L}%
_{k}\left(  Y\right)  $. If $\mathcal{F}$ is a sheaf of $k$-vector spaces on a
topological space $X$ then $\operatorname{Hom}_{k}\left(  \mathcal{F}\right)
$ denotes the sheaf $\operatorname{Hom}\left(  \mathcal{F},\mathcal{F}\right)
$ on $X$ of all $k$-endomorphisms of the sheaf $\mathcal{F}$. By \textit{a
filtered sheaf on }$X$ we mean a sheaf $\mathcal{F}$ of abelian groups on $X$
equipped with a (decaying) filtration $\left\{  \mathcal{F}_{m}\right\}  $ of
its subsheaves such that $\cap_{m}\mathcal{F}_{m}=\left\{  0\right\}  $. A
morphism $\varphi:\mathcal{F}\rightarrow\mathcal{P}$ of filtered sheaves on
$X$ is said to be \textit{a filtered morphism }if $\varphi\left(
\mathcal{F}_{m}\right)  \subseteq\mathcal{P}_{m}$ for all large $m$. If the
filtration of $\mathcal{F}$ is trivial then every sheaf morphism
$\varphi:\mathcal{F}\rightarrow\mathcal{P}$ of filtered sheaves is filtered. A
surjective sheaf morphism $\varphi:\mathcal{F}\rightarrow\mathcal{P}$ of
filtered sheaves on $X$ is called \textit{a filtered surjective morphism }if
$\varphi^{-1}\left(  \mathcal{P}_{m}\right)  =\mathcal{F}_{m}+\ker\left(
\varphi\right)  $ for all large $m$. A filtered surjective morphism is
filtered automatically.

\subsection{The commutation formulae\label{SubsecCF}}

The algebra $k\left[  \mathbf{z}\right]  $ of all $k$-polynomials in several
variables $\mathbf{z}=\left(  z_{1},\ldots,z_{v}\right)  $ possesses the
partial derivative operators $\partial_{1},\ldots,\partial_{v}$ with respect
to the variables $z_{1},\ldots,z_{v}$, respectively. Note that $\partial
_{s}^{m}p$ just means $\dfrac{\partial^{m}p}{\partial z_{s}^{m}}$, $1\leq
s\leq v$, where $p\in k\left[  \mathbf{z}\right]  $. For a multiplicative
subset $S\subseteq k\left[  \mathbf{z}\right]  $ we have the localization
$S^{-1}k\left[  \mathbf{z}\right]  $ to be a $k$-algebra extension of
$k\left[  \mathbf{z}\right]  $ with the related extended derivations
$\partial_{s}$ such that $\partial_{s}\left(  1/p\right)  =-\partial
_{s}\left(  p\right)  /p^{2}$ for all $p\in S$. We use the notations
$\partial_{\mathbf{z}}^{\mathbf{i}}f$ for the partial derivatives
$\partial_{1}^{i_{1}}\cdots\partial_{v}^{i_{v}}f$ with $\mathbf{i}%
\in\mathbb{Z}_{+}^{v}$ and $f\in S^{-1}k\left[  \mathbf{z}\right]  $. Let us
introduce the following differential operators
\begin{equation}
\partial^{\mathbf{i}},\overline{\partial}^{\mathbf{i}}\in\mathcal{L}%
_{k}\left(  S^{-1}k\left[  \mathbf{z}\right]  \right)  ,\quad\partial
^{\mathbf{i}}f=-\dfrac{1}{\mathbf{i}!}\partial_{\mathbf{z}}^{\mathbf{i}%
}f,\quad\overline{\partial}^{\mathbf{i}}f=\dfrac{\left(  -1\right)
^{\left\vert \mathbf{i}\right\vert }}{\mathbf{i}!}\partial_{\mathbf{z}%
}^{\mathbf{i}}f, \label{DforM}%
\end{equation}
related to the tuples $\mathbf{i}\in\mathbb{Z}_{+}^{v}$. If $\mathbf{z}%
_{s}=\left(  z_{s},\ldots,z_{v}\right)  $ then we write $\partial
_{\mathbf{z}_{s}}^{\mathbf{j}}f$ instead of $\partial_{s}^{j_{s}}%
\cdots\partial_{v}^{j_{v}}f$, where $\mathbf{j}\in\mathbb{Z}_{+}^{v-s+1}$. One
can easily verify that $\overline{\partial}^{\mathbf{k}}\left(  fg\right)
=\sum_{\mathbf{i+j}=\mathbf{k}}\overline{\partial}^{\mathbf{i}}\left(
f\right)  \overline{\partial}^{\mathbf{j}}\left(  g\right)  $ for all $f,g\in
S^{-1}k\left[  \mathbf{z}\right]  $.

Let $\mathcal{A}$ be a unital associative $k$-algebra with an $n$-tuple
$\mathbf{z}=\left(  z_{1},\ldots,z_{v}\right)  $ of its fixed elements called
\textit{a noncommutative variables}. There is a well defined homomorphism
(noncommutative polynomial calculus) $k\left\langle \mathbf{z}\right\rangle
\rightarrow\mathcal{A}$ from the free associative algebra $k\left\langle
\mathbf{z}\right\rangle $ generated by $\mathbf{z}$ into $\mathcal{A}$
extending the identity mapping over $\mathbf{z}$. Further, there is a well
defined $k$-linear mapping $k\left[  \mathbf{z}\right]  \rightarrow
k\left\langle \mathbf{z}\right\rangle $ sending the commutative monomials
$\mathbf{z}^{\alpha}$ into the related ordered noncommutative monomials in
$k\left\langle \mathbf{z}\right\rangle $. In particular, if $p\in k\left[
\mathbf{z}\right]  $ then we have the element $p\left(  \mathbf{z}\right)
\in\mathcal{A}$ being the range of $p$ in $k\left\langle \mathbf{z}%
\right\rangle $ throughout the indicated linear mapping called\textit{ the
ordered functional calculus}. As in (\ref{DforM}) we introduce the notations
$\operatorname{ad}\left(  \mathbf{z}\right)  ^{\mathbf{i}}=\operatorname{ad}%
\left(  z_{1}\right)  ^{i_{1}}\cdots\operatorname{ad}\left(  z_{v}\right)
^{i_{v}}$ and $\operatorname{ad}\left(  \overline{\mathbf{z}}\right)
^{\mathbf{i}}=\operatorname{ad}\left(  z_{v}\right)  ^{i_{v}}\cdots
\operatorname{ad}\left(  z_{1}\right)  ^{i_{1}}$ for the $k$-operators from
$\mathcal{L}_{k}\left(  \mathcal{A}\right)  $, where $\mathbf{i}=\left(
i_{1},\ldots,i_{v}\right)  \in\mathbb{Z}_{+}^{v}$.

The following assertion is known \cite{KM} as\textit{ the commutation
formulae} in noncommutative analysis. For the sake of a reader we provide its
proof in our special case.

\begin{lemma}
\label{lemNon1}If $p\in k\left[  \mathbf{z}\right]  $ then
\[
\left[  a,p\left(  \mathbf{z}\right)  \right]  =\sum_{\mathbf{i}\in
\mathbb{Z}_{+}^{v},\left\vert \mathbf{i}\right\vert >0}\operatorname{ad}%
\left(  \mathbf{z}\right)  ^{\mathbf{i}}\left(  a\right)  \left(
\partial^{\mathbf{i}}p\right)  \left(  \mathbf{z}\right)  =\sum_{\mathbf{i}%
\in\mathbb{Z}_{+}^{v},\left\vert \mathbf{i}\right\vert >0}\left(
\overline{\partial}^{\mathbf{i}}p\right)  \left(  \mathbf{z}\right)
\operatorname{ad}\left(  \overline{\mathbf{z}}\right)  ^{\mathbf{i}}\left(
a\right)  ,
\]
where $a,p\left(  \mathbf{z}\right)  \in\mathcal{A}$.
\end{lemma}

\begin{proof}
First note that if $p\left(  \mathbf{z}\right)  =z_{s}^{m}$ is a monomial in a
single variable, then the assertion follows from the following formula%
\begin{equation}
\left[  a,z_{s}^{m}\right]  =-\sum_{i=1}^{m}\dbinom{m}{i}\operatorname{ad}%
\left(  z_{s}\right)  ^{i}\left(  a\right)  z_{s}^{m-i}=\sum_{i=1}^{m}\left(
-1\right)  ^{i}\dbinom{m}{i}z_{s}^{m-i}\operatorname{ad}\left(  z_{s}\right)
^{i}\left(  a\right)  , \label{nonf1}%
\end{equation}
which can easily be proven by induction on $m$. Based on the linear property
of the partial derivative operators, we can assume that $p\left(
\mathbf{z}\right)  =\mathbf{z}^{\mathbf{m}}$ is an ordered monomial in
$\mathcal{A}$ with $\mathbf{m}=\left(  m_{1},\ldots,m_{v}\right)  $, and put
$b=\left[  a,p\left(  \mathbf{z}\right)  \right]  $. Using (\ref{nonf1}), we
derive that%
\begin{align*}
b  &  =-\sum_{s=1}^{v}\sum_{i_{s}\geq1}\dbinom{m_{s}}{i_{s}}z_{1}^{m_{1}%
}\cdots z_{s-1}^{m_{s-1}}\operatorname{ad}\left(  z_{s}\right)  ^{i_{s}%
}\left(  a\right)  z_{s}^{m_{s}-i_{s}}z_{s+1}^{m_{s+1}}\cdots z_{v}^{m_{v}}\\
&  =-\sum_{s}\sum_{i_{s}\geq1}\dfrac{1}{i_{s}!}\operatorname{ad}\left(
z_{s}\right)  ^{i_{s}}\left(  a\right)  \left(  \partial_{s}^{i_{s}}p\right)
\left(  \mathbf{z}\right)  +\sum_{s}\sum_{i_{s}\geq1}\dbinom{m_{s}}{i_{s}%
}\left[  a_{i_{s}},z_{1}^{m_{1}}\cdots z_{s-1}^{m_{s-1}}\right]  z_{s}%
^{m_{s}-i_{s}}z_{s+1}^{m_{s+1}}\cdots z_{v}^{m_{v}},
\end{align*}
where $a_{i_{s}}=\operatorname{ad}\left(  z_{s}\right)  ^{i_{s}}\left(
a\right)  $. Similarly, put $a_{i_{t}i_{s}}=\operatorname{ad}\left(
z_{t}\right)  ^{i_{t}}\operatorname{ad}\left(  z_{s}\right)  ^{i_{s}}\left(
a\right)  $. Using again (\ref{nonf1}) several times, we obtain that (see
(\ref{DforM}))
\begin{align*}
b  &  =\sum_{s}\sum_{i_{s}\geq1}\operatorname{ad}\left(  z_{s}\right)
^{i_{s}}\left(  a\right)  \left(  \partial^{i_{s}}p\right)  \left(
\mathbf{z}\right)  +\sum_{t<s}\sum_{i_{s},i_{t}\geq1}\operatorname{ad}\left(
z_{t}\right)  ^{i_{t}}\left(  a_{i_{s}}\right)  \left(  \partial^{i_{t}%
}\partial^{i_{s}}p\right)  \left(  \mathbf{z}\right) \\
&  +\sum_{t<s}\sum_{i_{s},i_{t}\geq1}\dbinom{m_{s}}{i_{s}}\dbinom{m_{t}}%
{i_{s}}\left[  a_{i_{t}i_{s}},z_{1}^{m_{1}}\cdots z_{t-1}^{m_{t-1}}\right]
z_{t}^{m_{t}-i_{t}}\cdots z_{s}^{m_{s}-i_{s}}\cdots z_{v}^{m_{v}}\\
&  =\sum_{\mathbf{i}\in\mathbb{Z}_{+}^{v},\left\vert \mathbf{i}\right\vert
>0}\operatorname{ad}\left(  z_{1}\right)  ^{i_{1}}\cdots\operatorname{ad}%
\left(  z_{v}\right)  ^{i_{v}}\left(  a\right)  \left(  \partial^{\mathbf{i}%
}p\right)  \left(  \mathbf{z}\right)  ,
\end{align*}
that is, $b=\sum_{\mathbf{i}\in\mathbb{Z}_{+}^{v},\left\vert \mathbf{i}%
\right\vert >0}\operatorname{ad}\left(  \mathbf{z}\right)  ^{\mathbf{i}%
}\left(  a\right)  \left(  \partial^{\mathbf{i}}p\right)  \left(
\mathbf{z}\right)  $. The same argument is applicable for the right
commutation formula. Hence $ap\left(  \mathbf{z}\right)  =p\left(
\mathbf{z}\right)  a+\sum_{\mathbf{i}\in\mathbb{Z}_{+}^{v},\left\vert
\mathbf{i}\right\vert >0}\left(  \overline{\partial}^{\mathbf{i}}p\right)
\left(  \mathbf{z}\right)  \operatorname{ad}\left(  \overline{\mathbf{z}%
}\right)  ^{\mathbf{i}}\left(  a\right)  =\sum_{\mathbf{i}\in\mathbb{Z}%
_{+}^{v}}\left(  \overline{\partial}^{\mathbf{i}}p\right)  \left(
\mathbf{z}\right)  \operatorname{ad}\left(  \overline{\mathbf{z}}\right)
^{\mathbf{i}}\left(  a\right)  $.
\end{proof}

We refer the result of Lemma \ref{lemNon1} as the left and right commutation fomulae.

\subsection{NC-deformations of a commutative ring and
NC-schemes\label{subsecLieNCfil}}

An associative ring $R$ turns out to be a Lie ring with respect to the Lie
brackets $\left[  x,y\right]  =xy-yx$, $x,y\in R$. We use the notation
$R_{\mathfrak{lie}}$ to indicate that we come up with the related Lie ring.
For an abstract Lie ring $\mathfrak{L}$ we have its lower central series
$\left\{  \mathfrak{L}^{\left(  n\right)  }\right\}  $ of Lie ideals defined
by $\mathfrak{L}^{\left(  n+1\right)  }=\left[  \mathfrak{L},\mathfrak{L}%
^{\left(  n\right)  }\right]  $, $n\geq1$, $\mathfrak{L}^{\left(  1\right)
}=\mathfrak{L}$. A Lie ring $\mathfrak{L}$ is said to be \textit{a nilpotent
Lie ring of index }$q$\textit{ }if $\mathfrak{L}^{\left(  q\right)  }%
\neq\left\{  0\right\}  $ whereas $\mathfrak{L}^{\left(  q+1\right)
}=\left\{  0\right\}  $. The universal enveloping algebra of a Lie algebra
$\mathfrak{L}$ is denoted by $\mathcal{U}\left(  \mathfrak{L}\right)  $.
Recall that an associative ring $R$ can be equipped with the NC-topology
defined by the commutator (or NC) filtration $\left\{  F_{m}\left(  R\right)
\right\}  $, where $F_{m}\left(  R\right)  =\sum_{i_{1}+\cdots+i_{s}%
=m}I_{i_{1}}\cdots I_{i_{s}}$ and $I_{t}=I\left(  R_{\mathfrak{lie}}^{\left(
t+1\right)  }\right)  $ is the two-sided ideal in $R$ generated by $t$th
member $R_{\mathfrak{lie}}^{\left(  t+1\right)  }$ of the lower central series
of the Lie ring $R_{\mathfrak{lie}}$. The ring $R$ is called an NC-complete
ring (see \cite{Kap}) if it is Hausdorff and complete with respect to its
NC-topology. Thus an NC-complete ring $R$ is an inverse limit of NC-nilpotent
rings $R_{m}$, that is, $F_{s}\left(  R_{m}\right)  =\left\{  0\right\}  $ for
large $s$. Optionally, it is a complete filtered ring whose topology is given
by a filter base $\left\{  J_{m}\right\}  $ of its two-sided ideals such that
every $R_{m}=R/J_{m}$ is NC-nilpotent and $J_{1}=I_{1}$ (we also accept that
$J_{0}=R$). In this case, $\left\{  J_{m}\right\}  $ is called \textit{an
NC-filtration of} $R$, which is equivalent to $\left\{  F_{m}\left(  R\right)
\right\}  $ (see \cite[Corollary 4.4.1]{DComA}). The commutative ring
$R_{c}=R/I_{1}$ ($=J_{0}/J_{1}=R_{0}$) is called \textit{the commutativization
of} $R$ and $\operatorname{gr}\left(  R\right)  =\bigoplus\limits_{m=0}%
^{\infty}J_{m}/J_{m+1}$ is an $R$-bimodule.

One of the key properties of NC-complete rings is that they admit topological
localizations, which are in turn NC-complete rings (see \cite{Kap},
\cite{DIZV}). Moreover, the formal spectrum $X=\operatorname{Spf}\left(
R\right)  $ of an NC-complete ring $R$ is reduced to $\operatorname{Spec}%
\left(  R_{c}\right)  $ up to a homeomorphism. The structure sheaf
$\mathcal{O}_{X}$ is defined to be the sheaf of all continuous sections of the
covering space over $X$ given by the noncommutative topological localizations
of $R$. The couple $\left(  X,\mathcal{O}_{X}\right)  $ is called an affine
NC-complete scheme. Actually, $\left(  X,\mathcal{O}_{X}\right)  $ is the
formal scheme $\underleftarrow{\lim}\left\{  \operatorname{Spec}\left(
R_{m}\right)  \right\}  $ to be the inverse limit of the affine NC-nilpotent
schemes $\operatorname{Spec}\left(  R_{m}\right)  $. A general NC-scheme is a
ringed space $\left(  X,\mathcal{O}_{X}\right)  $ which is locally isomorphic
to an affine NC-scheme (nilpotent or complete).

Now let $R$ be an NC-complete ring with its NC-filtration $\left\{
J_{m}\right\}  $, and let $A=R_{c}$. Suppose $R$ has an $A$-module (not
$A$-algebra) structure such that $\left\{  J_{m}\right\}  $ are $A$-submodules
of $R$ too. Thus $\operatorname{gr}\left(  R\right)  $ turns out to be an
$A$-module too.

\begin{definition}
\label{defNCCD}The ring $R$ is called an NC-deformation of $A$ if the
$R$-bimodule structure of $\operatorname{gr}\left(  R\right)  $ can be driven
from its $A$-module one as the pull back through $\varepsilon:R\rightarrow A$.
In the case of an NC-nilpotent ring, it is called an NC-nilpotent deformation
of $A$.
\end{definition}

Thus $R$-bimodule structure of every $J_{m}/J_{m+1}$ is reduced to its
$A$-module one. In particular, $R/J_{1}$ is reduced to the original ring $A$,
that is, the quotient ring map $\varepsilon:R\rightarrow A$ and $\tau
:A\rightarrow R$, $\tau\left(  a\right)  =a_{R}$ are $A$-module maps, where
$a_{R}=a1\in R$. Since $\varepsilon\tau=1$, it follows that $\varepsilon$ is a
retraction in the category of $A$-modules.

Recall that the affine NC-space $\mathbb{A}_{\mathfrak{nc}}^{n}$ (over the
complex field) is defined as the formal scheme $\operatorname{Spf}\left(
\mathcal{O}_{\mathfrak{nc}}\left(  \mathbf{x}\right)  \right)  $ of the
NC-completion $\mathcal{O}_{\mathfrak{nc}}\left(  \mathbf{x}\right)  $ of the
free associative algebra $\mathbb{C}\left\langle \mathbf{x}\right\rangle $ in
the independent variables $\mathbf{x=}\left(  x_{1},\ldots,x_{n}\right)  $
\cite{Kap}. In this case, $\mathcal{O}_{\mathfrak{nc}}\left(  \mathbf{x}%
\right)  $ turns out to be an NC-complete deformation of the polynomial
algebra $\mathbb{C}\left[  \mathbf{x}\right]  $ (see \cite[5.4]{Dproj1}).
Below we explain the details in the graded case.

\subsection{The free nilpotent Lie algebra $\mathfrak{g}_{q}\left(
\mathbf{x}\right)  $\label{subsecGFNA}}

Let $\mathfrak{L}\left(  \mathbf{x}\right)  $ be the free Lie $k$-algebra
generated by independent variables $\mathbf{x=}\left(  x_{0},\ldots
,x_{n}\right)  $, $S=k\left[  \mathbf{x}\right]  $ the commutative (graded)
algebra of polynomials in $\mathbf{x}$, $\mathfrak{g}_{q}\left(
\mathbf{x}\right)  =\mathfrak{L}\left(  \mathbf{x}\right)  /\mathfrak{L}%
\left(  \mathbf{x}\right)  ^{\left(  q+1\right)  }$ the free nilpotent Lie
algebra of index $q$, and let $S_{q}$ be the universal enveloping algebra
$\mathcal{U}\left(  \mathfrak{g}_{q}\left(  \mathbf{x}\right)  \right)  $ of
$\mathfrak{g}_{q}\left(  \mathbf{x}\right)  $. We define $\mathfrak{g}%
_{q}\left(  \mathbf{x}\right)  _{i}$ to be a subspace in $\mathfrak{g}%
_{q}\left(  \mathbf{x}\right)  $ generated by all Lie monomials in
$\mathbf{x}$ of length $i$. Then $\mathfrak{g}_{q}\left(  \mathbf{x}\right)
=\bigoplus_{i=1}^{q}\mathfrak{g}_{q}\left(  \mathbf{x}\right)  _{i}$ is a
graded Lie algebra in the sense that $\left[  \mathfrak{g}_{q}\left(
\mathbf{x}\right)  _{i},\mathfrak{g}_{q}\left(  \mathbf{x}\right)
_{j}\right]  \subseteq\mathfrak{g}_{q}\left(  \mathbf{x}\right)  _{i+j}$ for
all $i,j$, and $\mathfrak{g}_{q}\left(  \mathbf{x}\right)  ^{\left(  m\right)
}=\bigoplus_{i\geq m}\mathfrak{g}_{q}\left(  \mathbf{x}\right)  _{i}$, $1\leq
m\leq q$, $\mathfrak{g}_{q}\left(  \mathbf{x}\right)  ^{\left(  m\right)
}=\left\{  0\right\}  $, $m>q$. For each $i$ choose a Hall basis
$\mathbf{y}_{\left(  i\right)  }$ in $\mathfrak{g}_{q}\left(  \mathbf{x}%
\right)  _{i+1}$ and put $\mathbf{y=y}_{\left(  1\right)  }\sqcup\cdots
\sqcup\mathbf{y}_{\left(  q-1\right)  }$ which is a basis for $\mathfrak{g}%
_{q}\left(  \mathbf{x}\right)  ^{\left(  2\right)  }$ called \textit{the
radical variables. }Thus $\mathbf{z=x}$\textbf{$\sqcup$}$\mathbf{y=}\left(
z_{0},\ldots,z_{v}\right)  $ is a Hall $k$-basis for $\mathfrak{g}_{q}\left(
\mathbf{x}\right)  $. For an ordered monomial in $\mathbf{z}$ we use the
notation $\mathbf{z}^{\gamma}$, where $\gamma\in\mathbb{Z}_{+}^{v+1}$. Note
that $\mathbf{z}^{\gamma}=\mathbf{x}^{\beta}\mathbf{y}^{\alpha}$ for
$\gamma=\left(  \beta,\alpha\right)  $ with uniquely defined functions
$\beta\in\mathbb{Z}_{+}^{n+1}$ and $\alpha\in\mathbb{Z}_{+}^{v-n}$. By
Poincar\'{e}--Birkhoff--Witt Theorem, the set of all ordered monomials
$\mathbf{z}^{\gamma}$ is a basis for $S_{q}$. In particular, each element
$a\in S_{q}$ has a unique expansion $a=\sum_{\alpha}a_{\alpha}\left(
\mathbf{x}\right)  \mathbf{y}^{\alpha}$, where each $a_{\alpha}\in S$\textbf{
}is a usual polynomial, and $a_{\alpha}\left(  \mathbf{x}\right)  $ its
realization in $S_{q}$ through the ordered functional calculus. Namely, there
is a well defined $k$-linear mapping $\tau:S\rightarrow S_{q}$ such that
$\tau\left(  \mathbf{x}^{\beta}\right)  =\mathbf{x}^{\beta}$ for all $\beta
\in\mathbb{Z}_{+}^{n+1}$, and we put $\tau\left(  p\right)  =p\left(
\mathbf{x}\right)  $ for each $p\in S$. If $\alpha=0$ as a function (or
multi-index), we have the term $a_{0}\left(  \mathbf{x}\right)  $ in the
expansion of $a\in S_{q}$. Thus $a=a_{0}\left(  \mathbf{x}\right)
+\sum_{\left\vert \alpha\right\vert >0}a_{\alpha}\left(  \mathbf{x}\right)
\mathbf{y}^{\alpha}$ and $\varepsilon:S_{q}\rightarrow S$, $\varepsilon\left(
a\right)  =a_{0}$ is the quotient (modulo $\mathcal{J}_{1}\left(
\mathbf{x}\right)  $) homomorphism with $\varepsilon\tau=1$, where
$\mathcal{J}_{1}\left(  \mathbf{x}\right)  =I\left(  \mathfrak{g}_{q}\left(
\mathbf{x}\right)  ^{\left(  2\right)  }\right)  $. Moreover, $\tau
\varepsilon\left(  a\right)  =a_{0}\left(  \mathbf{x}\right)  $, $a\in S_{q}$,
that is, $\tau\varepsilon$ is a projection onto the subspace
$\operatorname{im}\left(  \tau\right)  $. Thus $S_{q}$ has the $S$-module
structure given by $\tau$, that is, $sa=\sum_{\alpha}\left(  sa_{\alpha
}\right)  \left(  \mathbf{x}\right)  \mathbf{y}^{\alpha}$ whenever $s\in S$
and $a\in S_{q}$.

We define \textit{the degree of the monomials }in $\mathbf{z}$ in the
following way. Put $\deg\left(  \mathbf{y}_{\left(  m\right)  }\right)  =m+1$,
that is, $\deg\left(  y_{u}\right)  =m+1$ for all $y_{u}\in\mathbf{y}_{\left(
m\right)  }$ and $m\in\mathbb{N}$. Further, we put $\deg\left(  \mathbf{x}%
\right)  =1$ and $\deg\left(  z_{\gamma_{1}}\cdots z_{\gamma_{s}}\right)
=\deg\left(  z_{\gamma_{1}}\right)  +\cdots+\deg\left(  z_{\gamma_{s}}\right)
$ for a non-ordered monomial in $\mathbf{z}$. For $\alpha\in\mathbb{Z}%
_{+}^{v+1}$ we use the notation $\left\langle \alpha\right\rangle $ instead of
$\deg\left(  \mathbf{z}^{\alpha}\right)  $ whereas $\left\vert \alpha
\right\vert $ means just the sum $\sum_{u}\alpha\left(  z_{u}\right)  $ of all
values of the function $\alpha$. Note that $\left\langle \alpha\right\rangle
=\sum_{i=0}^{n}\alpha\left(  x_{i}\right)  +\sum_{m\geq1}\left(  m+1\right)
\sum_{y_{u}\in\mathbf{y}_{\left(  m\right)  }}\alpha\left(  y_{u}\right)  $ is
a weighed sum of values of $\alpha$. Consider the subspace $R_{q}\left\langle
\mathbf{y}\right\rangle $ in $S_{q}$ generated by all powers $\mathbf{y}%
^{\alpha}$, which is a unital subalgebra in $S_{q}$ generated by
$\mathfrak{g}_{q}\left(  \mathbf{x}\right)  ^{\left(  2\right)  }$. Actually,
it admits the grading $R_{q}\left\langle \mathbf{y}\right\rangle
=\oplus_{m\geq1}R_{q}^{m}\left\langle \mathbf{y}\right\rangle $, where
$R_{q}^{m}\left\langle \mathbf{y}\right\rangle $ consists of all sums
$\sum_{\left\langle \alpha\right\rangle =m}\lambda_{\alpha}\mathbf{y}^{\alpha
}$. Thus $R_{q}^{s}\left\langle \mathbf{y}\right\rangle R_{q}^{m}\left\langle
\mathbf{y}\right\rangle \subseteq R_{q}^{s+m}\left\langle \mathbf{y}%
\right\rangle $ for all $s,m$. Moreover, $S_{q}=S\otimes R_{q}\left\langle
\mathbf{y}\right\rangle =\bigoplus\limits_{m\in\mathbb{Z}_{+}}S\otimes
R_{q}^{m}\left\langle \mathbf{y}\right\rangle $. For each $m$, we put
$\mathcal{J}_{m}\left(  \mathbf{x}\right)  =\bigoplus\limits_{s\geq m}S\otimes
R_{q}^{s}\left\langle \mathbf{y}\right\rangle $, and $\left\{  \mathcal{J}%
_{m}\left(  \mathbf{x}\right)  \right\}  $ is a filtration of the two-sided
ideals in $S_{q}$, which are in turn $S$-submodules. It is an NC-filtration of
$S_{q}$ (see \cite{DIZV}, \cite[Corollary 3.5]{Dproj1}) and $\mathcal{O}%
_{q}\left(  \mathbf{x}\right)  =\prod\limits_{m}S\otimes R_{q}^{m}\left\langle
\mathbf{y}\right\rangle $ is the NC-completion of $S_{q}$ (or the completion
with respect to $\left\{  \mathcal{J}_{m}\left(  \mathbf{x}\right)  \right\}
$). In this case, $\operatorname{gr}\left(  \mathcal{O}_{q}\left(
\mathbf{x}\right)  \right)  =\operatorname{gr}\left(  S_{q}\right)
=\bigoplus\limits_{m}S\otimes R_{q}^{m}\left\langle \mathbf{y}\right\rangle $
and its $S_{q}$-bimodule structure is reduced to its $S$-module one (see
\cite[Sections 4, 5]{Dproj1}). It follows that $\mathcal{O}_{q}\left(
\mathbf{x}\right)  $ is the NC-complete deformation of the polynomial ring $S$
(see Definition \ref{defNCCD}).

\subsection{Noncommutative projective $q$-spaces $\mathbb{P}_{k,q}^{n}%
$\label{subsecPQN}}

Everywhere below we fix our homogeneous coordinates $\mathbf{x}=\left(
x_{0},\ldots,x_{n}\right)  $ of the projective space $\mathbb{P}_{k}^{n}$ over
the field $k$. Recall that $\mathbb{P}_{k}^{n}=\operatorname{Proj}\left(
S\right)  $ for the graded algebra $S=k\left[  \mathbf{x}\right]  $ with its
natural grading $S=\bigoplus_{d}S^{d}$. We treat the family $\mathbf{x}$ as a
free generators of the algebra $S_{q}$ as well (see Subsection
\ref{subsecGFNA}). Note that $S_{q}$ is identified with $k\left[
\mathbf{z}\right]  $ as the $k$-vector spaces. Put $S_{q}^{d}=\bigoplus
_{m}S^{d-m}\otimes R_{q}^{m}\left\langle \mathbf{y}\right\rangle $ for every
$d\in\mathbb{Z}$. Note that $S_{q}^{d}=\left\{  0\right\}  $ for $d<0$, and
$S_{q}^{d}$ consists of all sum $\sum_{\left\langle \alpha\right\rangle
=d}\lambda_{\alpha}\mathbf{z}^{\alpha}$. In particular, $S_{q}^{e}\cdot
S_{q}^{d}\subseteq S_{q}^{e+d}$ for all $e,d$, that is, the algebra $S_{q}$
possesses a new grading $S_{q}=\bigoplus_{d}S_{q}^{d}$ (see \cite[Sections 4,
7]{Dproj1}). We define $\mathbb{P}_{k,q}^{n}$ to be the set
$\operatorname{Proj}^{\circ}\left(  S_{q}\right)  $ of all open (with respect
to the NC-topology) graded prime ideals of $S_{q}$ which do not contain
$\bigoplus\limits_{d>0}S_{q}^{d}$. As a topological space $\mathbb{P}%
_{k,q}^{n}$ is identified with the projective space $\mathbb{P}_{k}^{n}$ up to
a homeomorphism. The structure sheaf $\mathcal{O}_{q}$ on $\mathbb{P}%
_{k,q}^{n}$ is defined by means of the zero degree term of the
(noncommutative) topological localizations $S_{q,h\left(  \mathbf{x}\right)
}$ of $S_{q}$ at homogeneous $h\in S_{+}$. The ringed space $\left(
\mathbb{P}_{k}^{n},\mathcal{O}_{q}\right)  $ denoted by $\mathbb{P}_{k,q}^{n}$
is called a \textit{projective }$q$\textit{-space}. It is proved that
$\mathcal{O}_{q}=\prod_{m=0}^{\infty}\mathcal{O}\left(  -m\right)  \otimes
R_{q}^{m}\left\langle \mathbf{y}\right\rangle $ up to a sheaf isomorphism
\cite[Proposition 7.1]{Dproj1}, where $\mathcal{O}\left(  d\right)  $,
$d\in\mathbb{Z}$ are twisted sheaves of the structure sheaf $\mathcal{O}$ of
$\mathbb{P}_{k}^{n}$ (see Appendix Section \ref{SecApp}). Thus the projective
$q$-space $\mathbb{P}_{k,q}^{n}$ is described in terms of the twisted sheaves
of $\mathbb{P}_{k}^{n}$. There is a canonical filtration $\left\{
\mathcal{J}_{m}\right\}  $ of the sheaf $\mathcal{O}_{q}$ given by the
subsheaves $\mathcal{J}_{m}=\prod_{s=m}^{\infty}\mathcal{O}\left(  -s\right)
\otimes R_{q}^{s}\left\langle \mathbf{y}\right\rangle $ of two-sided ideals in
$\mathcal{O}_{q}$. In this case, $\mathcal{O}_{q}/\mathcal{J}_{m}$ is the
coherent sheaf $\bigoplus\limits_{s<m}\mathcal{O}\left(  -s\right)  \otimes
R_{q}^{s}\left\langle \mathbf{y}\right\rangle $ of NC-nilpotent algebras on
$\mathbb{P}_{k}^{n}$, and $\mathcal{O}_{q}$ being a filtered sheaf turns out
to be an inverse limit of the coherent $\mathcal{O}$-modules of NC-nilpotent
$k$-algebras. One can also use the formal functional calculus for sheaves on
$\mathbb{P}_{k}^{n}$ to describe $\mathcal{O}_{q}$ based on the operations
$\otimes_{\mathcal{O}}$, $\oplus$ and $\prod$ \cite[Theorem 7.1]{Dproj1}. If
$\mathbf{t}=\left(  t_{1},,\ldots t_{q-1}\right)  $ are independent
(commuting) variables that correspond to the tuple $\left(  \mathcal{O}\left(
-2\right)  ,\ldots,\mathcal{O}\left(  -q\right)  \right)  \mathcal{\ }$of the
invertible sheaves on $\mathbb{P}_{k}^{n}$, and $f\left(  \mathbf{t}\right)
=\prod\limits_{i=1}^{q-1}\dfrac{1}{n_{i}!}\dfrac{d^{n_{i}}}{dt_{i}^{n_{i}}%
}\left(  1-t_{i}\right)  ^{-1}$ is the formal power series from $\mathbb{Z}%
_{+}\left[  \left[  \mathbf{t}\right]  \right]  $, then $\mathcal{O}%
_{q}=f\left(  \mathcal{O}\left(  -2\right)  ,\ldots,\mathcal{O}\left(
-q\right)  \right)  $, where $\mathcal{O}$ stands for the unit, $n_{i}%
+1=\operatorname{Card}\left(  \mathbf{y}_{\left(  i\right)  }\right)  $,
$1\leq i\leq q-1$.

\section{NC-deformations and their morphisms\label{SecNCdef}}

In this section we introduce NC-deformations of the commutative schemes and
their morphisms to make clear our target on the deformation quantization of
the projective schemes.

\subsection{NC-deformations of schemes\label{subsecNCdefsch}}

Let $X$ be a (commutative) scheme with its structure sheaf $\mathcal{O}_{X}$,
and let $\mathcal{F}_{X}$ be a sheaf of $\mathcal{O}_{X}$-modules on $X$,
which in turn is a sheaf of noncommutative rings whose stalks $\mathcal{F}%
_{X,x}$ are local rings with their maximal two-sided ideals $\mathfrak{m}%
_{X,x}$, $x\in X$. The two-sided ideal sheaf in $\mathcal{F}_{X}$ generated by
the commutator $\left[  \mathcal{F}_{X},\mathcal{F}_{X}\right]  $ is denoted
by $\mathcal{C}_{X}$. It is a sheaf associated to the presheaf $U\mapsto
I\left(  \left[  \mathcal{F}_{X}\left(  U\right)  ,\mathcal{F}_{X}\left(
U\right)  \right]  \right)  $ on $X$.

We say that $\mathcal{F}_{X}$ is \textit{a sheaf of NC-complete rings }if
there is a filter base $\left\{  \mathcal{J}_{X,m}\right\}  $ of two-sided
ideal subsheaves of $\mathcal{O}_{X}$-modules in $\mathcal{F}_{X}$ such that
all $\mathcal{F}_{X}/\mathcal{J}_{X,m}$ are quasi-coherent $\mathcal{O}_{X}%
$-modules of NC-nilpotent rings, $\mathcal{C}_{X}=\mathcal{J}_{X,1}$,
$\mathcal{F}_{X}/\mathcal{J}_{X,1}=\mathcal{O}_{X}$ and $\mathcal{F}%
_{X}=\underleftarrow{\lim}\left\{  \mathcal{F}_{X}/\mathcal{J}_{X,m}\right\}
$. We put $\mathcal{F}_{X,m}=\mathcal{F}_{X}/\mathcal{J}_{X,m}$ to be the
quasi-coherent $\mathcal{O}_{X}$-modules of NC-nilpotent rings, which can be
included into exact sequences%
\begin{equation}
0\rightarrow\mathcal{J}_{X,m}/\mathcal{J}_{X,m+1}\rightarrow\mathcal{F}%
_{X,m+1}\rightarrow\mathcal{F}_{X,m}\rightarrow0 \label{sqc}%
\end{equation}
of $\mathcal{O}_{X}$-modules. In particular, every quotient $\mathcal{J}%
_{X,m}/\mathcal{J}_{X,m+1}$ is a quasi-coherent $\mathcal{O}_{X}$-module being
the kernel of a morphism of quasi-coherent $\mathcal{O}_{X}$-modules
\cite[2.5.7]{Harts}. For $m=0$ we obtain the exact sequence $0\rightarrow
\mathcal{J}_{X,1}\rightarrow\mathcal{F}_{X}\overset{\varepsilon_{X}%
}{\rightarrow}\mathcal{O}_{X}\rightarrow0$ of $\mathcal{O}_{X}$-modules such
that $\varepsilon_{X,x}^{-1}\left(  \mathfrak{n}_{X,x}\right)  =\mathfrak{m}%
_{X,x}$ for every point $x\in X$, where $\mathfrak{n}_{X,x}$ is the maximal
ideal of the commutative local ring $\mathcal{O}_{X,x}$. Thus $\varepsilon
_{X,x}$ is a local homomorphism that allows us to identity $\mathcal{F}%
_{X,x}/\mathcal{C}_{X,x}=\mathcal{O}_{X,x}$ the related local rings. Put
$\operatorname{gr}\left(  \mathcal{F}_{X}\right)  =\bigoplus\limits_{m}%
\mathcal{J}_{X,m}/\mathcal{J}_{X,m+1}$ to be an $\mathcal{F}_{X}$-bimodule,
which is a sum of the quasi-coherent $\mathcal{O}_{X}$-modules. If the filter
base $\left\{  \mathcal{J}_{X,m}\right\}  $ is vanishing (or trivial), then
$\mathcal{F}_{X}$ turns out to be a quasi-coherent $\mathcal{O}_{X}$-module of
NC-nilpotent rings, and $\operatorname{gr}\left(  \mathcal{F}_{X}\right)  $ is
a quasi-coherent $\mathcal{O}_{X}$-module.

\begin{definition}
\label{defNCS}A couple $\left(  X,\mathcal{F}_{X}\right)  $ is called
\textit{an NC-complete deformation of }$\left(  X,\mathcal{O}_{X}\right)
$\textit{ }if all $\left(  X,\mathcal{F}_{X,m}\right)  $ are NC-nilpotent
schemes (see Subsection \ref{subsecLieNCfil}) such that $\mathcal{F}_{X}%
$-bimodule structure of $\operatorname{gr}\left(  \mathcal{F}_{X}\right)  $
can be driven from its $\mathcal{O}_{X}$-module one as the pull back through
the morphism $\varepsilon_{X}:\mathcal{F}_{X}\rightarrow\mathcal{O}_{X}$.
\end{definition}

In the case of a trivial filtration, $\left(  X,\mathcal{F}_{X}\right)  $ is
called an \textit{NC-nilpotent deformation of }$X$. To cover up all cases, we
briefly say that $\left(  X,\mathcal{F}_{X}\right)  $ is \textit{an
NC-deformation of }$\left(  X,\mathcal{O}_{X}\right)  $. A commutative scheme
$\left(  X,\mathcal{O}_{X}\right)  $ itself is an NC-complete scheme with
$\mathcal{F}_{X}=\mathcal{O}_{X}$. Since $\mathcal{F}_{X,m}$ is a sheaf of
NC-nilpotent rings whose stalks are noncommutative local rings, we obtain that
$\left(  X,\mathcal{F}_{X,m}\right)  $ is an NC-nilpotent deformation of
$\left(  X,\mathcal{O}_{X}\right)  $ by Definition \ref{defNCS}, and $\left(
X,\mathcal{F}_{X}\right)  $ is an inverse limit of the NC-nilpotent
deformations $\left(  X,\mathcal{F}_{X,m}\right)  $.

\begin{theorem}
\label{thNCDaff}An NC-deformation of an affine scheme is an affine NC-scheme.
\end{theorem}

\begin{proof}
First assume that $\left(  X,\mathcal{F}_{X}\right)  $ is an NC-nilpotent
deformation of $X=\operatorname{Spec}\left(  A\right)  $, and let
$R=\Gamma\left(  X,\mathcal{F}_{X}\right)  $ be the ring of all global
sections of the quasi-coherent $\mathcal{O}_{X}$-module $\mathcal{F}_{X}$.
Then $R$ is an $A$-module, $\mathcal{F}_{X}=\widetilde{R}$ (see \cite[2.5.5]%
{Harts}), and $\mathcal{C}_{X}$ being the kernel of the morphism
$\varepsilon_{X}$ of the quasi-coherent $\mathcal{O}_{X}$-modules turns out to
be a quasi-coherent $\mathcal{O}_{X}$-module, that is, $\mathcal{C}%
_{X}=\widetilde{J_{1}}$ for a certain $A$-submodule $J_{1}\subseteq R$. Since
$J_{1}=\Gamma\left(  X,\mathcal{C}_{X}\right)  $, it follows that $J_{1}$ is a
two-sided ideal of $R$ and $R/J_{1}=A$. In particular, $I_{1}=I\left(  \left[
R,R\right]  \right)  \subseteq J_{1}$. Since $\mathcal{J}_{X,1}=\mathcal{C}%
_{X}$ and all $\mathcal{J}_{X,m}/\mathcal{J}_{X,m+1}$ are quasi-coherent
$\mathcal{O}_{X}$-modules, we derive by induction on $m$ that every
$\mathcal{J}_{X,m}$ is a quasi-coherent $\mathcal{O}_{X}$-module being the
kernel of the $\mathcal{O}_{X}$-morphism $\mathcal{J}_{X,m-1}\rightarrow
\mathcal{J}_{X,m-1}/\mathcal{J}_{X,m}$ (see \cite[2.5.7]{Harts}). Thus
$\mathcal{J}_{X,m}=\widetilde{J_{m}}$ for a certain $A$-submodule
$J_{m}\subseteq R$, which is a two-sided ideal of $R$. It turns out that
$\operatorname{gr}\left(  \mathcal{F}_{X}\right)  =\bigoplus\limits_{m}\left(
J_{m}/J_{m+1}\right)  ^{\sim}$ and its $\mathcal{F}_{X}$-bimodule structure is
reduced to its $\mathcal{O}_{X}$-module one, that is, the $R$-bimodule
structure of $\operatorname{gr}\left(  R\right)  =\bigoplus\limits_{m}%
J_{m}/J_{m+1}$ is just its $A$-module one compatible with the quotient map
$R\rightarrow A$.

We claim that $R$ is an NC-deformation of $A$ in the sense of Definition
\ref{defNCCD}. We have to prove that $I_{1}=J_{1}$. Pick $a\in A$, which is
not nilpotent. It defines the element $a_{R}=a1$ of $R$ whose class modulo
$J_{1}$ is $a$. Therefore $a_{R}\notin\mathfrak{Nil}\left(  R\right)  $, where
$\mathfrak{Nil}\left(  R\right)  $ is the nilradical of the NC-nilpotent ring
$R$ to be the set of all nilpotent elements in $R$ (see \cite[Proposition
3.2.2]{DComA}). The element $a_{R}$ defines the left (resp., right)
multiplication operator $L_{a_{R}}$ (resp., $R_{a_{R}}$) on $R$ and the
$A$-module action of $a$ on $R$ is denoted by $a$ too. Put $\Delta
_{a}=L_{a_{R}}-a$ and $\nabla_{a}=R_{a_{R}}-a$ to be $\mathbb{Z}%
$-endomorphisms on $R$. Note that $\mathcal{F}_{X}\left(  D\left(  a\right)
\right)  =R_{a}$ is the (commutative) $A_{a}$-module localization of the
$A$-module $R$, where $D\left(  a\right)  \subseteq X$ is the principal open
subset related to $a$. Since the $A$-module map $R\rightarrow R_{a}$
represents the restriction map of the sheaf $\mathcal{F}_{X}$, it is a ring
homomorphism of NC-nilpotent rings. In particular, $a$ has the similar actions
on $R_{a}$, and $\mathcal{J}_{X,m}\left(  D\left(  a\right)  \right)
=J_{m,a}$ defines a (trivial) filtration of two-sided ideals of $A_{a}%
$-submodules. By Definition \ref{defNCS}, both operators $\Delta_{a}$ and
$\nabla_{a}$ are vanishing on $\left(  J_{m}/J_{m+1}\right)  _{a}$ for every
$m$. But $a$ has an invertible action on $R_{a}$ with the inverse operator
$\dfrac{1}{a}$, and $\left(  \dfrac{1}{a}\Delta_{a}\right)  |\left(
J_{m}/J_{m+1}\right)  _{a}=0$ for every $m$. It follows that $\dfrac{1}%
{a}\Delta_{a}$ has a nilpotent action on $R$, say, $\left(  \dfrac{1}{a}%
\Delta_{a}\right)  ^{N+1}=0$. Put $T=\sum_{s=0}^{N}\left(  -1\right)
^{s}\left(  \dfrac{1}{a}\Delta_{a}\right)  ^{s}\dfrac{1}{a}$ to be a
$\mathbb{Z}$-endomorphism on $R_{a}$. Taking into account that $L_{a_{R}%
}=a+\Delta_{a}$, we deduce that $TL_{a_{R}}=L_{a_{R}}T=1$, that is, $a_{R}$ is
invertible in $R_{a}$ with $a_{R}^{-1}=T\left(  1\right)  $. Using the
universal projective property of the noncommutative localizations (see
\cite[Proposition 2.2.1]{DComA}), we obtain a unique ring homomorphism
$R\left[  a_{R}^{-1}\right]  \rightarrow R_{a}$ that makes the diagram
\[%
\begin{array}
[c]{ccc}
& R & \\
\swarrow &  & \searrow\\
R\left[  a_{R}^{-1}\right]  & \rightarrow & R_{a}%
\end{array}
\]
with the canonical homomorphisms commutative. Further, the filtration
$\left\{  J_{m}\right\}  $ of $R$ defines the two-sided ideal filtration
$\left\{  J_{m}\left(  a_{R}\right)  \right\}  $ of $R\left[  a_{R}%
^{-1}\right]  $ such that $R\left[  a_{R}^{-1}\right]  /J_{m}\left(
a_{R}\right)  =\left(  R/J_{m}\right)  \left[  a_{R}^{-1}\right]  $ (see
\cite[Lemma 4.3.1]{DComA}, recall that every $J_{m}$ is open in $R$), where
$J_{m}\left(  a_{R}\right)  =\left\{  a_{R}^{-s}z:z\in J_{m},s\geq0\right\}  $
and we identify $a_{R}$ with its class in every $R/J_{m}$. Since every
$R/J_{m}$ is an NC-nilpotent ring, the same argument allows us to conclude
that $\left(  R/J_{m}\right)  \left[  a_{R}^{-1}\right]  $ is the quotient
ring of $\left(  R/J_{m+1}\right)  \left[  a_{R}^{-1}\right]  $ modulo
$\left(  J_{m}/J_{m+1}\right)  \left(  a_{R}\right)  $. In particular,
$\left(  J_{m}/J_{m+1}\right)  \left(  a_{R}\right)  =J_{m}\left(
a_{R}\right)  /J_{m+1}\left(  a_{R}\right)  $ up to the canonical
identification. Moreover, as above there is a ring homomorphism $\left(
R/J_{m+1}\right)  \left[  a_{R}^{-1}\right]  \rightarrow\left(  R/J_{m+1}%
\right)  _{a}$ obtained from the NC-nilpotent scheme $\left(  X,\mathcal{F}%
_{X,m+1}\right)  $. By assumption, the $R$-bimodule structure of
$J_{m}/J_{m+1}$ is reduced to its $A$-module structure. In particular,
$L_{a_{R}}=R_{a_{R}}=a$ over $J_{m}/J_{m+1}$. Then $\left(  J_{m}%
/J_{m+1}\right)  \left(  a_{R}\right)  =\left(  J_{m}/J_{m+1}\right)  _{a}$
through the ring homomorphism $\left(  R/J_{m+1}\right)  \left[  a_{R}%
^{-1}\right]  \rightarrow\left(  R/J_{m+1}\right)  _{a}$. Indeed, if
$a_{R}^{-s}z^{\sim}\in\left(  J_{m}/J_{m+1}\right)  \left(  a_{R}\right)  $ is
vanishing in $\left(  J_{m}/J_{m+1}\right)  _{a}$ then $a^{r}z^{\sim}=0$ in
$J_{m}/J_{m+1}$ for some $r\geq1$. But $a_{R}^{r}z^{\sim}=a^{r}z^{\sim}$ in
$J_{m}/J_{m+1}$, therefore $a_{R}^{-s}z^{\sim}=0$. If $z^{\sim}/a^{s}%
\in\left(  J_{m}/J_{m+1}\right)  _{a}$ then $a_{R}^{s}\left(  z^{\sim}%
/a^{s}\right)  =a^{s}\left(  z^{\sim}/a^{s}\right)  =z^{\sim}/1$ in $\left(
J_{m}/J_{m+1}\right)  _{a}$ (the $R$-module action is reduced to the
$A$-module one). But $a_{R}$ has an invertible action on $\left(
R/J_{m+1}\right)  _{a}$ as we proved above. Therefore $z^{\sim}/a^{s}%
=a_{R}^{-s}z^{\sim}$ and it belongs to $\left(  J_{m}/J_{m+1}\right)  \left(
a_{R}\right)  $.

The canonical ring homomorphism $R\left[  a_{R}^{-1}\right]  \rightarrow
R_{a}$ induces the ring homomorphisms $R\left[  a_{R}^{-1}\right]
/J_{m}\left(  a_{R}\right)  $ $\rightarrow\left(  R/J_{m}\right)  _{a}$ such
that the diagram%
\[%
\begin{array}
[c]{ccccccc}%
0\rightarrow & \left(  J_{m}/J_{m+1}\right)  \left(  a_{R}\right)  &
\longrightarrow & \left(  R/J_{m+1}\right)  \left[  a_{R}^{-1}\right]  &
\longrightarrow & \left(  R/J_{m}\right)  \left[  a_{R}^{-1}\right]  &
\rightarrow0\\
& \updownarrow &  & \downarrow &  & \downarrow & \\
0\rightarrow & \left(  J_{m}/J_{m+1}\right)  _{a} & \longrightarrow & \left(
R/J_{m+1}\right)  _{a} & \longrightarrow & \left(  R/J_{m}\right)  _{a} &
\rightarrow0
\end{array}
\]
commutes. For $m=1$ we have $\left(  R/J_{1}\right)  \left[  a_{R}%
^{-1}\right]  =A\left[  a_{R}^{-1}\right]  =A_{a}=\left(  R/J_{1}\right)
_{a}$ and $\left(  J_{1}/J_{2}\right)  \left(  a_{R}\right)  =\left(
J_{1}/J_{2}\right)  _{a}$, therefore $\left(  R/J_{2}\right)  \left[
a_{R}^{-1}\right]  \rightarrow\left(  R/J_{2}\right)  _{a}$ is an isomorphism.
By induction hypothesis the right vertical arrow is an isomorphism, therefore
so is the middle one. By induction on $m$, we conclude that all middle
vertical arrows are ring isomorphisms. In particular, for large $m$ we obtain
that $R\left[  a_{R}^{-1}\right]  \rightarrow R_{a}$ is an isomorphism.

Since $\left(  X,\mathcal{F}_{X}\right)  $ is an NC-nilpotent scheme, for
every $x\in X$ there is a small neighborhood $D\left(  a\right)  $ such that
$\left(  D\left(  a\right)  ,R_{a}\right)  $ is the affine NC-nilpotent scheme
$\operatorname{Spec}\left(  R_{a}\right)  $. In particular, $R_{a}/I\left(
\left[  R_{a},R_{a}\right]  \right)  =A_{a}$ and $\operatorname{Spec}\left(
R_{a}\right)  =\operatorname{Spec}\left(  A_{a}\right)  $ as the topological
spaces. As we proved above $R_{a}=R\left[  a_{R}^{-1}\right]  $, $I\left(
\left[  R_{a},R_{a}\right]  \right)  =I_{1}\left(  a_{R}\right)  $
\cite[Corollary 4.3.2]{DComA} and $I_{1}\left(  a_{R}\right)  =\left\{
z/a^{s}:z\in I_{1}\right\}  =I_{1,a}\subseteq R_{a}$. Moreover, $I_{1}\left(
a_{R}\right)  \subseteq\mathfrak{Nil}\left(  R_{a}\right)  $, the nilradical
$\mathfrak{Nil}\left(  R_{a}\right)  $ is a two-sided ideal of $R_{a}$, and
$R_{a}$ is commutative modulo $\mathfrak{Nil}\left(  R_{a}\right)  $
\cite[Lemma 3.4.1]{DComA}. If $\xi\in J_{1}$ then its stalk $\xi_{x}$ at $x$
belongs to $\mathcal{C}_{X,x}$, which means that $\xi|D\left(  a\right)  \in
I_{1}\left(  a_{R}\right)  $ for a certain $D\left(  a\right)  \subseteq X$
containing $x$. Then $a^{n}\xi\in I_{1}$ for a certain $n$. Since $X$ is
quasi-compact, we deduce that $\left\{  a_{i}^{n}\xi\right\}  \subseteq I_{1}$
for a finite covering $X=\cup_{i=1}^{m}D\left(  a_{i}\right)  $. Taking into
account that $1=\sum_{i=1}^{m}b_{i}a_{i}^{n}$ in $A$ for some $\left\{
b_{i}\right\}  \subseteq A$, we derive that $\xi=\sum_{i=1}^{m}b_{i}a_{i}%
^{n}\xi\in I_{1}$. Thus $I_{1}=J_{1}$ and $A=R/J_{1}=R_{c}$.

Notice also that all noncommutative localizations of $R$ are reduced to the
$A$-module localizations. Namely, pick $\zeta\in R$, which is not nilpotent
and put $a\in A$ to be the class of $\zeta$ modulo $I_{1}$, that is,
$\zeta-a_{R}\in I_{1}$. But $I_{1}\subseteq\mathfrak{Nil}\left(  R\right)  $,
for $R$ is commutative modulo $\mathfrak{Nil}\left(  R\right)  $ (see
\cite[Lemma 3.4.1]{DComA}). Moreover, both noncommutative localizations
$R\left[  \zeta^{-1}\right]  $ and $R_{a}$ ($=R\left[  a_{R}^{-1}\right]  $)
possess the same properties \cite[Lemma 3.4.1]{DComA}. In particular, both
$\zeta$ and $a$ are invertible in $R\left[  \zeta^{-1}\right]  $ and $R_{a}$.
Taking into account the universal projective property of the noncommutative
localizations (see \cite[Proposition 2.2.1]{DComA}), we obtain the unique ring
homomorphisms $R\left[  \zeta^{-1}\right]  \rightleftarrows R_{a}$ that make
the diagram
\[%
\begin{array}
[c]{ccc}
& R & \\
\swarrow &  & \searrow\\
R\left[  \zeta^{-1}\right]  & \rightleftarrows & R_{a}%
\end{array}
\]
commutative. It follows that $R\left[  \zeta^{-1}\right]  =R_{a}%
=\mathcal{F}_{X}\left(  D_{a}\right)  $. Hence $X=\operatorname{Spec}\left(
R\right)  $ is the affine NC-nilpotent scheme.

In the general case of an NC-complete deformation of the affine scheme $X$ we
obtain that $R_{m}=\Gamma\left(  X,\mathcal{F}_{X,m}\right)  $ is a family of
NC-nilpotent ring, which are $A$-modules. Using the exact sequence
(\ref{sqc}), we deduce that
\[
0\rightarrow\Gamma\left(  X,\mathcal{J}_{X,m}/\mathcal{J}_{X,m+1}\right)
\rightarrow R_{m+1}\rightarrow R_{m}\rightarrow0
\]
remains exact \cite[2.5.7, 2.5.6]{Harts}. Thus the connecting $A$%
-homomorphisms $R_{m+1}\rightarrow R_{m}$ are surjective, and
$R=\underleftarrow{\lim}\left\{  R_{m}\right\}  $ turns out to be an
NC-complete ring with the surjective canonical homomorphisms $R\rightarrow
R_{m}$, that is, $R_{m}=R/J_{m}$ for some two sided (open) ideals $J_{m}$ of
$R$. But $R_{1}=\Gamma\left(  X,\mathcal{F}_{X,1}\right)  =\Gamma\left(
X,\mathcal{O}_{X}\right)  =A$ and $\left\{  J_{m}\right\}  $ defines the
NC-topology of $R$, therefore $I_{1}\subseteq J_{1}$. Moreover, $I_{1}$ is
open (as in every NC-complete ring), therefore it is closed. As we proved
above, $J_{1}/J_{m}=I_{m,1}$ in $R_{m}=R/J_{m}$, and $I_{1}$ is mapped onto
$I_{m,1}$, where $I_{m,1}=I\left(  \left[  R_{m},R_{m}\right]  \right)  $.
Pick $x\in J_{1}$. Then $x-x_{m}\in J_{m}$ for some $x_{m}\in I_{1}$. It
follows that $x=\lim_{m}x_{m}$ for the sequence $\left\{  x_{m}\right\}
\subseteq I_{1}$, and $x\in I_{1}$. Thus $J_{1}\subseteq\cap_{m}\left(
I_{1}+J_{m}\right)  =\overline{I}=I$, and $R/I_{1}=A$.

Finally, as we have proved above $\left(  X,\mathcal{F}_{X,m}\right)
=\operatorname{Spec}\left(  R_{m}\right)  $ is an affine NC-nilpotent scheme
of $R_{m}$, and $R\left[  \zeta^{-1}\right]  =\underleftarrow{\lim}\left\{
R_{m}\left[  \zeta^{-1}\right]  \right\}  =\underleftarrow{\lim}\left\{
\left(  R_{m}\right)  _{a}\right\}  =R_{a}$ as above. Hence $\left(
X,\mathcal{F}_{X}\right)  $ stands for the NC-complete scheme
$\operatorname{Spf}\left(  R\right)  $.
\end{proof}

The noncommutative projective $q$-space $\mathbb{P}_{k,q}^{n}$ equipped with
the sheaf $\mathcal{O}_{q}$ (see Subsection \ref{subsecPQN}) and its
filtration $\left\{  \mathcal{J}_{m}\right\}  $ of two-sided ideal subsheaves
is an NC-complete deformation of the (nonaffine) scheme $\mathbb{P}_{k}^{n}$.
In this case $\mathcal{O}_{q}/\mathcal{J}_{1}$ is reduced to the structure
sheaf $\mathcal{O}$ of the projective space $\mathbb{P}_{k}^{n}$ by means of
the canonical morphism $\varepsilon:\mathcal{O}_{q}\rightarrow\mathcal{O}$
(see Subsection \ref{subsecGFNA}).

\subsection{The morphisms of NC-deformations\label{subsecMNC}}

Let $\left(  X,\mathcal{F}_{X}\right)  $ and $\left(  Y,\mathcal{F}%
_{Y}\right)  $ be NC-deformations of the commutative schemes $\left(
X,\mathcal{O}_{X}\right)  $ and $\left(  Y,\mathcal{O}_{Y}\right)  $,
respectively. By a morphism $\left(  X,\mathcal{F}_{X}\right)  \rightarrow
\left(  Y,\mathcal{F}_{Y}\right)  $ of NC-deformations we mean a couple
$\left(  f,f^{+}\right)  $ of a continuous mapping $f:X\rightarrow Y$ and a
filtered sheaf morphism $f^{+}:\mathcal{F}_{Y}\rightarrow f_{\ast}%
\mathcal{F}_{X}$ of the filtered $\mathcal{O}_{Y}$-modules such that the
canonical mapping $f_{x}^{+}:\mathcal{F}_{Y,f\left(  x\right)  }%
\rightarrow\mathcal{F}_{X,x}$, $x\in X$ over stalks is a local homomorphism of
noncommutative rings, that is, $\left(  f_{x}^{+}\right)  ^{-1}\left(
\mathfrak{m}_{X,x}\right)  =\mathfrak{m}_{Y,f\left(  x\right)  }$. Recall that
the property to be a filtered morphism (see Section \ref{SecPrel}) means that
$f^{+}\left(  \mathcal{J}_{Y,m}\right)  \subseteq f_{\ast}\mathcal{J}_{X,m}$
for all $m$. In particular, $f^{+}\left(  V\right)  :\mathcal{F}_{Y}\left(
V\right)  \rightarrow\mathcal{F}_{X}\left(  f^{-1}\left(  V\right)  \right)  $
is a continuous homomorphism of NC-complete rings for every open subset
$V\subseteq Y$. The corresponding kernel $\ker\left(  f^{+}\right)  $ turns
out to be an NC-complete, two-sided ideal subsheaf of $\mathcal{F}_{Y}$.

Actually, the NC-morphism $\left(  f,f^{+}\right)  $ defines the morphism
$\left(  f,f^{\times}\right)  :\left(  X,\mathcal{O}_{X}\right)
\rightarrow\left(  Y,\mathcal{O}_{Y}\right)  $ of the supported commutative
schemes. Namely, $f_{\ast}\mathcal{O}_{X}$ is a sheaf of commutative rings and
$\left(  f_{\ast}\varepsilon_{X}\right)  f^{+}:\mathcal{F}_{Y}\rightarrow
f_{\ast}\mathcal{O}_{X}$ is a sheaf morphism, therefore $\mathcal{J}%
_{Y,1}=\mathcal{C}_{Y}\subseteq\ker\left(  \left(  f_{\ast}\varepsilon
_{X}\right)  f^{+}\right)  $. Thus there is a unique sheaf morphism
$f^{\times}:\mathcal{O}_{Y}\rightarrow f_{\ast}\mathcal{O}_{X}$ such that the
following diagram
\[%
\begin{array}
[c]{ccc}%
\mathcal{F}_{Y} & \overset{f^{+}}{\longrightarrow} & f_{\ast}\mathcal{F}_{X}\\
\downarrow_{\varepsilon_{Y}} &  & \downarrow_{f_{\ast}\varepsilon_{X}}\\
\mathcal{O}_{Y} & \overset{f^{\times}}{\longrightarrow} & f_{\ast}%
\mathcal{O}_{X}%
\end{array}
\]
commutes. The latter diagram in turn generates the following commutative
diagram%
\[%
\begin{array}
[c]{ccc}%
\mathcal{F}_{Y,f\left(  x\right)  } & \overset{f_{x}^{+}}{\longrightarrow} &
\mathcal{F}_{X,x}\\
\downarrow_{\varepsilon_{Y,f\left(  x\right)  }} &  & \downarrow
_{\varepsilon_{X,x}}\\
\mathcal{O}_{Y,f\left(  x\right)  } & \overset{f_{x}^{\times}}{\longrightarrow
} & \mathcal{O}_{X,x}\\
\downarrow &  & \downarrow\\
0 &  & 0
\end{array}
\]
for every point $x\in X$, whose vertical homomorphisms are surjective by
assumption. But the upper row is a local homomorphism of noncommutative local
rings. It follows that%
\begin{align*}
\left(  f_{x}^{\times}\right)  ^{-1}\left(  \mathfrak{n}_{X,x}\right)   &
=\varepsilon_{Y,f\left(  x\right)  }\left(  \varepsilon_{Y,f\left(  x\right)
}\right)  ^{-1}\left(  \left(  f_{x}^{\times}\right)  ^{-1}\left(
\mathfrak{n}_{X,x}\right)  \right)  =\varepsilon_{Y,f\left(  x\right)
}\left(  f_{x}^{\times}\varepsilon_{Y,f\left(  x\right)  }\right)
^{-1}\left(  \mathfrak{n}_{X,x}\right) \\
&  =\varepsilon_{Y,f\left(  x\right)  }\left(  \varepsilon_{X,x}f_{x}%
^{+}\right)  ^{-1}\left(  \mathfrak{n}_{X,x}\right)  =\varepsilon_{Y,f\left(
x\right)  }\left(  f_{x}^{+}\right)  ^{-1}\left(  \varepsilon_{X,x}\right)
^{-1}\left(  \mathfrak{n}_{X,x}\right) \\
&  =\varepsilon_{Y,f\left(  x\right)  }\left(  f_{x}^{+}\right)  ^{-1}\left(
\mathfrak{m}_{X,x}\right)  =\varepsilon_{Y,f\left(  x\right)  }\left(
\mathfrak{m}_{Y,f\left(  x\right)  }\right)  =\mathfrak{n}_{Y,f\left(
x\right)  }%
\end{align*}
for all $x\in X$. Hence $f_{x}^{\times}$ is a local homomorphism of the
commutative local rings for all $x\in X$. In particular, $\left(  f,f^{\times
}\right)  :\left(  X,\mathcal{O}_{X}\right)  \rightarrow\left(  Y,\mathcal{O}%
_{Y}\right)  $ is a scheme morphism of the commutative schemes. Thus a
morphism $\left(  X,\mathcal{F}_{X}\right)  \rightarrow\left(  Y,\mathcal{F}%
_{Y}\right)  $ of NC-deformations is automatically defines the morphism of the
related commutative schemes. In particular, the presence of a morphism
$\left(  X,\mathcal{F}_{X}\right)  \rightarrow\operatorname{Spec}\left(
k\right)  $ is equivalent to say that $X$ is a scheme over $k$ and
$\mathcal{F}_{X}$ is a sheaf of noncommutative $k$-algebras. In this case
$\left(  X,\mathcal{F}_{X}\right)  $ is an NC-scheme over $k$. The projective
$q$-space $\mathbb{P}_{k,q}^{n}$ is an example of an NC-complete scheme over
$k$ with its canonical projection $\mathbb{P}_{k,q}^{n}\rightarrow
\operatorname{Spec}\left(  k\right)  $.

\subsection{The closed NC-immersions of NC-deformations\label{subsecCIPS}}

A morphism $\left(  f,f^{+}\right)  :\left(  X,\mathcal{F}_{X}\right)
\rightarrow\left(  Y,\mathcal{F}_{Y}\right)  $ of NC-deformations is called
\textit{a closed NC-immersion }if $f$ is a homeomorphism onto its range and
$f^{+}:\mathcal{F}_{Y}\rightarrow f_{\ast}\mathcal{F}_{X}$ is a filtered
surjective morphism (see Section \ref{SecPrel}). In this case we say that
$\left(  X,\mathcal{F}_{X}\right)  $ (actually the equivalence class of
morphisms) is \textit{a closed NC-subscheme of the noncommutative scheme}
$\left(  Y,\mathcal{F}_{Y}\right)  $. First note that $f_{\ast}\varepsilon
_{X}:f_{\ast}\mathcal{F}_{X}\rightarrow f_{\ast}\mathcal{O}_{X}$ is
surjective, for $f$ is a homeomorphism onto its range. Thus we have the
following commutative diagram%
\[%
\begin{array}
[c]{cccc}%
\mathcal{F}_{Y} & \overset{f^{+}}{\longrightarrow} & f_{\ast}\mathcal{F}_{X} &
\rightarrow0\\
\downarrow_{\varepsilon_{Y}} &  & \downarrow_{f_{\ast}\varepsilon_{X}} & \\
\mathcal{O}_{Y} & \overset{f^{\times}}{\longrightarrow} & f_{\ast}%
\mathcal{O}_{X} & \\
\downarrow &  & \downarrow & \\
0 &  & 0 &
\end{array}
\]
with exact columns and upper row. Therefore so is its lower row, that is,
$\left(  f,f^{\times}\right)  :\left(  X,\mathcal{O}_{X}\right)
\rightarrow\left(  Y,\mathcal{O}_{Y}\right)  $ is a closed immersion of the
commutative schemes. Note also the corresponding two-sided ideal sheaf
$\ker\left(  f^{+}\right)  $ is denoted by $\mathcal{I}_{X}$, which is an
NC-complete subsheaf of $\mathcal{F}_{Y}$.

Now let $\left(  X,\mathcal{F}_{X}\right)  $ be an NC-complete scheme. By its
very definition, the commutative scheme $\left(  X,\mathcal{O}_{X}\right)  $
turns out to be a closed NC-subscheme of $\left(  X,\mathcal{F}_{X}\right)  $.
As above the NC-nilpotent schemes $\left(  X,\mathcal{F}_{X,m}\right)  $ (with
their quasi-coherent sheaves $\mathcal{F}_{X,m}$) are closed NC-subschemes of
the original NC-complete scheme $\left(  X,\mathcal{F}_{X}\right)  $ too.

\begin{proposition}
\label{propCI}Every morphism\ (resp., closed NC-immersion) of NC-complete
deformations is an inverse limit of morphisms (resp., closed NC-immersions) of
NC-nilpotent ones. If $\mathcal{I}_{X}$ is the two-sided ideal sheaf of a
closed NC-immersion $\left(  f,f^{+}\right)  :\left(  X,\mathcal{F}%
_{X}\right)  \rightarrow\left(  Y,\mathcal{F}_{Y}\right)  $ of NC-complete
schemes then $\mathcal{I}_{X}=\underleftarrow{\lim}\left\{  \mathcal{I}%
_{X,m}\right\}  $ is an inverse limit of the two-sided ideal sheaves of the
closed NC-immersions of NC-nilpotent schemes. In particular, $\mathcal{I}_{X}$
is an inverse limit of the quasi-coherent $\mathcal{O}_{Y}$-modules and
\[
0\rightarrow\mathcal{I}_{X}\longrightarrow\mathcal{F}_{Y}\overset{f^{+}%
}{\longrightarrow}f_{\ast}\mathcal{F}_{X}\rightarrow0
\]
is an exact sequence of the sheaves of NC-complete rings.
\end{proposition}

\begin{proof}
We prove the assertion for the closed immersions. A similar argument is
applicable to the general case. Let $\left(  f,f^{+}\right)  :\left(
X,\mathcal{F}_{X}\right)  \rightarrow\left(  Y,\mathcal{F}_{Y}\right)  $ be a
closed immersion of NC-complete schemes. Since $f^{+}:\mathcal{F}%
_{Y}\rightarrow f_{\ast}\mathcal{F}_{X}$ is a filtered morphism of the
filtered sheaves, we have $f^{+}\left(  \mathcal{J}_{Y,m}\right)  \subseteq
f_{\ast}\mathcal{J}_{X,m}$ and there are sheaf morphisms $f_{m}^{+}%
:\mathcal{F}_{Y,m}\rightarrow f_{\ast}\mathcal{F}_{X,m}$ of $\mathcal{O}_{Y}%
$-modules such that the following diagram
\[%
\begin{array}
[c]{cccc}%
\mathcal{F}_{Y} & \overset{f^{+}}{\longrightarrow} & f_{\ast}\mathcal{F}_{X} &
\rightarrow0\\
\downarrow &  & \downarrow & \\
\mathcal{F}_{Y,m} & \overset{f_{m}^{+}}{\longrightarrow} & f_{\ast}%
\mathcal{F}_{X,m} & \\
\downarrow &  & \downarrow & \\
0 &  & 0 &
\end{array}
\]
commutes, whose columns and upper row are surjective. As follows from the
diagram, all $f_{m}^{+}$ are surjective morphisms, that is, all $\left(
f,f_{m}^{+}\right)  :\left(  X,\mathcal{F}_{X,m}\right)  \rightarrow\left(
Y,\mathcal{F}_{Y,m}\right)  $ are closed NC-immersions of the NC-nilpotent
schemes. Thus there are commutative diagrams
\[%
\begin{array}
[c]{ccc}%
\left(  X,\mathcal{F}_{X}\right)  & \longrightarrow & \left(  Y,\mathcal{F}%
_{Y}\right) \\
\uparrow &  & \uparrow\\
\left(  X,\mathcal{F}_{X,m}\right)  & \longrightarrow & \left(  Y,\mathcal{F}%
_{Y,m}\right)
\end{array}
\]
of (compatible) closed NC-immersions, and $\left(  f,f^{+}\right)
=\underleftarrow{\lim}\left\{  \left(  f,f_{m}^{+}\right)  \right\}  $.

Finally, $\left(  f^{+}\right)  ^{-1}\left(  f_{\ast}\mathcal{J}_{X,m}\right)
=\mathcal{J}_{Y,m}+\mathcal{I}_{X}$ for all large $m$. Put $\mathcal{I}%
_{X,m}=\mathcal{I}_{X}/\left(  \mathcal{J}_{Y,m}\cap\mathcal{I}_{X}\right)  $.
It follows that
\[
0\rightarrow\mathcal{I}_{X,m}\longrightarrow\mathcal{F}_{Y,m}\overset{f_{m}%
^{+}}{\longrightarrow}f_{\ast}\mathcal{F}_{X,m}\rightarrow0
\]
are exact sequences of $\mathcal{O}_{Y}$-modules. In particular,
$\mathcal{I}_{X,m}$ is an ideal sheaf of the closed NC-immersion $\left(
f,f_{m}^{+}\right)  :\left(  X,\mathcal{F}_{X,m}\right)  \rightarrow\left(
Y,\mathcal{F}_{Y,m}\right)  $ of the NC-nilpotent schemes. Since $f$ is a
finite morphism and $\mathcal{F}_{X,m}$ are the quasi-coherent $\mathcal{O}%
_{X}$-modules, it follows that so are all $f_{\ast}\mathcal{F}_{X,m}$ as
$\mathcal{O}_{Y}$-modules \cite[2.5.8]{Harts}. Then $\mathcal{I}_{X,m}$ is a
quasi-coherent $\mathcal{O}_{Y}$-module being the kernel of the morphism
$f_{m}^{+}$ \cite[2.5.7]{Harts}. But $\left\{  \mathcal{I}_{X,m}\right\}  $ is
an (ML) inverse system of $\mathcal{O}_{Y}$-modules and $\underleftarrow{\lim
}\left\{  \mathcal{I}_{X,m}\right\}  $ is the NC-completion of $\mathcal{I}%
_{X}$, therefore $\mathcal{I}_{X}=\underleftarrow{\lim}\left\{  \mathcal{I}%
_{X,m}\right\}  $.
\end{proof}

\begin{remark}
To state the reverse assertion to Proposition \ref{propCI} one needs to use R.
Hartshorne's Theorem \cite[1.4.5]{Harts2}. Namely, an inverse limit of closed
NC-immersions of NC-nilpotent deformations is a closed immersion of
NC-complete deformations (see also \cite[Remark 3.5.1]{Harts}).
\end{remark}

Now let $\left(  Y,\mathcal{F}_{Y}\right)  $ be an NC-deformation over $k$. We
say it is a \textit{projective NC-deformation over} $k$ or $\left(
Y,\mathcal{F}_{Y}\right)  \rightarrow\operatorname{Spec}\left(  k\right)  $ is
a\textit{ projective morphism} if it factors into a closed NC-immersion
$\left(  \iota,\iota^{+}\right)  :\left(  Y,\mathcal{F}_{Y}\right)
\rightarrow\mathbb{P}_{k,q}^{n}$ followed by the projection $\mathbb{P}%
_{k,q}^{n}\rightarrow\operatorname{Spec}\left(  k\right)  $. So are all
noncommutative schemes obtained from the differential chains in $S_{q}$
considered below in Section \ref{SecPSDC}. In the case of an NC-nilpotent
scheme $\left(  Y,\mathcal{F}_{Y}\right)  $ we say that it is \textit{a
projective NC-nilpotent scheme} \textit{over }$k$. For example, all $\left(
\mathbb{P}_{k,q}^{n},\mathcal{O}_{q}/\mathcal{J}_{m}\right)  $ are projective
NC-nilpotent schemes over $k$ (see Subsection \ref{subsecPQN}).

\begin{corollary}
\label{corCOH1}Let $\left(  Y,\mathcal{F}_{Y}\right)  $ be a projective
NC-complete scheme over $k$ with its two-sided ideal sheaf $\mathcal{I}_{q,Y}%
$. Then all $\mathcal{F}_{Y,m}$ are coherent $\mathcal{O}_{Y}$-modules, both
$\mathcal{F}_{Y}$ and $\mathcal{I}_{q,Y}$ are inverse limits of the coherent
$\mathcal{O}_{Y}$-modules.
\end{corollary}

\begin{proof}
Let $\left(  \iota,\iota^{\times}\right)  :\left(  Y,\mathcal{F}_{Y}\right)
\rightarrow\mathbb{P}_{k,q}^{n}$ be a closed NC-immersion that factors
$\left(  Y,\mathcal{F}_{Y}\right)  \rightarrow\operatorname{Spec}\left(
k\right)  $. Then $\iota:Y\rightarrow\mathbb{P}_{k}^{n}$ is a homeomorphism
onto its range, and $0\rightarrow\mathcal{I}_{q,Y}\longrightarrow
\mathcal{O}_{q}\overset{\iota^{+}}{\longrightarrow}\iota_{\ast}\mathcal{F}%
_{Y}\rightarrow0$ is an exact sequence of $\mathcal{O}$-modules with the
filtered surjective morphism $\iota^{+}$. By Proposition \ref{propCI}, the
latter exact sequence turns out to be the inverse limit of the following exact
sequences $0\rightarrow\mathcal{I}_{Y,m}\longrightarrow\mathcal{O}%
_{q}/\mathcal{J}_{m}\overset{\iota_{m}^{+}}{\longrightarrow}\iota_{\ast
}\mathcal{F}_{Y,m}\rightarrow0$ of the quasi-coherent $\mathcal{O}$-modules.
Take an open affine $U=D_{+}\left(  h\right)  $ with $h\in S_{+}$ homogeneous.
Then $\left(  \mathcal{O}_{q}/\mathcal{J}_{m}\right)  \left(  U\right)
=\bigoplus\limits_{l=0}^{m-1}S_{h}^{-l}\otimes R_{q}^{l}\left\langle
\mathbf{y}\right\rangle $ is $S_{\left(  h\right)  }$-module finite. In
particular, $\mathcal{I}_{Y,m}\left(  U\right)  $ being an $S_{\left(
h\right)  }$-submodule of $\left(  \mathcal{O}_{q}/\mathcal{J}_{m}\right)
\left(  U\right)  $ is noetherian. But $U=\operatorname{Spec}\left(
S_{\left(  h\right)  }\right)  $ and $\mathcal{I}_{Y,m}\left(  U\right)
=\Gamma\left(  U,\mathcal{I}_{Y,m}|_{U}\right)  $. It follows that
$\mathcal{I}_{Y,m}$ is a coherent sheaf. In particular, so is $\mathcal{F}%
_{Y,m}$.
\end{proof}

\section{Projective NC-schemes associated with the differential
chains\label{SecPSDC}}

In this section we introduce the differential chains in the algebra $S_{q}$
and describe the related noncommutative projective schemes.

\subsection{The differential operators on $S_{q}$\label{subsecDOSQ}}

As in Subsection \ref{subsecGFNA} fix a Hall basis $\mathbf{z=x}%
$\textbf{$\sqcup$}$\mathbf{y=}\left(  z_{0},\ldots,z_{v}\right)  $ for
$\mathfrak{g}_{q}\left(  \mathbf{x}\right)  $ with their degrees $e_{i}%
=\deg\left(  z_{i}\right)  $, $0\leq i\leq v$. Thus $e_{i}=1$ for all $i$,
$0\leq i\leq n$, and $e_{i}\geq2$ for all $i>n$. Since $S_{q}=k\left[
\mathbf{z}\right]  $ as the vector spaces, every element $z_{j}$ of the basis
$\mathbf{z}$ defines the (commutative) multiplication operator $z_{j}%
\in\mathcal{L}_{k}\left(  S_{q}\right)  $, $z_{j}\left(  p\left(
\mathbf{z}\right)  \right)  =\left(  z_{j}p\right)  \left(  \mathbf{z}\right)
$. We have also the differential operators $\partial^{\beta},\overline
{\partial}^{\beta}\in\mathcal{L}_{k}\left(  S_{q}\right)  $, $\beta
\in\mathbb{Z}_{+}^{v+1}$ (see (\ref{DforM})). For every $j$ and $\beta
\in\mathbb{Z}_{+}^{v+1}$ we use the notation $\left\vert \beta\right\vert
_{j}=\beta_{0}+\cdots+\beta_{j}$.

\begin{remark}
\label{remDZ}Note that $\overline{\partial}^{\beta}\left(  z_{j}g\right)
=-\overline{\partial}^{\beta_{\left(  j\right)  }}g+z_{j}\overline{\partial
}^{\beta}g$ with $\beta_{\left(  j\right)  }=\left(  \beta_{0},\ldots
,\beta_{j}-1,\ldots,\beta_{n}\right)  $ (see (\ref{DforM})) in the polynomial
algebra $k\left[  \mathbf{z}\right]  $. Whence $\operatorname{ad}\left(
z_{j}\right)  \left(  \overline{\partial}^{\beta}\right)  =\overline{\partial
}^{\beta_{\left(  j\right)  }}$ over $k\left[  \mathbf{z}\right]  $. Note also
that $\left[  z_{j},\overline{\partial}^{\beta}\right]  =0$ whenever
$\left\vert \beta\right\vert _{j}=0$.
\end{remark}

The algebraic structure of $S_{q}$ defines the left regular representation
$L:S_{q}\rightarrow\mathcal{L}_{k}\left(  S_{q}\right)  $, $L\left(  a\right)
=L_{a}$, and the right regular (anti-)representation $R:S_{q}\rightarrow
\mathcal{L}_{k}\left(  S_{q}\right)  $, $R\left(  a\right)  =R_{a}$. For
brevity we put $L_{j}=L_{z_{j}}$ and $R_{j}=R_{z_{j}}$. For every tuple
$\mathbf{i}=\left(  i_{0},\ldots,i_{v}\right)  \in\mathbb{Z}_{+}^{v+1}$ we put
$\mathbf{i}_{x}=\left(  i_{0},\ldots,i_{n}\right)  \in\mathbb{Z}_{+}^{n+1}$
and $\mathbf{i}_{y}=\left(  i_{n+1},\ldots,i_{v}\right)  \in\mathbb{Z}%
_{+}^{v-n}$ corresponding to the the bases $\mathbf{x}$ and $\mathbf{y}$,
respectively. Let us introduce the following differential operators
\begin{align*}
\Delta_{j,k}  &  =\sum_{\left\vert \mathbf{i}_{x}\right\vert =k}%
\sum_{\mathbf{i}_{y},\left\vert \mathbf{i}\right\vert _{j}>0}R\left(
\operatorname{ad}\left(  \overline{\mathbf{y}}\right)  ^{\mathbf{i}_{y}%
}\operatorname{ad}\left(  \overline{\mathbf{x}}\right)  ^{\mathbf{i}_{x}%
}\left(  z_{j}\right)  \right)  \overline{\partial}^{\mathbf{i}},\\
\nabla_{j,k}  &  =-\sum_{\left\vert \mathbf{i}_{x}\right\vert =k}%
\sum_{\mathbf{i}_{y},\left\vert \mathbf{i}\right\vert _{j}=0}R\left(
\operatorname{ad}\left(  \overline{\mathbf{y}}\right)  ^{\mathbf{i}_{y}%
}\operatorname{ad}\left(  \overline{\mathbf{x}}\right)  ^{\mathbf{i}_{x}%
}\left(  z_{j}\right)  \right)  \overline{\partial}^{\mathbf{i}}%
\end{align*}
on $S_{q}$, where $0\leq k\leq q-e_{j}$ ($\operatorname{ad}\left(
\overline{\mathbf{x}}\right)  ^{\mathbf{i}_{x}}\left(  z_{j}\right)
\in\mathfrak{g}_{q}\left(  \mathbf{x}\right)  _{e_{j}+k}$). Note that
$\mathbf{i}=\left(  i_{0},\ldots,i_{v}\right)  \in\mathbb{Z}_{+}^{v+1}$ with
$\left\vert \mathbf{i}\right\vert _{j}=0$ means that $i_{0}=\cdots=i_{j}=0$,
which in turn implies that $\left\vert \mathbf{i}_{x}\right\vert =0$ whenever
$j\geq n$. Thus for $j\geq n$ we obtain that $\nabla_{j,0}=-\sum_{\left\vert
\mathbf{i}\right\vert _{j}=0}R\left(  \operatorname{ad}\left(  \overline
{\mathbf{y}}\right)  ^{\mathbf{i}_{y}}\left(  z_{j}\right)  \right)
\overline{\partial}^{\mathbf{i}_{y}}$ and $\nabla_{j,k}=0$, $k>0$. Further, we
put $\Delta_{j}=\sum_{k=0}^{q-e_{j}}\Delta_{j,k}$ and $\nabla_{j}=\sum
_{k=0}^{q-e_{j}}\nabla_{j,k}$. Then
\[
\Delta_{j}=\sum_{\mathbf{i}\in\mathbb{Z}_{+}^{v+1},\left\vert \mathbf{i}%
\right\vert _{j}>0}R\left(  \operatorname{ad}\left(  \overline{\mathbf{z}%
}\right)  ^{\mathbf{i}}\left(  z_{j}\right)  \right)  \overline{\partial
}^{\mathbf{i}},\text{ }\nabla_{j}=-\sum_{\mathbf{i}\in\mathbb{Z}_{+}%
^{v+1},\left\vert \mathbf{i}\right\vert _{j}=0}R\left(  \operatorname{ad}%
\left(  \overline{\mathbf{z}}\right)  ^{\mathbf{i}}\left(  z_{j}\right)
\right)  \overline{\partial}^{\mathbf{i}}\in\mathcal{L}_{k}\left(
S_{q}\right)  .
\]
Thus $\nabla_{j}=\nabla_{j,0}$ whenever $j\geq n$.

\begin{remark}
\label{remSpec1}In the case of $q=2$ we have $\Delta_{j}=0$ for all $j>n$,
and
\[
\Delta_{j}=\sum_{\mathbf{i}\in\mathbb{Z}_{+}^{n+1},1=\left\vert \mathbf{i}%
\right\vert \geq\left\vert \mathbf{i}\right\vert _{j}>0}R\left(
\operatorname{ad}\left(  \overline{\mathbf{x}}\right)  ^{\mathbf{i}}\left(
x_{j}\right)  \right)  \overline{\partial}^{\mathbf{i}}=-\sum_{i<j}R\left(
\left[  x_{i},x_{j}\right]  \right)  \partial_{i}%
\]
for all $j\leq n$. Similarly, $\nabla_{j}=0$ for all $j\geq n$, and
$\nabla_{j}=\sum_{i>j}R\left(  \left[  x_{i},x_{j}\right]  \right)
\partial_{i}$, $j<n$. Since $y_{ij}=\left[  x_{i},x_{j}\right]  $, $i<j$ are
central elements, we obtain that $\Delta_{j}=-\sum_{i<j}\partial_{i}y_{ij}$
and $\nabla_{j}=-\sum_{j<i}\partial_{i}y_{ji}$.
\end{remark}

Note that $z_{j}\left(  S^{e}\otimes R_{q}^{m}\left\langle \mathbf{y}%
\right\rangle \right)  \subseteq S^{e+1}\otimes R_{q}^{m}\left\langle
\mathbf{y}\right\rangle $, $0\leq j\leq n$, and $z_{j}\left(  S^{e}\otimes
R_{q}^{m}\left\langle \mathbf{y}\right\rangle \right)  \subseteq S^{e}\otimes
R_{q}^{m+e_{j}}\left\langle \mathbf{y}\right\rangle $, $j>n$ for all $e$ and
$m$. In particular, $z_{j}\left(  S_{q}^{d}\right)  \subseteq S_{q}^{d+e_{j}}$
for all $d$ (we used to write $d=e+m$). Moreover, $\overline{\partial
}^{\mathbf{i}}\left(  S_{q}^{d}\right)  \subseteq S_{q}^{d-\left\langle
\mathbf{i}\right\rangle }$ for all $\mathbf{i}$.

\begin{lemma}
\label{lemOp0}Fix $d$ and $m$. Then $\Delta_{j,k}\left(  S^{e}\otimes
R_{q}^{m}\left\langle \mathbf{y}\right\rangle \right)  +\nabla_{j,k}\left(
S^{e}\otimes R_{q}^{m}\left\langle \mathbf{y}\right\rangle \right)  \subseteq
S^{e-k}\otimes R_{q}^{m+e_{j}+k}\left\langle \mathbf{y}\right\rangle $ for all
$j$, $k$. In particular, $\Delta_{j}\left(  S_{q}^{d}\right)  +\nabla
_{j}\left(  S_{q}^{d}\right)  \subseteq S_{q}^{d+e_{j}}$ for all $j$.
\end{lemma}

\begin{proof}
Take $u=a\left(  \mathbf{x}\right)  \mathbf{y}^{\alpha}\in S^{e}\otimes
R_{q}^{m}\left\langle \mathbf{y}\right\rangle $ with $a\in S^{e}$,
$\left\langle \alpha\right\rangle =m$. By its very definition we have
\[
\Delta_{j,k}\left(  u\right)  =\sum_{\left\vert \mathbf{i}_{x}\right\vert
=k}\sum_{\mathbf{i}_{y},\left\vert \mathbf{i}\right\vert _{j}>0}\left(
\overline{\partial}^{\mathbf{i}_{x}}a\right)  \left(  \mathbf{x}\right)
\left(  \overline{\partial}^{\mathbf{i}_{y}}\mathbf{y}^{\alpha}\right)
\operatorname{ad}\left(  \overline{\mathbf{y}}\right)  ^{\mathbf{i}_{y}%
}\operatorname{ad}\left(  \overline{\mathbf{x}}\right)  ^{\mathbf{i}_{x}%
}\left(  z_{j}\right)  .
\]
But $\operatorname{ad}\left(  \overline{\mathbf{x}}\right)  ^{\mathbf{i}_{x}%
}\left(  z_{j}\right)  \in\mathfrak{g}_{q}\left(  \mathbf{x}\right)
_{e_{j}+k}$, $\overline{\partial}^{\mathbf{i}_{y}}\mathbf{y}^{\alpha}\in
R_{q}^{m-\left\langle \mathbf{i}_{y}\right\rangle }\left\langle \mathbf{y}%
\right\rangle $ and $\operatorname{ad}\left(  \overline{\mathbf{y}}\right)
^{\mathbf{i}_{y}}\operatorname{ad}\left(  \overline{\mathbf{x}}\right)
^{\mathbf{i}_{x}}\left(  z_{j}\right)  \in\mathfrak{g}_{q}\left(
\mathbf{x}\right)  _{e_{j}+k+\left\langle \mathbf{i}_{y}\right\rangle }$. Then%
\[
\left(  \overline{\partial}^{\mathbf{i}_{y}}\mathbf{y}^{\alpha}\right)
\operatorname{ad}\left(  \overline{\mathbf{y}}\right)  ^{\mathbf{i}_{y}%
}\operatorname{ad}\left(  \overline{\mathbf{x}}\right)  ^{\mathbf{i}_{x}%
}\left(  z_{j}\right)  \in R_{q}^{m-\left\langle \mathbf{i}_{y}\right\rangle
}\left\langle \mathbf{y}\right\rangle R_{q}^{e_{j}+k+\left\langle
\mathbf{i}_{y}\right\rangle }\left\langle \mathbf{y}\right\rangle \subseteq
R_{q}^{m-\left\langle \mathbf{i}_{y}\right\rangle +e_{j}+k+\left\langle
\mathbf{i}_{y}\right\rangle }\left\langle \mathbf{y}\right\rangle
=R_{q}^{m+e_{j}+k}\left\langle \mathbf{y}\right\rangle .
\]
It follows that $\Delta_{j,k}\left(  u\right)  \in S^{e-k}\otimes
R_{q}^{m+e_{j}+k}\left\langle \mathbf{y}\right\rangle $, that is,
$\Delta_{j,k}\left(  S^{e}\otimes R_{q}^{m}\left\langle \mathbf{y}%
\right\rangle \right)  \subseteq S^{e-k}\otimes R_{q}^{m+e_{j}+k}\left\langle
\mathbf{y}\right\rangle $. In particular,
\[
\Delta_{j}\left(  S^{e}\otimes R_{q}^{m}\left\langle \mathbf{y}\right\rangle
\right)  =\sum_{k=0}^{q-e_{j}}\Delta_{j,k}\left(  S^{e}\otimes R_{q}%
^{m}\left\langle \mathbf{y}\right\rangle \right)  \subseteq\sum_{k=0}%
^{q-e_{j}}S^{e-k}\otimes R_{q}^{m+e_{j}+k}\left\langle \mathbf{y}\right\rangle
,
\]
which in turn implies that $\Delta_{j}\left(  S_{q}^{d}\right)  \subseteq
S_{q}^{d+e_{j}}$. A similar inclusion takes place for the operators
$\nabla_{j}$. Namely, if $j\geq n$ then
\[
\nabla_{j}\left(  u\right)  =\nabla_{j,0}\left(  u\right)  =-\sum_{\left\vert
\mathbf{i}\right\vert _{j}=0}a\left(  \mathbf{x}\right)  \left(
\overline{\partial}^{\mathbf{i}_{y}}\mathbf{y}^{\alpha}\right)
\operatorname{ad}\left(  \overline{\mathbf{y}}\right)  ^{\mathbf{i}_{y}%
}\left(  z_{j}\right)  \in S^{e}\otimes R_{q}^{m+e_{j}}\left\langle
\mathbf{y}\right\rangle ,
\]
that is, $\nabla_{j}\left(  S^{e}\otimes R_{q}^{m}\left\langle \mathbf{y}%
\right\rangle \right)  \subseteq S^{e}\otimes R_{q}^{m+e_{j}}\left\langle
\mathbf{y}\right\rangle $. But if $j<n$ then
\[
\nabla_{j,k}\left(  u\right)  =-\sum_{\left\vert \mathbf{i}_{x}\right\vert
=k}\sum_{\left\vert \mathbf{i}\right\vert _{j}=0}\left(  \overline{\partial
}^{\mathbf{i}_{x}}a\right)  \left(  \mathbf{x}\right)  \left(  \overline
{\partial}^{\mathbf{i}_{y}}\mathbf{y}^{\alpha}\right)  \operatorname{ad}%
\left(  \overline{\mathbf{y}}\right)  ^{\mathbf{i}_{y}}\operatorname{ad}%
\left(  \overline{\mathbf{x}}\right)  ^{\mathbf{i}_{x}}\left(  z_{j}\right)
\in S^{e-k}\otimes R_{q}^{m+e_{j}+k}\left\langle \mathbf{y}\right\rangle .
\]
Whence $\nabla_{j,k}\left(  S^{e}\otimes R_{q}^{m}\left\langle \mathbf{y}%
\right\rangle \right)  \subseteq S^{e-k}\otimes R_{q}^{m+e_{j}+k}\left\langle
\mathbf{y}\right\rangle $ for all $j$ and $k$.
\end{proof}

\begin{lemma}
\label{lemOp1}For every $j$ we have $L_{j}=z_{j}+\Delta_{j}$ and $R_{j}%
=z_{j}+\nabla_{j}$ with $\left[  z_{j},\nabla_{t}\right]  =0$, $j\leq t$.
Moreover,
\[
\operatorname{ad}\left(  z_{0}\right)  ^{m_{0}}\cdots\operatorname{ad}\left(
z_{j}\right)  ^{m_{j}}\left(  \Delta_{j}\right)  =\sum_{\left\vert
\beta+\mathbf{m}\right\vert _{j}>0}R\left(  \operatorname{ad}\left(
z_{v}\right)  ^{\beta_{v}}\cdots\operatorname{ad}\left(  z_{j}\right)
^{\beta_{j}+m_{j}}\cdots\operatorname{ad}\left(  z_{0}\right)  ^{\beta
_{0}+m_{0}}\left(  z_{j}\right)  \right)  \overline{\partial}^{\beta}%
\]
for every $\mathbf{m}\in\mathbb{Z}_{+}^{j+1}$, which is identified with
$\left(  m_{0},\ldots,m_{j},0,\ldots,0\right)  \in\mathbb{Z}_{+}^{v+1}$. In
particular,%
\[
\operatorname{ad}\left(  \mathbf{x}\right)  ^{\mathbf{i}}\left(  \Delta
_{j}\right)  =\sum_{\left\vert \beta+\mathbf{i}\right\vert _{j}>0}R\left(
\operatorname{ad}\left(  \overline{\mathbf{y}}\right)  ^{\beta_{y}%
}\operatorname{ad}\left(  \overline{\mathbf{x}}\right)  ^{\beta_{x}%
+\mathbf{i}}\left(  z_{j}\right)  \right)  \overline{\partial}^{\beta}%
\]
for all $\mathbf{i}\in\mathbb{Z}_{+}^{n+1}$. If $\mathbf{i}\in\mathbb{Z}%
_{+}^{n+1}$ with $\left\vert \mathbf{i}\right\vert _{j}=0$ and $j\leq n$,
then
\[
\operatorname{ad}\left(  \mathbf{x}\right)  ^{\mathbf{i}}\left(  \nabla
_{j}\right)  =\sum_{\left\vert \beta\right\vert _{j}=0}R\left(
\operatorname{ad}\left(  \overline{\mathbf{y}}\right)  ^{\beta_{y}%
}\operatorname{ad}\left(  \overline{\mathbf{x}}\right)  ^{\beta_{x}%
+\mathbf{i}}\left(  z_{j}\right)  \right)  \overline{\partial}^{\beta}\text{,}%
\]
and $\operatorname{ad}\left(  \mathbf{x}\right)  ^{\mathbf{i}}\left(
\nabla_{j}\right)  =0$ otherwise.
\end{lemma}

\begin{proof}
Take an ordered monomial $p\left(  \mathbf{z}\right)  =\mathbf{z}^{\mathbf{d}%
}$ in $S_{q}$ with $\mathbf{d}\in\mathbb{Z}_{+}^{v+1}$. Using Lemma
\ref{lemNon1}, we derive that
\begin{align*}
L_{j}\left(  p\left(  \mathbf{z}\right)  \right)   &  =z_{j}z_{0}^{d_{0}%
}\cdots z_{v}^{d_{v}}=z_{0}^{d_{0}}\cdots z_{j}^{d_{j}+1}\cdots z_{v}^{d_{v}%
}+\left[  z_{j},z_{0}^{d_{0}}\cdots z_{j-1}^{d_{j-1}}\right]  z_{j}^{d_{j}%
}\cdots z_{v}^{d_{v}}\\
&  =\left(  z_{j}p\right)  \left(  \mathbf{z}\right)  +\sum_{\mathbf{i}%
\in\mathbb{Z}_{+}^{j},\left\vert \mathbf{i}\right\vert >0}\left(
\overline{\partial}^{\mathbf{i}}z_{0}^{d_{0}}\cdots z_{j-1}^{d_{j-1}}\right)
\left(  \mathbf{z}\right)  \operatorname{ad}\left(  \overline{\mathbf{z}%
}\right)  ^{\mathbf{i}}\left(  z_{j}\right)  z_{j}^{d_{j}}\cdots z_{v}^{d_{v}%
}\\
&  =\sum_{\mathbf{i}\in\mathbb{Z}_{+}^{j}}\left(  \overline{\partial
}^{\mathbf{i}}p\right)  \left(  \mathbf{z}\right)  \operatorname{ad}\left(
\overline{\mathbf{z}}\right)  ^{\mathbf{i}}\left(  z_{j}\right)
+\sum_{\mathbf{i}\in\mathbb{Z}_{+}^{j}}\left(  \overline{\partial}%
^{\mathbf{i}}z_{0}^{d_{0}}\cdots z_{j-1}^{d_{j-1}}\right)  \left(
\mathbf{z}\right)  \left[  \operatorname{ad}\left(  \overline{\mathbf{z}%
}\right)  ^{\mathbf{i}}\left(  z_{j}\right)  ,z_{j}^{d_{j}}\cdots z_{v}%
^{d_{v}}\right] \\
&  =\sum_{\mathbf{i}\in\mathbb{Z}_{+}^{j}}\left(  \overline{\partial
}^{\mathbf{i}}p\right)  \left(  \mathbf{z}\right)  \operatorname{ad}\left(
\overline{\mathbf{z}}\right)  ^{\mathbf{i}}\left(  z_{j}\right)
+\sum_{\mathbf{i},\mathbf{k}}\left(  \overline{\partial}^{\mathbf{i}}%
z_{0}^{d_{0}}\cdots z_{j-1}^{d_{j-1}}\right)  \left(  \mathbf{z}\right)
\left(  \overline{\partial}^{\mathbf{k}}z_{j}^{d_{j}}\cdots z_{v}^{d_{v}%
}\right)  \left(  \mathbf{z}\right)  \operatorname{ad}\left(  \overline
{\mathbf{z}}\right)  ^{\mathbf{i}\cup\mathbf{k}}\left(  z_{j}\right) \\
&  =\left(  z_{j}p\right)  \left(  \mathbf{z}\right)  +\sum_{\mathbf{i}%
\in\mathbb{Z}_{+}^{v+1},\left\vert \mathbf{i}\right\vert _{j}>0}\left(
\overline{\partial}^{\mathbf{i}}p\right)  \left(  \mathbf{z}\right)
\operatorname{ad}\left(  \overline{\mathbf{z}}\right)  ^{\mathbf{i}}\left(
z_{j}\right)
\end{align*}
(notice that $\left(  \overline{\partial}^{\mathbf{i}}z_{0}^{d_{0}}\cdots
z_{j-1}^{d_{j-1}}\right)  \left(  \mathbf{z}\right)  \left(  \overline
{\partial}^{\mathbf{k}}z_{j}^{d_{j}}\cdots z_{v}^{d_{v}}\right)  \left(
\mathbf{z}\right)  =\dfrac{\left(  -1\right)  ^{\left\vert \mathbf{i}%
\right\vert }}{\mathbf{i}!}\dfrac{\left(  -1\right)  ^{\left\vert
\mathbf{k}\right\vert }}{\mathbf{k}!}\left(  \partial_{\mathbf{z}}%
^{\mathbf{i}}\partial_{\mathbf{z}}^{\mathbf{k}}p\right)  \left(
\mathbf{z}\right)  =\left(  \overline{\partial}^{\mathbf{i}\cup\mathbf{k}%
}p\right)  \left(  \mathbf{z}\right)  $), which in turn implies that
$L_{j}\left(  a\right)  =z_{j}\left(  a\right)  +\sum_{\mathbf{i}\in
\mathbb{Z}_{+}^{v+1},\left\vert \mathbf{i}\right\vert _{j}>0}R\left(
\operatorname{ad}\left(  \overline{\mathbf{z}}\right)  ^{\mathbf{i}}\left(
z_{j}\right)  \right)  \overline{\partial}^{\mathbf{i}}\left(  a\right)
=z_{j}\left(  a\right)  +\Delta_{j}\left(  a\right)  $ for every $a\in S_{q}$,
that is, $L_{j}=z_{j}+\Delta_{j}$. Similarly,%
\begin{align*}
R_{j}\left(  p\left(  \mathbf{z}\right)  \right)   &  =z_{0}^{d_{0}}\cdots
z_{v}^{d_{v}}z_{j}=z_{0}^{d_{0}}\cdots z_{j}^{d_{j}+1}\cdots z_{v}^{d_{v}%
}-z_{0}^{d_{0}}\cdots z_{j}^{d_{j}}\left[  z_{j},z_{j+1}^{d_{j+1}}\cdots
z_{v}^{d_{v}}\right] \\
&  =\left(  z_{j}p\right)  \left(  \mathbf{x}\right)  -\sum_{\mathbf{i}%
\in\mathbb{Z}_{+}^{v-j},\left\vert \mathbf{i}\right\vert >0}z_{0}^{d_{0}%
}\cdots z_{j}^{d_{j}}\left(  \overline{\partial}^{\mathbf{i}}z_{j+1}^{d_{j+1}%
}\cdots z_{v}^{d_{v}}\right)  \left(  \mathbf{z}\right)  \operatorname{ad}%
\left(  \overline{\mathbf{z}}\right)  ^{\mathbf{i}}\left(  z_{j}\right) \\
&  =\left(  z_{j}p\right)  \left(  \mathbf{x}\right)  -\sum_{\mathbf{i}%
\in\mathbb{Z}_{+}^{v+1},\left\vert \mathbf{i}\right\vert _{j}=0}\left(
\overline{\partial}^{\mathbf{i}}p\right)  \left(  \mathbf{z}\right)
\operatorname{ad}\left(  \overline{\mathbf{z}}\right)  ^{\mathbf{i}}\left(
z_{j}\right)  ,
\end{align*}
that is, $R_{j}=z_{j}-\sum_{\mathbf{i}\in\mathbb{Z}_{+}^{v+1},\left\vert
\mathbf{i}\right\vert _{j}=0}R\left(  \operatorname{ad}\left(  \overline
{\mathbf{z}}\right)  ^{\mathbf{i}}\left(  z_{j}\right)  \right)
\overline{\partial}^{\mathbf{i}}=z_{j}+\nabla_{j}$. Now let us prove that
$\left[  z_{j},\nabla_{t}\right]  =0$ whenever $j\leq t$. If $e_{t}\geq q-1$
then $\operatorname{ad}\left(  \overline{\mathbf{z}}\right)  ^{\mathbf{i}%
}\left(  z_{t}\right)  $ (for $\left\vert \mathbf{i}\right\vert >0$) is a
linear combination of the basis (central) elements $z_{u}$ with $e_{u}=q$,
and
\[
\nabla_{t}=-\sum_{\mathbf{i}\in\mathbb{Z}_{+}^{v+1},\left\vert \mathbf{i}%
\right\vert _{t}=0}R\left(  \operatorname{ad}\left(  \overline{\mathbf{z}%
}\right)  ^{\mathbf{i}}\left(  z_{t}\right)  \right)  \overline{\partial
}^{\mathbf{i}}=\sum_{e_{u}=q,\left\vert \mathbf{i}\right\vert _{t}=0}%
\lambda_{u,\mathbf{i}}R\left(  z_{u}\right)  \overline{\partial}^{\mathbf{i}%
}=\sum_{e_{u}=q,\left\vert \mathbf{i}\right\vert _{t}=0}\lambda_{u,\mathbf{i}%
}z_{u}\overline{\partial}^{\mathbf{i}}%
\]
for some $\lambda_{u,\mathbf{i}}\in k$. But $\left[  z_{j},z_{u}%
\overline{\partial}^{\mathbf{i}}\right]  =\left[  z_{j},z_{u}\right]
\overline{\partial}^{\mathbf{i}}+z_{u}\left[  z_{j},\overline{\partial
}^{\mathbf{i}}\right]  =0$ for all $u$, $e_{u}=q$, and $\mathbf{i}$,
$\left\vert \mathbf{i}\right\vert _{t}=0$ (see Remark \ref{remDZ}), that is,
$\left[  z_{j},\nabla_{t}\right]  =0$. If $e_{t}<q-1$, then $\operatorname{ad}%
\left(  \overline{\mathbf{z}}\right)  ^{\mathbf{i}}\left(  z_{t}\right)  $ is
a linear combination of the basis elements $z_{u}$ with $e_{u}>\max\left\{
1,e_{t}\right\}  $, and $\nabla_{t}=\sum_{e_{u}>\max\left\{  1,e_{t}\right\}
,\left\vert \mathbf{i}\right\vert _{t}=0}\lambda_{u,\mathbf{i}}R\left(
z_{u}\right)  \overline{\partial}^{\mathbf{i}}$ for some $\lambda
_{u,\mathbf{i}}$. By induction hypothesis, we have $\left[  z_{j},R\left(
z_{u}\right)  \right]  =\left[  z_{j},z_{u}+\nabla_{u}\right]  =0$. But
$\left[  z_{j},\overline{\partial}^{\mathbf{i}}\right]  =0$ whenever
$\left\vert \mathbf{i}\right\vert _{t}=0$ and $j\leq t$ (see Remark
\ref{remDZ}). Therefore $\left[  z_{j},\nabla_{t}\right]  =0$ for $j\leq t$.

Further, as above $\Delta_{j}=\sum_{\beta\in\mathbb{Z}_{+}^{v+1},\left\vert
\beta\right\vert _{j}>0}R\left(  \operatorname{ad}\left(  \overline
{\mathbf{z}}\right)  ^{\beta}\left(  z_{j}\right)  \right)  \overline
{\partial}^{\beta}$ and
\[
R\left(  \operatorname{ad}\left(  \overline{\mathbf{z}}\right)  ^{\beta
}\left(  z_{j}\right)  \right)  =\sum_{e_{u}>\max\left\{  1,e_{j}\right\}
}\lambda_{u,\beta}R\left(  z_{u}\right)  .
\]
But $\left[  z_{i},R\left(  z_{u}\right)  \right]  =\left[  z_{i},z_{u}%
+\nabla_{u}\right]  =\left[  z_{i},\nabla_{u}\right]  =0$ whenever $i\leq j$
and $\max\left\{  1,e_{j}\right\}  <e_{u}$. In particular, $\left[
x_{i},R\left(  z_{u}\right)  \right]  =0$ for all $i$, $0\leq i\leq n$. Thus
$\left[  z_{i},R\left(  \operatorname{ad}\left(  \overline{\mathbf{z}}\right)
^{\beta}\left(  z_{j}\right)  \right)  \right]  =0$ whenever $i\leq j$. Using
Remark \ref{remDZ}, we conclude that
\begin{align*}
\left[  z_{i},\Delta_{j}\right]   &  =\sum_{\beta\in\mathbb{Z}_{+}%
^{v+1},\left\vert \beta\right\vert _{j}>0}R\left(  \operatorname{ad}\left(
\overline{\mathbf{z}}\right)  ^{\beta}\left(  z_{j}\right)  \right)  \left[
z_{i},\overline{\partial}^{\beta}\right]  =\sum_{\beta\in\mathbb{Z}_{+}%
^{v},\left\vert \beta\right\vert _{j}>0}R\left(  \operatorname{ad}\left(
\overline{\mathbf{z}}\right)  ^{\beta}\left(  z_{j}\right)  \right)
\overline{\partial}^{\beta_{\left(  i\right)  }}\\
&  =\sum_{\beta}R\left(  \operatorname{ad}\left(  z_{v}\right)  ^{\beta_{v}%
}\cdots\operatorname{ad}\left(  z_{i}\right)  ^{\beta_{i}+1}\cdots
\operatorname{ad}\left(  z_{0}\right)  ^{\beta_{0}}\left(  z_{j}\right)
\right)  \overline{\partial}^{\beta}%
\end{align*}
for all $i\leq j$ (or just $i\leq n$). Thus if $\mathbf{m}\in\mathbb{Z}%
_{+}^{j+1}$, then $\operatorname{ad}\left(  \mathbf{z}\right)  ^{\mathbf{m}%
}\left(  \Delta_{j}\right)  =\operatorname{ad}\left(  z_{0}\right)  ^{m_{0}%
}\cdots\operatorname{ad}\left(  z_{j}\right)  ^{m_{j}}\left(  \Delta
_{j}\right)  $ and
\[
\operatorname{ad}\left(  \mathbf{z}\right)  ^{\mathbf{m}}\left(  \Delta
_{j}\right)  =\sum_{\left\vert \beta+\mathbf{m}\right\vert _{j}>0}R\left(
\operatorname{ad}\left(  z_{v}\right)  ^{\beta_{v}}\cdots\operatorname{ad}%
\left(  z_{j}\right)  ^{\beta_{j}+m_{j}}\cdots\operatorname{ad}\left(
z_{0}\right)  ^{\beta_{0}+m_{0}}\left(  z_{j}\right)  \right)  \overline
{\partial}^{\beta}.
\]
In particular, $\operatorname{ad}\left(  \mathbf{x}\right)  ^{\mathbf{i}%
}\left(  \Delta_{j}\right)  =\sum_{\left\vert \beta+\mathbf{i}\right\vert
_{j}>0}R\left(  \operatorname{ad}\left(  \overline{\mathbf{y}}\right)
^{\beta_{y}}\operatorname{ad}\left(  \overline{\mathbf{x}}\right)  ^{\beta
_{x}+\mathbf{i}}\left(  z_{j}\right)  \right)  \overline{\partial}^{\beta}$
for all $\mathbf{i}\in\mathbb{Z}_{+}^{n+1}$.

Now let us prove the same relation for the operator $\nabla_{j}$. First
$\operatorname{ad}\left(  x_{i}\right)  \left(  \nabla_{j}\right)  =0$
whenever $i\leq j$ or $j>n$. In particular, $\operatorname{ad}\left(
\mathbf{x}\right)  ^{\mathbf{i}}\left(  \nabla_{j}\right)  =0$ whenever $j\geq
n$. If $j<i\leq n$ then as above we have%
\[
\left[  x_{i},\nabla_{j}\right]  =\sum_{\beta\in\mathbb{Z}_{+}^{v+1}%
,\left\vert \beta\right\vert _{j}=0}R\left(  \operatorname{ad}\left(
\overline{\mathbf{z}}\right)  ^{\beta}\left(  x_{j}\right)  \right)  \left[
x_{i},\overline{\partial}^{\beta}\right]  =\sum_{\left\vert \beta\right\vert
_{j}=0}R\left(  \operatorname{ad}\left(  z_{v}\right)  ^{\beta_{v}}%
\cdots\operatorname{ad}\left(  z_{i}\right)  ^{\beta_{i}+1}\cdots
\operatorname{ad}\left(  z_{0}\right)  ^{\beta_{0}}\left(  x_{j}\right)
\right)  \overline{\partial}^{\beta}.
\]
It follows that
\[
\operatorname{ad}\left(  \mathbf{x}\right)  ^{\mathbf{i}}\left(  \nabla
_{j}\right)  =\operatorname{ad}\left(  x_{j+1}\right)  ^{i_{j+1}}%
\cdots\operatorname{ad}\left(  x_{n}\right)  ^{i_{n}}\left(  \nabla
_{j}\right)  =\sum_{\left\vert \beta\right\vert _{j}=0}R\left(
\operatorname{ad}\left(  \overline{\mathbf{y}}\right)  ^{\beta_{y}%
}\operatorname{ad}\left(  \overline{\mathbf{x}}\right)  ^{\beta_{x}%
+\mathbf{i}}\left(  x_{j}\right)  \right)  \overline{\partial}^{\beta},
\]
where $\mathbf{i}\in\mathbb{Z}_{+}^{n+1}$ with $\left\vert \mathbf{i}%
\right\vert _{j}=0$.
\end{proof}

\begin{corollary}
\label{corKey1}For every $j$ and $\mathbf{i}\in\mathbb{Z}_{+}^{n+1}$ with
$l=\left\vert \mathbf{i}\right\vert $ we have
\[
\operatorname{ad}\left(  \mathbf{x}\right)  ^{\mathbf{i}}\left(  \Delta
_{j}\right)  \left(  S^{e}\otimes R_{q}^{m}\left\langle \mathbf{y}%
\right\rangle \right)  +\operatorname{ad}\left(  \mathbf{x}\right)
^{\mathbf{i}}\left(  \nabla_{j}\right)  \left(  S^{e}\otimes R_{q}%
^{m}\left\langle \mathbf{y}\right\rangle \right)  \subseteq\sum_{t=0}%
^{q-l-e_{j}}S^{e-t}\otimes R_{q}^{m+t+l+e_{j}}\left\langle \mathbf{y}%
\right\rangle .
\]
In particular, $\operatorname{ad}\left(  \mathbf{x}\right)  ^{\mathbf{i}%
}\left(  \Delta_{j}\right)  \left(  S_{q}^{d}\right)  +\operatorname{ad}%
\left(  \mathbf{x}\right)  ^{\mathbf{i}}\left(  \nabla_{j}\right)  \left(
S_{q}^{d}\right)  \subseteq S_{q}^{d+l+e_{j}}$ and
\[
\operatorname{ad}\left(  \mathbf{x}\right)  ^{\mathbf{i}}\left(  \Delta
_{j}\right)  \left(  S\otimes R_{q}^{m}\left\langle \mathbf{y}\right\rangle
\right)  +\operatorname{ad}\left(  \mathbf{x}\right)  ^{\mathbf{i}}\left(
\nabla_{j}\right)  \left(  S\otimes R_{q}^{m}\left\langle \mathbf{y}%
\right\rangle \right)  \subseteq\sum_{k=0}^{q-e_{j}}S\otimes R_{q}^{m+e_{j}%
+k}\left\langle \mathbf{y}\right\rangle
\]
for all $\mathbf{i}\in\mathbb{Z}_{+}^{n+1}$.
\end{corollary}

\begin{proof}
As in the proof of Lemma \ref{lemOp0}, take $u=a\left(  \mathbf{x}\right)
\mathbf{y}^{\alpha}\in S^{e}\otimes R_{q}^{m}\left\langle \mathbf{y}%
\right\rangle $ with $a\in S^{e}$, $\left\langle \alpha\right\rangle =m$.
Then
\[
\left(  \overline{\partial}^{\beta_{y}}\mathbf{y}^{\alpha}\right)
\operatorname{ad}\left(  \overline{\mathbf{y}}\right)  ^{\beta_{y}%
}\operatorname{ad}\left(  \overline{\mathbf{x}}\right)  ^{\beta_{x}%
+\mathbf{i}}\left(  z_{j}\right)  \in R_{q}^{m-\left\langle \beta
_{y}\right\rangle }\left\langle \mathbf{y}\right\rangle R_{q}^{e_{j}%
+\left\vert \beta_{x}\right\vert +l+\left\langle \beta_{y}\right\rangle
}\left\langle \mathbf{y}\right\rangle \subseteq R_{q}^{m+e_{j}+\left\vert
\beta_{x}\right\vert +l}\left\langle \mathbf{y}\right\rangle
\]
with $l=\left\vert \mathbf{i}\right\vert $ and $\left\vert \beta
_{x}\right\vert +l+e_{j}\leq q$ (the Lie algebra $\mathfrak{g}_{q}\left(
\mathbf{x}\right)  $ is nilpotent). It follows that
\[
\overline{\partial}^{\beta_{x}}\left(  a\right)  \left(  \overline{\partial
}^{\beta_{y}}\mathbf{y}^{\alpha}\right)  \operatorname{ad}\left(
\overline{\mathbf{y}}\right)  ^{\beta_{y}}\operatorname{ad}\left(
\overline{\mathbf{x}}\right)  ^{\beta_{x}+\mathbf{i}}\left(  z_{j}\right)  \in
S^{e-\left\vert \beta_{x}\right\vert }\otimes R_{q}^{m+e_{j}+\left\vert
\beta_{x}\right\vert +l}\left\langle \mathbf{y}\right\rangle .
\]
Using Lemma \ref{lemOp1}, we derive that%
\[
\operatorname{ad}\left(  \mathbf{x}\right)  ^{\mathbf{i}}\left(  \Delta
_{j}\right)  \left(  u\right)  =\sum_{\left\vert \beta+\mathbf{i}\right\vert
_{j}>0}\overline{\partial}^{\beta}\left(  u\right)  \operatorname{ad}\left(
\overline{\mathbf{y}}\right)  ^{\beta_{y}}\operatorname{ad}\left(
\overline{\mathbf{x}}\right)  ^{\beta_{x}+\mathbf{i}}\left(  z_{j}\right)
\in\sum_{\left\vert \beta+\mathbf{i}\right\vert _{j}>0}S^{e-\left\vert
\beta_{x}\right\vert }\otimes R_{q}^{m+e_{j}+\left\vert \beta_{x}\right\vert
+l}\left\langle \mathbf{y}\right\rangle .
\]
But $\sum_{\left\vert \beta+\mathbf{i}\right\vert _{j}>0}S^{e-\left\vert
\beta_{x}\right\vert }\otimes R_{q}^{m+e_{j}+\left\vert \beta_{x}\right\vert
+l}\left\langle \mathbf{y}\right\rangle \subseteq\sum_{t=0}^{q-l-e_{j}}%
S^{e-t}\otimes R_{q}^{m+t+e_{j}+l}\left\langle \mathbf{y}\right\rangle $.
Hence
\[
\operatorname{ad}\left(  \mathbf{x}\right)  ^{\mathbf{i}}\left(  \Delta
_{j}\right)  \left(  S^{e}\otimes R_{q}^{m}\left\langle \mathbf{y}%
\right\rangle \right)  \subseteq\sum_{t=0}^{q-l-e_{j}}S^{e-t}\otimes
R_{q}^{m+t+e_{j}+l}\left\langle \mathbf{y}\right\rangle .
\]
A similar argument is applicable to $\operatorname{ad}\left(  \mathbf{x}%
\right)  ^{\mathbf{i}}\left(  \nabla_{j}\right)  $. In particular,
\[
\operatorname{ad}\left(  \mathbf{x}\right)  ^{\mathbf{i}}\left(  \Delta
_{j}\right)  \left(  S^{e}\otimes R_{q}^{m}\left\langle \mathbf{y}%
\right\rangle \right)  +\operatorname{ad}\left(  \mathbf{x}\right)
^{\mathbf{i}}\left(  \nabla_{j}\right)  \left(  S^{e}\otimes R_{q}%
^{m}\left\langle \mathbf{y}\right\rangle \right)  \subseteq S_{q}^{d+l+e_{j}}%
\]
whenever $e+m=d$, which in turn implies that $\operatorname{ad}\left(
\mathbf{x}\right)  ^{\mathbf{i}}\left(  \Delta_{j}\right)  \left(  S_{q}%
^{d}\right)  +\operatorname{ad}\left(  \mathbf{x}\right)  ^{\mathbf{i}}\left(
\nabla_{j}\right)  \left(  S_{q}^{d}\right)  \subseteq S_{q}^{d+l+e_{j}}$.
Finally,
\begin{align*}
\operatorname{ad}\left(  \mathbf{x}\right)  ^{\mathbf{i}}\left(  \Delta
_{j}\right)  \left(  S\otimes R_{q}^{m}\left\langle \mathbf{y}\right\rangle
\right)  +\operatorname{ad}\left(  \mathbf{x}\right)  ^{\mathbf{i}}\left(
\nabla_{j}\right)  \left(  S\otimes R_{q}^{m}\left\langle \mathbf{y}%
\right\rangle \right)   &  \subseteq\sum_{e}\sum_{t=0}^{q-l-e_{j}}%
S^{e-t}\otimes R_{q}^{m+e_{j}+t+l}\left\langle \mathbf{y}\right\rangle \\
&  \subseteq\sum_{k=0}^{q-e_{j}}S\otimes R_{q}^{m+e_{j}+k}\left\langle
\mathbf{y}\right\rangle ,
\end{align*}
which proves the last required inclusion.
\end{proof}

\begin{remark}
\label{remRight}Based on Lemma \ref{lemNon1}, one can generate another formula
$R_{j}=z_{j}+\nabla_{j}^{\prime}$ with
\[
\nabla_{j}^{\prime}=-\sum_{\mathbf{i}\in\mathbb{Z}_{+}^{v+1},\left\vert
\mathbf{i}\right\vert ^{j}>0}L\left(  \operatorname{ad}\left(  \mathbf{z}%
\right)  ^{\mathbf{i}}\left(  z_{j}\right)  \right)  \partial^{\mathbf{i}}%
\in\mathcal{L}_{k}\left(  S_{q}\right)  ,
\]
where $\left\vert \mathbf{i}\right\vert ^{j}=i_{j}+\cdots+i_{v}$. As above
notice that
\[
\left(  \partial^{\mathbf{k}}z_{0}^{d_{0}}\cdots z_{j}^{d_{j}}\right)  \left(
\mathbf{z}\right)  \left(  \partial^{\mathbf{i}}z_{j+1}^{d_{j+1}}\cdots
z_{v}^{d_{v}}\right)  \left(  \mathbf{z}\right)  =\dfrac{-1}{\mathbf{k}%
!}\dfrac{-1}{\mathbf{i}!}\left(  \partial_{\mathbf{z}}^{\mathbf{k}}%
\partial_{\mathbf{z}}^{\mathbf{i}}p\right)  \left(  \mathbf{z}\right)
=-\left(  \partial^{\mathbf{k}\cup\mathbf{i}}p\right)  \left(  \mathbf{z}%
\right)  .
\]
The main feature of this formula is to reduce the right multiplication
operators $R_{j}$\ to the left ones $L_{i}$ that will be used in the proof of
Lemma \ref{lemLR0}.
\end{remark}

\begin{remark}
\label{remDIK}Based on Lemma \ref{lemOp1}, we obtain the following
differential operators
\[
D_{il}=\sum_{\left\vert \beta+l\right\vert _{i}>0}R\left(  \operatorname{ad}%
\left(  \overline{\mathbf{y}}\right)  ^{\beta_{y}}\operatorname{ad}\left(
x_{n}\right)  ^{\beta_{n}}\cdots\operatorname{ad}\left(  x_{i}\right)
^{\beta_{i}+l}\cdots\operatorname{ad}\left(  x_{0}\right)  ^{\beta_{0}}\left(
x_{i}\right)  \right)  \overline{\partial}^{\beta}=\operatorname{ad}\left(
x_{i}\right)  ^{l}\left(  \Delta_{i}\right)
\]
on the algebra $S_{q}$, where $0\leq i\leq n$ and $0\leq l<q$. Note that
$D_{il}\left(  S_{q}^{d}\right)  \subseteq S_{q}^{d+l+1}$ for all $d$ (see
Corollary \ref{corKey1}), and $D_{i0}=\Delta_{i}$ for all $0\leq i\leq n$.
\end{remark}

Note that each summand $S\otimes R_{q}^{m}\left\langle \mathbf{y}\right\rangle
$ of $S_{q}$ turns out to be a commutative graded $k$-algebra with its grading
$S\otimes R_{q}^{m}\left\langle \mathbf{y}\right\rangle =\bigoplus_{d}%
S^{d}\otimes R_{q}^{m}\left\langle \mathbf{y}\right\rangle $, and the diagonal
mapping $\epsilon_{m}:S\rightarrow S\otimes R_{q}^{m}\left\langle
\mathbf{y}\right\rangle $, $\epsilon_{m}\left(  a\right)  =\sum_{\left\langle
\alpha\right\rangle =m}a\left(  \mathbf{x}\right)  \mathbf{y}^{\alpha}$ is an
algebra homomorphism (see Subsection \ref{subsecDU}). In particular, $S\otimes
R_{q}^{m}\left\langle \mathbf{y}\right\rangle $ is an $S$-module along
$\epsilon_{m}$. Since $S^{e}\cdot\left(  S^{d}\otimes R_{q}^{m}\left\langle
\mathbf{y}\right\rangle \right)  \subseteq S^{d+e}\otimes R_{q}^{m}%
\left\langle \mathbf{y}\right\rangle $, it follows that $S\otimes R_{q}%
^{m}\left\langle \mathbf{y}\right\rangle $ is a graded $S$-module.

\subsection{The differential chains in $S_{q}$\label{SubsecDC1}}

Let $I_{m}\subseteq S\otimes R_{q}^{m}\left\langle \mathbf{y}\right\rangle $
be a graded $S$-submodule, that is, $I_{m}=\oplus_{d}I_{m}^{d}$ with its
$k$-subspaces $I_{m}^{d}\subseteq S^{d}\otimes R_{q}^{m}\left\langle
\mathbf{y}\right\rangle $ such that $S^{e}\cdot I_{m}^{d}\subseteq I_{m}%
^{e+d}$ for all $e$ and $d$. For example, so is every graded ideal $I_{m}$ of
the commutative $k$-algebra $S\otimes R_{q}^{m}\left\langle \mathbf{y}%
\right\rangle $. A family $\left\{  I_{m}\right\}  $ of graded $S$-submodules
is said to be \textit{a chain }if $z_{j}\left(  I_{m}\right)  \subseteq
I_{m+e_{j}}$ in $S_{q}$ for all $m$ and $j>n$. Since $I_{m}$ is an
$S$-submodule of $S\otimes R_{q}^{m}\left\langle \mathbf{y}\right\rangle $, we
have $z_{j}\left(  I_{m}\right)  =\epsilon_{m}\left(  z_{j}\right)
I_{m}=z_{j}I_{m}\subseteq I_{m}$ for all $0\leq j\leq n$ and $m$. We define
\textit{the sum of a chain }$\left\{  I_{m}\right\}  $ to be the subspace
$I=\oplus_{m}I_{m}$ in $S_{q}$. A chain $\left\{  I_{m}\right\}  $ is called
\textit{a terminating chain }if $I_{m}=S\otimes R_{q}^{m}\left\langle
\mathbf{y}\right\rangle $ for large $m$.

Consider a chain $\left\{  I_{m}\right\}  $ and put $\epsilon=\epsilon_{m}$
for the fixed $m$. Note that $I_{0}\subseteq S$, $I_{0}\cdot z_{j}%
=z_{j}\left(  I_{0}\right)  \subseteq I_{e_{j}}$ for all $j>n$, and $a\left(
\mathbf{x}\right)  \mathbf{y}^{\alpha}=a\left(  \mathbf{x}\right)
z_{n+1}^{\alpha_{n+1}}\cdots z_{v}^{\alpha_{v}}=z_{v}^{\alpha_{v}}\cdots
z_{n+1}^{\alpha_{n+1}}\left(  a\left(  \mathbf{x}\right)  \right)  $ for all
$a\in S$. It follows that $I_{0}\cdot\mathbf{y}^{\alpha}=\overline{\mathbf{y}%
}^{\alpha}\left(  I_{0}\right)  \subseteq I_{m}$ whenever $\left\langle
\alpha\right\rangle =m$, that is, $\epsilon\left(  I_{0}\right)  \subseteq
I_{m}$. We have the quotient $S$-module $T_{m}=\left(  S\otimes R_{q}%
^{m}\left\langle \mathbf{y}\right\rangle \right)  /I_{m}$. There is a unique
$S$-module homomorphism $\omega:S/I_{0}\rightarrow T_{m}$ (or $\omega_{m}$)
such that the following diagram%
\begin{equation}%
\begin{array}
[c]{ccc}%
S & \overset{\epsilon}{\longrightarrow} & S\otimes R_{q}^{m}\left\langle
\mathbf{y}\right\rangle \\
\downarrow &  & \downarrow\\
S/I_{0} & \overset{\omega}{\longrightarrow} & T_{m}%
\end{array}
\label{d1}%
\end{equation}
commutes whose columns are the quotient homomorphisms.

\begin{lemma}
\label{lemDC0}Let $\left\{  I_{m}\right\}  $ be a chain in $S_{q}$. Then
$I_{0}\cdot\left(  S\otimes R_{q}^{m}\left\langle \mathbf{y}\right\rangle
\right)  \subseteq I_{m}$ (in the $S$-module sense) and the canonical
$S$-homomorphism $T_{m}\rightarrow S/I_{0}\otimes_{S}T_{m}$, $x\mapsto
1\otimes_{S}x$ is an isomorphism. In particular, $T_{m}$ turns into an
$S/I_{0}$-module such that $\omega$ is a morphism of $S/I_{0}$-modules and its
$S$-module structure via pull back along the quotient homomorphism
$S\rightarrow S/I_{0}$ is reduced to the original one.
\end{lemma}

\begin{proof}
First note that $I_{0}$ is a graded ideal of the algebra $S$ and
$\epsilon\left(  I_{0}\right)  \subseteq I_{m}$. Take $h\in I_{0}$ and
$f=\left(  f_{\alpha}\right)  _{\alpha}\in S\otimes R_{q}^{m}\left\langle
\mathbf{y}\right\rangle $ with $f_{\alpha}\in S$, $\left\langle \alpha
\right\rangle =m$. Then $h\cdot f=\epsilon\left(  h\right)  f=\left(
hf_{\alpha}\right)  _{\alpha}=\sum_{\left\langle \alpha\right\rangle
=m}\left(  hf_{\alpha}\right)  \mathbf{y}^{\alpha}$ in $S\otimes R_{q}%
^{m}\left\langle \mathbf{y}\right\rangle $. But $\left(  hf_{\alpha}\right)
\mathbf{y}^{\alpha}=\left(  f_{\alpha}h\right)  \mathbf{y}^{\alpha}%
=\epsilon\left(  f_{\alpha}\right)  \left(  h\mathbf{y}^{\alpha}\right)
=f_{\alpha}\cdot\left(  h\mathbf{y}^{\alpha}\right)  $. As we have seen above
$h\mathbf{y}^{\alpha}=h\cdot\mathbf{y}^{\alpha}\in I_{0}\cdot\mathbf{y}%
^{\alpha}\subseteq I_{m}$. Therefore $\left(  hf_{\alpha}\right)
\mathbf{y}^{\alpha}=f_{\alpha}\cdot\left(  h\mathbf{y}^{\alpha}\right)  \in
f_{\alpha}\cdot I_{m}\subseteq S\cdot I_{m}\subseteq I_{m}$ for all $\alpha$,
$\left\langle \alpha\right\rangle =m$, that is, $h\cdot f\in I_{m}$. Whence
$I_{0}\cdot\left(  S\otimes R_{q}^{m}\left\langle \mathbf{y}\right\rangle
\right)  \subseteq I_{m}$.

Further, take $u=s^{\sim}\otimes_{S}z^{\sim}\in S/I_{0}\otimes_{S}T_{m}$ with
$s\in S$ and $z\in S\otimes R_{q}^{m}\left\langle \mathbf{y}\right\rangle $.
Then $u=s\cdot1\otimes_{S}z^{\sim}=1\otimes_{S}s\cdot z^{\sim}=1\otimes
_{S}\left(  s\cdot z\right)  ^{\sim}$. Thus the $S$-homomorphism
$\lambda:T_{m}\rightarrow S/I_{0}\otimes_{S}T_{m}$, $\lambda\left(  x\right)
=1\otimes_{S}x$ is onto. For each couple $\left(  s^{\sim},z^{\sim}\right)
\in S/I_{0}\times T_{m}$ define $\mu\left(  s^{\sim},z^{\sim}\right)  =s\cdot
z^{\sim}$. If $s^{\sim}=0$ then $s\cdot z\in I_{0}\cdot\left(  S\otimes
R_{q}^{m}\left\langle \mathbf{y}\right\rangle \right)  \subseteq I_{m}$ as we
have just proven above. It follows that $s\cdot z^{\sim}=\left(  s\cdot
z\right)  ^{\sim}=0$. Hence there is a well defined $S$-homomorphism
$\mu:S/I_{0}\otimes_{S}T_{m}\rightarrow T_{m}$, $\mu\left(  s^{\sim}%
\otimes_{S}x\right)  =s\cdot x$. Moreover, $\mu\left(  \lambda\left(
x\right)  \right)  =\mu\left(  1\otimes_{S}x\right)  =x$ for all $x\in T_{m}$.
Thus $\lambda$ is an isomorphism of $S$-modules. In particular, $T_{m}$ has
the $S/I_{0}$-module structure given by $s^{\sim}\cdot x=s\cdot x$ for all
$s\in S$ and $x\in T_{m}$. In this case, $\omega:S/I_{0}\rightarrow T_{m}$ is
an $S/I_{0}$-homomorphism, for $\omega\left(  s^{\sim}\cdot t^{\sim}\right)
=\epsilon\left(  s\cdot t\right)  ^{\sim}=s\cdot\epsilon\left(  t\right)
^{\sim}=s^{\sim}\cdot\epsilon\left(  t\right)  ^{\sim}=s^{\sim}\cdot
\omega\left(  t^{\sim}\right)  $, $s,t\in S$.
\end{proof}

Now we reflect the noncommutative property of the algebra $S_{q}$ over chains
based on Lemma \ref{lemOp0}.

\begin{definition}
\label{defDC}A chain $\left\{  I_{m}\right\}  $ in $S_{q}$ is said to be a
differential chain if $\Delta_{j,k}\left(  I_{m}\right)  +\nabla_{j,k}\left(
I_{m}\right)  \subseteq I_{m+e_{j}+k}$ for all $j$, $k$ and $m$.
\end{definition}

A trivial example of a differential chain is $\left\{  S\otimes R_{q}%
^{m}\left\langle \mathbf{y}\right\rangle \right\}  $ itself (see Lemma
\ref{lemOp0}). Since all $S$-modules $I_{m}$ are graded, we derive that a
chain $\left\{  I_{m}\right\}  $ is a differential chain iff $\Delta
_{j}\left(  I_{m}\right)  +\nabla_{j}\left(  I_{m}\right)  \subseteq\sum
_{k=0}^{q-e_{j}}I_{m+e_{j}+k}$ for all $j$, $k$ and $m$ thanks to Lemma
\ref{lemOp0}. If $I$ is a sum of a differential chain, then we put
$I^{d}=\oplus_{m=0}^{d}I_{m}^{d-m}$ to be a subspace of $S_{q}$. Taking into
account that $S_{q}$ is a graded noncommutative algebra with its grading
$S_{q}=\oplus_{d}S_{q}^{d}$, $S_{q}^{d}=\oplus_{m=0}^{d}S^{d-m}\otimes
R_{q}^{m}\left\langle \mathbf{y}\right\rangle $ (see Subsection
\ref{subsecPQN}), we obtain that $\oplus_{d}I^{d}=\oplus_{d}\oplus_{m=0}%
^{d}I_{m}^{d-m}=\oplus_{m}\oplus_{d\geq m}I_{m}^{d-m}=\oplus_{m}\oplus
_{e}I_{m}^{e}=\oplus_{m}I_{m}=I$.

\begin{lemma}
\label{lemDC1}If $\left\{  I_{m}\right\}  $ is a differential chain in $S_{q}$
then its sum $I=\oplus_{m}I_{m}$ is a graded two-sided ideal in $S_{q}$ with
the grading $I=\oplus_{d}I^{d}$.
\end{lemma}

\begin{proof}
First note that $z_{j}\left(  I_{m}^{e}\right)  \subseteq I_{m}^{e+e_{j}}$ for
$j\leq n$, and $z_{j}\left(  I_{m}^{e}\right)  \subseteq I_{m+e_{j}}^{e}$ for
$j>n$, therefore $z_{j}\left(  I_{m}^{e}\right)  \subseteq I_{m}^{e+e_{j}%
}+I_{m+e_{j}}^{e}$ for all $j$ and $e+m=d$. Using Lemma \ref{lemOp0} and
\ref{lemOp1}, we derive that
\begin{align*}
L_{j}\left(  I_{m}^{e}\right)   &  \subseteq z_{j}\left(  I_{m}^{e}\right)
+\Delta_{j}\left(  I_{m}^{e}\right)  =z_{j}\left(  I_{m}^{e}\right)
+\sum_{k=0}^{q-e_{j}}\Delta_{j,k}\left(  I_{m}^{e}\right) \\
&  \subseteq\left(  I_{m}^{e+e_{j}}+I_{m+e_{j}}^{e}\right)  +\sum
_{k=0}^{q-e_{j}}I_{m+e_{j}+k}^{e-k}\subseteq\oplus_{s=0}^{d+e_{j}}%
I_{s}^{d+e_{j}-s}=I^{d+e_{j}},
\end{align*}
that is, $L_{j}\left(  I^{d}\right)  \subseteq I^{d+e_{j}}$ for all $j$ and
$d$. Similarly, $R_{j}\left(  I^{d}\right)  \subseteq I^{d+e_{j}}$. In
particular, $\mathbf{z}^{\alpha}\cdot I^{d}+I^{d}\cdot\mathbf{z}^{\alpha
}=L_{\mathbf{z}}^{\alpha}\left(  I^{d}\right)  +R_{\overline{\mathbf{z}}%
}^{\alpha}\left(  I^{d}\right)  \subseteq I^{d+\left\langle \alpha
\right\rangle }$ for every ordered monomial $\mathbf{z}^{\alpha}$ in $S_{q}$.
Thus $I$ is a graded two-sided ideal of the algebra $S_{q}$.
\end{proof}

\begin{remark}
\label{remOp1}If the differential operators $t_{j,\mathbf{i}}=R\left(
\operatorname{ad}\left(  \overline{\mathbf{y}}\right)  ^{\mathbf{i}_{y}%
}\operatorname{ad}\left(  \overline{\mathbf{x}}\right)  ^{\mathbf{i}_{x}%
}\left(  z_{j}\right)  \right)  \overline{\partial}^{\mathbf{i}}$ (which are
building blocks of $\Delta_{j}$ and $\nabla_{j}$) do leave invariant the chain
$\left\{  I_{m}\right\}  $ in the sense that $t_{j,\mathbf{i}}\left(
I_{m}\right)  \subseteq I_{m+e_{j}+\left\vert \mathbf{i}_{x}\right\vert }$ for
all $j$, $\mathbf{i}$ and $m$, then $\left\{  I_{m}\right\}  $ is obviously a
differential chain. But the reverse is not true as it follows from the
forthcoming example of two lines in $\mathbb{P}_{k,q}^{2}$ for $q=2$.
\end{remark}

\subsection{The union of two lines in $\mathbb{P}_{k,2}^{2}$ as a differential
chain\label{subsecTL}}

For a while assume that $n=2$ and $q=2$, that is, we have $\mathbf{x}=\left(
x_{0},x_{1},x_{2}\right)  $ and $\mathbf{y=}\left(  y_{1},y_{2},y_{3}\right)
$ with $y_{1}=y_{02}$, $y_{2}=y_{01}$, $y_{3}=y_{12}$ (see Remark
\ref{remSpec1}). Thus the set of all open homogeneous prime ideals of the
algebra $S_{q}$ is reduced to $\mathbb{P}_{k}^{2}$. In the projective plane
$\mathbb{P}_{k}^{2}$ consider the union of two lines defined by the
(commutative) homogeneous ideal $I_{0}=\left(  f\right)  \subseteq S$, where
$f=\left(  x_{0}-x_{2}\right)  x_{1}^{d-1}\in S^{d}$ with $d\geq2$. Let
$J=\left(  x_{1},x_{0}-x_{2}\right)  \subseteq S$ be the commutative ideal of
the point $\left(  1:0:1\right)  \in\mathbb{P}_{k}^{2}$. Define $I_{1}=\left(
x_{1}^{d-1}\right)  y_{1}\oplus Jy_{2}\oplus Jy_{3}$ and $I_{m}=\left(
x_{1}^{d-1}\right)  y_{1}^{m}\oplus\bigoplus\limits_{\left\vert \alpha
\right\vert =m,\alpha_{1}\neq m}S\mathbf{y}^{\alpha}$ for all $m\geq2$. Note
that $I_{m}\subseteq S\otimes R_{q}^{2m}\left\langle \mathbf{y}\right\rangle $
is a graded (commutative) ideal for every $m$. Since $I_{0}\subseteq\left(
x_{1}^{d-1}\right)  \subseteq J$, it follows that $\left\{  I_{m}\right\}  $
is a chain in $S_{q}$.

\begin{lemma}
\label{lemDC2}The chain $\left\{  I_{m}\right\}  $ turns out to be a
differential chain in $S_{q}$, whose sum $I=\oplus_{m}I_{m}$ is not an open
two-sided ideal of $S_{q}$ equipped with NC-topology. Moreover, the operator
$\partial_{1}y_{1}$ does not leave invariant the chain $\left\{
I_{m}\right\}  $.
\end{lemma}

\begin{proof}
First let us write down the list of all differential operators $\Delta_{j}$
and $\nabla_{j}$ occurred in this special case (see Remark \ref{remSpec1}).
They are the following
\[
\Delta_{0}=0,\quad\Delta_{1}=-\partial_{0}y_{2},\quad\Delta_{2}=-\partial
_{0}y_{1}-\partial_{1}y_{3},\quad\nabla_{0}=-\partial_{1}y_{2}-\partial
_{2}y_{1},\quad\nabla_{1}=-\partial_{2}y_{3},\quad\nabla_{2}=0.
\]
Note that $\partial_{0}\left(  I_{0}\right)  +\partial_{2}\left(
I_{0}\right)  \subseteq\left(  x_{1}^{d-1}\right)  \subseteq J$ and
$\partial_{1}\left(  I_{0}\right)  \subseteq J$. It follows that $\Delta
_{1}\left(  I_{0}\right)  \subseteq\left(  x_{1}^{d-1}\right)  y_{2}\subseteq
Jy_{2}$, $\Delta_{2}\left(  I_{0}\right)  \subseteq\left(  x_{1}^{d-1}\right)
y_{1}+Jy_{3}$, $\nabla_{0}\left(  I_{0}\right)  \subseteq\left(  x_{1}%
^{d-1}\right)  y_{1}+Jy_{2}$ and $\nabla_{1}\left(  I_{0}\right)
\subseteq\left(  x_{1}^{d-1}\right)  y_{3}\subseteq Jy_{3}$, that is,
$\Delta_{i}\left(  I_{0}\right)  +\nabla_{j}\left(  I_{0}\right)  \subseteq
I_{1}$ for all possible $i$ and $j$. For the rest we need just to catch up
powers of $y_{1}$, which are occurred due to of $\Delta_{2}$ and $\nabla_{0}$.
Namely, $\Delta_{2}\left(  \left(  x_{1}^{d-1}\right)  y_{1}\right)
\subseteq\partial_{0}\left(  x_{1}^{d-1}\right)  y_{1}^{2}+Sy_{1}%
y_{3}\subseteq\left(  x_{1}^{d-1}\right)  y_{1}^{2}+Sy_{1}y_{3}$ and
$\nabla_{0}\left(  \left(  x_{1}^{d-1}\right)  y_{1}\right)  \subseteq
\partial_{1}\left(  x_{1}^{d-1}\right)  y_{1}y_{2}+\partial_{2}\left(
x_{1}^{d-1}\right)  y_{1}^{2}\subseteq\left(  x_{1}^{d-1}\right)  y_{1}%
^{2}+Sy_{1}y_{2}$. By induction on $k$, we derive that
\begin{align*}
\Delta_{2}\left(  \left(  x_{1}^{d-1}\right)  y_{1}^{k}\right)   &
\subseteq\partial_{0}\left(  x_{1}^{d-1}\right)  y_{1}^{k+1}+Sy_{1}^{k}%
y_{3}\subseteq\left(  x_{1}^{d-1}\right)  y_{1}^{k+1}+Sy_{1}^{k}y_{3},\\
\nabla_{0}\left(  \left(  x_{1}^{d-1}\right)  y_{1}^{k}\right)   &
=\partial_{1}\left(  x_{1}^{d-1}\right)  y_{1}^{k}y_{2}+\partial_{2}\left(
x_{1}^{d-1}\right)  y_{1}^{k+1}\subseteq\left(  x_{1}^{d-1}\right)
y_{1}^{k+1}+Sy_{1}^{k}y_{2}.
\end{align*}
Thus $\left\{  I_{m}\right\}  $ is a differential chain in $S_{q}$ (see
Definition \ref{defDC}). By Lemma \ref{lemDC1}, $I=\oplus_{m}I_{m}$ is a
two-sided ideal of $S_{q}$, and $\mathcal{J}_{2m}\left(  \mathbf{x}\right)
=\bigoplus_{k\geq m}\bigoplus\limits_{\left\vert \alpha\right\vert
=k}S\mathbf{y}^{\alpha}\nsubseteq I$ for all $m$, which means that $I$ is not
open. Finally, $\left(  \partial_{1}y_{1}\right)  \left(  \left(  x_{1}%
^{d-1}\right)  y_{1}^{m}\right)  \nsubseteq\left(  x_{1}^{d-1}\right)
y_{1}^{m+1}$, that is, $\partial_{1}y_{1}$ does not leave invariant the chain
$\left\{  I_{m}\right\}  $.
\end{proof}

Thus as we have confirmed above in Remark \ref{remOp1}, the conditions
$\Delta_{j}\left(  I_{m}\right)  \subseteq I_{m+1}$ and $\nabla_{j}\left(
I_{m}\right)  \subseteq I_{m+1}$ for all $j$ and $m$ are not equivalent to
$\partial_{t}y_{ij}\left(  I_{m}\right)  \subseteq I_{m+1}$ for all $t$, $i$,
$j$, $m$, $i<j$ as it follows from Lemma \ref{lemDC2}.

\subsection{Projective line of Heisenberg $\mathbb{P}_{k,2}^{1}$}

Consider the case of $n=1$ and $q=2$, that is, we deal with the projective
line of Heisenberg $\mathbb{P}_{k,2}^{1}$ \cite{Dproj1}. In this case
$\mathbf{x=}\left(  x_{0},x_{1}\right)  $ and $\mathbf{y=}\left(  y\right)  $
with $y=\left[  x_{0},x_{1}\right]  $. All subspaces $R_{q}^{2m}\left\langle
\mathbf{y}\right\rangle $ are one dimensional, $S_{q}^{d}=\bigoplus_{m=0}%
^{d}S^{d-2m}y^{m}$ and $S_{q}=\oplus_{m}Sy^{m}$. Moreover, $\mathcal{O}%
_{q}=\prod_{m=0}^{\infty}\mathcal{O}\left(  -2m\right)  $. A chain $\left\{
I_{m}\right\}  $ in $S_{q}$ is given by an increasing family of graded ideals
in $S$, that is, $I_{0}\subseteq I_{1}\subseteq\cdots$. The elementary
differential operators are the following $\Delta_{1}=-\partial_{0}y$ and
$\nabla_{1}=-\partial_{1}y$.

\begin{lemma}
A family $\left\{  I_{m}\right\}  $ of graded ideals in $S$ is a differential
chain iff $\partial_{i}\left(  I_{m}\right)  \subseteq I_{m+1}$ for all
$i=0,1$ and $m$. The sum $I=\oplus_{m}I_{m}$ of a differential chain $\left\{
I_{m}\right\}  $ is always open in $S_{q}$.
\end{lemma}

\begin{proof}
By Definition \ref{defDC}, $\left\{  I_{m}\right\}  $ is a differential chain
if it is an increasing chain and $\partial_{i}\left(  I_{m}\right)  \subseteq
I_{m+1}$ for all $i=0,1$ and $m$. Conversely, suppose the partial derivatives
leave invariant $\left\{  I_{m}\right\}  $ in the sense that $\partial
_{i}\left(  I_{m}\right)  \subseteq I_{m+1}$. Since the family $\left\{
I_{m}\right\}  $ consists of graded ideals, we derive that $I_{m}\subseteq
I_{m+1}$ based on Euler's lemma (see \cite[Exercise 1.5.8]{Harts}, recall that
the field $k$ has zero characteristic), that is, $\left\{  I_{m}\right\}  $ is
a chain, thereby it is a differential chain. Finally, if $\left\{
I_{m}\right\}  $ is a nonzero differential chain then $\partial_{0}%
^{e}\partial_{1}^{d}\left(  I_{m}\right)  \subseteq I_{m+e+d}$ for all $e,d$
and $m$, that is, eventually $I_{m}$ contains the unit. It follows that
$\left\{  I_{m}\right\}  $ is a terminating chain. In particular,
$I=\oplus_{m}I_{m}$ is an open graded ideal of $S_{q}$.
\end{proof}

Thus in the case of the projective line of Heisenberg $\mathbb{P}_{k,2}^{1}$
all differential chains are terminating.

\subsection{NC-graded ideals of $S_{q}$}

Now let $I=\oplus_{d}I^{d}$ be a two-sided graded ideal of $S_{q}$. We say
that $I$ is \textit{an NC-graded ideal} if $I^{d}=\bigoplus_{m=0}^{d}I^{d}%
\cap\left(  S^{d-m}\otimes R_{q}^{m}\left\langle \mathbf{y}\right\rangle
\right)  $ for every $d$. Recall that $I$ is open whenever $\mathcal{J}%
_{m}\left(  \mathbf{x}\right)  \subseteq I$ for large $m$, where
$\mathcal{J}_{m}\left(  \mathbf{x}\right)  =\oplus_{t\geq m}S\otimes R_{q}%
^{t}\left\langle \mathbf{y}\right\rangle $.

As above let us consider a differential chain $\left\{  I_{m}\right\}  $ in
$S_{q}$. By Lemma \ref{lemDC1}, its sum $I=\oplus_{m}I_{m}$ is a graded
two-sided ideal in $S_{q}$ with the grading $I=\oplus_{d}I^{d}$, $I^{d}%
=\oplus_{m=0}^{d}I_{m}^{d-m}$, where $I_{m}^{d-m}\subseteq S^{d-m}\otimes
R_{q}^{m}\left\langle \mathbf{y}\right\rangle $ is a subspace. It follows that
$I_{m}^{d-m}=I^{d}\cap\left(  S^{d-m}\otimes R_{q}^{m}\left\langle
\mathbf{y}\right\rangle \right)  $ for all $m$, that is, $I$ is an NC-graded ideal.

\begin{lemma}
\label{lemSGI1}Let $I\subseteq S_{q}$ be an open, two-sided, graded ideal.
Then $I$ is an NC-graded ideal.
\end{lemma}

\begin{proof}
First assume that $\mathcal{J}_{1}\left(  \mathbf{x}\right)  \subseteq I$.
Take $a\in I^{d}\subseteq S_{q}^{d}$ with its unique expansion $a=\sum
_{i=0}^{d}a_{i}$, $a_{i}\in S^{d-i}\otimes R_{q}^{i}\left\langle
\mathbf{y}\right\rangle $. Since $b=\sum_{i=1}^{d}a_{i}\in\mathcal{J}%
_{1}\left(  \mathbf{x}\right)  \subseteq I$ and the ideal $I$ is graded, it
follows that $a_{0}=a-b\in I\cap S^{d}\subseteq I\cap S_{q}^{d}=I^{d}$ or
$a_{0}\in I^{d}\cap S^{d}$. Thus $I^{d}=\left(  I^{d}\cap S^{d}\right)
\oplus\bigoplus_{m>0}\left(  S^{d-m}\otimes R_{q}^{m}\left\langle
\mathbf{y}\right\rangle \right)  $, that is, $I$ is an NC-graded ideal.

Now suppose that $\mathcal{J}_{m}\left(  \mathbf{x}\right)  \subseteq I$ for
$m>1$. Note that $A=S_{q}/\mathcal{J}_{m}\left(  \mathbf{x}\right)  $ is a
graded NC-nilpotent algebra, namely, $A=\oplus_{d}A^{d}$ with $A^{d}%
=\bigoplus_{t<m}S^{d-t}\otimes R_{q}^{t}\left\langle \mathbf{y}\right\rangle
$. Moreover, $I$ is identified with a graded ideal of $A$. Put $B$ to be the
quotient algebra $S_{q}/\mathcal{J}_{m-1}\left(  \mathbf{x}\right)  $ with the
quotient homomorphism $\pi:A\rightarrow B$. Notice that $\mathcal{J}%
_{m-1}\left(  \mathbf{x}\right)  =\left(  S\otimes R_{q}^{m-1}\left\langle
\mathbf{y}\right\rangle \right)  \oplus\mathcal{J}_{m}\left(  \mathbf{x}%
\right)  $ and $\ker\left(  \pi\right)  =S\otimes R_{q}^{m-1}\left\langle
\mathbf{y}\right\rangle $. Since $I=\oplus_{d}I^{d}$, $I^{d}\subseteq A^{d}$,
and $\pi$ is onto, it follows that $\pi\left(  I^{d}\right)  \subseteq
\pi\left(  A^{d}\right)  =\bigoplus_{t<m-1}S^{d-t}\otimes R_{q}^{t}%
\left\langle \mathbf{y}\right\rangle =B^{d}$ and $\pi\left(  I\right)
=\oplus_{d}\pi\left(  I^{d}\right)  $ is a graded ideal of $B$. Take $a\in
I^{d}$ with its expansion $a=a_{0}+\cdots+a_{m-1}$ in $A^{d}$. By induction
hypothesis, $\pi\left(  a_{i}\right)  =\pi\left(  b_{i}\right)  $ for some
$b_{i}\in I^{d}\cap\left(  S^{d-i}\otimes R_{q}^{i}\left\langle \mathbf{y}%
\right\rangle \right)  $, $0\leq i\leq m-2$. Hence $a_{i}=b_{i}+c_{i}$ for
some $c_{i}\in S\otimes R_{q}^{m-1}\left\langle \mathbf{y}\right\rangle $. But
$A=\bigoplus_{t=0}^{m-1}S\otimes R_{q}^{t}\left\langle \mathbf{y}\right\rangle
$, therefore $c_{i}=0$ for all $i$. In particular, $a_{m-1}=a-b_{0}%
-\cdots-b_{m-2}\in I^{d}\cap\left(  S^{d-m+1}\otimes R_{q}^{m-1}\left\langle
\mathbf{y}\right\rangle \right)  $. Thus $I$ is an NC-graded ideal of $S_{q}$.
\end{proof}

The construction from Subsection \ref{subsecTL} provides an example of an
NC-graded ideal which is not open.

\begin{proposition}
\label{propSG1}Let $I$ be an NC-graded ideal of $S_{q}$. Then $I$ is the sum
of a certain differential chain in $S_{q}$. If $I$ is open, it is the sum of a
terminating differential chain.
\end{proposition}

\begin{proof}
Put $I_{m}^{e}=I^{d}\cap\left(  S^{e}\otimes R_{q}^{m}\left\langle
\mathbf{y}\right\rangle \right)  $ with $d=e+m$, and $I_{m}=\oplus_{e}%
I_{m}^{e}\subseteq I\cap\left(  S\otimes R_{q}^{m}\left\langle \mathbf{y}%
\right\rangle \right)  $. Take $a\in I_{m}^{e}$. Note that
\[
z_{j}\cdot a\in S_{q}^{e_{j}}\cdot I_{m}^{e}\subseteq S_{q}^{e_{j}}\cdot
I^{d}\subseteq I^{d+e_{j}}=\bigoplus_{t=0}^{d+e_{j}}I^{d+e_{j}}\cap\left(
S^{d+e_{j}-t}\otimes R_{q}^{t}\left\langle \mathbf{y}\right\rangle \right)
=\bigoplus_{t=0}^{d+e_{j}}I_{k}^{d+e_{j}-t}%
\]
for all $j$. If $j\leq n$ then $z_{j}=x_{j}$ and
\begin{align*}
x_{j}\cdot a  &  =x_{j}\left(  a\right)  +\Delta_{j}\left(  a\right)  ,\quad
x_{j}\left(  a\right)  \in S^{e+1}\otimes R_{q}^{m}\left\langle \mathbf{y}%
\right\rangle ,\\
\Delta_{j}\left(  a\right)   &  \in\Delta_{j}\left(  S^{e}\otimes R_{q}%
^{m}\left\langle \mathbf{y}\right\rangle \right)  \subseteq\sum_{k=0}%
^{q-e_{j}}S^{e-k}\otimes R_{q}^{m+k+e_{j}}\left\langle \mathbf{y}%
\right\rangle
\end{align*}
thanks to Lemma \ref{lemOp1}, that is, $x_{j}\cdot a\in S^{e+1}\otimes
R_{q}^{m}\left\langle \mathbf{y}\right\rangle +\sum_{k=0}^{q-1}S^{e-k}\otimes
R_{q}^{m+k+1}\left\langle \mathbf{y}\right\rangle $. If $j>n$ then
$z_{j}\left(  a\right)  \in S^{e}\otimes R_{q}^{m+e_{j}}\left\langle
\mathbf{y}\right\rangle $ and
\[
z_{j}\cdot a\in S^{e}\otimes R_{q}^{m+e_{j}}\left\langle \mathbf{y}%
\right\rangle +\sum_{k=0}^{q-e_{j}}S^{e-k}\otimes R_{q}^{m+k+e_{j}%
}\left\langle \mathbf{y}\right\rangle =\sum_{k=0}^{q-e_{j}}S^{e-k}\otimes
R_{q}^{m+k+e_{j}}\left\langle \mathbf{y}\right\rangle .
\]
But as we have seen above $z_{j}\cdot a\in\bigoplus_{t=0}^{d+e_{j}}%
I_{t}^{d+e_{j}-t}$. It follows that $x_{j}\left(  a\right)  \in I_{m}^{e+1}$
(put $t=m$), $z_{j}\left(  a\right)  \in I_{m+e_{j}}^{e}$, $j>n$ (put
$t=m+e_{j}$), and $\Delta_{j}\left(  a\right)  \in\sum_{k=0}^{q-e_{j}%
}I_{m+k+e_{j}}^{e-k}$ (put $t=m+k+e_{j}$ for $0\leq k\leq q-e_{j}$). Thus%
\[
x_{j}\left(  I_{m}^{e}\right)  \subseteq I_{m}^{e+1},\quad z_{j}\left(
I_{m}^{e}\right)  \in I_{m+e_{j}}^{e},\text{ }j>n,\text{ and }\Delta
_{j,k}\left(  I_{m}^{e}\right)  \subseteq I_{m+k+e_{j}}^{e-k}\text{, }0\leq
j\leq v\text{, }0\leq k\leq q-e_{j}.
\]
Similar inclusions take place for the operators $\nabla_{j}$. Consequently
$I_{m}=\oplus_{e}I_{m}^{e}$ is a graded $S$-submodule in $S\otimes R_{q}%
^{m}\left\langle \mathbf{y}\right\rangle $ such that $z_{j}\left(
I_{m}\right)  \subseteq I_{m+e_{j}}$, $\Delta_{j,k}\left(  I_{m}\right)
\subseteq I_{m+k+e_{j}}$ and $\nabla_{j,k}\left(  I_{m}\right)  \subseteq
I_{m+k+e_{j}}$ in $S_{q}$ for all $m$, $i,j$, $k$. By Definition \ref{defDC},
$\left\{  I_{m}\right\}  $ is a differential chain in $S_{q}$. Finally,
$I^{d}=\oplus_{m=0}^{d}I^{d}\cap\left(  S^{d-m}\otimes R_{q}^{m}\left\langle
\mathbf{y}\right\rangle \right)  =\oplus_{m=0}^{d}I_{m}^{d-m}$ and%
\[
I=\oplus_{d}I^{d}=\oplus_{d}\oplus_{m=0}^{d}I_{m}^{d-m}=\oplus_{m}%
\oplus_{d\geq m}I_{m}^{d-m}=\oplus_{k}\oplus_{e}I_{m}^{e}=\oplus_{m}I_{m},
\]
that is, $I$ is the sum of the differential chain $\left\{  I_{m}\right\}  $.
If $I$ contains $\mathcal{J}_{m}\left(  \mathbf{x}\right)  $ then $S\otimes
R_{q}^{t}\left\langle \mathbf{y}\right\rangle \subseteq I$ for all $t\geq m$,
which in turn implies that%
\[
I_{m}=\oplus_{e}I_{m}^{e}=\oplus_{e}I^{e+m}\cap\left(  S^{e}\otimes R_{q}%
^{m}\left\langle \mathbf{y}\right\rangle \right)  =\oplus_{e}\left(
S^{e}\otimes R_{q}^{m}\left\langle \mathbf{y}\right\rangle \right)  =S\otimes
R_{q}^{m}\left\langle \mathbf{y}\right\rangle ,
\]
that is, $\left\{  I_{m}\right\}  $ is a terminating chain.
\end{proof}

\subsection{Noncommutative topological localizations\label{subsecNTL}}

Take a homogeneous polynomial $h\in S$ of positive degree. Actually, all
arguments below take place for an arbitrary nonconstant polynomial but we
mainly focus on the homogenous case. There is a well defined topological
localization $S_{q,h\left(  \mathbf{x}\right)  }$ of the algebra $S_{q}$ at
$h\left(  \mathbf{x}\right)  $(see \cite{Kap}, \cite{DIZV}), where $h\left(
\mathbf{x}\right)  $ is the realization of $h$ in $S_{q}$ through the ordered
calculus. Note that $S_{q,h\left(  \mathbf{x}\right)  }=\prod\limits_{m=0}%
^{\infty}S_{h}\otimes R_{q}^{m}\left\langle \mathbf{y}\right\rangle $ is an
NC-complete algebra, whose NC-topology is given by the filtration $\left\{
\mathcal{J}_{m}\left(  h\right)  \right\}  $ of two-sided ideals, where
$\mathcal{J}_{m}\left(  h\right)  =\prod\limits_{j=m}^{\infty}S_{h}\otimes
R_{q}^{j}\left\langle \mathbf{y}\right\rangle $ (see \cite{DIZV}). Moreover,
$S_{q,h\left(  \mathbf{x}\right)  }$ admits the topological grading
$S_{q,h\left(  \mathbf{x}\right)  }=\widetilde{\bigoplus}_{d\in\mathbb{Z}%
}S_{q,h\left(  \mathbf{x}\right)  }^{d}$ (the algebraic sum is dense in) with
$S_{q,h\left(  \mathbf{x}\right)  }^{d}=\prod\limits_{m=0}^{\infty}S_{h}%
^{d-m}\otimes R_{q}^{m}\left\langle \mathbf{y}\right\rangle $. Thus
$S_{q,h\left(  \mathbf{x}\right)  }^{d}\cdot S_{q,h\left(  \mathbf{x}\right)
}^{e}\subseteq S_{q,h\left(  \mathbf{x}\right)  }^{d+e}$ for all $d$ and $e$
(see \cite[Section 4]{Dproj1}), and $S_{q}\subseteq S_{q,h\left(
\mathbf{x}\right)  }$ is a noncommutative subalgebra along the diagonal inclusion.

Note that each (commutative) fraction $g\in S_{h}$ defines the diagonal
operator on $S_{q,h\left(  \mathbf{x}\right)  }$ denoted by $g$ as well. Since
$\partial_{i}\left(  S_{h}\right)  \subseteq S_{h}$, $0\leq i\leq n$, it
follows that the operators $\Delta_{j,k}$ and $\nabla_{j,k}$ admit extensions
up to operators over $S_{q,h\left(  \mathbf{x}\right)  }$. We use the same
notations for all these extensions. For example, $\Delta_{j,k}=\sum
_{\left\vert \mathbf{i}_{x}\right\vert =k}\sum_{\mathbf{i}_{y},\left\vert
\mathbf{i}\right\vert _{j}>0}R\left(  \operatorname{ad}\left(  \overline
{\mathbf{y}}\right)  ^{\mathbf{i}_{y}}\operatorname{ad}\left(  \overline
{\mathbf{x}}\right)  ^{\mathbf{i}_{x}}\left(  z_{j}\right)  \right)
\overline{\partial}^{\mathbf{i}},$ and $\nabla_{j,k}=-\sum_{\left\vert
\mathbf{i}_{x}\right\vert =k}\sum_{\mathbf{i}_{y},\left\vert \mathbf{i}%
\right\vert _{j}=0}R\left(  \operatorname{ad}\left(  \overline{\mathbf{y}%
}\right)  ^{\mathbf{i}_{y}}\operatorname{ad}\left(  \overline{\mathbf{x}%
}\right)  ^{\mathbf{i}_{x}}\left(  z_{j}\right)  \right)  \overline{\partial
}^{\mathbf{i}}$. Based on the results from \cite{Dproj1}, we conclude that
$S_{q,h\left(  \mathbf{x}\right)  }$ is obtained from $S_{h}$ using a special
construction $\mathcal{F}_{q}\left(  S_{h}\right)  $, where the multiplication
rule obeys the same rule as in $S_{q}$. Using Corollary \ref{corKey1}, we
deduce that
\begin{equation}
\operatorname{ad}\left(  \mathbf{x}\right)  ^{\mathbf{i}}\left(  \Delta
_{j}\right)  \left(  S_{h}^{e}\otimes R_{q}^{m}\left\langle \mathbf{y}%
\right\rangle \right)  +\operatorname{ad}\left(  \mathbf{x}\right)
^{\mathbf{i}}\left(  \nabla_{j}\right)  \left(  S_{h}^{e}\otimes R_{q}%
^{m}\left\langle \mathbf{y}\right\rangle \right)  \subseteq\sum_{t=0}%
^{q-\left\vert \mathbf{i}\right\vert -e_{j}}S_{h}^{e-t}\otimes R_{q}%
^{m+t+\left\vert \mathbf{i}\right\vert +e_{j}}\left\langle \mathbf{y}%
\right\rangle \label{s1}%
\end{equation}
for all $\mathbf{i}\in\mathbb{Z}_{+}^{n+1}$, $j$, $e$ and $m$. By Lemma
\ref{lemOp1}, $L_{j}=L_{z_{j}}=z_{j}+\Delta_{j}$ over $S_{q,h\left(
\mathbf{x}\right)  }$ too, which is given by the left regular representation
$L:S_{q}\rightarrow\mathcal{L}_{k}\left(  S_{q,h\left(  \mathbf{x}\right)
}\right)  $, $L\left(  a\right)  =L_{a}$, $L_{a}\left(  x\right)  =a\cdot x$,
$a\in S_{q}$, $x\in S_{q,h\left(  \mathbf{x}\right)  }$ of $S_{q}$ over the
localization $S_{q,h\left(  \mathbf{x}\right)  }$. Similarly, we have the
right regular (anti) representation $R:S_{q}\rightarrow\mathcal{L}_{k}\left(
S_{q,h\left(  \mathbf{x}\right)  }\right)  $ such that $R_{j}=z_{j}+\nabla
_{j}$ over $S_{q,h\left(  \mathbf{x}\right)  }$ and $\left[  z_{j},\nabla
_{j}\right]  =0$ for all $j$ (see Lemma \ref{lemOp1}). All these considered
operators leave invariant the filtration $\left\{  \mathcal{J}_{m}\left(
h\right)  \right\}  $, which means that they are continuous operators with
respect to the NC-topology of $S_{q,h\left(  \mathbf{x}\right)  }$.

\begin{lemma}
\label{lemLR0}For every $j$ we have $L_{j}\left(  S_{q,h\left(  \mathbf{x}%
\right)  }^{d}\right)  +R_{j}\left(  S_{q,h\left(  \mathbf{x}\right)  }%
^{d}\right)  \subseteq S_{q,h\left(  \mathbf{x}\right)  }^{d+e_{j}}$.
\end{lemma}

\begin{proof}
If $e_{j}=q$ then $L_{j}=R_{j}=z_{j}$ and $z_{j}\left(  S_{h}^{d-m}\otimes
R_{q}^{m}\left\langle \mathbf{y}\right\rangle \right)  \subseteq S_{h}%
^{d-m}\otimes R_{q}^{m+q}\left\langle \mathbf{y}\right\rangle $ for all $d$
and $m$. Therefore $z_{j}\left(  S_{q,h\left(  \mathbf{x}\right)  }%
^{d}\right)  \subseteq S_{q,h\left(  \mathbf{x}\right)  }^{d+q}$ for all $d$.
In the general case, we have $R_{j}\left(  S_{q,h\left(  \mathbf{x}\right)
}^{d}\right)  \subseteq z_{j}\left(  S_{q,h\left(  \mathbf{x}\right)  }%
^{d}\right)  +\nabla_{j}^{\prime}\left(  S_{q,h\left(  \mathbf{x}\right)
}^{d}\right)  \subseteq S_{q,h\left(  \mathbf{x}\right)  }^{d+e_{j}}%
+\nabla_{j}^{\prime}\left(  S_{q,h\left(  \mathbf{x}\right)  }^{d}\right)  $
due to Remark \ref{remRight}, where
\[
\nabla_{j}^{\prime}=-\sum_{\mathbf{i}\in\mathbb{Z}_{+}^{v+1},\left\vert
\mathbf{i}\right\vert ^{j}>0}L\left(  \operatorname{ad}\left(  \mathbf{z}%
\right)  ^{\mathbf{i}}\left(  z_{j}\right)  \right)  \partial^{\mathbf{i}%
}=\sum_{\mathbf{i}\in\mathbb{Z}_{+}^{v+1},\left\vert \mathbf{i}\right\vert
^{j}>0}\sum_{e_{u}=\left\langle \mathbf{i}\right\rangle +e_{j}}\lambda
_{u,\mathbf{i}}L_{u}\partial^{\mathbf{i}}.
\]
By induction hypothesis,
\[
L_{u}\partial^{\mathbf{i}}\left(  S_{h}^{d-m}\otimes R_{q}^{m}\left\langle
\mathbf{y}\right\rangle \right)  \subseteq L_{u}\left(  S_{h}^{d-m-\left\vert
\mathbf{i}_{x}\right\vert }\otimes R_{q}^{m-\left\langle \mathbf{i}%
_{y}\right\rangle }\left\langle \mathbf{y}\right\rangle \right)  \subseteq
L_{u}\left(  S_{q,h\left(  \mathbf{x}\right)  }^{d-\left\langle \mathbf{i}%
\right\rangle }\right)  \subseteq S_{q,h\left(  \mathbf{x}\right)
}^{d-\left\langle \mathbf{i}\right\rangle +e_{u}}=S_{q,h\left(  \mathbf{x}%
\right)  }^{d+e_{j}},
\]
where $e_{u}-\left\langle \mathbf{i}\right\rangle =e_{j}$. Similarly, using
the formula $L_{j}=z_{j}+\Delta_{j}$ one can reduce the assertion to the right
operators $R_{u}$ with $e_{u}>e_{j}$ and apply induction hypothesis as above.
\end{proof}

Now let $I_{m}\subseteq S\otimes R^{m}\left\langle \mathbf{y}\right\rangle $
be $S$-submodules (not necessarily graded) such that $I=\oplus_{m}I_{m}$ is a
two-sided ideal of $S_{q}$. Since $L_{j}\left(  I_{m}\right)  +R_{j}\left(
I_{m}\right)  \subseteq I$, we derive that $x_{j}\left(  I_{m}\right)
\subseteq I_{m}$, $z_{j}\left(  I_{m}\right)  \subseteq I_{m+e_{j}}$, $j>n$,
and $\Delta_{j,k}\left(  I_{m}\right)  +\nabla_{j,k}\left(  I_{m}\right)
\subseteq I_{m+e_{j}+k}$ for all $j$, $m$ and $0\leq k\leq q-e_{j}$. So is the
sum of a differential chain $\left\{  I_{m}\right\}  $ in $S_{q}$ (see to the
proof of Proposition \ref{propSG1}). For every $m$ there is a well defined
commutative localization $I_{m,h}$ which is an $S_{h}$-submodule of
$S_{h}\otimes R^{m}\left\langle \mathbf{y}\right\rangle $. Namely,
$I_{m,h}=\left\{  a/h^{l}:a\in I_{m},l\geq0\right\}  $, where $a/h^{l}$ is a
commutative fraction and $a/h^{l}=\sum_{\left\langle \alpha\right\rangle
=m}\left(  a_{\alpha}/h^{l}\right)  \left(  \mathbf{x}\right)  \mathbf{y}%
^{\alpha}$ whenever $a=\sum_{\left\langle \alpha\right\rangle =m}a_{\alpha
}\left(  \mathbf{x}\right)  \mathbf{y}^{\alpha}\in I_{m}$. These submodules
define the closed subspace $M=\prod\limits_{m=0}^{\infty}I_{m,h}$ in
$S_{q,h\left(  \mathbf{x}\right)  }$. If $\left\{  I_{m}\right\}  $ is a
differential chain or $I$ is an NC-graded ideal (see Proposition
\ref{propSG1}), then we have a grading of $M$ as a topological direct sum
$M=\widetilde{\oplus}_{d\in\mathbb{Z}}M^{d}$ of its closed $k$-subspaces
$M^{d}=\prod\limits_{m=0}^{\infty}I_{m,h}^{d-m}$, where $I_{m,h}^{e}%
=I_{m,h}\cap\left(  S_{h}^{e}\otimes R^{m}\left\langle \mathbf{y}\right\rangle
\right)  $, $e\in\mathbb{Z}$, $m\in\mathbb{Z}_{+}$. We have also\textit{ the
topological localization} $I_{h\left(  \mathbf{x}\right)  }$ \textit{of }%
$I$\textit{ at} $h\left(  \mathbf{x}\right)  $ to be the closure of the linear
span of all noncommutative right fractions $\left\{  a\cdot h\left(
\mathbf{x}\right)  ^{-d}:a\in I,d\geq0\right\}  $ in $S_{q,h\left(
\mathbf{x}\right)  }$. It can be proven (see \cite{Kap}, \cite{DIZV} or
\cite[Remark 3.2.2]{DComA}) that $I_{h\left(  \mathbf{x}\right)  }$ coincides
with the closed linear span of all noncommutative left fractions $h\left(
\mathbf{x}\right)  ^{-d}\cdot a$, $a\in I$, $d\geq0$ (see also below Lemma
\ref{lemrightF}).

\begin{lemma}
\label{lemtp1}Let $I$ be an NC-graded ideal of $S_{q}$, and let $M=\prod
\limits_{m=0}^{\infty}I_{m,h}$ be the related closed $k$-subspace in
$S_{q,h\left(  \mathbf{x}\right)  }$ with its grading $M=\widetilde{\oplus
}_{d\in\mathbb{Z}}M^{d}$. All operators $g$, $\Delta_{j}$ and $\nabla_{j}$ on
$S_{q,h\left(  \mathbf{x}\right)  }$ leave invariant the subspace $M$, where
$g\in S_{h}$ is acting as the diagonal operator. Moreover, $g\left(
M^{d}\right)  \subseteq M^{d+\deg\left(  g\right)  }$ and $\Delta_{j}\left(
M^{d}\right)  +\nabla_{j}\left(  M^{d}\right)  \subseteq M^{d+e_{j}}$ for all
homogeneous $g\in S_{h}$, $j$ and $d$. In particular, $L_{j}\left(
M^{d}\right)  +R_{j}\left(  M^{d}\right)  \subseteq M^{d+e_{j}}$ for all $j$,
$d$.
\end{lemma}

\begin{proof}
One can assume that $g=b/h^{p}\in S_{h}^{t}$ is a homogenous fraction. If
$a/h^{l}\in I_{m,h}^{e}$ then $g\left(  a/h^{l}\right)  =\left(  ba\right)
/h^{p+l}\in I_{m,h}^{e+t}$, that is, $g\left(  I_{m,h}^{e}\right)  \subseteq
M^{d+t}$ whenever $e+m=d$. Thus $g\left(  M^{d}\right)  \subseteq M^{d+t}$.
Take $u=a/h^{l}\in I_{m,h}^{e}$ with a homogeneous $a\in I_{m}$. It means that
$a=\sum_{\left\langle \alpha\right\rangle =m}a_{\alpha}\left(  \mathbf{x}%
\right)  \mathbf{y}^{\alpha}$, $a_{\alpha}\in S^{\deg\left(  a\right)  }$, and
$u=\sum_{\alpha}u_{\alpha}$ with $u_{\alpha}=\left(  a_{\alpha}/h^{l}\right)
\left(  \mathbf{x}\right)  \mathbf{y}^{\alpha}$, $\deg\left(  a_{\alpha}%
/h^{l}\right)  =e$. Since we deal with the linear operators, one can assume
that $u=\left(  a/h^{l}\right)  \left(  \mathbf{x}\right)  \mathbf{y}^{\alpha
}$ with $a/h^{l}\in S_{h}^{e}$. By its very definition of $\Delta_{j}$, we
have
\[
\Delta_{j}\left(  u\right)  =\sum_{\mathbf{i}\in\mathbb{Z}_{+}^{v+1}%
,\left\vert \mathbf{i}\right\vert _{j}>0}\overline{\partial}^{\mathbf{i}%
}\left(  u\right)  \operatorname{ad}\left(  \overline{\mathbf{z}}\right)
^{\mathbf{i}}\left(  z_{j}\right)  =\sum_{\mathbf{i}\in\mathbb{Z}_{+}%
^{v+1},\left\vert \mathbf{i}\right\vert _{j}>0}\overline{\partial}%
^{\mathbf{i}_{x}}\left(  a/h^{l}\right)  \left(  \mathbf{x}\right)
\overline{\partial}^{\mathbf{i}_{y}}\left(  \mathbf{y}^{\alpha}\right)
\operatorname{ad}\left(  \overline{\mathbf{y}}\right)  ^{\mathbf{i}_{y}%
}\operatorname{ad}\left(  \overline{\mathbf{x}}\right)  ^{\mathbf{i}_{x}%
}\left(  z_{j}\right)  .
\]
But $\overline{\partial}^{\mathbf{i}_{x}}\left(  a/h^{l}\right)  =\sum
_{\gamma_{x}+\beta_{x}=\mathbf{i}_{x}}\overline{\partial}^{\gamma_{x}}\left(
1/h^{l}\right)  \overline{\partial}^{\beta_{x}}\left(  a\right)  $ (see
Subsection \ref{SubsecCF}). Using Lemma \ref{lemOp1}, we derive that
\begin{align*}
\Delta_{j}\left(  u\right)   &  =\sum_{\gamma_{x}}\overline{\partial}%
^{\gamma_{x}}\left(  1/h^{l}\right)  \sum_{\left\vert \beta+\gamma\right\vert
_{j}>0}\overline{\partial}^{\beta_{x}}\left(  a\right)  \left(  \mathbf{x}%
\right)  \overline{\partial}^{\beta_{y}}\left(  \mathbf{y}^{\alpha}\right)
\operatorname{ad}\left(  \overline{\mathbf{y}}\right)  ^{\beta_{y}%
}\operatorname{ad}\left(  \overline{\mathbf{x}}\right)  ^{\beta_{x}+\gamma
_{x}}\left(  z_{j}\right) \\
&  =\sum_{\gamma_{x}}\overline{\partial}^{\gamma_{x}}\left(  1/h^{l}\right)
\operatorname{ad}\left(  \mathbf{x}\right)  ^{\gamma_{x}}\left(  \Delta
_{j}\right)  \left(  a\left(  \mathbf{x}\right)  \mathbf{y}^{\alpha}\right)  .
\end{align*}
If $j>n$ then $-\nabla_{j}\left(  u\right)  =\sum_{\mathbf{i}\in\mathbb{Z}%
_{+}^{v+1},\left\vert \mathbf{i}\right\vert _{j}=0}\overline{\partial
}^{\mathbf{i}}\left(  u\right)  \operatorname{ad}\left(  \overline{\mathbf{z}%
}\right)  ^{\mathbf{i}}\left(  z_{j}\right)  =\sum_{\left\vert \mathbf{i}%
\right\vert _{j}=0}u\left(  \mathbf{x}\right)  \overline{\partial}%
^{\mathbf{i}_{y}}\left(  \mathbf{y}^{\alpha}\right)  \operatorname{ad}\left(
\overline{\mathbf{y}}\right)  ^{\mathbf{i}_{y}}\left(  z_{j}\right)  $. But if
$j\leq n$ then%
\begin{align*}
-\nabla_{j}\left(  u\right)   &  =\sum_{\mathbf{i}\in\mathbb{Z}_{+}%
^{v+1},\left\vert \mathbf{i}\right\vert _{j}=0}\overline{\partial}%
^{\mathbf{i}}\left(  u\right)  \operatorname{ad}\left(  \overline{\mathbf{z}%
}\right)  ^{\mathbf{i}}\left(  z_{j}\right) \\
&  =\sum_{\left\vert \beta+\gamma\right\vert _{j}=0}\overline{\partial
}^{\gamma_{x}}\left(  1/h^{l}\right)  \overline{\partial}^{\beta_{x}}\left(
a\right)  \left(  \mathbf{x}\right)  \overline{\partial}^{\beta_{y}}\left(
\mathbf{y}^{\alpha}\right)  \operatorname{ad}\left(  \overline{\mathbf{y}%
}\right)  ^{\beta_{y}}\operatorname{ad}\left(  \overline{\mathbf{x}}\right)
^{\beta_{x}+\gamma_{x}}\left(  z_{j}\right) \\
&  =\sum_{\left\vert \gamma_{x}\right\vert _{j}=0}\overline{\partial}%
^{\gamma_{x}}\left(  1/h^{l}\right)  \sum_{\left\vert \beta_{x}\right\vert
_{j}=0}\overline{\partial}^{\beta_{x}}\left(  a\right)  \left(  \mathbf{x}%
\right)  \overline{\partial}^{\beta_{y}}\left(  \mathbf{y}^{\alpha}\right)
\operatorname{ad}\left(  \overline{\mathbf{y}}\right)  ^{\beta_{y}%
}\operatorname{ad}\left(  \overline{\mathbf{x}}\right)  ^{\beta_{x}+\gamma
_{x}}\left(  z_{j}\right) \\
&  =\sum_{\left\vert \gamma_{x}\right\vert _{j}=0}\overline{\partial}%
^{\gamma_{x}}\left(  1/h^{l}\right)  \operatorname{ad}\left(  \mathbf{x}%
\right)  ^{\gamma_{x}}\left(  \nabla_{j}\right)  \left(  a\left(
\mathbf{x}\right)  \mathbf{y}^{\alpha}\right)  .
\end{align*}
By assumption, $\Delta_{j}\left(  I_{m}\right)  +\nabla_{j}\left(
I_{m}\right)  \subseteq\bigoplus\limits_{k=0}^{q-e_{j}}I_{m+e_{j}+k}$ and
$x_{i}\left(  I_{s}^{d}\right)  \subseteq I_{s}^{d+1}$ (so $x_{i}\left(
I_{s}\right)  \subseteq I_{s}$) for all $i$, $d$ and $s$. Using Corollary
\ref{corKey1} (or (\ref{s1})), we deduce that $\operatorname{ad}\left(
\mathbf{x}\right)  ^{\gamma_{x}}\left(  \Delta_{j}\right)  \left(
I_{m}\right)  +\operatorname{ad}\left(  \mathbf{x}\right)  ^{\gamma_{x}%
}\left(  \nabla_{j}\right)  \left(  I_{m}\right)  \subseteq\bigoplus
\limits_{k=0}^{q-e_{j}}I_{m+e_{j}+k}$ for all $\gamma_{x}$. It follows that
$\Delta_{j}\left(  I_{m,h}\right)  \subseteq\bigoplus\limits_{k=0}^{q-e_{j}%
}I_{m+e_{j}+k,h}$. A similar inclusion takes place for $\nabla_{j}$. In
particular, $\Delta_{j,k}\left(  I_{m,h}\right)  +\nabla_{j,k}\left(
I_{m,h}\right)  \subseteq I_{m+e_{j}+k,h}$ for all $k$. But $\Delta
_{j,k}\left(  S_{h}^{e}\otimes R_{q}^{m}\left\langle \mathbf{y}\right\rangle
\right)  +\nabla_{j,k}\left(  S_{h}^{e}\otimes R_{q}^{m}\left\langle
\mathbf{y}\right\rangle \right)  \subseteq S_{h}^{e-k}\otimes R_{q}%
^{m+e_{j}+k}\left\langle \mathbf{y}\right\rangle $ (see Lemma \ref{lemOp0})
for all $j$. Therefore $\Delta_{j,k}\left(  I_{m,h}^{e}\right)  +\nabla
_{j,k}\left(  I_{m,h}^{e}\right)  \subseteq I_{m+e_{j}+k,h}^{e-k}$ for all
$k$. Using the continuity argument, we conclude that $\Delta_{j}\left(
M^{d}\right)  +\nabla_{j}\left(  M^{d}\right)  \subseteq M^{d+e_{j}}$ and
$L_{j}\left(  M^{d}\right)  =\left(  z_{j}+\Delta_{j}\right)  \left(
M^{d}\right)  \subseteq M^{d+e_{j}}$. Similarly, we have $R_{j}\left(
M^{d}\right)  \subseteq M^{d+e_{j}}$.
\end{proof}

Now take $a\in S$. Then $L_{a\left(  \mathbf{x}\right)  }=a\left(
L_{\mathbf{x}}\right)  =a\left(  x_{0}+\Delta_{0},\ldots,x_{n}+\Delta
_{n}\right)  =a+\Delta_{a}$ for some operator $\Delta_{a}\in\mathcal{L}%
_{k}\left(  S_{q,h\left(  \mathbf{x}\right)  }\right)  $ such that $\Delta
_{a}\left(  \mathcal{J}_{m}\left(  h\right)  \right)  \subseteq\mathcal{J}%
_{m+1}\left(  h\right)  $ (see (\ref{s1})) for all $m$. If $a\in S^{e}$ with
its expansion $a=\sum_{\left\vert \lambda\right\vert =e}c_{\lambda}%
\mathbf{x}^{\lambda}$ then
\[
\Delta_{a}=\sum_{\left\vert \lambda\right\vert =e}\sum_{\lambda_{k}=\left\vert
\mathbf{i}_{k}\right\vert +\left\vert \mathbf{i}_{k}\right\vert \geq\left\vert
\mathbf{j}_{k}\right\vert >0}c_{\lambda}\left(  x_{0}^{\mathbf{i}_{0}}%
\Delta_{0}^{\mathbf{j}_{0}}\right)  \cdots\left(  x_{n}^{\mathbf{i}_{n}}%
\Delta_{n}^{\mathbf{j}_{n}}\right)  ,
\]
where $x_{k}^{\mathbf{i}}\Delta_{k}^{\mathbf{j}}=x_{k}^{i_{1}}\Delta
_{k}^{j_{1}}\cdots x_{k}^{i_{p_{k}}}\Delta_{k}^{j_{p_{k}}}$. Thus all
$\Delta_{a}$ are continuous (with respect to NC-topology) operators on
$S_{q,h\left(  \mathbf{x}\right)  }$, and $\Delta_{x_{j}}=\Delta_{j}$ for all
$j$. By (\ref{s1}), $\Delta_{j}\left(  S_{q,h\left(  \mathbf{x}\right)  }%
^{d}\right)  \subseteq S_{q,h\left(  \mathbf{x}\right)  }^{d+e_{j}}$ and
$z_{j}\left(  S_{q,h\left(  \mathbf{x}\right)  }^{d}\right)  \subseteq
S_{q,h\left(  \mathbf{x}\right)  }^{d+e_{j}}$ for all $j$ and $d$, which in
turn implies that $a\left(  S_{q,h\left(  \mathbf{x}\right)  }^{d}\right)
\subseteq S_{q,h\left(  \mathbf{x}\right)  }^{d+e}$, $\Delta_{a}\left(
S_{q,h\left(  \mathbf{x}\right)  }^{d}\right)  \subseteq S_{q,h\left(
\mathbf{x}\right)  }^{d+e}$ and $L_{a\left(  \mathbf{x}\right)  }\left(
S_{q,h\left(  \mathbf{x}\right)  }^{d}\right)  \subseteq S_{q,h\left(
\mathbf{x}\right)  }^{d+e}$. By Lemma \ref{lemtp1}, $a$ and $\Delta_{a}$ leave
invariant $M$. It follows that $a\left(  M^{d}\right)  \subseteq M^{d+e}$ and
$\Delta_{a}\left(  M^{d}\right)  \subseteq M^{d+e}$, which in turn implies
that $L_{a\left(  \mathbf{x}\right)  }\left(  M^{d}\right)  \subseteq M^{d+e}$
for all $d$ whenever $a$ is a homogeneous polynomial of degree $e$. In the
general (nonhomogeneous) case, $L_{a\left(  \mathbf{x}\right)  }$ just leaves
invariant $M$. Note also that $L_{h\left(  \mathbf{x}\right)  }\left(
M^{d}\right)  \subseteq M^{d+\deg\left(  h\right)  }$ and $L_{h\left(
\mathbf{x}\right)  }$ is an invertible operator on $S_{q,h\left(
\mathbf{x}\right)  }$ with its inverse $L_{h\left(  \mathbf{x}\right)  ^{-1}}$.

\begin{lemma}
\label{lemtp2}Let $a$ be an algebraic power of a factor of $h$. Then
$L_{a\left(  \mathbf{x}\right)  }$ is an invertible operator on $S_{q,h\left(
\mathbf{x}\right)  }$, the operator $L_{a\left(  \mathbf{x}\right)  }^{-1}$
leaves invariant the subspace $M$, and
\[
L_{a\left(  \mathbf{x}\right)  }^{-1}=\sum_{k=0}^{\infty}\left(  -1\right)
^{k}\left(  \dfrac{1}{a}\Delta_{a}\right)  ^{k}\dfrac{1}{a}=\left(
1+\dfrac{1}{a}\Delta_{a}\right)  ^{-1}\dfrac{1}{a}\text{.}%
\]
Moreover, $L_{a\left(  \mathbf{x}\right)  }^{-1}\left(  S_{q,h\left(
\mathbf{x}\right)  }^{d}\right)  \subseteq S_{q,h\left(  \mathbf{x}\right)
}^{d-\deg\left(  a\right)  }$ and $L_{a\left(  \mathbf{x}\right)  }%
^{-1}\left(  M^{d}\right)  \subseteq M^{d-\deg\left(  a\right)  }$.
\end{lemma}

\begin{proof}
By assumption $a$ is invertible on $S_{h}$, therefore it turns out to be an
invertible diagonal operator on $S_{q,h\left(  \mathbf{x}\right)  }%
=\prod\limits_{m=0}^{\infty}S_{h}\otimes R_{q}^{m}\left\langle \mathbf{y}%
\right\rangle $. Since $\Delta_{a}\left(  \mathcal{J}_{m}\left(  h\right)
\right)  \subseteq\mathcal{J}_{m+1}\left(  h\right)  $ for all $m$, it follows
that $\sum_{k=m}^{N}\left(  -1\right)  ^{k}\left(  \dfrac{1}{a}\Delta
_{a}\right)  ^{k}\dfrac{1}{a}\left(  f\right)  \in\mathcal{J}_{m}\left(
h\right)  $ for all $m$, $N>m$ and $f\in S_{q,h\left(  \mathbf{x}\right)  }$.
But $S_{q,h\left(  \mathbf{x}\right)  }$ is an NC-complete algebra, therefore
the series $\sum_{k=0}^{\infty}\left(  -1\right)  ^{k}\left(  \dfrac{1}%
{a}\Delta_{a}\right)  ^{k}\dfrac{1}{a}\left(  f\right)  $ converges in
$S_{q,h\left(  \mathbf{x}\right)  }$ thanks to the Cauchy Criteria. Thus the
operator series $\sum_{k=0}^{\infty}\left(  -1\right)  ^{k}\left(  \dfrac
{1}{a}\Delta_{a}\right)  ^{k}\dfrac{1}{a}$ converges with respect to the
strong operator topology (SOT), and its sum $T$ is a well defined continuous
operator on $S_{q,h\left(  \mathbf{x}\right)  }$. Note that
\[
TL_{a\left(  \mathbf{x}\right)  }\left(  f\right)  =\sum_{k=0}^{\infty}\left(
-1\right)  ^{k}\left(  \dfrac{1}{a}\Delta_{a}\right)  ^{k}\dfrac{1}{a}\left(
a+\Delta_{a}\right)  \left(  f\right)  =\sum_{k=0}^{\infty}\left(  -1\right)
^{k}\left(  \dfrac{1}{a}\Delta_{a}\right)  ^{k}\left(  1+\left(  \dfrac{1}%
{a}\Delta_{a}\right)  \right)  \left(  f\right)  =f
\]
for all $f\in S_{q,h\left(  \mathbf{x}\right)  }$. But $T=\left(  1+\dfrac
{1}{a}\Delta_{a}\right)  ^{-1}\dfrac{1}{a}$ formally, therefore $L_{a\left(
\mathbf{x}\right)  }T=a\left(  1+\dfrac{1}{a}\Delta_{a}\right)  T=a\dfrac
{1}{a}=1$, that is, $T=L_{a\left(  \mathbf{x}\right)  }^{-1}$. Further, both
operators $\Delta_{a}$ and $\dfrac{1}{a}$ leave invariant the subspace $M$,
therefore so do the partial sums $\sum_{k=0}^{m}\left(  -1\right)  ^{k}\left(
\dfrac{1}{a}\Delta_{a}\right)  ^{k}\dfrac{1}{a}$. Since $M$ is closed, it
follows that $L_{a\left(  \mathbf{x}\right)  }^{-1}\left(  M\right)  \subseteq
M$. Moreover, $\left(  \dfrac{1}{a}\Delta_{a}\right)  \left(  M^{d}\right)
\subseteq\dfrac{1}{a}\left(  M^{d+\deg\left(  a\right)  }\right)  \subseteq
M^{d}$ implies that $\left(  \dfrac{1}{a}\Delta_{a}\right)  ^{k}\dfrac{1}%
{a}\left(  M^{d}\right)  \subseteq\left(  \dfrac{1}{a}\Delta_{a}\right)
^{k}\left(  M^{d-\deg\left(  a\right)  }\right)  \subseteq M^{d-\deg\left(
a\right)  }$ for all $k$ and $d$. But every $M^{e}$ is closed, thereby
$L_{a\left(  \mathbf{x}\right)  }^{-1}\left(  M^{d}\right)  \subseteq
\sum_{k=0}^{\infty}\left(  -1\right)  ^{k}\left(  \dfrac{1}{a}\Delta
_{a}\right)  ^{k}\dfrac{1}{a}\left(  M^{d}\right)  \subseteq M^{d-\deg\left(
a\right)  }$ for all $d$.
\end{proof}

Now fix $d\in\mathbb{Z}$ and let $\mathfrak{R}_{q,h\left(  \mathbf{x}\right)
}^{d}$ be the linear span (or just finite sums) in $S_{q,h\left(
\mathbf{x}\right)  }^{d}$ of all right fractions $h\left(  \mathbf{x}\right)
^{-p}\cdot a$ with $a\in S^{e}\otimes R_{q}^{m}\left\langle \mathbf{y}%
\right\rangle $, $e+m-p\deg\left(  h\right)  =d$, $p\geq0$. Similarly,
$\mathfrak{L}_{q,h\left(  \mathbf{x}\right)  }^{d}$ denotes the linear span of
the related left fractions.

\begin{lemma}
\label{lemrightF}The closure of $\mathfrak{R}_{q,h\left(  \mathbf{x}\right)
}^{d}$ in $S_{q,h\left(  \mathbf{x}\right)  }$ coincides with the closure of
$\mathfrak{L}_{q,h\left(  \mathbf{x}\right)  }^{d}$, which in turn coincides
with the subspace $S_{q,h\left(  \mathbf{x}\right)  }^{d}$.
\end{lemma}

\begin{proof}
First take $u=h\left(  \mathbf{x}\right)  ^{-p}\cdot b\left(  \mathbf{x}%
\right)  $ with homogenous $b\in S$ such that $\deg\left(  b\right)
-p\deg\left(  h\right)  =d$. By Lemma \ref{lemtp2}, we have
\[
u=L_{h\left(  \mathbf{x}\right)  }^{-p}\left(  b\left(  \mathbf{x}\right)
\right)  =\left(  \left(  1+\dfrac{1}{h}\Delta_{h}\right)  ^{-1}\dfrac{1}%
{h}\right)  ^{p}\left(  b\left(  \mathbf{x}\right)  \right)  =\left(
\sum_{k=0}^{\infty}\left(  -1\right)  ^{k}\left(  \dfrac{1}{h}\Delta
_{h}\right)  ^{k}\dfrac{1}{h}\right)  ^{p}\left(  b\left(  \mathbf{x}\right)
\right)  ,
\]
and $\dfrac{1}{h}\Delta_{h}\left(  S_{q,h\left(  \mathbf{x}\right)  }%
^{e}\right)  \subseteq\dfrac{1}{h}\left(  S_{q,h\left(  \mathbf{x}\right)
}^{e+\deg\left(  h\right)  }\right)  \subseteq S_{q,h\left(  \mathbf{x}%
\right)  }^{e}$ for all $e$. Therefore $\left(  1+\dfrac{1}{h}\Delta
_{h}\right)  ^{-1}\dfrac{1}{h}\left(  S_{q,h\left(  \mathbf{x}\right)  }%
^{e}\right)  \subseteq S_{q,h\left(  \mathbf{x}\right)  }^{e-\deg\left(
h\right)  }$. In particular, $u\in S_{q,h\left(  \mathbf{x}\right)  }%
^{\deg\left(  b\right)  -p\deg\left(  h\right)  }=S_{q,h\left(  \mathbf{x}%
\right)  }^{d}$. Now take $v=h\left(  \mathbf{x}\right)  ^{-p}\cdot
b\in\mathfrak{R}_{q,h\left(  \mathbf{x}\right)  }^{d}$ with $b\in S^{e}\otimes
R_{q}^{m}\left\langle \mathbf{y}\right\rangle $, $e+m-p\deg\left(  h\right)
=d$. Then $v=\sum_{\left\langle \alpha\right\rangle =m}u_{\alpha}%
\mathbf{y}^{\alpha}$ with $u_{\alpha}=h\left(  \mathbf{x}\right)  ^{-p}\cdot
b_{\alpha}\left(  \mathbf{x}\right)  $, $\deg\left(  b_{\alpha}\right)  =e$.
But $\deg\left(  b_{\alpha}\right)  -p\deg\left(  h\right)  =e-p\deg\left(
h\right)  =d-m$. Using the fact just proven above, we obtain that $u_{\alpha
}\in S_{q,h\left(  \mathbf{x}\right)  }^{d-m}$. It follows that $v\in
\sum_{\left\langle \alpha\right\rangle =m}S_{q,h\left(  \mathbf{x}\right)
}^{d-m}\cdot\mathbf{y}^{\alpha}\subseteq S_{q,h\left(  \mathbf{x}\right)
}^{d}$. Thus $\mathfrak{R}_{q,h\left(  \mathbf{x}\right)  }^{d}\subseteq
S_{q,h\left(  \mathbf{x}\right)  }^{d}$.

Conversely, let us prove that $S_{q,h\left(  \mathbf{x}\right)  }^{d}%
\subseteq\mathfrak{R}_{q,h\left(  \mathbf{x}\right)  }^{d}+\mathcal{J}%
_{m}\left(  h\right)  $ for every $m$. It suffices to prove that $S_{h}%
^{d-t}\otimes R_{q}^{t}\left\langle \mathbf{y}\right\rangle \subseteq
\mathfrak{R}_{q,h\left(  \mathbf{x}\right)  }^{d}+\mathcal{J}_{m}\left(
h\right)  $ for all $0\leq t\leq m$. The assertion is trivial for $t=m$. Take
$u=\left(  b/h^{p}\right)  \left(  \mathbf{x}\right)  \mathbf{y}^{\alpha}\in
S_{h}^{d-t}\otimes R_{q}^{t}\left\langle \mathbf{y}\right\rangle $ with $t<m$.
Put $v=\left(  b/h^{p}\right)  \left(  \mathbf{x}\right)  \in S_{h}^{d-t}$.
Using again Lemma \ref{lemtp2}, we have
\begin{align*}
v  &  =\left(  \dfrac{1}{h}\right)  ^{p}\left(  b\left(  \mathbf{x}\right)
\right)  =\left(  L_{h\left(  \mathbf{x}\right)  }^{-1}+\sum_{k=1}^{\infty
}\left(  -1\right)  ^{k+1}\left(  \dfrac{1}{h}\Delta_{h}\right)  ^{k}\dfrac
{1}{h}\right)  ^{p}\left(  b\left(  \mathbf{x}\right)  \right) \\
&  =h\left(  \mathbf{x}\right)  ^{-p}\cdot b\left(  \mathbf{x}\right)
+u_{1}+\cdots+u_{m-1}+v_{m}%
\end{align*}
for some homogenous $\left\{  u_{i}\right\}  \subseteq S_{q,h\left(
\mathbf{x}\right)  }^{l}$, $u_{i}\in S_{h}^{l-i}\otimes R_{q}^{i}\left\langle
\mathbf{y}\right\rangle $ and $v_{m}\in\mathcal{J}_{m}\left(  h\right)  $.
Notice that both $L_{h\left(  \mathbf{x}\right)  }^{-1}$ and $\sum
_{k=1}^{\infty}\left(  -1\right)  ^{k+1}\left(  \dfrac{1}{h}\Delta_{h}\right)
^{k}\dfrac{1}{h}$ preserve the degrees in $S_{q,h\left(  \mathbf{x}\right)  }%
$, therefore all $\left\{  u_{i}\right\}  $ have the same degree $l$. In
particular,
\[
u=h\left(  \mathbf{x}\right)  ^{-p}\cdot b\left(  \mathbf{x}\right)
\mathbf{y}^{\alpha}+u_{1}\mathbf{y}^{\alpha}+\cdots+u_{m-1}\mathbf{y}^{\alpha
}+v_{m}\mathbf{y}^{\alpha}%
\]
with homogenous $\left\{  u_{i}\mathbf{y}^{\alpha}\right\}  \subseteq
S_{q,h\left(  \mathbf{x}\right)  }^{l+t}$. Since $h\left(  \mathbf{x}\right)
^{-p}\cdot b\left(  \mathbf{x}\right)  \mathbf{y}^{\alpha}\in\mathfrak{R}%
_{q,h\left(  \mathbf{x}\right)  }^{d}\subseteq S_{q,h\left(  \mathbf{x}%
\right)  }^{d}$, it follows that $l+t=d$ and $u_{i}\mathbf{y}^{\alpha}\in
S_{h}^{d-\left(  t+i\right)  }\otimes R_{q}^{t+i}\left\langle \mathbf{y}%
\right\rangle $ for all $i$. By induction hypothesis, all $u_{i}%
\mathbf{y}^{\alpha}\in\mathfrak{R}_{q,h\left(  \mathbf{x}\right)  }%
^{d}+\mathcal{J}_{m}\left(  h\right)  $. Therefore $u\in\mathfrak{R}%
_{q,h\left(  \mathbf{x}\right)  }^{d}+\mathcal{J}_{m}\left(  h\right)  $.
Finally, the inclusions $\mathfrak{R}_{q,h\left(  \mathbf{x}\right)  }%
^{d}\subseteq S_{q,h\left(  \mathbf{x}\right)  }^{d}\subseteq\bigcap
\limits_{m}\left(  \mathfrak{R}_{q,h\left(  \mathbf{x}\right)  }%
^{d}+\mathcal{J}_{m}\left(  h\right)  \right)  =\overline{\mathfrak{R}%
_{q,h\left(  \mathbf{x}\right)  }^{d}}$ imply that $S_{q,h\left(
\mathbf{x}\right)  }^{d}=\overline{\mathfrak{R}_{q,h\left(  \mathbf{x}\right)
}^{d}}$. Similarly, $S_{q,h\left(  \mathbf{x}\right)  }^{d}=\overline
{\mathfrak{L}_{q,h\left(  \mathbf{x}\right)  }^{d}}$.
\end{proof}

As above let $I$ be an NC-graded ideal of $S_{q}$ with its related closed
subspace $M=\prod\limits_{m=0}^{\infty}I_{m,h}$ in $S_{q,h\left(
\mathbf{x}\right)  }$.

\begin{proposition}
\label{proptp1}The subspace $M$ is a closed, graded, two-sided ideal of the
algebra $S_{q,h\left(  \mathbf{x}\right)  }$ and $M=I_{h\left(  \mathbf{x}%
\right)  }$.
\end{proposition}

\begin{proof}
Since $\Delta_{j}\left(  I_{m,h}\right)  \subseteq\bigoplus\limits_{k=0}%
^{q-e_{j}}I_{m+e_{j}+k,h}$ for all $j>n$ (see (\ref{s1}) and Lemma
\ref{lemtp1}), it follows that $\mathbf{y}^{\alpha}\cdot M=L_{\mathbf{y}%
}^{\alpha}\left(  M\right)  \subseteq M$ for all $\alpha$. Take a
noncommutative right fraction $g=h\left(  \mathbf{x}\right)  ^{-p}\cdot b\in
S_{q,h\left(  \mathbf{x}\right)  }$ with $p\geq0$, $b\in S\otimes R_{q}%
^{m}\left\langle \mathbf{y}\right\rangle $ and $f\in M$. Prove that $g\cdot
f\in M$. Since $\mathbf{y}^{\alpha}\cdot M\subseteq M$ for all $\alpha$, we
can assume that $b=b\left(  \mathbf{x}\right)  \in S$. Using Lemma
\ref{lemtp1} and \ref{lemtp2}, we obtain that $g\cdot f=L_{h\left(
\mathbf{x}\right)  }^{-p}L_{b\left(  \mathbf{x}\right)  }\left(  f\right)  \in
M$. In the general case of $g$, it is a limit of actual right fractions
$g=\lim_{k}g_{k}$ in $S_{q,h\left(  \mathbf{x}\right)  }$ (see \cite[Lemma
2.2.2]{DComA}). Taking into account that $M$ is a closed subspace, we conclude
that $g\cdot f=\lim_{k}g_{k}\cdot f\in M$. Similarly, $f\cdot g\in M$. Thus
$M$ is a closed two-sided ideal. Now take $u=h\left(  \mathbf{x}\right)
^{-p}\cdot a\left(  \mathbf{x}\right)  \cdot\mathbf{y}^{\alpha}\in
\mathfrak{R}_{q,h\left(  \mathbf{x}\right)  }^{t}\subseteq S_{q,h\left(
\mathbf{x}\right)  }^{t}$ with $a\in S^{e}$, $\left\langle \alpha\right\rangle
=m$, $e+m-p\deg\left(  h\right)  =t$. Using again Lemma \ref{lemtp1} and
\ref{lemtp2}, we derive that
\begin{align*}
u\cdot M^{d}  &  =h\left(  \mathbf{x}\right)  ^{-p}\cdot a\left(
\mathbf{x}\right)  \cdot\mathbf{y}^{\alpha}\cdot M^{d}\subseteq h\left(
\mathbf{x}\right)  ^{-p}\cdot a\left(  \mathbf{x}\right)  \cdot M^{d+m}%
=L_{h\left(  \mathbf{x}\right)  }^{-p}L_{a\left(  \mathbf{x}\right)  }\left(
M^{d+m}\right) \\
&  \subseteq L_{h\left(  \mathbf{x}\right)  }^{-p}\left(  M^{d+m+e}\right)
\subseteq M^{d+m+e-p\deg\left(  h\right)  }=M^{d+t},
\end{align*}
that is, $\mathfrak{R}_{q,h\left(  \mathbf{x}\right)  }^{t}M^{d}\subseteq
M^{d+t}$ for all $d$. By Lemma \ref{lemrightF}, $\mathfrak{R}_{q,h\left(
\mathbf{x}\right)  }^{t}$ is dense in $S_{q,h\left(  \mathbf{x}\right)  }^{t}%
$, and $S_{q,h\left(  \mathbf{x}\right)  }$ is a topological algebra. Hence
$S_{q,h\left(  \mathbf{x}\right)  }^{t}\cdot M^{d}\subseteq M^{t+d}$, that is,
$M$ is a graded two-sided ideal of $S_{q,h\left(  \mathbf{x}\right)  }$.

Finally, let us prove that $M=I_{h\left(  \mathbf{x}\right)  }$. First note
that $I=\oplus_{m}I_{m}\subseteq\prod\limits_{m=0}^{\infty}I_{m,h}=M\subseteq
S_{q,h\left(  \mathbf{x}\right)  }$ is an inclusion of subalgebras and
$I_{h\left(  \mathbf{x}\right)  }$ coincides with the closure of
noncommutative right fractions $h\left(  \mathbf{x}\right)  ^{-p}\cdot a$,
$a\in I$, $p\geq0$ \cite[Section 4.3]{DComA}. If $f=h\left(  \mathbf{x}%
\right)  ^{-p}\cdot a\in I_{h\left(  \mathbf{x}\right)  }$ with $a\in I$, then
$f\in h\left(  \mathbf{x}\right)  ^{-p}\cdot M\subseteq M$, for $M$ is a
two-sided ideal. Since $M$ is closed, it follows that $I_{h\left(
\mathbf{x}\right)  }\subseteq M$. Conversely, it suffices to prove that
$I_{m,h}\subseteq I_{h\left(  \mathbf{x}\right)  }$ for all $m$. Take
$f=a/h^{p}\in I_{m,h}$ with $a\in I_{m}$. Since $\dfrac{1}{h}\Delta_{h}\left(
\mathcal{J}_{m}\left(  h\right)  \right)  \subseteq\mathcal{J}_{m+1}\left(
h\right)  $, it follows that
\[
f=\left(  \dfrac{1}{h}\right)  ^{p}\left(  a\right)  =\left(  L_{h\left(
\mathbf{x}\right)  }^{-1}+\dfrac{1}{h}\Delta_{h}L_{h\left(  \mathbf{x}\right)
}^{-1}\right)  ^{p}\left(  a\right)  \in L_{h\left(  \mathbf{x}\right)  }%
^{-p}\left(  a\right)  +\mathcal{J}_{m+1}\left(  h\right)
\]
thanks to Lemma \ref{lemtp2}. But $\mathcal{J}_{m+1}\left(  h\right)  =\left(
S_{h}\otimes R_{q}^{m+1}\left\langle \mathbf{y}\right\rangle \right)
\oplus\mathcal{J}_{m+2}\left(  h\right)  $, that is, $f=L_{h\left(
\mathbf{x}\right)  }^{-p}\left(  a\right)  +a_{1}/h^{p_{1}}+f_{1}$ for some
$a_{1}/h^{p_{1}}\in S_{h}\otimes R_{q}^{m+1}\left\langle \mathbf{y}%
\right\rangle $ and $f_{1}\in\mathcal{J}_{m+2}\left(  h\right)  $. Since
$L_{h\left(  \mathbf{x}\right)  }^{-p}\left(  a\right)  \in I_{h\left(
\mathbf{x}\right)  }\subseteq M$, we have $f-L_{h\left(  \mathbf{x}\right)
}^{-p}\left(  a\right)  \in M$ and $M=\prod\limits_{m=0}^{\infty}I_{m,h}$. It
follows that $a_{1}/h^{p_{1}}\in I_{m+1,h}$ or $a_{1}\in I_{m}$. Using the
same argument, we derive that $f=L_{h\left(  \mathbf{x}\right)  }^{-p}\left(
a\right)  +L_{h\left(  \mathbf{x}\right)  }^{-p_{1}}\left(  a_{1}\right)
+a_{2}/h^{p_{2}}+f_{2}$ for some $a_{2}/h^{p_{2}}\in I_{m+2,h}$ and $f_{2}%
\in\mathcal{J}_{m+3}\left(  h\right)  $. By iterating, we obtain a sequence
$\left(  g_{i}\right)  \subseteq I_{h\left(  \mathbf{x}\right)  }$ such that
$f-g_{i}\in\mathcal{J}_{m+k}\left(  h\right)  $ for all $i$. Being
$I_{h\left(  \mathbf{x}\right)  }$ a closed ideal, we conclude that
$f=\lim_{i}g_{i}\in I_{h\left(  \mathbf{x}\right)  }$. Hence $M=I_{h\left(
\mathbf{x}\right)  }$.
\end{proof}

\begin{remark}
The key idea on the invertibility of $L_{a}$ was formally used in the proof of
Theorem \ref{thNCDaff}.
\end{remark}

Thus if $I$ is the sum of a differential chain $\left\{  I_{m}\right\}  $ and
$h$ is a homogeneous polynomial from $S$ of positive degree, then the
topological localization $I_{h\left(  \mathbf{x}\right)  }$ of $I$ at
$h\left(  \mathbf{x}\right)  $ turns out to be a two-sided (topologically)
graded ideal in $S_{q,h\left(  \mathbf{x}\right)  }$. Namely, $I_{h\left(
\mathbf{x}\right)  }=\widetilde{\oplus}_{d\in\mathbb{Z}}I_{h\left(
\mathbf{x}\right)  }^{d}$ with $I_{h\left(  \mathbf{x}\right)  }^{d}%
=\prod\limits_{m=0}^{\infty}I_{m,h}^{d-m}$ such that $S_{q,h\left(
\mathbf{x}\right)  }^{e}\cdot I_{h\left(  \mathbf{x}\right)  }^{d}\cdot
S_{q,h\left(  \mathbf{x}\right)  }^{t}\subseteq I_{h\left(  \mathbf{x}\right)
}^{e+d+t}$ for all $e,d,t$, thanks to Proposition \ref{proptp1}. In
particular, $I_{h\left(  \mathbf{x}\right)  }^{0}$ denoted by $I_{\left(
h\left(  \mathbf{x}\right)  \right)  }$ is a closed two-sided ideal of the
algebra $S_{q,\left(  h\left(  \mathbf{x}\right)  \right)  }$ ,where
$S_{q,\left(  h\left(  \mathbf{x}\right)  \right)  }=S_{q,h\left(
\mathbf{x}\right)  }^{0}=\prod\limits_{m=0}^{\infty}S_{h}^{-m}\otimes
R_{q}^{m}\left\langle \mathbf{y}\right\rangle $.

\subsection{The $S_{q}$-bimodule structure on $\mathcal{O}_{q,\ast}%
$\label{subsecGmod}}

As we have seen above the structure sheaf $\mathcal{O}_{q}$ of the projective
$q$-space $\mathbb{P}_{k,q}^{n}$ is an inverse limit of the coherent
$\mathcal{O}$-modules, namely, $\mathcal{O}_{q}=\prod_{m=0}^{\infty
}\mathcal{O}\left(  -m\right)  \otimes R_{q}^{m}\left\langle \mathbf{y}%
\right\rangle $. Recall (see \cite[Section 7]{Dproj1}) that the topological
twisting $\mathcal{O}_{q}\widetilde{\otimes}\mathcal{O}\left(  d\right)  $ of
the $\mathcal{O}$-module $\mathcal{O}_{q}$ is reduced to $\mathcal{O}%
_{q}\left(  d\right)  =\prod_{m=0}^{\infty}\mathcal{O}\left(  d-m\right)
\otimes R_{q}^{m}\left\langle \mathbf{y}\right\rangle $ for every
$d\in\mathbb{Z}$. We define $\mathcal{O}_{q,\ast}$ to be the $\mathcal{O}%
$-module $\bigoplus\limits_{d}\mathcal{O}_{q}\left(  d\right)  $. Take a
homogeneous $h\in S_{+}$ and consider the affine open subset $D_{+}\left(
h\right)  \subseteq\mathbb{P}_{k}^{n}$. Since $\mathbb{P}_{k}^{n}$ is
noetherian, we derive that
\[
\mathcal{O}_{q,\ast}\left(  D_{+}\left(  h\right)  \right)  =\bigoplus
\limits_{d}\mathcal{O}_{q}\left(  d\right)  \left(  D_{+}\left(  h\right)
\right)  =\bigoplus\limits_{d}\prod_{m=0}^{\infty}S_{h}^{d-m}\otimes R_{q}%
^{m}\left\langle \mathbf{y}\right\rangle =\bigoplus\limits_{d}S_{q,h\left(
\mathbf{x}\right)  }^{d},
\]
which is a dense graded subalgebra in the topological localization
$S_{q,h\left(  \mathbf{x}\right)  }$. Note that $\Gamma\left(  \mathbb{P}%
_{k}^{n},\mathcal{O}_{q,\ast}\right)  =\oplus_{d}\Gamma\left(  \mathbb{P}%
_{k}^{n},\mathcal{O}_{q}\left(  d\right)  \right)  =\oplus_{d}S_{q}^{d}=S_{q}%
$. Moreover, $S_{q}^{e}\cdot S_{q,h\left(  \mathbf{x}\right)  }^{d}\cdot
S_{q}^{r}\subseteq S_{q,h\left(  \mathbf{x}\right)  }^{e}\cdot S_{q,h\left(
\mathbf{x}\right)  }^{d}\cdot S_{q,h\left(  \mathbf{x}\right)  }^{r}\subseteq
S_{q,h\left(  \mathbf{x}\right)  }^{e+d+r}$ for all $e,d,r$. Thus
$\mathcal{O}_{q,\ast}$ is a sheaf of graded algebras and it has a natural
$S_{q}$-bimodule structure that reflects the noncommutative algebra $S_{q}$.
For every open subset $U\subseteq\mathbb{P}_{k}^{n}$ we obtain the sheaf
morphisms $L,R:\mathcal{O}_{q,\ast}\left(  U\right)  \rightarrow
\operatorname{Hom}_{k}\left(  \mathcal{O}_{q,\ast}|U\right)  $ given by the
left and right regular representations, where $\mathcal{O}_{q,\ast}\left(
U\right)  $ stands for the constant sheaf on $U$, that is, the global sections
of the sheaf $\mathcal{O}_{q,\ast}|U$ are acting on it. In particular, the
left and right regular representations $S_{q}\rightarrow\mathcal{L}_{k}\left(
S_{q,h\left(  \mathbf{x}\right)  }\right)  $ considered above define $S_{q}%
$-bimodule structure on $\mathcal{O}_{q,\ast}$ for $U=\mathbb{P}_{k}^{n}$. By
Lemma \ref{lemLR0}, we conclude that
\[
L_{j}\left(  \mathcal{O}_{q}\left(  d\right)  \right)  +R_{j}\left(
\mathcal{O}_{q}\left(  d\right)  \right)  \subseteq\mathcal{O}_{q}\left(
d+e_{j}\right)
\]
for all $j$ and $d$. In particular, there are well defined differential
operators
\[
\Delta_{j},\nabla_{j}\in\operatorname{Hom}_{k}\left(  \mathcal{O}_{q,\ast
}\right)  \text{,\quad}\Delta_{j}\left(  \mathcal{O}_{q}\left(  d\right)
\right)  +\nabla_{j}\left(  \mathcal{O}_{q}\left(  d\right)  \right)
\subseteq\mathcal{O}_{q}\left(  d+e_{j}\right)
\]
for all $j$ and $d$. In this case, $L_{j}=z_{j}+\Delta_{j}$, $R_{j}%
=z_{j}+\nabla_{j}$, and we have the adjoint representation $\mathfrak{g}%
_{q}\left(  \mathbf{x}\right)  \rightarrow\operatorname{Hom}_{k}\left(
\mathcal{O}_{q,\ast}\right)  $, $u\mapsto\operatorname{ad}\left(  u\right)  $
with $\operatorname{ad}\left(  z_{j}\right)  =\Delta_{j}-\nabla_{j}$ for all
$j$.

Now fix the principal open subset $U_{i}=D_{+}\left(  x_{i}\right)
\subseteq\mathbb{P}_{k}^{n}$ with the representations
\[
L,R:\mathcal{O}_{q,\ast}\left(  U_{i}\right)  \rightarrow\operatorname{Hom}%
_{k}\left(  \mathcal{O}_{q,\ast}|U_{i}\right)  ,
\]
where $\mathcal{O}_{q,\ast}\left(  U_{i}\right)  =\oplus_{d}S_{q,x_{i}}^{d}$.
Take a homogeneous $h\in S_{+}$ with $D_{+}\left(  h\right)  \subseteq
D_{+}\left(  x_{i}\right)  $, that is, $S_{x_{i}}$ is a commutative subalgebra
of $S_{h}$ and $h=x_{i}g$ with a homogeneous $g\in S$. In this case, the
diagonal embedding $S_{q,x_{i}}\hookrightarrow S_{q,h\left(  \mathbf{x}%
\right)  }$ turns out to be a continuous graded algebra homomorphism (see
\cite[Section 7.1]{Dproj1}), that is, $S_{q,x_{i}}^{d}\subseteq S_{q,h\left(
\mathbf{x}\right)  }^{d}$ for all $d\in\mathbb{Z}$. In particular,
$S_{q,\left(  x_{i}\right)  }$ is a noncommutative subalgebra of $S_{q,\left(
h\left(  \mathbf{x}\right)  \right)  }$, which defines the restriction map
$\mathcal{O}_{q}\left(  U_{i}\right)  \rightarrow\mathcal{O}_{q}\left(
D_{+}\left(  h\right)  \right)  $ of the sheaf $\mathcal{O}_{q}$. There is a
well defined diagonal operator $\dfrac{1}{x_{i}}:S_{q,h\left(  \mathbf{x}%
\right)  }\rightarrow S_{q,h\left(  \mathbf{x}\right)  }$ such that $\dfrac
{1}{x_{i}}\left(  S_{q,h\left(  \mathbf{x}\right)  }^{d}\right)  \subseteq
S_{q,h\left(  \mathbf{x}\right)  }^{d-1}$ and $\left(  1/x_{i}\right)  \left(
a/h^{d}\right)  =\left(  ag/h^{d+1}\right)  $ for all $a/h^{d}\in S_{h}$ and
$d\in\mathbb{Z}$. Thus $x_{i}^{e}\in\operatorname{Hom}_{k}\left(
\mathcal{O}_{q,\ast}|U_{i}\right)  $, $e\in\mathbb{Z}$ with
\[
x_{i}^{e}\left(  S_{q,h\left(  \mathbf{x}\right)  }^{d}\right)  =\prod
_{m=0}^{\infty}x_{i}^{e}\left(  S_{h}^{d-m}\right)  \otimes R_{q}%
^{m}\left\langle \mathbf{y}\right\rangle =\prod_{m=0}^{\infty}S_{h}%
^{e+d-m}\otimes R_{q}^{m}\left\langle \mathbf{y}\right\rangle =S_{q,h\left(
\mathbf{x}\right)  }^{e+d},
\]
which means that $x_{i}^{e}$ implements an $\mathcal{O}$-isomorphism between
the subsheaves $\mathcal{O}_{q}\left(  d\right)  $ and $\mathcal{O}_{q}\left(
e+d\right)  $ on $U_{i}$ for all $e\in\mathbb{Z}$.

Since $L_{j}\in\mathcal{L}_{k}\left(  S_{q,h\left(  \mathbf{x}\right)
}\right)  $ with $L_{j}\left(  S_{q,h\left(  \mathbf{x}\right)  }^{d}\right)
\subseteq S_{q,h\left(  \mathbf{x}\right)  }^{d+e_{j}}$, we obtain the
operators $\dfrac{1}{x_{i}^{e_{j}}}L_{j}$ and $\dfrac{1}{x_{i}^{e_{j}}}%
\Delta_{j}$ over the closed subalgebra $S_{q,\left(  h\left(  \mathbf{x}%
\right)  \right)  }$, where $e_{j}=\deg\left(  z_{j}\right)  $. In particular,
$\dfrac{1}{x_{i}^{l}}L_{i}^{l}$, $\dfrac{1}{x_{i}^{l}}\Delta_{i}^{l}%
\in\operatorname{Hom}_{k}\left(  \mathcal{O}_{q}|U_{i}\right)  $. Similarly,
we have the operators \ $\dfrac{1}{x_{i}^{l}}R_{i}^{l}$, $\dfrac{1}{x_{i}^{l}%
}\nabla_{i}^{l}\in\operatorname{Hom}_{k}\left(  \mathcal{O}_{q}|U_{i}\right)
$.

Put $D_{il}=\operatorname{ad}\left(  x_{i}\right)  ^{l}\left(  L_{i}\right)
$, $0\leq l<q$. If $l\geq1$ then $D_{il}=\operatorname{ad}\left(
x_{i}\right)  ^{l}\left(  x_{i}+\Delta_{i}\right)  =\operatorname{ad}\left(
x_{i}\right)  ^{l}\left(  \Delta_{i}\right)  $, and $D_{i0}=L_{i}$ (see Remark
\ref{remDIK}, we use the same notation $D_{i0}$ in the case of $l=0$). The
$S_{q}$-bimodule structure on $\mathcal{O}_{q,\ast}$ define also the following
local operators (see (\ref{s1}))
\[
T_{il}=\dfrac{1}{x_{i}^{l+1}}D_{il}\in\operatorname{Hom}_{k}\left(
\mathcal{O}_{q}|U_{i}\right)  ,\quad0\leq l<q.
\]
All operators on $\mathcal{O}_{q}|U_{i}$ introduced so far in this subsection
leave invariant the subsheaves $\mathcal{J}_{m}|U_{i}$ of ideals, that is,
they are filtered operators on $\mathcal{O}_{q}|U_{i}$. Note also that
$\dfrac{1}{x_{i}}L_{i}$ leaves invariant a subsheaf of $\mathcal{O}|U_{i}%
$-modules iff so does $\dfrac{1}{x_{i}}\Delta_{i}$, for $L_{i}=x_{i}%
+\Delta_{i}$. That property of invariance of $T_{i0}$ does not affect the
double treatment in the notation of $D_{i0}$.

\begin{proposition}
\label{propSBM1}Let $\mathcal{I}$ be a subsheaf in $\mathcal{O}_{q}$ of closed
subspaces. The operators $\dfrac{1}{x_{i}^{l}}L_{i}^{l}$, $1\leq l\leq q$
leave invariant $\mathcal{I}|U_{i}$ iff so do $T_{il}$, $0\leq l<q$. In this
case, the operators $\dfrac{1}{x_{i}^{e_{j}}}L_{j}$ leave invariant
$\mathcal{I}|U_{i}$ whenever $\mathcal{I}$ is a sheaf of closed left ideals.
Moreover, $\dfrac{z_{j}}{x_{i}^{e_{j}}}\left(  \mathcal{I}|U_{i}\right)
\subseteq\mathcal{I}|U_{i}$ iff $\dfrac{1}{x_{i}^{e_{j}}}\Delta_{j}$ leaves
invariant the sheaf $\mathcal{I}|U_{i}$ of closed left ideals.
\end{proposition}

\begin{proof}
As above fix $i$ and $h\in S_{+}$ with $D_{+}\left(  h\right)  \subseteq
D_{+}\left(  x_{i}\right)  $, and put $I=\mathcal{I}\left(  D_{+}\left(
h\right)  \right)  $ to be a closed $k$-subspace of the algebra $S_{q,\left(
h\left(  \mathbf{x}\right)  \right)  }$. Having $i$ fixed, one can use the
notations $S_{l}=\dfrac{1}{x_{i}^{l}}L_{i}^{l}$ and $T_{l}$ instead of
$T_{il}$. First assume that $T_{l}\left(  I\right)  \subseteq I$ for all
$0\leq l<q$. Note that $S_{1}=T_{0}$, therefore $S_{1}\left(  I\right)
\subseteq I$. Prove that $S_{l}\left(  I\right)  \subseteq I$ for $l>1$. Put
$V_{l}=\dfrac{1}{x_{i}^{l+1}}\left[  x_{i}^{l},L_{i}\right]  $, where the Lie
brackets is taken in the unital algebra $\operatorname{Hom}_{k}\left(
\mathcal{O}_{q,\ast}\left(  D_{+}\left(  h\right)  \right)  \right)  $. Using
Lemma \ref{lemNon1} (see (\ref{nonf1})), we obtain that
\[
V_{l}=\sum_{\tau=1}^{l}\left(  -1\right)  ^{\tau+1}\dbinom{l}{\tau}%
\dfrac{x_{i}^{l-\tau}}{x_{i}^{l+1}}\operatorname{ad}\left(  x_{i}\right)
^{\tau}\left(  L_{i}\right)  =\sum_{\tau=1}^{l}\left(  -1\right)  ^{\tau
+1}\dbinom{l}{\tau}\dfrac{1}{x_{i}^{\tau+1}}D_{i\tau}=\sum_{\tau=1}^{l}\left(
-1\right)  ^{\tau+1}\dbinom{l}{\tau}T_{\tau}.
\]
In particular, all $V_{l}$ leave invariant the subspace $I$. Note that%
\[
S_{l}=\dfrac{1}{x_{i}}\dfrac{1}{x_{i}^{l-1}}L_{i}L_{i}^{l-1}=\dfrac{1}{x_{i}%
}L_{i}\dfrac{1}{x_{i}^{l-1}}L_{i}^{l-1}-\dfrac{1}{x_{i}^{l}}\left[
x_{i}^{l-1},L_{i}\right]  \dfrac{1}{x_{i}^{l-1}}L_{i}^{l-1}=S_{1}%
S_{l-1}-V_{l-1}S_{l-1},
\]
that is,
\begin{equation}
S_{l}=S_{1}S_{l-1}-V_{l-1}S_{l-1} \label{sv}%
\end{equation}
for all $l\geq1$. By induction hypothesis, we have $S_{l-1}\left(  I\right)
\subseteq I$. Therefore $S_{l}\left(  I\right)  \subseteq I$ by (\ref{sv}).

Now assume that $S_{l}\left(  I\right)  \subseteq I$ for all $1\leq l\leq q$.
Since $T_{0}=S_{1}$, we have $T_{0}\left(  I\right)  \subseteq I$. Fix $l$
with $0<l<q$. Assume that $T_{r}\left(  I\right)  \subseteq I$ for all $0\leq
r\leq l-1$, and prove that $T_{l}\left(  I\right)  \subseteq I$. As above
$V_{l}=V_{l}^{\prime}+\left(  -1\right)  ^{l+1}T_{l}$, where $V_{l}^{\prime
}=\sum_{\tau=1}^{l-1}\left(  -1\right)  ^{\tau+1}\dbinom{l}{\tau}T_{\tau}$
leaves invariant $I$. It suffices to prove that $V_{l}\left(  I\right)
\subseteq I$. By Lemma \ref{lemtp2}, $L_{i}^{l}=L_{a}$ is an invertible
operator on $S_{q,h\left(  \mathbf{x}\right)  }$ and $L_{a\left(
\mathbf{x}\right)  }^{-1}=\left(  1+\dfrac{1}{a}\Delta_{a}\right)  ^{-1}%
\dfrac{1}{a}$, where $a=x_{i}^{l}$ and $L_{a\left(  \mathbf{x}\right)
}=a+\Delta_{a}$. But $S_{l}=\dfrac{1}{a}L_{a\left(  \mathbf{x}\right)
}=\dfrac{1}{a}\left(  a+\Delta_{a}\right)  =1+\dfrac{1}{a}\Delta_{a}$ leaves
invariant $I$, therefore so does $\dfrac{1}{a}\Delta_{a}$. Note that
$S_{l}^{-1}=\left(  1+\dfrac{1}{a}\Delta_{a}\right)  ^{-1}=\sum_{k=0}^{\infty
}\left(  -1\right)  ^{k}\left(  \dfrac{1}{a}\Delta_{a}\right)  ^{k}$ (SOT).
Since $\left(  \dfrac{1}{a}\Delta_{a}\right)  \left(  I\right)  \subseteq I$
and $I$ is closed, we derive that $S_{l}^{-1}$ leaves invariant $I$ too. Using
(\ref{sv}), we derive that $S_{l+1}=S_{1}S_{l}-V_{l}S_{l}$. In particular,
$V_{l}=S_{1}-S_{l+1}S_{l}^{-1}$ leaves invariant $I$.

Finally, assume that $I$ is a closed left ideal of $S_{q,\left(  h\left(
\mathbf{x}\right)  \right)  }$ and all $S_{l}$ (or $T_{l}$) leave invariant
$I$. By Lemma \ref{lemrightF}, $x_{i}^{-e_{j}}\cdot z_{j}\in S_{q,\left(
x_{i}\right)  }$ and $S_{q,\left(  x_{i}\right)  }\subseteq S_{q,\left(
h\left(  \mathbf{x}\right)  \right)  }$, therefore $x_{i}^{-e_{j}}\cdot
z_{j}\cdot I\in S_{q,\left(  h\left(  \mathbf{x}\right)  \right)  }\cdot
I\subseteq I$. It follows that
\[
\dfrac{1}{x_{i}^{e_{j}}}L_{j}\left(  I\right)  =\dfrac{1}{x_{i}^{e_{j}}}%
L_{i}^{e_{j}}\left(  L_{i}^{-e_{j}}L_{j}\left(  I\right)  \right)
\subseteq\dfrac{1}{x_{i}^{e_{j}}}L_{i}^{e_{j}}\left(  x_{i}^{-e_{j}}\cdot
z_{j}\cdot I\right)  \subseteq\dfrac{1}{x_{i}^{e_{j}}}L_{i}^{e_{j}}\left(
I\right)  \subseteq I.
\]
If $\dfrac{z_{j}}{x_{i}^{e_{j}}}\left(  I\right)  \subseteq I$, then
$\dfrac{1}{x_{i}^{e_{j}}}\Delta_{j}\left(  I\right)  =\dfrac{1}{x_{i}^{e_{j}}%
}\left(  L_{j}-z_{j}\right)  \left(  I\right)  \subseteq\dfrac{1}{x_{i}%
^{e_{j}}}L_{j}\left(  I\right)  -\dfrac{z_{j}}{x_{i}^{e_{j}}}\left(  I\right)
\subseteq I$, that is, $\dfrac{1}{x_{i}^{e_{j}}}\Delta_{j}$ leaves invariant
$I$ iff so does $\dfrac{z_{j}}{x_{i}^{e_{j}}}$.
\end{proof}

\begin{remark}
\label{remTop1}Based on Lemma \ref{lemOp1} (see also Remark \ref{remDIK}), we
obtain that
\begin{align*}
T_{il}  &  =\dfrac{1}{x_{i}^{l+1}}D_{il}=\sum_{\left\vert \beta+l\right\vert
_{i}>0}R\left(  \operatorname{ad}\left(  \overline{\mathbf{y}}\right)
^{\beta_{y}}\operatorname{ad}\left(  x_{n}\right)  ^{\beta_{n}}\cdots
\operatorname{ad}\left(  x_{i}\right)  ^{\beta_{i}+l}\cdots\operatorname{ad}%
\left(  x_{0}\right)  ^{\beta_{0}}\left(  x_{i}\right)  \right)  \dfrac
{1}{x_{i}^{l+1}}\overline{\partial}^{\beta}\\
&  =\sum_{u_{\beta}}\lambda_{l,u_{\beta}}R\left(  z_{u_{\beta}}\right)
\dfrac{1}{x_{i}^{l+1}}\overline{\partial}^{\beta},
\end{align*}
for some scalars $\lambda_{l,u_{\beta}}$, where the summation is taken over
all $z_{u_{\beta}}$ with $\deg\left(  z_{u_{\beta}}\right)  =\left\langle
\beta\right\rangle +l+1$, $1\leq l<q$. In the summation we have $\beta_{j}>0$
for at least one $j$ with $j<i$ otherwise the related term is vanishing. But
$\operatorname{ad}\left(  x_{i}\right)  ^{q-1}\operatorname{ad}\left(
x_{j}\right)  \left(  x_{i}\right)  =0$ for $j<i$, therefore $T_{il}=0$ for
all $l\geq q-1$. In the case of $q=2$ all these operators for $l\geq1$ are
vanishing. In the case of $q=3$ we have
\[
T_{i1}=\dfrac{1}{x_{i}^{2}}\operatorname{ad}\left(  x_{i}\right)  \left(
\Delta_{i}\right)  =\sum_{j<i}R\left(  \left[  x_{i},\left[  x_{j}%
,x_{i}\right]  \right]  \right)  \dfrac{1}{x_{i}^{2}}\overline{\partial}%
_{j}\text{, and }T_{il}=0\text{, }l\geq2\text{.}%
\]
The operators $T_{il}$ are prominent in the case of $q\geq3$.
\end{remark}

\begin{remark}
\label{remTop2}Let us introduce the following (local) operator
\[
\Gamma=\sum_{\tau=0}^{q-2}\left(  -1\right)  ^{\tau}\dbinom{q-2}{\tau}%
\dfrac{1}{x_{i}^{\tau+1}}\operatorname{ad}\left(  x_{i}\right)  ^{\tau}\left(
L_{i}\right)  \in\operatorname{Hom}_{k}\left(  \mathcal{O}_{q}|U_{i}\right)
.
\]
Using the notations from the proof of Proposition \ref{propSBM1}, we have
$\Gamma=\sum_{\tau=0}^{q-2}\left(  -1\right)  ^{\tau}\dbinom{q-2}{\tau}%
T_{\tau}=S_{1}-V_{q-2}$ and $V_{l}=V_{q-2}$ for all $l\geq q-1$ (see Remark
\ref{remTop1}). By (\ref{sv}), $S_{q-1}=S_{1}S_{q-2}-V_{q-2}S_{q-2}$ and
$\Gamma=S_{q-1}S_{q-2}^{-1}$. Actually, $\Gamma=S_{l}S_{l-1}^{-1}$ for all
$l\geq q-1$. Suppose $S_{l}\left(  I\right)  \subseteq I$ for all large $l$.
Then $\Gamma\left(  I\right)  \subseteq I$. Indeed, suppose $S_{l-1}\left(
I\right)  +S_{l}\left(  I\right)  \subseteq I$ for $l\geq q-1$. Then
$S_{l-1}^{-1}\left(  I\right)  \subseteq I$ (see to the proof of Proposition
\ref{propSBM1}), which in turn implies that $\Gamma\left(  I\right)  \subseteq
I$. In the case of $q=2$ we obtain that $S_{1}=\Gamma$. Therefore
$S_{1}\left(  I\right)  \subseteq I$ iff $S_{l}\left(  I\right)  \subseteq I$
for all large $l$. Indeed, just notice that $V_{l}=0$ (see Remark
\ref{remTop1}) and $S_{l}=S_{1}^{l}$ (see (\ref{sv})) for all $l$.
\end{remark}

\begin{remark}
\label{remTop3}Since $R_{i}=x_{i}+\nabla_{i}$ and $\left[  x_{i},\nabla
_{i}\right]  =0$, it follows that $\left[  \dfrac{1}{x_{i}},R_{i}\right]  =0$.
In particular, all $\dfrac{1}{x_{i}^{l}}R_{i}^{l}$ leave invariant $I$
whenever so does $\dfrac{1}{x_{i}}R_{i}$. Similar assertion is true for the
left multiplication operators in the case of $q=2$ (see Remark \ref{remTop2}).
But the assertion fails for $q\geq3$. Based on Lemma \ref{lemOp1}, we conclude
that both types of differential operators $T_{il}$ and $\dfrac{1}{x_{i}%
^{e_{j}}}\Delta_{j}$, $e_{j}=l+1$ considered above admit very similar
expansions $\sum_{u_{\beta}}\lambda_{u_{\beta}}R\left(  z_{u_{\beta}}\right)
\overline{\partial}^{\beta}$ taken over all $z_{u_{\beta}}$ with $\deg\left(
z_{u_{\beta}}\right)  =\left\langle \beta\right\rangle +l+1$.
\end{remark}

Now let $\mathcal{I}$ be an $\mathcal{O}$-submodule in $\mathcal{O}_{q}$ which
is a subsheaf of closed subspaces. As we have seen above there are well
defined sheaf isomorphisms $x_{i}^{d}:\mathcal{O}_{q}|U_{i}\rightarrow
\mathcal{O}_{q}\left(  d\right)  |U_{i}$, $0\leq i\leq n$ preserving the
related filtrations. Consider the subsheaves $x_{i}^{d}\left(  \mathcal{I}%
|U_{i}\right)  \subseteq\mathcal{O}_{q}\left(  d\right)  |U_{i}$. Since
$\mathcal{I}$ is an $\mathcal{O}$-module, we have $\left(  \dfrac{x_{i}}%
{x_{s}}\right)  ^{d}\left(  \mathcal{I}|U_{i}\cap U_{s}\right)  \subseteq
\mathcal{I}|U_{i}\cap U_{s}$ for all $d\in\mathbb{Z}$. It follows that
$x_{i}^{d}\left(  \mathcal{I}|U_{i}\right)  |U_{i}\cap U_{s}=x_{s}^{d}\left(
\mathcal{I}|_{U_{s}}\right)  |U_{i}\cap U_{s}$. In particular, the sheaves
$x_{i}^{d}\left(  \mathcal{I}|U_{i}\right)  $, $0\leq i\leq n$ are glued to
form a unique subsheaf in $\mathcal{O}_{q}\left(  d\right)  $ denoted by
$\mathcal{I}\left(  d\right)  $. Thus $\mathcal{I}\left(  d\right)
|U_{i}=x_{i}^{d}\left(  \mathcal{I}|U_{i}\right)  $ for all $i$, and
$\mathcal{I}\left(  d\right)  $ is a subsheaf in $\mathcal{O}_{q}\left(
d\right)  $ of closed subspaces. Define $\mathcal{I}_{\ast}$ to be the
subsheaf $\bigoplus_{d}\mathcal{I}\left(  d\right)  $ of $\mathcal{O}_{q,\ast
}$.

\begin{proposition}
\label{propSBM2}Let $\mathcal{I}$ be an $\mathcal{O}$-submodule in
$\mathcal{O}_{q}$ which in turn is a sheaf of closed left ideals. If the
operators $T_{il}$ leave invariant $\mathcal{I}|U_{i}$ for every $i$, then
$\mathcal{I}_{\ast}$ is a left $S_{q}$-submodule of $\mathcal{O}_{q,\ast}$,
that is, the left representation $L:S_{q}\rightarrow\operatorname{Hom}%
_{k}\left(  \mathcal{O}_{q,\ast}\right)  $ leaves invariant $\mathcal{I}%
_{\ast}$, and $L_{j}\left(  \mathcal{I}\left(  d\right)  \right)
\subseteq\mathcal{I}\left(  d+e_{j}\right)  $ for all $j$ and $d$. If
$\dfrac{z_{j}}{x_{i}^{e_{j}}}\left(  \mathcal{I}|U_{i}\right)  \subseteq
\mathcal{I}|U_{i}$ for all $i$ and $j$ additionally, then $\Delta_{j}\left(
\mathcal{I}\left(  d\right)  \right)  \subseteq\mathcal{I}\left(
d+e_{j}\right)  $ for all $j$ and $d$.
\end{proposition}

\begin{proof}
Since $L_{j}$, $j>n$ are commutators of $L_{i}$, $0\leq i\leq n$, it suffices
to prove that $L_{i}\left(  \mathcal{I}\left(  d\right)  \right)
\subseteq\mathcal{I}\left(  d+1\right)  $ for all $i$, $0\leq i\leq n$, and
$d\in\mathbb{Z}$. First prove that $L_{i}\left(  \mathcal{I}\left(  d\right)
|U_{i}\right)  \subseteq\mathcal{I}\left(  d+1\right)  |U_{i}$. Based on Lemma
\ref{lemNon1}, we have%
\[
L_{i}x_{i}^{d}=x_{i}^{d}L_{i}+\sum_{\sigma=1}^{d}\left(  -1\right)  ^{\sigma
}\dbinom{d}{\sigma}x_{i}^{d-\sigma}\operatorname{ad}\left(  x_{i}\right)
^{\sigma}\left(  L_{i}\right)  =x_{i}^{d+1}T_{i},
\]
where $T_{i}=\sum_{\sigma=0}^{d}\left(  -1\right)  ^{\sigma}\dbinom{d}{\sigma
}T_{i\sigma}$. It follows that $L_{i}\left(  \mathcal{I}\left(  d\right)
|U_{i}\right)  =L_{i}x_{i}^{d}\left(  \mathcal{I}|U_{i}\right)  =x_{i}%
^{d+1}T_{i}\left(  \mathcal{I}|U_{i}\right)  \subseteq x_{i}^{d+1}\left(
\mathcal{I}|U_{i}\right)  =\mathcal{I}\left(  d+1\right)  |U_{i}$. In
particular, $\Delta_{i}\left(  \mathcal{I}\left(  d\right)  |U_{i}\right)
\subseteq\mathcal{I}\left(  d+1\right)  |U_{i}$ and $\left(  x_{i}^{-1}%
\Delta_{i}\right)  \left(  \mathcal{I}\left(  d\right)  |U_{i}\right)
\subseteq\mathcal{I}\left(  d\right)  |U_{i}$ for all $d$. Then
\[
\left(  \dfrac{1}{x_{i}}\Delta_{i}\right)  ^{k}\dfrac{1}{x_{i}}\left(
\mathcal{I}\left(  d\right)  |U_{i}\right)  =\left(  \dfrac{1}{x_{i}}%
\Delta_{i}\right)  ^{k}x_{i}^{d-1}\left(  \mathcal{I}|U_{i}\right)  =\left(
\dfrac{1}{x_{i}}\Delta_{i}\right)  ^{k}\left(  \mathcal{I}\left(  d-1\right)
|U_{i}\right)  \subseteq\mathcal{I}\left(  d-1\right)  |U_{i}%
\]
for all $k$. Using Lemma \ref{lemtp2}, we obtain that $L_{i}^{-1}\left(
\mathcal{I}\left(  d\right)  |U_{i}\right)  \subseteq\mathcal{I}\left(
d-1\right)  |U_{i}$ for all $d$. Therefore $L_{i}\left(  \mathcal{I}\left(
d\right)  |U_{i}\right)  =\mathcal{I}\left(  d+1\right)  |U_{i}$. In the
general case, we have (see Lemma \ref{lemrightF})%
\begin{align*}
L_{j}\left(  \mathcal{I}\left(  d\right)  |U_{i}\right)   &  =L_{j}L_{i}%
^{d}\left(  \mathcal{I}|U_{i}\right)  =L_{i}^{d+1}L_{i}^{-d-1}L_{j}L_{i}%
^{d}\left(  \mathcal{I}|U_{i}\right)  =L_{i}^{d+1}\left(  x_{i}^{-d-1}\cdot
x_{j}\cdot x_{i}^{d}\cdot\mathcal{I}|U_{i}\right) \\
&  \subseteq L_{i}^{d+1}\left(  \mathcal{O}_{q}|U_{i}\cdot\mathcal{I}%
|U_{i}\right)  \subseteq L_{i}^{d+1}\left(  \mathcal{I}|U_{i}\right)
\subseteq\mathcal{I}\left(  d+1\right)  |U_{i}%
\end{align*}
for all $i$. Hence $L_{j}\left(  \mathcal{I}\left(  d\right)  \right)
\subseteq\mathcal{I}\left(  d+1\right)  $ for all $0\leq j\leq n$.

Finally, assume that $\dfrac{z_{j}}{x_{i}^{e_{j}}}\left(  \mathcal{I}%
|U_{i}\right)  \subseteq\mathcal{I}|U_{i}$ for all $i$ and $j$ additionally.
By Proposition \ref{propSBM1}, the operator $\dfrac{1}{x_{i}^{e_{j}}}%
\Delta_{j}$ leaves invariant the subsheaf $\mathcal{I}|U_{i}$. Moreover,
\[
z_{j}\left(  \mathcal{I}\left(  d\right)  |U_{i}\right)  =\dfrac{z_{j}}%
{x_{i}^{e_{j}}}x_{i}^{e_{j}+d}\left(  \mathcal{I}|U_{i}\right)  =x_{i}%
^{e_{j}+d}\dfrac{z_{j}}{x_{i}^{e_{j}}}\left(  \mathcal{I}|U_{i}\right)
\subseteq x_{i}^{e_{j}+d}\left(  \mathcal{I}|U_{i}\right)  =\mathcal{I}\left(
d+e_{j}\right)  |U_{i}%
\]
for all $i$ (the diagonal operators commute). It follows that $z_{j}\left(
\mathcal{I}\left(  d\right)  \right)  \subseteq\mathcal{I}\left(
d+e_{j}\right)  $ and $\Delta_{j}\left(  \mathcal{I}\left(  d\right)  \right)
=\left(  L_{j}-z_{j}\right)  \left(  \mathcal{I}\left(  d\right)  \right)
\subseteq\mathcal{I}\left(  d+e_{j}\right)  $ for all $j$ and $d$.
\end{proof}

\begin{remark}
\label{remTop4}A similar assertion for the right operators is trivial (see
Remark \ref{remTop3} and Lemma \ref{lemOp1}). Namely, if $\dfrac{1}{x_{i}%
}R_{i}$ leaves invariant $\mathcal{I}|U_{i}$ for every $i$, then
$\mathcal{I}_{\ast}$ is a right $S_{q}$-submodule of $\mathcal{O}_{q,\ast}$.
Indeed, first $R_{i}\left(  \mathcal{I}\left(  d\right)  |U_{i}\right)
=R_{i}x_{i}^{d}\left(  \mathcal{I}|U_{i}\right)  =x_{i}^{d}R_{i}\left(
\mathcal{I}|U_{i}\right)  =x_{i}^{d+1}\dfrac{1}{x_{i}}R_{i}\left(
\mathcal{I}|U_{i}\right)  \subseteq x_{i}^{d+1}\left(  \mathcal{I}%
|U_{i}\right)  =\mathcal{I}\left(  d+1\right)  |U_{i}$, and then $R_{j}\left(
\mathcal{I}\left(  d\right)  |U_{i}\right)  =R_{j}R_{i}^{d}\left(
\mathcal{I}|U_{i}\right)  =R_{i}^{d+1}R_{i}^{-d-1}R_{j}R_{i}^{d}\left(
\mathcal{I}|U_{i}\right)  =R_{i}^{d+1}\left(  \left(  \mathcal{I}%
|U_{i}\right)  \cdot x_{i}^{d}\cdot x_{j}\cdot x_{i}^{-d-1}\right)  \subseteq
R_{i}^{d+1}\left(  \left(  \mathcal{I}|U_{i}\right)  \cdot\mathcal{O}%
_{q}|U_{i}\right)  \subseteq\mathcal{I}\left(  d+1\right)  |U_{i}$.
\end{remark}

\subsection{Noncommutative sheaves\label{subsecNS}}

Let $\left\{  I_{m}\right\}  $ be a differential chain in $S_{q}$,
$Y=\operatorname{Proj}\left(  S/I_{0}\right)  $ the closed subscheme of
$\mathbb{P}_{k}^{n}$ associated with the graded ideal $I_{0}$, and let
$\left(  \iota,\iota^{\times}\right)  :Y\rightarrow\mathbb{P}_{k}^{n}$ be the
closed immersion, where $\iota:Y\rightarrow\mathbb{P}_{k}^{n}$ is the
homeomorphism onto its range and $\iota^{\times}:\mathcal{O}\rightarrow
\iota_{\ast}\mathcal{O}_{Y}$ is the related surjective sheaf morphism on
$\mathbb{P}_{k}^{n}$. Note that the structure sheaf $\mathcal{O}_{Y}$ of the
projective scheme $Y$ is the coherent sheaf $\left(  S/I_{0}\right)  ^{\sim}$
defined by means of the graded $k$-algebra $S/I_{0}$ (see Subsection
\ref{SubsecPSA}). Moreover, $\iota_{\ast}\mathcal{O}_{Y}=\iota_{\ast}\left(
S/I_{0}\right)  ^{\sim}=\left(  _{S}S/I_{0}\right)  ^{\sim}$ (see (\ref{fdi}))
is the coherent sheaf on $\mathbb{P}_{k}^{n}$ obtained from the $S$-module
$S/I_{0}$. Thereby $\iota^{\times}$ is the morphism $\widetilde{S}%
\rightarrow\left(  _{S}S/I_{0}\right)  ^{\sim}$ obtained from the first
vertical arrow of the diagram (\ref{d1}). Further, based on the arguments from
Subsection \ref{subsecDU}, we conclude that $\operatorname{Proj}\left(
S\otimes R_{q}^{m}\left\langle \mathbf{y}\right\rangle \right)  =\bigsqcup
\limits^{r_{m}}\mathbb{P}_{k}^{n}$ is the disjoint union of $r_{m}$ copies of
the projective space $\mathbb{P}_{k}^{n}$, where $r_{m}=\dim\left(  R_{q}%
^{m}\left\langle \mathbf{y}\right\rangle \right)  $. Moreover, the diagonal
homomorphism $\epsilon:S\rightarrow S\otimes R_{q}^{m}\left\langle
\mathbf{y}\right\rangle $ defines the canonical morphism $\sigma
:\bigsqcup\limits^{r_{m}}\mathbb{P}_{k}^{n}\rightarrow\mathbb{P}_{k}^{n}$ of
schemes such that $\sigma^{\times}:\mathcal{O}\rightarrow\mathcal{O}^{r_{m}%
}=\mathcal{O}\otimes R_{q}^{m}\left\langle \mathbf{y}\right\rangle $ is the
diagonal morphism of sheaves on $\mathbb{P}_{k}^{n}$ by Lemma \ref{lemDU2}.
Thus the diagram (\ref{d1}) generates the following diagram
\[%
\begin{array}
[c]{ccc}%
\mathcal{O} & \overset{\sigma^{\times}}{\longrightarrow} & \mathcal{O}\otimes
R_{q}^{m}\left\langle \mathbf{y}\right\rangle \\
^{\iota^{\times}}\downarrow &  & \downarrow\\
\iota_{\ast}\mathcal{O}_{Y} & \overset{\tau}{\longrightarrow} &
\widetilde{T_{m}}%
\end{array}
\]
of the coherent $\mathcal{O}$-modules on $\mathbb{P}_{k}^{n}$, where
$\tau=\left(  _{S}\omega\right)  ^{\sim}:\left(  _{S}S/I_{0}\right)  ^{\sim
}\rightarrow\widetilde{T_{m}}$ is the canonical morphism obtained from the
$S$-homomorphism $_{S}\omega$. By Lemma \ref{lemDC0}, $T_{m}$ has the
$S/I_{0}$-module structure too such that $\omega:S/I_{0}\rightarrow T_{m}$ is
a morphism of $S/I_{0}$-modules. To fix that $S/I_{0}$-module structure on
$T_{m}$ we use the notation $N_{m}$ instead of $T_{m}$, that is, $T_{m}%
=_{S}N_{m}$. Then $\widetilde{N_{m}}$ is a coherent $\mathcal{O}_{Y}$-module
such that $\iota_{\ast}\widetilde{N_{m}}=\widetilde{T_{m}}$ (see (\ref{fdi}))
along $\iota:Y\rightarrow\mathbb{P}_{k}^{n}$. Notice that $S/I_{0}$-morphism
$\omega:S/I_{0}\rightarrow T_{m}$ defines also $\mathcal{O}_{Y}$-module
morphism $\omega^{\sim}:\mathcal{O}_{Y}\rightarrow\widetilde{N_{m}}$ such that
$\iota_{\ast}\left(  \omega^{\sim}\right)  =\left(  _{S}\omega\right)  ^{\sim
}=\tau$, where $\iota_{\ast}\left(  \omega^{\sim}\right)  :\iota_{\ast
}\mathcal{O}_{Y}\rightarrow\iota_{\ast}\widetilde{N_{m}}$ is the application
of the functor $\iota_{\ast}$ to the sheaf morphism $\omega^{\sim}$.

If $I_{m}$ is a graded ideal of the commutative graded algebra $S\otimes
R_{q}^{m}\left\langle \mathbf{y}\right\rangle $ then $N_{m}$ (or $T_{m}$) is a
graded algebra and (\ref{d1}) is a commutative diagram of the graded algebras
and homomorphisms. In particular, we have the closed subscheme $Y_{m}%
=\operatorname{Proj}\left(  N_{m}\right)  $ of the scheme $\bigsqcup
\limits^{r_{m}}\mathbb{P}_{k}^{n}$ such that $\mathcal{O}_{Y_{m}%
}=\widetilde{N_{m}}$, and the graded homomorphism $\omega:S/I_{0}\rightarrow
N_{m}$ defines the scheme morphism $\rho:Y_{m}\rightarrow Y$ (or $\rho_{m}$)
such that $\rho^{\times}:\mathcal{O}_{Y}\rightarrow\rho_{\ast}\mathcal{O}%
_{Y_{m}}$ is reduced to $\omega^{\sim}:\left(  S/I_{0}\right)  ^{\sim
}\rightarrow\widetilde{N_{m}}$. In this case, (\ref{d1}) is equivalent to the
following diagram of schemes
\[%
\begin{array}
[c]{ccc}%
\mathbb{P}_{k}^{n} & \overset{\sigma}{\longleftarrow} & \bigsqcup
\limits^{r_{m}}\mathbb{P}_{k}^{n}\\
\uparrow_{\iota} &  & \uparrow_{\iota_{m}}\\
Y & \overset{\rho}{\longleftarrow} & Y_{m}%
\end{array}
\]
whose vertical arrows are closed immersions. Moreover, $\tau$ is obtained from
the sheaf morphism $\rho^{\times}:\mathcal{O}_{Y}\rightarrow\rho_{\ast
}\mathcal{O}_{Y_{m}}$ on $Y$ by applying the functor $\iota_{\ast}$, that is,
$\tau=\iota_{\ast}\left(  \rho^{\times}\right)  $.

\begin{lemma}
\label{lemExactSh}Let $\left\{  I_{m}\right\}  $ be a differential chain in
$S_{q}$. Then $\mathcal{O}_{q,Y}=\prod\limits_{m=0}^{\infty}\widetilde{N_{m}%
}\left(  -m\right)  $ is a sheaf of noncommutative $k$-algebras on $Y$ such
that $\left(  Y,\mathcal{O}_{q,Y}\right)  $ is an NC-complete subscheme of
$\mathbb{P}_{k,q}^{n}$ with its sheaf of ideals $\mathcal{I}_{q,Y}%
=\prod\limits_{m=0}^{\infty}\widetilde{I_{m}}\left(  -m\right)  $. Moreover,
$0\rightarrow\mathcal{I}_{q,Y}\left(  d\right)  \longrightarrow\mathcal{O}%
_{q}\left(  d\right)  \longrightarrow\iota_{\ast}\mathcal{O}_{q,Y}\left(
d\right)  \rightarrow0$ remains exact for all $d\in\mathbb{Z}$, where
$\mathcal{I}_{q,Y}\left(  d\right)  =\prod\limits_{m=0}^{\infty}%
\widetilde{I_{m}}\left(  d-m\right)  $ and $\mathcal{O}_{q,Y}\left(  d\right)
=\prod\limits_{m=0}^{\infty}\widetilde{N_{m}}\left(  d-m\right)  $ are
topological twistings of the sheaves $\mathcal{I}_{q,Y}$ and $\mathcal{O}%
_{q,Y}$, respectively.
\end{lemma}

\begin{proof}
First recall that $\epsilon\left(  I_{0}\right)  \subseteq I_{m}$ (see
Subsection \ref{SubsecDC1}). The following diagram of $\mathcal{O}$-modules%
\[%
\begin{array}
[c]{ccccccc}%
0\rightarrow & \widetilde{I_{m}} & \longrightarrow & \mathcal{O}\otimes
R_{q}^{m}\left\langle \mathbf{y}\right\rangle  & \longrightarrow &
\widetilde{T_{m}} & \rightarrow0\\
& \uparrow &  & \uparrow_{\sigma^{\times}} &  & \uparrow_{\tau} & \\
0\rightarrow & \widetilde{I_{0}} & \longrightarrow & \mathcal{O} &
\overset{\iota^{\times}}{\longrightarrow} & \iota_{\ast}\mathcal{O}_{Y} &
\rightarrow0
\end{array}
\]
with its exact rows commutes. Since the twisting functor (see Subsection
\ref{SubsecPSA}) is exact and $\widetilde{T_{m}}\left(  l\right)  =\iota
_{\ast}\left(  \widetilde{N_{m}}\right)  \left(  l\right)  =\iota_{\ast
}\left(  \widetilde{N_{m}}\left(  l\right)  \right)  $, $l\in\mathbb{Z}$ (see
(\ref{fdi})), it follows that%
\[
0\rightarrow\left\{  \prod\limits_{s=0}^{m}\widetilde{I_{k}}\left(
d-s\right)  \right\}  \rightarrow\left\{  \prod\limits_{s=0}^{m}%
\mathcal{O}\left(  d-s\right)  \otimes R_{q}^{s}\left\langle \mathbf{y}%
\right\rangle \right\}  \rightarrow\left\{  \prod\limits_{s=0}^{m}\iota_{\ast
}\left(  \widetilde{N_{k}}\left(  d-s\right)  \right)  \right\}  \rightarrow0
\]
is an exact sequence of inverse systems of the coherent sheaves for every
$d\in\mathbb{Z}$. But $\left\{  \prod\limits_{s=0}^{m}\widetilde{I_{s}}\left(
d-s\right)  \right\}  $ satisfies (ML) and it is a system of the coherent
sheaves on $\mathbb{P}_{k}^{n}$. Therefore $\underleftarrow{\lim}^{\left(
1\right)  }\left\{  \prod\limits_{s=0}^{m}\widetilde{I_{s}}\left(  d-s\right)
\right\}  =0$ (see \cite[1.4.9]{Harts2}) and the inverse limit sequence
\begin{equation}
0\rightarrow\mathcal{I}_{q,Y}\left(  d\right)  \longrightarrow\mathcal{O}%
_{q}\left(  d\right)  \longrightarrow\iota_{\ast}\prod\limits_{m=0}^{\infty
}\widetilde{N_{m}}\left(  d-m\right)  \rightarrow0 \label{d2}%
\end{equation}
remains exact being a sequence of $\mathcal{O}$-modules on $\mathbb{P}_{k}%
^{n}$, where $\mathcal{O}_{q}\left(  d\right)  =\prod\limits_{m=0}^{\infty
}\mathcal{O}\left(  d-m\right)  \otimes R_{q}^{m}\left\langle \mathbf{y}%
\right\rangle $ and $\mathcal{I}_{q,Y}\left(  d\right)  =\prod\limits_{m=0}%
^{\infty}\widetilde{I_{m}}\left(  d-m\right)  $ are topological twistings of
the sheaves $\mathcal{O}_{q}$ and $\mathcal{I}_{q,Y}$($=\mathcal{I}%
_{q,Y}\left(  0\right)  $), respectively. Take a homogeneous $h\in S_{+}$ and
consider the affine open subset $D_{+}\left(  h\right)  \subseteq
\mathbb{P}_{k}^{n}$. Using Proposition \ref{proptp1}, we derive that
$\mathcal{I}_{q,Y}\left(  D_{+}\left(  h\right)  \right)  =\prod
\limits_{m=0}^{\infty}I_{m,h}^{-m}=I_{h\left(  \mathbf{x}\right)  }%
^{0}=I_{\left(  h\left(  \mathbf{x}\right)  \right)  }$ is a closed two-sided
ideal of the algebra $S_{q,\left(  h\left(  \mathbf{x}\right)  \right)  }$.
But $S_{q,\left(  h\left(  \mathbf{x}\right)  \right)  }=S_{q,h\left(
\mathbf{x}\right)  }^{0}=\prod\limits_{m=0}^{\infty}S_{h}^{-m}\otimes
R^{m}\left\langle \mathbf{y}\right\rangle =\prod\limits_{m=0}^{\infty
}\mathcal{O}\left(  -m\right)  \otimes R_{q}^{m}\left\langle \mathbf{y}%
\right\rangle \left(  D_{+}\left(  h\right)  \right)  =\mathcal{O}_{q}\left(
D_{+}\left(  h\right)  \right)  $. Thus $\mathcal{I}_{q,Y}$ is a sheaf of
closed two-sided ideals, therefore $\prod\limits_{m=0}^{\infty}%
\widetilde{N_{m}}\left(  -m\right)  $ is a sheaf of noncommutative
$k$-algebras on the projective scheme $Y$ denoted by $\mathcal{O}_{q,Y}$. The
sequence (\ref{d2}) for $d=0$ provides an exact sequence of the noncommutative
$k$-algebras which makes $\left(  \iota,\iota^{+}\right)  :\left(
Y,\mathcal{O}_{q,Y}\right)  \rightarrow\mathbb{P}_{k,q}^{n}$ to be a closed
immersion of the NC-complete schemes.
\end{proof}

\begin{remark}
\label{remzoverx}As above take $D_{+}\left(  h\right)  \subseteq U_{i}%
=D_{+}\left(  x_{i}\right)  $. Since every $I_{m,h}$ is a graded $S_{h}%
$-module (in particular, $S_{x_{i}}$-module), it follows that $x_{i}%
^{d}\left(  I_{m,h}^{e}\right)  =I_{m,h}^{d+e}$ for all $e,d\in\mathbb{Z}$.
Then $x_{i}^{d}\left(  \mathcal{I}_{q,Y}\left(  e\right)  |U_{i}\right)
=\mathcal{I}_{q,Y}\left(  d+e\right)  |U_{i}$, that is, the diagonal operators
$x_{i}^{d}$ define the topological twistings $\mathcal{I}_{q,Y}\left(
d\right)  $. In particular, $\dfrac{1}{x_{i}^{e_{j}}}\left(  \mathcal{I}%
_{q,Y}|U_{i}\right)  \subseteq\mathcal{I}_{q,Y}\left(  -e_{j}\right)  |U_{i}$
for all $j$. But $\left\{  I_{m}\right\}  $ is a chain, therefore
$z_{j}\left(  \mathcal{I}_{q,Y}|U_{i}\right)  \subseteq\mathcal{I}%
_{q,Y}\left(  e_{j}\right)  |U_{i}$ for all $j>n$. Whence $\dfrac{z_{j}}%
{x_{i}^{e_{j}}}\left(  \mathcal{I}_{q,Y}|U_{i}\right)  \subseteq
\mathcal{I}_{q,Y}|U_{i}$ for all $j$.
\end{remark}

Recall that $\mathcal{O}_{q,\ast}$ has the $S_{q}$-bimodule structure given by
the regular representations $L,R:S_{q}\rightarrow\operatorname{Hom}_{k}\left(
\mathcal{O}_{q,\ast}\right)  $ (see Subsection \ref{subsecGmod}), and
$\mathcal{I}_{q,Y,\ast}=\bigoplus\limits_{d}\mathcal{I}_{q,Y}\left(  d\right)
$ turns out to be a subsheaf of $\mathcal{O}_{q,\ast}$. Based on Lemma
\ref{lemExactSh}, we also put $\mathcal{O}_{q,Y,\ast}=\bigoplus\limits_{d}%
\mathcal{O}_{q,Y}\left(  d\right)  $ to be a sheaf of $k$-spaces on $Y$.

\begin{theorem}
\label{propDCS1}There is a natural $S_{q}$-bimodule structure on the
$\mathcal{O}$-module $\iota_{\ast}\mathcal{O}_{q,Y,\ast}$ such that
\[
0\rightarrow\mathcal{I}_{q,Y,\ast}\rightarrow\mathcal{O}_{q,\ast}%
\rightarrow\iota_{\ast}\mathcal{O}_{q,Y,\ast}\rightarrow0
\]
is an exact sequence of $S_{q}$-bimodules.
\end{theorem}

\begin{proof}
The $S_{q}$-bimodule structure on $\mathcal{O}_{q,\ast}$ is given by the
operators $L_{j}:\mathcal{O}_{q}\rightarrow\mathcal{O}_{q}\left(
e_{j}\right)  $ and $R_{j}:\mathcal{O}_{q}\rightarrow\mathcal{O}_{q}\left(
e_{j}\right)  $ (see Subsection \ref{subsecGmod}). By Lemma \ref{lemtp1},
$L_{i}\left(  \mathcal{I}_{q,Y}\left(  d\right)  \right)  \subseteq
\mathcal{I}_{q,Y}\left(  d+1\right)  $ for all $i$, $0\leq i\leq n$, and
$d\in\mathbb{Z}$. In particular, $\dfrac{1}{x_{i}^{k}}L_{i}^{k}\left(
\mathcal{I}_{q,Y}|U_{i}\right)  \subseteq\mathcal{I}_{q,Y}|U_{i}$. By
Proposition \ref{propSBM1}, all operators $T_{ik}$ leave invariant
$\mathcal{I}_{q,Y}|U_{i}$. Note that $\mathcal{I}_{q,Y}$ is an $\mathcal{O}%
$-submodule of $\mathcal{O}_{q}$ which in turn is a sheaf of closed two-sided
ideals by Lemma \ref{lemExactSh}. Based on Remark \ref{remzoverx}, we have
$\dfrac{z_{j}}{x_{i}^{e_{j}}}\left(  \mathcal{I}_{q,Y}|U_{i}\right)
\subseteq\mathcal{I}_{q,Y}|U_{i}$ for all $j$. Using Proposition
\ref{propSBM2}, we conclude that $\mathcal{I}_{q,Y,\ast}$ is a left $S_{q}%
$-submodule of $\mathcal{O}_{q,\ast}$, and $\Delta_{j}\left(  \mathcal{I}%
_{q,Y}\left(  d\right)  \right)  \subseteq\mathcal{I}_{q,Y}\left(
d+e_{j}\right)  $ for all $j$ and $d$. Similarly, $R_{i}\left(  \mathcal{I}%
_{q,Y}\left(  d\right)  \right)  \subseteq\mathcal{I}_{q,Y}\left(
d+e_{j}\right)  $ implies that $\mathcal{I}_{q,Y,\ast}$ is a right $S_{q}%
$-submodule of $\mathcal{O}_{q,\ast}$ (see Remark \ref{remTop4}) and
$\nabla_{j}\left(  \mathcal{I}_{q,Y}\left(  d\right)  \right)  \subseteq
\mathcal{I}_{q,Y}\left(  d+e_{j}\right)  $ for all $d$ and $j$. It follows
that there are well defined operators $\Delta_{j,Y}$, $\nabla_{j,Y}%
:\mathcal{O}_{q,Y}\left(  d\right)  \rightarrow\mathcal{O}_{q,Y}\left(
d+e_{j}\right)  $ such that $\iota_{\ast}\Delta_{j,Y}$, $\iota_{\ast}%
\nabla_{j,Y}$ commute with the quotient mapping $\mathcal{O}_{q,\ast
}\longrightarrow\iota_{\ast}\mathcal{O}_{q,Y,\ast}$ (see Lemma
\ref{lemExactSh}). Thus we have the $S_{q}$-bimodule structure on the
$\mathcal{O}$-module $\mathcal{O}_{q,Y}$ such that $0\rightarrow
\mathcal{I}_{q,Y,\ast}\rightarrow\mathcal{O}_{q,\ast}\rightarrow\iota_{\ast
}\mathcal{O}_{q,Y,\ast}\rightarrow0$ is an exact sequence of $S_{q}$-modules.
Equivalently, the sheaf representations $L,R:S_{q}\rightarrow
\operatorname{Hom}_{k}\left(  \mathcal{O}_{q,\ast}\right)  $ leave invariant
the subsheaf $\mathcal{I}_{q,Y,\ast}$, and they are in turn generate the sheaf
morphisms $L_{Y},R_{Y}:S_{q}\rightarrow\operatorname{Hom}_{k}\left(
\mathcal{O}_{q,Y,\ast}\right)  $.
\end{proof}

Note that $\mathcal{O}_{q,Y}$ is an inverse limit of the coherent
$\mathcal{O}_{Y}$-modules. If the differential chain $\left\{  I_{m}\right\}
$ consists of (commutative) ideals then there are closed subschemes $Y_{m}$ in
$\bigsqcup\limits^{r_{m}}\mathbb{P}_{k}^{n}$ and%
\begin{equation}
\mathcal{O}_{q,Y}=\prod\limits_{m=0}^{\infty}\rho_{m,\ast}\mathcal{O}_{Y_{m}%
}\left(  -m\right)  , \label{FI2}%
\end{equation}
where $\rho_{m}:Y_{m}\rightarrow Y$ are scheme morphisms obtained from the
canonical morphisms $\sigma_{m}:\bigsqcup\limits^{r_{m}}\mathbb{P}_{k}%
^{n}\rightarrow\mathbb{P}_{k}^{n}$ (see Subsection \ref{subsecDU}),
respectively. In this case, we say that $\left(  Y,\mathcal{O}_{q,Y}\right)  $
is \textit{a geometric quantization} of the scheme $\left(  Y,\mathcal{O}%
_{Y}\right)  $. Being a commutative ideal of $S\otimes R_{q}^{m}\left\langle
\mathbf{y}\right\rangle $, we obtain that $I_{m}=\bigoplus_{\left\langle
\alpha\right\rangle =m}I_{m,\alpha}\mathbf{y}^{\alpha}$ for some graded ideals
$I_{m,\alpha}\subseteq S$, and $I_{m}^{d}=\bigoplus_{\left\langle
\alpha\right\rangle =m}I_{m,\alpha}^{d}\mathbf{y}^{\alpha}$ for all $d$.
Moreover, $N_{m}=\bigoplus_{\left\langle \alpha\right\rangle =m}\left(
S/I_{m,\alpha}\right)  \mathbf{y}^{\alpha}$ and $Y_{m}=\bigsqcup
\limits_{\left\langle \alpha\right\rangle =m}Y_{m,\alpha}$ for a family of
closed subschemes $\left\{  Y_{m,\alpha}\right\}  $ of $Y_{m}$. Finally,
$\rho_{m,\ast}\mathcal{O}_{Y_{m}}=\bigoplus_{\left\langle \alpha\right\rangle
=m}\mathcal{O}_{Y_{m,\alpha}}$ thanks to Lemma \ref{lemDU2}. Below we will see
several examples of the geometric quantizations of the projective schemes.

\section{Projective $q$-schemes and their differential chains\label{Sec5}}

In this section we introduce the projective $q$-schemes and classify them as
the schemes from differential chains.

\subsection{Projective $q$-schemes}

Let $\left(  Y,\mathcal{F}_{Y}\right)  $ be a projective NC-deformation over
$k$. It means that $\left(  Y,\mathcal{F}_{Y}\right)  \rightarrow
\operatorname{Spec}\left(  k\right)  $ is a projective morphism being
factorized into a closed NC-immersion $\left(  \iota,\iota^{+}\right)
:\left(  Y,\mathcal{F}_{Y}\right)  \rightarrow\mathbb{P}_{k,q}^{n}$ for some
$n$, $q$ followed by the projection $\mathbb{P}_{k,q}^{n}\rightarrow
\operatorname{Spec}\left(  k\right)  $ (see Subsection \ref{subsecCIPS}). In
this case, $\left(  Y,\mathcal{O}_{Y}\right)  $ is a projective scheme over
$k$ and there is a very ample sheaf $\mathcal{O}_{Y}\left(  1\right)  $ on
$Y$. Namely, the closed NC-immersion $\left(  \iota,\iota^{+}\right)  $
defines the closed immersion $\left(  \iota,\iota^{\times}\right)  :\left(
Y,\mathcal{O}_{Y}\right)  \rightarrow\mathbb{P}_{k}^{n}$ and we put
$\mathcal{O}_{Y}\left(  1\right)  =\iota^{\ast}\left(  \mathcal{O}\left(
1\right)  \right)  $ to be the inverse image of $\mathcal{O}\left(  1\right)
$ along $\iota$. Using Corollary \ref{corCOH1}, we have the exact sequences
$0\rightarrow\mathcal{I}_{Y,m}\longrightarrow\mathcal{O}_{q}/\mathcal{J}%
_{m}\overset{\iota_{m}^{+}}{\longrightarrow}\iota_{\ast}\mathcal{F}%
_{Y,m}\rightarrow0$ of the coherent $\mathcal{O}$-modules, $\iota
^{+}=\underleftarrow{\lim}\left\{  \iota_{m}^{+}\right\}  $, and all canonical
maps $\mathcal{I}_{Y,m+1}\rightarrow\mathcal{I}_{Y,m}$ are surjective. Taking
into account that the algebraic twisting functor is exact (see Subsection
\ref{SubsecPSA}), we conclude that all sequences $0\rightarrow\mathcal{I}%
_{Y,m}\left(  d\right)  \longrightarrow\mathcal{O}_{q}/\mathcal{J}_{m}\left(
d\right)  \overset{\iota_{m}^{+}}{\longrightarrow}\iota_{\ast}\mathcal{F}%
_{Y,m}\left(  d\right)  \rightarrow0$, $d\in\mathbb{Z}$ remain exact, and
$\mathcal{O}_{q}\left(  d\right)  =\underleftarrow{\lim}\left\{
\mathcal{O}_{q}/\mathcal{J}_{m}\left(  d\right)  \right\}  $. We define the
topological twisting $\mathcal{F}_{Y}\left(  d\right)  $ of $\mathcal{F}_{Y}$
to be $\underleftarrow{\lim}\left\{  \mathcal{F}_{Y,m}\left(  d\right)
\right\}  $ with $\mathcal{F}_{Y,m}\left(  d\right)  =\mathcal{F}_{Y,m}%
\otimes_{\mathcal{O}_{Y}}\mathcal{O}_{Y}\left(  d\right)  $. Notice that
\[
\iota_{\ast}\left(  \mathcal{F}_{Y,m}\left(  d\right)  \right)  =\iota_{\ast
}\left(  \mathcal{F}_{Y,m}\otimes_{\mathcal{O}_{Y}}\iota^{\ast}\left(
\mathcal{O}\left(  d\right)  \right)  \right)  =\iota_{\ast}\left(
\mathcal{F}_{Y,m}\right)  \otimes_{\mathcal{O}}\mathcal{O}\left(  d\right)
=\iota_{\ast}\left(  \mathcal{F}_{Y,m}\right)  \left(  d\right)
\]
(see \cite[Exercise 2.5.1 (d)]{Harts}), which in turn implies that
$\iota_{\ast}\left(  \mathcal{F}_{Y}\left(  d\right)  \right)  =\iota_{\ast
}\left(  \mathcal{F}_{Y}\right)  \left(  d\right)  $ for all $d\in\mathbb{Z}$.

Using again \cite[1.4.9]{Harts2}, we derive that the inverse limit
$0\rightarrow\underleftarrow{\lim}\left\{  \mathcal{I}_{Y,m}\left(  d\right)
\right\}  \longrightarrow\mathcal{O}_{q}\left(  d\right)  \overset{\iota
^{+}}{\longrightarrow}\iota_{\ast}\mathcal{F}_{Y}\left(  d\right)
\rightarrow0$ remains exact. In particular, $\mathcal{I}_{q,Y}%
=\underleftarrow{\lim}\left\{  \mathcal{I}_{Y,m}\right\}  $ is a sheaf of
closed two-sided ideals and its topological twisting is reduced to
$\underleftarrow{\lim}\left\{  \mathcal{I}_{Y,m}\left(  d\right)  \right\}  $.

Now fix the following differential operators (see Remark \ref{remDIK})
\[
D_{il}:\mathcal{O}_{q}\rightarrow\mathcal{O}_{q}\left(  l+1\right)  ,\quad
D_{il}=\sum_{\left\vert \beta+l\right\vert _{i}>0}R\left(  \operatorname{ad}%
\left(  \overline{\mathbf{y}}\right)  ^{\beta_{y}}\operatorname{ad}\left(
x_{n}\right)  ^{\beta_{n}}\cdots\operatorname{ad}\left(  x_{i}\right)
^{\beta_{i}+l}\cdots\operatorname{ad}\left(  x_{0}\right)  ^{\beta_{0}}\left(
x_{i}\right)  \right)  \overline{\partial}^{\beta},
\]
and
\[
\nabla_{i}:\mathcal{O}_{q}\rightarrow\mathcal{O}_{q}\left(  1\right)
,\quad\nabla_{i}=-\sum_{\mathbf{i}\in\mathbb{Z}_{+}^{v+1},\left\vert
\mathbf{i}\right\vert _{i}=0}R\left(  \operatorname{ad}\left(  \overline
{\mathbf{z}}\right)  ^{\mathbf{i}}\left(  z_{i}\right)  \right)
\overline{\partial}^{\mathbf{i}},
\]
where $0\leq l<q$, $0\leq i\leq n$. Note that $D_{i0}=\Delta_{i}$ for all $i$.
These operators can be treated as the elements of $\operatorname{Hom}%
_{k}\left(  \mathcal{O}_{q,\ast}\right)  $, whereas $x_{i}^{-l-1}D_{il}%
=T_{il}\in\operatorname{Hom}_{k}\left(  \mathcal{O}_{q}|U_{i}\right)  $ (see
Subsection \ref{subsecGmod}) for all $i,l$.

\begin{definition}
\label{defPQS}Let $\left(  Y,\mathcal{F}_{Y}\right)  $ be a projective
NC-deformation of the projective scheme $\left(  Y,\mathcal{O}_{Y}\right)  $
over $k$, and let $\mathcal{O}_{Y}\left(  1\right)  $ be a very ample sheaf on
$Y$. If there are operators $D_{il,Y}:\mathcal{F}_{Y}\rightarrow
\mathcal{F}_{q}\left(  l+1\right)  $ and $\nabla_{i,Y}:\mathcal{F}%
_{Y}\rightarrow\mathcal{F}_{Y}\left(  1\right)  $, $0\leq l<q$, $0\leq i\leq
n$ lifting the related differential operators through a closed NC-immersion
$\left(  Y,\mathcal{F}_{Y}\right)  \rightarrow\mathbb{P}_{k,q}^{n}$ then
$\left(  Y,\mathcal{F}_{Y}\right)  $ is called a projective $q$-scheme over
$k$.
\end{definition}

Thus if $\left(  \iota,\iota^{+}\right)  :\left(  Y,\mathcal{F}_{Y}\right)
\rightarrow\mathbb{P}_{k,q}^{n}$ is a closed NC-immersion, then $\iota
^{+}\left(  l+1\right)  D_{il}=\left(  \iota_{\ast}D_{il,Y}\right)  \iota^{+}$
and $\iota^{+}\left(  1\right)  \nabla_{i}=\left(  \iota_{\ast}\nabla
_{i,Y}\right)  \iota^{+}$ for all $i,l$. Note that $\mathbb{P}_{k,q}^{n}$
itself is a projective $q$-scheme over $k$.

\begin{proposition}
\label{propPROJ1}Let $\left(  Y,\mathcal{F}_{Y}\right)  $ be a projective
NC-deformation over $k$. The following assertions are equivalent: $\left(
i\right)  $ $\left(  Y,\mathcal{F}_{Y}\right)  $ is a projective $q$-scheme
over $k$;

$\left(  ii\right)  $ There exists a closed NC-immersion $\left(
Y,\mathcal{F}_{Y}\right)  \rightarrow\mathbb{P}_{k,q}^{n}$ such that the local
operators $\dfrac{1}{x_{i}^{l}}L_{i}^{l}$ and $\dfrac{1}{x_{i}}R_{i}$ on
$\mathcal{O}_{q}|U_{i}$ leave invariant the subsheaf $\mathcal{I}_{q,Y}|U_{i}$
for all $i,l$;

$\left(  iii\right)  $ The sheaf $\mathcal{F}_{Y,\ast}$ has an $S_{q}%
$-bimodule structure being the quotient of $\mathcal{O}_{q,\ast}$ along a
closed NC-immersion $\left(  Y,\mathcal{F}_{Y}\right)  \rightarrow
\mathbb{P}_{k,q}^{n}$.
\end{proposition}

\begin{proof}
First assume that $\left(  Y,\mathcal{F}_{Y}\right)  $ is a projective
$q$-scheme over $k$. By Definition \ref{defPQS}, there is a closed
NC-immersion $\left(  \iota,\iota^{+}\right)  :\left(  Y,\mathcal{F}%
_{Y}\right)  \rightarrow\mathbb{P}_{k,q}^{n}$ such that the operators
$D_{il,Y}:\mathcal{F}_{Y}\rightarrow\mathcal{F}_{q}\left(  l+1\right)  $ are
lifting the related differential operators on $\mathcal{O}_{q,\ast}$. Since
$\mathcal{F}_{Y}\left(  e\right)  =\underleftarrow{\lim}\left\{
\mathcal{F}_{Y,m}\left(  e\right)  \right\}  $ and all $\mathcal{F}%
_{Y,m}\left(  e\right)  $ are the coherent $\mathcal{O}_{Y}$-modules, we have
$x_{i}^{d}\left(  \mathcal{F}_{Y}\left(  e\right)  \right)
=\underleftarrow{\lim}\left\{  x_{i}^{d}\mathcal{F}_{Y,m}\left(  e\right)
\right\}  =\underleftarrow{\lim}\left\{  \mathcal{F}_{Y,m}\left(  d+e\right)
\right\}  =\mathcal{F}_{Y}\left(  d+e\right)  $ on $\iota^{-1}\left(
U_{i}\right)  $ for all $d,e\in\mathbb{Z}$. The operators $\dfrac{1}%
{x_{i}^{l+1}}D_{il,Y}$ on $\mathcal{F}_{Y}|\iota^{-1}\left(  U_{i}\right)  $
are lifting the local operators $T_{il}$ on $\mathcal{O}_{q}|U_{i}$,
respectively. In particular, all $T_{il}$ leave invariant the subsheaf
$\mathcal{I}_{q,Y}|U_{i}$. By Proposition \ref{propSBM1}, all $\dfrac{1}%
{x_{i}^{l}}L_{i}^{l}$ (and $\dfrac{1}{x_{i}^{e_{j}}}L_{j}$) leave invariant
$\mathcal{I}_{q,Y}|U_{i}$. A similar statement takes place for the right
multiplication operators. Hence $\left(  i\right)  \Rightarrow\left(
ii\right)  $.

Now assume that $\left(  ii\right)  $ holds. Since the operators $\dfrac
{1}{x_{i}^{l}}L_{i}^{l}$ on $\mathcal{O}_{q}|U_{i}$ leave invariant the
subsheaf $\mathcal{I}_{q,Y}|U_{i}$, it follows that $T_{il}\left(
\mathcal{I}_{q,Y}|U_{i}\right)  \subseteq\mathcal{I}_{q,Y}|U_{i}$ for all
$i,l$ by virtue of Proposition \ref{propSBM1}. Using Proposition
\ref{propSBM2}, we conclude that the left representation $L:S_{q}%
\rightarrow\operatorname{Hom}_{k}\left(  \mathcal{O}_{q,\ast}\right)  $ leaves
invariant $\mathcal{I}_{q,Y,\ast}$. Hence $\mathcal{F}_{Y,\ast}$ turns out to
be an $S_{q}$-bimodule structure being the quotient of $\mathcal{O}_{q,\ast}$
along $\left(  \iota,\iota^{+}\right)  $, that is, $\left(  ii\right)
\Rightarrow\left(  iii\right)  $.

Finally, prove that $\left(  iii\right)  \Rightarrow\left(  i\right)  $.
Suppose that $\mathcal{F}_{Y,\ast}$ has an $S_{q}$-bimodule structure being
the quotient of $\mathcal{O}_{q,\ast}$ along a closed NC-immersion $\left(
Y,\mathcal{F}_{Y}\right)  \rightarrow\mathbb{P}_{k,q}^{n}$. There are the
multiplication operators $L_{j,Y}:\mathcal{F}_{Y}\rightarrow\mathcal{F}%
_{Y}\left(  e_{j}\right)  $ lifting $L_{j}$. Since $D_{il}=\operatorname{ad}%
\left(  x_{i}\right)  ^{l}\left(  L_{i}\right)  $ (see Remark \ref{remDIK}),
it follows that $D_{il}\left(  \mathcal{I}_{q,Y,\ast}\right)  \subseteq
\mathcal{I}_{q,Y,\ast}$ for all $i,l$. In particular, there are operators
$D_{il,Y}:\mathcal{F}_{Y}\rightarrow\mathcal{F}_{Y}\left(  l+1\right)  $
lifting $D_{il}:\mathcal{O}_{q}\rightarrow\mathcal{O}_{q}\left(  l+1\right)
$. Similarly, $\nabla_{i}=R_{i}-x_{i}$ leaves invariant $\mathcal{I}%
_{q,Y,\ast}$, therefore there is an operator $\nabla_{i,Y}:\mathcal{F}%
_{Y}\rightarrow\mathcal{F}_{Y}\left(  1\right)  $ lifting $\nabla_{i}$. Whence
$\left(  Y,\mathcal{F}_{Y}\right)  $ is a projective $q$-scheme.
\end{proof}

\begin{remark}
Notice that just the presence of the operators $\Delta_{i,Y}:\mathcal{F}%
_{Y}\rightarrow\mathcal{F}_{Y}\left(  1\right)  $ is not enough to be a
projective $q$-scheme. That is the case even for $q=2$.
\end{remark}

\subsection{The NC-nilpotent case}

Now assume for a while that $\left(  Y,\mathcal{F}_{Y}\right)  $ is an
NC-nilpotent (with the trivial filtration), projective NC-deformation over
$k$, and let $\mathcal{O}_{Y}\left(  1\right)  $ be a very ample sheaf on $Y$.
If $\left(  \iota,\iota^{+}\right)  :\left(  Y,\mathcal{F}_{Y}\right)
\rightarrow\mathbb{P}_{k,q}^{n}$ is a closed NC-immersion, then $\iota^{+}$ is
a filtered surjective morphism and $\mathcal{F}_{Y}$ is a coherent
$\mathcal{O}_{Y}$-module (see Corollary \ref{corCOH1}). It follows that
$\iota^{+}$ can be factorized through a surjective morphism $\iota
^{+}:\mathcal{O}_{q}/\mathcal{J}_{m}\rightarrow\iota_{\ast}\mathcal{F}_{Y}$
for some $m$, that is, we can assume that $\iota$ is a closed NC-immersion
$\left(  Y,\mathcal{F}_{Y}\right)  \rightarrow\left(  \mathbb{P}_{k,q}%
^{n},\mathcal{O}_{q}/\mathcal{J}_{m}\right)  $ of the NC-nilpotent schemes. If
$\mathcal{I}_{q,Y}$ is the two-sided ideal sheaf of the NC-nilpotent scheme
$\left(  Y,\mathcal{F}_{Y}\right)  $, then $\mathcal{J}_{m}\subseteq
\mathcal{I}_{q,Y}$ and $\mathcal{I}_{q,Y}$ turns out to be a subsheaf of open
(therefore closed) two-sided ideals in $\mathcal{O}_{q}$. Then $\mathcal{I}%
_{q,Y}$ can be identified with the related two-sided ideal sheaf of
$\mathcal{O}_{q}/\mathcal{J}_{m}$ such that $0\rightarrow\mathcal{I}%
_{q,Y}\longrightarrow\mathcal{O}_{q}/\mathcal{J}_{m}\overset{\iota
^{+}}{\longrightarrow}\iota_{\ast}\mathcal{F}_{Y}\rightarrow0$ is an exact
sequence of$\ \mathcal{O}$-modules. Since both $\mathcal{O}_{q}/\mathcal{J}%
_{m}$ and $\iota_{\ast}\mathcal{F}_{Y}$ are the coherent $\mathcal{O}$-modules
\cite[2.5.8 (c)]{Harts}, it follows that so is $\mathcal{I}_{q,Y}$
\cite[2.5.7]{Harts}. Based on the Theorem of Serre \cite[2.5.15]{Harts}, we
conclude that $\mathcal{I}_{q,Y}=\widetilde{I}$ with the graded $S$-module
$I=\Gamma_{\ast}\left(  \mathcal{I}_{q,Y}\right)  =\oplus_{d}I^{d}$ such that%
\[
I^{d}=\Gamma\left(  \mathbb{P}_{k}^{n},\mathcal{I}_{q,Y}\left(  d\right)
\right)  \subseteq\Gamma\left(  \mathbb{P}_{k}^{n},\left(  \mathcal{O}%
_{q}/\mathcal{J}_{m}\right)  \left(  d\right)  \right)  =\bigoplus
\limits_{s=0}^{m-1}S^{d-s}\otimes R_{q}^{s}\left\langle \mathbf{y}%
\right\rangle
\]
is an $S$-submodule for every $d$.

\begin{lemma}
\label{lemCNC1}If $\left(  Y,\mathcal{F}_{Y}\right)  $ is an NC-nilpotent,
projective $q$-scheme, then $I_{h}^{d}=I_{\left(  h\right)  }\cdot x_{i}%
^{d}=x_{i}^{d}\cdot I_{\left(  h\right)  }$ with respect to the noncommutative
multiplication of the topological localization $S_{q,h}$ whenever
$D_{+}\left(  h\right)  \subseteq U_{i}$ and $d\in\mathbb{Z}$. In particular,
$\mathcal{O}_{q}\left(  e\right)  \cdot\mathcal{I}_{q,Y}\left(  d\right)
\cdot\mathcal{O}_{q}\left(  r\right)  |U_{i}\subseteq\mathcal{I}_{q,Y}\left(
e+d+r\right)  |U_{i}$ for all $e,d,r\in\mathbb{Z}$, $0\leq i\leq n$, and the
$S$-submodule $I=\oplus_{d}I^{d}\subseteq S_{q}$ is a graded two-sided ideal.
\end{lemma}

\begin{proof}
Since $\mathcal{I}_{q,Y}=\widetilde{I}$, it follows that $\mathcal{I}%
_{q,Y}\left(  d\right)  \left(  D_{+}\left(  h\right)  \right)  =I_{h}^{d}$
(in the commutative sense) whenever $D_{+}\left(  h\right)  \subseteq U_{i}$
and $d\in\mathbb{Z}$. Actually, the noncommutative localizations over the
affine scheme $U_{i}$ are reduced to the commutative ones (see to the proof of
Theorem \ref{thNCDaff}). Based on Propositions \ref{propPROJ1} and
\ref{propSBM2}, we obtain that $L_{i}^{e}\mathcal{I}\left(  d\right)
|U_{i}=\mathcal{I}\left(  e+d\right)  |U_{i}$ for all $e,d\in\mathbb{Z}$. It
follows that $x_{i}^{e}\cdot I_{h}^{d}=I_{h}^{e+d}$, therefore $x_{i}^{d}\cdot
I_{\left(  h\right)  }=x_{i}^{d}\cdot I_{h}^{0}=I_{h}^{d}$ for all
$d\in\mathbb{Z}$. A similar assertion takes place for the right multiplication
operators. Hence $I_{h}^{d}=I_{\left(  h\right)  }\cdot x_{i}^{d}$ for all
$d\in\mathbb{Z}$. In particular, $x_{i}^{d}\cdot S_{q,\left(  h\right)
}=S_{q,\left(  h\right)  }\cdot x_{i}^{d}=S_{q,h}^{d}$ for all $d\in
\mathbb{Z}$. Then
\[
S_{q,h}^{e}\cdot I_{h}^{d}\cdot S_{q,h}^{r}=S_{q,\left(  h\right)  }\cdot
x_{i}^{e}\cdot x_{i}^{d}\cdot I_{\left(  h\right)  }\cdot S_{q,\left(
h\right)  }\cdot x_{i}^{r}=x_{i}^{e+d}\cdot S_{q,\left(  h\right)  }\cdot
I_{\left(  h\right)  }\cdot S_{q,\left(  h\right)  }\cdot x_{i}^{r}\subseteq
x_{i}^{e+d}\cdot I_{\left(  h\right)  }\cdot x_{i}^{r}\subseteq I_{h}^{e+d+r}%
\]
for all integers $e,d,r$. Hence $\mathcal{O}_{q}\left(  e\right)
\cdot\mathcal{I}_{q,Y}\left(  d\right)  \cdot\mathcal{O}_{q}\left(  r\right)
|U_{i}\subseteq\mathcal{I}_{q,Y}\left(  e+d+r\right)  |U_{i}$ for all $i$.

Finally, prove that $S_{q}^{e}\cdot I^{d}\cdot S_{q}^{r}\subseteq I^{e+d+r}$
for all nonnegative $e,d,r$. Take the global sections $f\in S_{q}^{e}$, $h\in
S_{q}^{r}$ and $g\in I^{d}$. Then $u=f\cdot g\cdot h$ is a global section of
the sheaf $\mathcal{O}_{q}\left(  e+d+r\right)  $, and $u|U_{i}\in
\mathcal{I}_{q,Y}\left(  e+d+r\right)  \left(  U_{i}\right)  =I_{x_{i}%
}^{e+d+r}$ for all $i$. Since $\left\{  U_{i}\right\}  $ is an open covering
of $\mathbb{P}_{k}^{n}$, we conclude that $u\in\Gamma\left(  \mathbb{P}%
_{k}^{n},\mathcal{I}_{q,Y}\left(  e+d+r\right)  \right)  =I^{e+d+r}$. Whence
$I=\oplus_{d}I^{d}$ is a graded ideal of the algebra $S_{q}$.
\end{proof}

\subsection{The main result}

Now we prove our main result on the projective NC\textit{-}deformations
over.$k$.

\begin{theorem}
\label{theoremMain}Let $\left(  Y,\mathcal{F}_{Y}\right)  $ be an
NC-deformation over $k$. Then $\left(  Y,\mathcal{F}_{Y}\right)  $ is an
NC-subscheme of $\mathbb{P}_{k,q}^{n}$ given by a differential chain in
$S_{q}$ if and only if $\left(  Y,\mathcal{F}_{Y}\right)  $ is a projective
$q$\textit{-}scheme over $k$. Thus the projective $q$-schemes over $k$ are
only projective NC-deformations over $k$ obtained from the differential chains
in $S_{q}$.
\end{theorem}

\begin{proof}
First note that a projective NC-complete scheme obtained from a differential
chain in $S_{q}$ is a projective $q$-scheme thanks to Theorem \ref{propDCS1}
and Proposition \ref{propPROJ1}. Namely, the differential operators
$D_{ik},\nabla_{i}\in\operatorname{Hom}_{k}\left(  \mathcal{O}_{q,\ast
}\right)  $ leave invariant the subsheaf $\mathcal{I}_{q,Y,\ast}$, they in
turn generate the operators $D_{ik,Y},\nabla_{i,Y}\in\operatorname{Hom}%
_{k}\left(  \mathcal{F}_{Y,\ast}\right)  $. It means that the NC-complete
scheme $\left(  Y,\mathcal{O}_{q,Y}\right)  $ of a differential chain is a
projective $q$-scheme (see Definition \ref{defPQS}).

Conversely, let $\left(  Y,\mathcal{F}_{Y}\right)  $ be a projective
$q$\textit{-}scheme. First assume that $\left(  Y,\mathcal{F}_{Y}\right)  $ is
an NC-nilpotent scheme. As we have seen above $\mathcal{I}_{q,Y}%
=\widetilde{I}$. Since $\mathcal{J}_{m}\left(  d\right)  \subseteq
\mathcal{I}_{q,Y}\left(  d\right)  $ for all $d$, it follows that
$\mathcal{J}_{m}\left(  \mathbf{x}\right)  \subseteq I$. By Lemma
\ref{lemCNC1}, $I$ is an open graded ideal of $S_{q}$. Using Lemma
\ref{lemSGI1}, we obtain that $I$ is an NC-graded ideal. By Proposition
\ref{propSG1}, $I$ is the sum $I=\oplus_{s}I_{s}$ of a terminating
differential chain $\left\{  I_{s}\right\}  $ in $S_{q}$, and $I^{d}%
=\oplus_{s=0}^{d}I_{s}^{d-s}$ for all $d$, where $I_{s}\subseteq S\otimes
R_{q}^{s}\left\langle \mathbf{y}\right\rangle $ is a graded $S$-submodule.
Since $I$ is identified with a two-sided ideal of $S_{q}/\mathcal{J}%
_{m}\left(  \mathbf{x}\right)  $, we have $I=\oplus_{s=0}^{m-1}I_{s}$. Fix
$i$, $0\leq i\leq n$. The commutative localization $I_{x_{i}}$ is reduced to
$\oplus_{s=0}^{m-1}I_{s,x_{i}}$. By Proposition \ref{proptp1}, $I_{x_{i}}$ is
in turn the noncommutative localization of $I$ at $x_{i}$. It follows that
\[
I_{\left(  x_{i}\right)  }=%
{\textstyle\sum\limits_{d}}
I^{d}\cdot x_{i}^{-d}=%
{\textstyle\sum\limits_{d}}
{\textstyle\sum\limits_{s=0}^{d}}
I_{s}^{d-s}\cdot x_{i}^{-d}=%
{\textstyle\sum\limits_{d}}
{\textstyle\sum\limits_{s=0}^{d}}
I_{s,x_{i}}^{-s}=%
{\textstyle\sum\limits_{s=0}^{m-1}}
I_{s,x_{i}}^{-s}=%
{\textstyle\bigoplus\limits_{s=0}^{m-1}}
\widetilde{I_{s}}\left(  -s\right)  \left(  U_{i}\right)
\]
for all $i$. Hence $\mathcal{I}_{q,Y}=\widetilde{I}=%
{\textstyle\bigoplus\limits_{s=0}^{m-1}}
\widetilde{I_{s}}\left(  -s\right)  $. In the general case, based on
Proposition \ref{propCI}, we have $\mathcal{I}_{q,Y}=\underleftarrow{\lim
}\left\{  \mathcal{I}_{Y,m}\right\}  $, $\mathcal{I}_{Y,m}=%
{\textstyle\bigoplus\limits_{s=0}^{m-1}}
\widetilde{I}_{m,s}\left(  -s\right)  $ for some ideals $I_{m}=\oplus
_{s}I_{m,s}\subseteq S_{q}/\mathcal{J}_{m}\left(  \mathbf{x}\right)  $, and
all canonical maps $\mathcal{I}_{Y,m+1}\rightarrow\mathcal{I}_{Y,m}$ are
surjective. Moreover, the diagram
\[%
\begin{array}
[c]{cccc}%
0\rightarrow & \mathcal{I}_{Y,m+1}\left(  d\right)  =%
{\textstyle\bigoplus\limits_{s=0}^{m}}
\widetilde{I}_{m+1,s}\left(  d-s\right)  & \longrightarrow & \left(
\mathcal{O}_{q}/\mathcal{J}_{m+1}\right)  \left(  d\right)  =%
{\textstyle\bigoplus\limits_{s=0}^{m}}
\mathcal{O}\left(  d-s\right)  \otimes R_{q}^{s}\left\langle \mathbf{y}%
\right\rangle \\
& \downarrow &  & \downarrow\\
0\rightarrow & \mathcal{I}_{Y,m}\left(  d\right)  =%
{\textstyle\bigoplus\limits_{s=0}^{m-1}}
\widetilde{I}_{m,s}\left(  d-s\right)  & \rightarrow & \left(  \mathcal{O}%
_{q}/\mathcal{J}_{m}\right)  \left(  d\right)  =%
{\textstyle\bigoplus\limits_{s=0}^{m-1}}
\mathcal{O}\left(  d-s\right)  \otimes R_{q}^{s}\left\langle \mathbf{y}%
\right\rangle
\end{array}
\]
commutes for all $d$ and $m$. Since the second column of the diagram is the
canonical projection being the identity mapping over $%
{\textstyle\bigoplus\limits_{s=0}^{m-1}}
\mathcal{O}\left(  d-s\right)  \otimes R_{q}^{s}\left\langle \mathbf{y}%
\right\rangle $ (in particular, over $%
{\textstyle\bigoplus\limits_{s=0}^{m-1}}
\widetilde{I}_{m+1,s}\left(  d-s\right)  $) for all $d$ and $m$, we obtain
that $I_{m,s}=\bigoplus\limits_{e}\Gamma\left(  \mathbb{P}_{k}^{n}%
,\widetilde{I}_{m,s}\left(  e\right)  \right)  =\bigoplus\limits_{e}%
\Gamma\left(  \mathbb{P}_{k}^{n},\widetilde{I}_{m+1,s}\left(  e\right)
\right)  =I_{m+1,s}$ for all $s$. One needs to fix $s$, $0\leq s\leq m-1$, and
put $d=e+s$ in the diagram. Put $I_{m,s}=I_{s}$ for all $m$. Since every sum
$\oplus_{s=0}^{m-1}I_{s}$ is a differential chain, it follows that so is
$\left\{  I_{s}\right\}  $ in $S_{q}$ (see Definition \ref{defDC}). It follows
that $\mathcal{I}_{q,Y}=\underleftarrow{\lim}\left\{
{\textstyle\bigoplus\limits_{s=0}^{m-1}}
\widetilde{I}_{s}\left(  -s\right)  \right\}  =\prod\limits_{s=0}^{\infty
}\widetilde{I}_{s}\left(  -s\right)  $. Using the sequence (\ref{d2}), we
identify $\mathcal{F}_{Y}$ with the sheaf $\prod\limits_{s=0}^{\infty
}\widetilde{N_{s}}\left(  -s\right)  $ of the projective NC-scheme obtained
from the differential chain $\left\{  I_{s}\right\}  $.
\end{proof}

Thus the projective $q$-schemes are geometric objects behind the differential
chains. Note also that the projective $q$-schemes considered in \cite{Dproj1}
are just commutative projective $q$-schemes over $\mathbb{C}$. The assertion
stated in Theorem \ref{theoremMain} provides the classification of all
noncommutative projective $q$-schemes over $k$.

\subsection{The union of two lines in $\mathbb{P}_{k,2}^{2}$ as a projective
$2$-scheme\label{subsecUTLP2}}

As in Subsection \ref{subsecTL}, let us consider the lines $L_{1}=Z\left(
x_{1}^{d-1}\right)  $, $d\geq2$, and $L_{2}=Z\left(  x_{0}-x_{2}\right)  $ in
$\mathbb{P}_{k}^{2}$. In this case, $S=k\left[  x_{0},x_{1},x_{2}\right]  $,
$Y=L_{1}\cup L_{2}$, and we have the differential chain $\left\{
I_{m}\right\}  $ given by $I_{0}=\left(  f\right)  \subseteq S$ with
$f=\left(  x_{0}-x_{2}\right)  x_{1}^{d-1}$, $I_{1}=\left(  x_{1}%
^{d-1}\right)  y_{1}\oplus Jy_{2}\oplus Jy_{3}$ and $I_{m}=\left(  x_{1}%
^{d-1}\right)  y_{1}^{m}\oplus\bigoplus\limits_{\left\vert \alpha\right\vert
=m,\alpha_{1}\neq m}S\mathbf{y}^{\alpha}$ for all $m\geq2$, where $J=\left(
x_{1},x_{0}-x_{2}\right)  \subseteq S$. Thus $I_{m}\subseteq S\otimes
R_{q}^{2m}\left\langle \mathbf{y}\right\rangle $ is a graded (commutative)
ideal such that $N_{1}=S^{3}/I_{1}=\left(  S/\left(  x_{1}^{d-1}\right)
\right)  y_{1}\oplus\left(  S/J\right)  y_{2}\oplus\left(  S/J\right)  y_{3}$
and $N_{m}=\left(  S/\left(  x_{1}^{d-1}\right)  \right)  y_{1}^{m}$ for all
$m\geq2$. We have the related schemes $Y_{1}=L_{1}\sqcup p\sqcup
p\subseteq\mathbb{P}_{k}^{2}\sqcup\mathbb{P}_{k}^{2}\sqcup\mathbb{P}_{k}^{2}$
and $Y_{m}=L_{1}\sqcup\varnothing\subseteq\bigsqcup\limits^{r_{m}}%
\mathbb{P}_{k}^{2}$, where $p=\left(  1:0:1\right)  =Z\left(  J\right)  $ is
the point on $Y$, and $r_{m}=\dbinom{m+2}{m}$. There are canonical morphisms
$\rho_{m}:Y_{m}\rightarrow Y$, $m\geq1$ obtained from $\sigma_{m}%
:\bigsqcup\limits^{r_{m}}\mathbb{P}_{k}^{2}\rightarrow\mathbb{P}_{k}^{2}$,
which is an identity map on ever copy of the disjoint union (see Subsection
\ref{subsecNS}). Thus $\sigma_{1}\left(  L_{1}\right)  =L_{1}\subseteq Y$,
$\sigma_{1}\left(  p\right)  =p\in Y$ and $\sigma_{m}\left(  L_{1}\right)
=L_{1}$ for all $m\geq2$. Using Lemma \ref{lemDU2}, we deduce that
$\rho_{1,\ast}\mathcal{O}_{Y_{1}}=\mathcal{O}_{L_{1}}\oplus\mathcal{O}%
_{p}\oplus\mathcal{O}_{p}$ and $\rho_{m,\ast}\mathcal{O}_{Y_{m}}%
=\mathcal{O}_{L_{1}}$ for all $m\geq2$, where $\mathcal{O}_{L_{1}}$ is
identified with $\kappa_{\ast}\mathcal{O}_{L_{1}}$ for the closed immersion
$\kappa:L_{1}\rightarrow Y$. Notice that $\left(  p,\mathcal{O}_{p}\right)
=\operatorname{Proj}S/J=\operatorname{Proj}k\left[  x_{0}\right]  $,
$\mathcal{O}_{p}\left(  d\right)  =k\left(  d\right)  =kx_{0}^{d}$,
$d\in\mathbb{Z}$, and $\mathcal{O}_{p,\ast}=k\left[  x_{0}\right]
_{0}=k\left[  x_{0}\right]  _{x_{0}}=\bigoplus\limits_{d\in\mathbb{Z}}%
kx_{0}^{d}=k^{\left(  \mathbb{Z}\right)  }$. Using (\ref{FI2}), we have
\begin{align*}
\mathcal{O}_{2,Y}  &  =\prod\limits_{m=0}^{\infty}\rho_{m,\ast}\mathcal{O}%
_{Y_{m}}\left(  -2m\right)  =\mathcal{O}_{Y}\oplus\mathcal{O}_{L_{1}}\left(
-2\right)  y_{1}\oplus k\left(  -2\right)  y_{2}\oplus k\left(  -2\right)
y_{3}\oplus\prod\limits_{m=2}^{\infty}\mathcal{O}_{L_{1}}\left(  -2m\right)
y_{1}^{m}\\
&  =\mathcal{O}_{Y}\oplus\mathcal{O}_{L_{1}}\left(  -2\right)  \oplus k\left(
-2\right)  ^{2}\oplus\prod\limits_{m=2}^{\infty}\mathcal{O}_{L_{1}}\left(
-2m\right)  .
\end{align*}
As in \cite[Theorem 7.1]{Dproj1} one can also use the formal functional
calculus for sheaves on $Y$ to describe $\mathcal{O}_{2,Y}$ using the
operations $\otimes_{\mathcal{O}_{Y}}$, $\oplus$ and $\prod$. If $t_{1}$ and
$t_{2}$ are independent (commuting) variables representing the sheaves
$\mathcal{O}_{L_{1}}\left(  -2\right)  $ and $k\left(  -2\right)  $ on $Y$,
respectively, then $f\left(  t_{1},t_{2}\right)  =\left(  1-t_{1}\right)
^{-1}+t_{2}^{2}=\sum_{m\geq0}t_{1}^{m}+t_{2}^{2}$ is a formal power series
from $\mathbb{Z}_{+}\left[  \left[  t_{1},t_{2}\right]  \right]  $. Since
$\mathcal{O}_{Y}$ represents the unit, we derive that
\[
f\left(  \mathcal{O}_{L_{1}}\left(  -2\right)  ,k\left(  -2\right)  \right)
=\left(  \mathcal{O}_{Y}-\mathcal{O}_{L_{1}}\left(  -2\right)  \right)
^{-1}+k\left(  -2\right)  ^{2}=\mathcal{O}_{Y}\oplus k\left(  -2\right)
^{2}\oplus\prod\limits_{m=1}^{\infty}\mathcal{O}_{L_{1}}\left(  -2\right)
^{\otimes m}=\mathcal{O}_{2,Y}\text{.}%
\]
Thus $f\left(  t_{1},t_{2}\right)  $ is an invariant of the projective
$2$-scheme $Y$. Further, take sections $g=\dfrac{x_{0}}{x_{2}}$ and
$h=\dfrac{x_{2}}{x_{0}}$ from $\mathcal{O}_{2,Y}\left(  Y\cap D_{+}\left(
x_{0}x_{2}\right)  \right)  $. Then
\[
g\cdot h=1+\sum_{m=1}^{\infty}\left(  -1\right)  ^{m}m!\dfrac{1}{\left(
x_{0}x_{2}\right)  ^{m}}|_{L_{1}}y_{1}^{m}.
\]
Indeed, based on \cite{Dproj1}, the multiplication in the sheaf $\mathcal{O}%
_{q}$ on $\mathbb{P}_{k,2}^{2}$ is given by means of the following formula
\[
f\left(  \mathbf{x}\right)  \cdot g\left(  \mathbf{x}\right)  =\sum_{\alpha
}\dfrac{\left(  -1\right)  ^{\left\vert \alpha\right\vert }}{\alpha!}\left(
\partial_{1}^{\alpha_{2}}\partial_{2}^{\alpha_{1}+\alpha_{3}}f\right)  \left(
\partial_{0}^{\alpha_{1}+\alpha_{2}}\partial_{1}^{\alpha_{3}}g\right)  \left(
\mathbf{x}\right)  y_{1}^{\alpha_{1}}y_{2}^{\alpha_{2}}y_{3}^{\alpha_{3}}%
\]
whenever $f,g\in\mathcal{O}$. It follows that $g\cdot h=1-\dfrac{1}{x_{0}%
x_{2}}|_{L_{1}}y_{1}+0y_{2}+0y_{3}+\sum_{m=2}^{\infty}\left(  -1\right)
^{m}m!\dfrac{1}{\left(  x_{0}x_{2}\right)  ^{m}}|_{L_{1}}y_{1}^{m}$. Further,
let us consider the following multiplication
\[
\dfrac{x_{2}}{x_{0}}\cdot\dfrac{x_{1}}{x_{0}}=\dfrac{x_{1}x_{2}}{x_{0}^{2}%
}+\dfrac{x_{1}}{x_{0}^{3}}|_{L_{1}}y_{1}-\dfrac{1}{x_{0}^{2}}|_{p}y_{3}%
=\dfrac{x_{1}x_{2}}{x_{0}^{2}}-y_{3}%
\]
on the chart $U_{0}$($=D_{+}\left(  x_{0}\right)  $). Similarly, $\dfrac
{x_{1}}{x_{0}}\cdot\dfrac{x_{2}}{x_{0}}=\dfrac{x_{1}x_{2}}{x_{0}^{2}}%
+\dfrac{x_{2}}{x_{0}^{3}}|_{p}y_{2}$, which in turn implies that $\left[
\dfrac{x_{1}}{x_{0}},\dfrac{x_{2}}{x_{0}}\right]  =y_{2}+y_{3}$ in
$\mathcal{O}_{q,Y}\left(  Y\cap U_{0}\right)  $. Notice that $Y$ is not an
NC-manifold. The construction just suggested can be treated as the
$2$-quantization of $Y$, or the geometric quantization of $Y$ within
$\mathbb{P}_{k,2}^{2}$.

\subsection{The geometric $3$-quantization of the projective scheme $Y$ in
$\mathbb{P}_{k,3}^{2}$}

Now let us demonstrate higher quantization of the union $Y$ of two lines in
$\mathbb{P}_{k}^{2}$ based on the differential chains. As above, suppose $n=2$
but $q=3$, that is, we deal with the graded algebra $S_{3}$ obtained from
$S=k\left[  x_{0},x_{1},x_{2}\right]  $ with the Hall basis $\mathbf{y}$ which
consists of degree $2$ elements $y_{1}=\left[  x_{0},x_{2}\right]  $,
$y_{2}=\left[  x_{0},x_{1}\right]  $, $y_{3}=\left[  x_{1},x_{2}\right]  $ and
degree $3$ elements
\begin{align*}
u_{1}  &  =\left[  x_{0},y_{2}\right]  ,u_{2}=\left[  x_{0},y_{1}\right]
,u_{3}=\left[  x_{2},y_{1}\right]  ,u_{4}=\left[  x_{1},y_{2}\right]
,u_{5}=\left[  x_{1},y_{1}\right]  ,\\
u_{6}  &  =\left[  x_{2},y_{2}\right]  ,u_{7}=\left[  x_{1},y_{3}\right]
,u_{8}=\left[  x_{2},y_{3}\right]
\end{align*}
in $\mathfrak{g}_{3}\left(  \mathbf{x}\right)  $, where $u_{i}=y_{3+i}$,
$1\leq i\leq8$. Thus $\mathbf{y}$ consists of $11$ elements and $\mathbf{z=x}%
$\textbf{$\sqcup$}$\mathbf{y}$ is a basis for $\mathfrak{g}_{3}\left(
\mathbf{x}\right)  $. Let us write down the list of all differential operators
occurred in this case (see Subsection \ref{subsecDOSQ}). Namely,
\begin{align*}
\Delta_{0}  &  =0,\\
\Delta_{1}  &  =-R\left(  \left[  x_{0},x_{1}\right]  \right)  \partial
_{0}+2^{-1}R\left(  \left[  x_{0},\left[  x_{0},x_{1}\right]  \right]
\right)  \partial_{0}^{2}+2^{-1}R\left(  \left[  x_{1},\left[  x_{0}%
,x_{1}\right]  \right]  \right)  \partial_{0}\partial_{1}+2^{-1}R\left(
\left[  x_{2},\left[  x_{0},x_{1}\right]  \right]  \right)  \partial
_{0}\partial_{2}\\
&  =-\partial_{0}y_{2}+2^{-1}\partial_{0}^{2}u_{1}+2^{-1}\partial_{0}%
\partial_{1}u_{4}+2^{-1}\partial_{0}\partial_{2}u_{6},\\
\Delta_{2}  &  =-R\left(  \left[  x_{0},x_{2}\right]  \right)  \partial
_{0}-R\left(  \left[  x_{1},x_{2}\right]  \right)  \partial_{1}+2^{-1}R\left(
\left[  x_{0},\left[  x_{0},x_{2}\right]  \right]  \right)  \partial_{0}%
^{2}+2^{-1}R\left(  \left[  x_{1},\left[  x_{0},x_{2}\right]  \right]
\right)  \partial_{0}\partial_{1}\\
&  +2^{-1}R\left(  \left[  x_{2},\left[  x_{0},x_{2}\right]  \right]  \right)
\partial_{0}\partial_{2}+2^{-1}R\left(  \left[  x_{1},\left[  x_{1}%
,x_{2}\right]  \right]  \right)  \partial_{1}^{2}+2^{-1}R\left(  \left[
x_{2},\left[  x_{1},x_{2}\right]  \right]  \right)  \partial_{1}\partial_{2}\\
&  =-\partial_{0}y_{1}-\partial_{1}y_{3}+2^{-1}\partial_{0}^{2}u_{2}%
+2^{-1}\partial_{0}\partial_{1}u_{5}+2^{-1}\partial_{0}\partial_{2}%
u_{3}+2^{-1}\partial_{1}^{2}u_{7}+2^{-1}\partial_{1}\partial_{2}u_{8},\\
D_{11}  &  =\operatorname{ad}\left(  x_{1}\right)  \left(  \Delta_{1}\right)
=-\partial_{0}\left[  x_{1},\left[  x_{0},x_{1}\right]  \right]
=-\partial_{0}u_{4},\\
D_{21}  &  =\operatorname{ad}\left(  x_{2}\right)  \left(  \Delta_{2}\right)
=-\partial_{0}\left[  x_{2},\left[  x_{0},x_{2}\right]  \right]  -\partial
_{1}\left[  x_{2},\left[  x_{1},x_{2}\right]  \right]  =-\partial_{0}%
u_{3}-\partial_{1}u_{8},
\end{align*}
which are responding to the left multiplication operators. Similarly, we have
\begin{align*}
\nabla_{0}  &  =R\left(  \left[  x_{1},x_{0}\right]  \right)  \partial
_{1}+R\left(  \left[  x_{2},x_{0}\right]  \right)  \partial_{2}-2^{-1}R\left(
\left[  x_{1},\left[  x_{1},x_{0}\right]  \right]  \right)  \partial_{1}^{2}\\
&  -2^{-1}R\left(  \left[  x_{2},\left[  x_{1},x_{0}\right]  \right]  \right)
\partial_{1}\partial_{2}-2^{-1}R\left(  \left[  x_{2},\left[  x_{2}%
,x_{0}\right]  \right]  \right)  \partial_{2}^{2}\\
&  =-\partial_{1}y_{2}-\partial_{2}y_{1}+2^{-1}\partial_{1}^{2}u_{4}%
+2^{-1}\partial_{1}\partial_{2}u_{6}+2^{-1}\partial_{2}^{2}u_{3},\\
\nabla_{1}  &  =R\left(  \left[  x_{2},x_{1}\right]  \right)  \partial
_{1}-2^{-1}R\left(  \left[  x_{2},\left[  x_{2},x_{1}\right]  \right]
\right)  \partial_{2}^{2}\\
&  =-\partial_{2}y_{3}+2^{-1}\partial_{2}^{2}u_{8},\\
\nabla_{2}  &  =0.
\end{align*}
Using the notations from Subsection \ref{subsecDOSQ}, we have the following
nontrivial differential operators
\begin{align*}
\Delta_{1,1}  &  =-\partial_{0}y_{2},\quad\Delta_{1,2}=2^{-1}\partial_{0}%
^{2}u_{1}+2^{-1}\partial_{0}\partial_{1}u_{4}+2^{-1}\partial_{0}\partial
_{2}u_{6},\\
\Delta_{2,1}  &  =-\partial_{0}y_{1}-\partial_{1}y_{3},\quad\Delta
_{2,2}=2^{-1}\partial_{0}^{2}u_{2}++2^{-1}\partial_{0}\partial_{2}u_{3}%
+2^{-1}\partial_{0}\partial_{1}u_{5}+2^{-1}\partial_{1}^{2}u_{7}%
+2^{-1}\partial_{1}\partial_{2}u_{8},\\
D_{11}  &  =-\partial_{0}u_{4},\quad D_{21}=-\partial_{0}u_{3}-\partial
_{1}u_{8},
\end{align*}
and
\begin{align*}
\nabla_{0,1}  &  =-\partial_{1}y_{2}-\partial_{2}y_{1},\quad\nabla
_{0,2}=2^{-1}\partial_{2}^{2}u_{3}+2^{-1}\partial_{1}^{2}u_{4}+2^{-1}%
\partial_{1}\partial_{2}u_{6},\\
\nabla_{1,1}  &  =-\partial_{2}y_{3},\quad\nabla_{1,2}=2^{-1}\partial_{2}%
^{2}u_{8}.
\end{align*}
As in Subsection \ref{subsecTL}, we define the following ideals
\begin{align*}
I_{0}  &  =\left(  f\right)  \subseteq S,\\
I_{2}  &  =\left(  x_{1}^{d-1}\right)  y_{1}\oplus Jy_{2}\oplus Jy_{3}%
\subseteq S\otimes R_{q}^{2}\left\langle \mathbf{y}\right\rangle =S^{3},\\
I_{3}  &  =\left(  x_{1}^{d-1}\right)  u_{1}\oplus\left(  x_{1}^{d-1}\right)
u_{2}\oplus\left(  x_{1}^{d-1}\right)  u_{3}\oplus\bigoplus_{k=4}^{8}%
Su_{k}\subseteq S\otimes R_{3}^{3}\left\langle \mathbf{y}\right\rangle
=S^{8},\\
I_{m}  &  =\bigoplus_{2\alpha_{1}+3\left(  \alpha_{4}+\alpha_{5}+\alpha
_{6}\right)  =m}\left(  x_{1}^{d-1}\right)  y_{1}^{\alpha_{1}}u_{1}%
^{\alpha_{4}}u_{2}^{\alpha_{5}}u_{3}^{\alpha_{6}}\oplus\bigoplus_{\left\langle
\alpha\right\rangle =m,\alpha_{2}+\alpha_{3}+\sum_{k=7}^{11}\alpha_{k}%
>0}S\mathbf{y}^{\alpha}\subseteq S\otimes R_{q}^{m}\left\langle \mathbf{y}%
\right\rangle ,m\geq4,
\end{align*}
where $f=\left(  x_{0}-x_{2}\right)  x_{1}^{d-1}\in S^{2}$, $J=\left(
x_{1},x_{0}-x_{2}\right)  \subseteq S$ is the ideal of the point $\left(
1:0:1\right)  \in\mathbb{P}_{k}^{2}$.

\begin{lemma}
\label{lem3Q1}The family $\left\{  I_{m}\right\}  $ of ideals is a
differential chain in $S_{3}$.
\end{lemma}

\begin{proof}
By construction, the family $\left\{  I_{m}\right\}  $ is a chain. Using
inclusions $\partial_{0}\left(  I_{0}\right)  +\partial_{2}\left(
I_{0}\right)  \subseteq\left(  x_{1}^{d-1}\right)  \subseteq J$ and
$\partial_{1}\left(  I_{0}\right)  \subseteq J$, we derive that
\begin{align*}
\Delta_{1,1}\left(  I_{0}\right)  +\Delta_{2,1}\left(  I_{0}\right)
+\nabla_{0,1}\left(  I_{0}\right)  +\nabla_{1,1}\left(  I_{0}\right)   &
\subseteq\left(  x_{1}^{d-1}\right)  y_{2}+\left(  x_{1}^{d-1}\right)
y_{1}+Jy_{3}+Jy_{2}+\left(  x_{1}^{d-1}\right)  y_{1}+\left(  x_{1}%
^{d-1}\right)  y_{3}\\
&  \subseteq\left(  x_{1}^{d-1}\right)  y_{1}+Jy_{2}+Jy_{3}=I_{2}%
\end{align*}
as in Lemma \ref{lemDC2}. By sorting out all operators listed above, we see
that the commutators $y_{1}$, $u_{1}$, $u_{2}$, $u_{3}$ have been involved in
the list with the following operators $\partial_{0}y_{1},\partial_{2}%
y_{1},\partial_{0}u_{3},\partial_{0}^{2}u_{1},\partial_{0}^{2}u_{2}%
,\partial_{2}^{2}u_{3},\partial_{0}\partial_{2}u_{3}$. Since $\partial_{0}$
and $\partial_{2}$ leave invariant the ideal $\left(  x_{1}^{d-1}\right)  $,
we obtain that $\Delta_{1,1}\left(  I_{m}\right)  +\Delta_{2,1}\left(
I_{m}\right)  +\nabla_{0,1}\left(  I_{m}\right)  +\nabla_{1,1}\left(
I_{m}\right)  \subseteq I_{m+2}$ and $\Delta_{1,2}\left(  I_{m}\right)
+\Delta_{2,2}\left(  I_{m}\right)  +\nabla_{0,2}\left(  I_{m}\right)
+\nabla_{1,2}\left(  I_{m}\right)  \subseteq I_{m+3}$ for all $m\geq2$. Whence
$\left\{  I_{m}\right\}  $ is a differential chain in $S_{3}$.
\end{proof}

Now we introduce the following list of the commutative graded algebras
\begin{align*}
N_{0}  &  =S/I_{0},\\
N_{2}  &  =S^{3}/I_{2}=\left(  S/\left(  x_{1}^{d-1}\right)  \right)
y_{1}\oplus\left(  S/J\right)  y_{2}\oplus\left(  S/J\right)  y_{3},\\
N_{3}  &  =S^{8}/I_{3}=\left(  S/\left(  x_{1}^{d-1}\right)  \right)
u_{1}\oplus\left(  S/\left(  x_{1}^{d-1}\right)  \right)  u_{2}\oplus\left(
S/\left(  x_{2}^{d-1}\right)  \right)  u_{3},\\
N_{m}  &  =\bigoplus_{2\alpha_{1}+3\left(  \alpha_{4}+\alpha_{5}+\alpha
_{6}\right)  =m}\left(  S/\left(  x_{1}^{d-1}\right)  \right)  y_{1}%
^{\alpha_{1}}u_{1}^{\alpha_{4}}u_{2}^{\alpha_{5}}u_{3}^{\alpha_{6}},\quad
m\geq4.
\end{align*}
These graded algebras are in turn define the schemes
\begin{align*}
Y  &  =\operatorname{Proj}E_{0}=L_{1}\cup L_{2}\subseteq\mathbb{P}_{k}^{2},\\
Y_{2}  &  =\operatorname{Proj}E_{2}=L_{1}\sqcup p\sqcup p\subseteq
\mathbb{P}_{k}^{2}\sqcup\mathbb{P}_{k}^{2}\sqcup\mathbb{P}_{k}^{2},\\
Y_{3}  &  =\operatorname{Proj}E_{3}=\left(  L_{1}\sqcup L_{1}\sqcup
L_{1}\right)  \sqcup\varnothing\subseteq\bigsqcup\limits^{8}\mathbb{P}_{k}%
^{2},\\
Y_{m}  &  =\left(  \bigsqcup\limits_{2\alpha_{1}+3\left(  \alpha_{4}%
+\alpha_{5}+\alpha_{6}\right)  =m}L_{1}\right)  \sqcup\varnothing
\subseteq\bigsqcup\limits^{r_{m}}\mathbb{P}_{k}^{2},\quad m\geq4,
\end{align*}
where $p=\left(  1:0:1\right)  =Z\left(  J\right)  $ is the point on $Y$, and
$r_{m}=\dim\left(  R_{2}^{m}\left\langle \mathbf{y}\right\rangle \right)  $.
There are canonical morphisms $\rho_{m}:Y_{m}\rightarrow Y$, $m\geq2$ obtained
from $\sigma_{m}:\bigsqcup\limits^{r_{m}}\mathbb{P}_{k}^{2}\rightarrow
\mathbb{P}_{k}^{2}$ such that $\sigma_{m}\left(  L_{1}\right)  =L_{1}\subseteq
Y$, $\sigma_{m}\left(  p\right)  =p\in Y$. Moreover, $\rho_{2,\ast}%
\mathcal{O}_{Y_{2}}=\mathcal{O}_{L_{1}}\oplus\mathcal{O}_{p}\oplus
\mathcal{O}_{p}$, $\rho_{3,\ast}\mathcal{O}_{Y_{3}}=\mathcal{O}_{L_{1}}^{3}$
and $\rho_{m,\ast}\mathcal{O}_{Y_{m}}=\bigoplus_{2\alpha_{1}+3\left(
\alpha_{4}+\alpha_{5}+\alpha_{6}\right)  =m}\mathcal{O}_{L_{1}}$ for all
$m\geq4$. Using (\ref{FI2}), we have
\begin{align*}
\mathcal{O}_{3,Y}  &  =\mathcal{O}_{Y}\oplus\mathcal{O}_{L_{1}}\left(
-2\right)  y_{1}\oplus k\left(  -2\right)  y_{2}\oplus k\left(  -2\right)
y_{3}\oplus\mathcal{O}_{L_{1}}\left(  -3\right)  u_{1}\oplus\mathcal{O}%
_{L_{1}}\left(  -3\right)  u_{2}\oplus\mathcal{O}_{L_{1}}\left(  -3\right)
u_{3}\\
&  \oplus\prod\limits_{m=4}^{\infty}\bigoplus_{2\alpha_{1}+3\left(  \alpha
_{4}+\alpha_{5}+\alpha_{6}\right)  =m}\mathcal{O}_{L_{1}}\left(  -m\right)
y_{1}^{\alpha_{1}}u_{1}^{\alpha_{4}}u_{2}^{\alpha_{5}}u_{3}^{\alpha_{6}}.
\end{align*}
Thus the projective scheme $\left(  Y,\mathcal{O}_{Y}\right)  $ is quantized
into a new projective $3$-scheme $\left(  Y,\mathcal{O}_{3,Y}\right)  $. For
the related formal power series we put $f\left(  t_{1},t_{2},t_{3}\right)
=\left(  1-t_{1}\right)  ^{-1}\dfrac{1}{2}\dfrac{d^{2}}{dt_{3}^{2}}\left(
1-t_{3}\right)  ^{-1}+2t_{2}$ to be an element of $\mathbb{Z}_{+}\left[
\left[  t_{1},t_{2},t_{3}\right]  \right]  $. Note that
\begin{align*}
f\left(  t_{1},t_{2},t_{3}\right)   &  =2t_{2}+\sum_{m_{1}}t_{1}^{m_{1}}%
\sum_{m_{3}}\dbinom{m_{3}+2}{2}t_{3}^{m_{3}}=2t_{2}+\sum_{m\geq0}\sum
_{2\alpha_{1}+3\left(  \alpha_{4}+\alpha_{5}+\alpha_{6}\right)  =m}%
t_{1}^{\alpha_{1}}t_{3}^{\alpha_{4}+\alpha_{5}+\alpha_{6}}\\
&  =1+t_{1}+2t_{2}+3t_{3}+\sum_{m\geq4}\sum_{2\alpha_{1}+3\left(  \alpha
_{4}+\alpha_{5}+\alpha_{6}\right)  =m}t_{1}^{\alpha_{1}}t_{3}^{\alpha
_{4}+\alpha_{5}+\alpha_{6}},
\end{align*}
which in turn implies that $f\left(  \mathcal{O}_{L_{1}}\left(  -2\right)
,k\left(  -2\right)  ,\mathcal{O}_{L_{1}}\left(  -3\right)  \right)
=\mathcal{O}_{3,Y}$.

\subsection{The cohomology of the projective $q$-schemes}

For calculations of the cohomology groups of the projective $q$-schemes we
need the following key result of R. Hartshorne \cite[Theorem 4.5]{Harts2}. Let
$\left(  X,\mathcal{O}_{X}\right)  $ be a scheme and let $\mathcal{F=}%
\prod\limits_{m=0}^{\infty}\mathcal{F}_{m}$ be the direct product of the
quasi-coherent $\mathcal{O}_{X}$-modules $\mathcal{F}_{m}$. Then
\begin{equation}
H^{i}\left(  X,\mathcal{F}\right)  =\prod\limits_{m=0}^{\infty}H^{i}\left(
X,\mathcal{F}_{m}\right)  ,\quad i\geq0. \label{Hi}%
\end{equation}
Namely, every $\mathcal{P}_{s}=\prod\limits_{m=0}^{s}\mathcal{F}_{m}$ is a
quasi-coherent $\mathcal{O}_{X}$-module and the inverse system $\left\{
\mathcal{P}_{s}\left(  U\right)  \right\}  $ is surjective with $H^{i}\left(
U,\mathcal{P}_{s}\right)  =\left\{  0\right\}  $, $i>0$ (\cite[3.3.5]{Harts})
for every open affine $U\subseteq X$. Moreover, $H^{i}\left(  X,\mathcal{P}%
_{s}\right)  =\prod\limits_{m=0}^{s}H^{i}\left(  X,\mathcal{F}_{m}\right)  $,
which means that $\left\{  H^{i}\left(  X,\mathcal{P}_{s}\right)  \right\}  $
satisfies (ML). Then $H^{i}\left(  X,\mathcal{F}\right)  =\underleftarrow{\lim
}\left\{  H^{i}\left(  X,\mathcal{P}_{s}\right)  \right\}  $. In the case of
$\left(  X,\mathcal{O}_{X}\right)  =\mathbb{P}_{k}^{n}$ one can also use the
standard affine covering $\mathfrak{U=}\left\{  U_{i}:0\leq i\leq n\right\}  $
and the related \v{C}ech complex.

Now let $\left(  Y,\mathcal{F}_{Y}\right)  $ be a projective $q$-scheme with
the related closed immersion $\left(  \iota,\iota^{+}\right)  :\left(
Y,\mathcal{F}_{Y}\right)  \rightarrow\mathbb{P}_{k,q}^{n}$. By Theorem
\ref{theoremMain}, $\mathcal{F}_{Y}\left(  d\right)  =\mathcal{O}_{q,Y}\left(
d\right)  =\prod\limits_{m=0}^{\infty}\widetilde{N_{m}}\left(  d-m\right)  $
for all $d\in\mathbb{Z}$. It follows that $\iota_{\ast}\mathcal{F}_{Y}\left(
d\right)  =\prod\limits_{m=0}^{\infty}\widetilde{T_{m}}\left(  d-m\right)  $
is the direct product of the quasi-coherent sheaves on $\mathbb{P}_{k}^{n}$
(see Lemma \ref{lemExactSh}). Using (\ref{Hi}), we deduce that%
\[
H^{i}\left(  Y,\mathcal{O}_{q,Y}\left(  d\right)  \right)  =H^{i}\left(
\mathbb{P}_{k,q}^{n},\iota_{\ast}\mathcal{O}_{q,Y}\left(  d\right)  \right)
=\prod\limits_{m=0}^{\infty}H^{i}\left(  \mathbb{P}_{k}^{n},\widetilde{T_{m}%
}\left(  d-m\right)  \right)  \mathbf{=}\prod\limits_{m=0}^{\infty}%
H^{i}\left(  Y,\widetilde{N_{m}}\left(  d-m\right)  \right)
\]
(see \cite[3.2.10]{Harts}) for all $i$.

Let us calculate the cohomology groups $H^{i}\left(  Y,\mathcal{O}%
_{2,Y}\right)  $, $i\geq0$ of the projective $2$-scheme $Y=L_{1}\cup L_{2}$
(see Subsection \ref{subsecUTLP2}) being the union of two lines $L_{1}%
=Z\left(  x_{1}^{d-1}\right)  $, $d\geq2$, and $L_{2}=Z\left(  x_{0}%
-x_{2}\right)  $ in $X=\mathbb{P}_{k}^{2}$. Thus $\mathcal{O}_{2,Y}%
=\mathcal{O}_{Y}\oplus\mathcal{O}_{L_{1}}\left(  -2\right)  \oplus k\left(
-2\right)  ^{2}\oplus\prod\limits_{m=2}^{\infty}\mathcal{O}_{L_{1}}\left(
-2m\right)  $ and
\[
H^{i}\left(  Y,\mathcal{O}_{2,Y}\right)  =H^{i}\left(  Y,\mathcal{O}%
_{Y}\right)  \oplus H^{i}\left(  L_{1},\mathcal{O}_{L_{1}}\left(  -2\right)
\right)  \oplus H^{i}\left(  p,k\left(  -2\right)  \right)  ^{2}\oplus
\prod\limits_{m=2}^{\infty}H^{i}\left(  L_{1},\mathcal{O}_{L_{1}}\left(
-2m\right)  \right)
\]
by virtue of (\ref{Hi}). In particular, $H^{0}\left(  Y,\mathcal{O}%
_{2,Y}\right)  =k$ and $H^{i}\left(  Y,\mathcal{O}_{2,Y}\right)  =0$ for all
$i\geq2$. To calculate $H^{1}\left(  Y,\mathcal{O}_{2,Y}\right)  $ we need
$H^{1}\left(  Y,\mathcal{O}_{Y}\right)  $ and $H^{1}\left(  L_{1}%
,\mathcal{O}_{L_{1}}\left(  m\right)  \right)  $, $m\in\mathbb{Z}$. Recall
\cite[3.5]{Harts} that $H^{0}\left(  X,\mathcal{O}\left(  m\right)  \right)
=S^{m}$, $H^{1}\left(  X,\mathcal{O}\left(  m\right)  \right)  =0$ for all
$m\in\mathbb{Z}$, and
\[
H^{2}\left(  X,\mathcal{O}\left(  m\right)  \right)  =\left\{
\begin{array}
[c]{ccc}%
\bigoplus\limits_{\left\vert \alpha\right\vert =m,\alpha<0}k\mathbf{x}%
^{\alpha} & \text{if} & m\leq-3\\
0 & \text{if} & m>-3
\end{array}
\right.
\]
where $\mathbf{x}=\left(  x_{0},x_{1},x_{2}\right)  $ and $\alpha=\left(
\alpha_{0},\alpha_{1},\alpha_{2}\right)  \in\mathbb{Z}^{3}$, $\alpha_{i}<0$,
$i=0,1,2$. Thus $H^{2}\left(  X,\mathcal{O}\left(  -m-3\right)  \right)
=H^{0}\left(  X,\mathcal{O}\left(  m\right)  \right)  ^{\ast}$ for all
$m\geq0$.

\begin{lemma}
\label{lemdimCG1}Let $Z$ be a closed subscheme of $\mathbb{P}_{k}^{2}$ defined
either by the equation $\left(  x_{0}-x_{2}\right)  x_{1}^{d-1}=0$ or
$x_{1}^{d}=0$. Then
\[
\dim_{k}H^{1}\left(  Z,\mathcal{O}_{Z}\left(  m\right)  \right)  =\left\{
\begin{array}
[c]{ccc}%
0 & \text{if} & m>d-3;\\
\frac{1}{2}\left(  -m+d-1\right)  \left(  -m+d-2\right)  & \text{if} &
-3<m\leq d-3;\\
d\left(  d-2m-3\right)  /2 & \text{if} & m\leq-3.
\end{array}
\right.
\]

\end{lemma}

\begin{proof}
First assume that $Z$ is a closed subscheme of $X=\mathbb{P}_{k}^{2}$ defined
by a single homogeneous equation $f\left(  x_{0},x_{1},x_{2}\right)  =0$ of
degree $d\geq2$. We do not assume that $f$ is irreducible. That is the case
for $Z=Y$. The polynomial $f$ defines the multiplication operator $f:S\left(
-d\right)  \rightarrow S$ of the graded $S$-modules. The exact sequence
$0\rightarrow S\left(  -d\right)  \overset{f}{\longrightarrow}S\longrightarrow
S/\left(  f\right)  \rightarrow0$ generates the exact sequence of the coherent
$\mathcal{O}$-modules $0\rightarrow\mathcal{O}\left(  -d\right)
\overset{f}{\longrightarrow}\mathcal{O}\longrightarrow j_{\ast}\mathcal{O}%
_{Z}\rightarrow0$, where $j:Z\rightarrow\mathbb{P}_{k}^{2}$ is the closed
immersion. Put $\mathcal{F=\oplus}_{m\in\mathbb{Z}}\mathcal{O}\left(
m\right)  $ and $\mathcal{F}_{Z}=\mathcal{\oplus}_{m\in\mathbb{Z}}%
\mathcal{O}_{Z}\left(  m\right)  $ to be quasi-coherent sheaves. Then
$\mathcal{F}\left(  -d\right)  \mathcal{=\oplus}_{m\in\mathbb{Z}}%
\mathcal{O}\left(  m-d\right)  $, $j_{\ast}\mathcal{F}_{Z}=\mathcal{\oplus
}_{m\in\mathbb{Z}}j_{\ast}\mathcal{O}_{Z}\left(  m\right)  $ and
\begin{equation}
0\rightarrow\mathcal{F}\left(  -d\right)  \overset{f}{\longrightarrow
}\mathcal{F}\longrightarrow j_{\ast}\mathcal{F}_{Z}\rightarrow0 \label{SFS1}%
\end{equation}
is an exact sequence of $\mathcal{O}$-modules on $\mathbb{P}_{k}^{2}$.
Moreover, $H^{i}\left(  Z,\mathcal{F}_{Z}\right)  =H^{i}\left(  X,j_{\ast
}\mathcal{F}_{Z}\right)  $, $i\geq0$ and $H^{i}\left(  Z,\mathcal{F}%
_{Z}\right)  =0$, $i\geq2$ by Vanishing Theorem of Grothendieck. As we have
mentioned above
\begin{align*}
H^{1}\left(  X,\mathcal{F}\right)   &  =H^{1}\left(  X,\mathcal{F}\left(
-d\right)  \right)  =0,\quad H^{2}\left(  X,\mathcal{F}\right)  =\bigoplus
\limits_{m\leq-3}\bigoplus\limits_{\left\vert \alpha\right\vert =m,\alpha
<0}k\mathbf{x}^{\alpha},\\
H^{2}\left(  X,\mathcal{F}\left(  -d\right)  \right)   &  =\bigoplus
\limits_{m\leq d-3}\bigoplus\limits_{\left\vert \alpha\right\vert
=m-d,\alpha<0}k\mathbf{x}^{\alpha}\text{,\quad}d\in\mathbb{Z}\text{.}%
\end{align*}
Consider the long cohomology sequence%
\begin{align*}
0  &  \rightarrow S\left(  -d\right)  \overset{f}{\longrightarrow
}S\longrightarrow S/\left(  f\right)  \rightarrow H^{1}\left(  X,\mathcal{F}%
\left(  -d\right)  \right)  \rightarrow H^{1}\left(  X,\mathcal{F}\right)
\rightarrow H^{1}\left(  Z,\mathcal{F}_{Z}\right) \\
&  \rightarrow H^{2}\left(  X,\mathcal{F}\left(  -d\right)  \right)
\rightarrow H^{2}\left(  X,\mathcal{F}\right)  \rightarrow H^{2}\left(
Z,\mathcal{F}_{Z}\right)  \rightarrow0
\end{align*}
associated to (\ref{SFS1}), which preserves the grading. It follows that for
every $m\in\mathbb{Z}$ the following sequence
\[
0\rightarrow H^{1}\left(  Z,\mathcal{O}_{Z}\left(  m\right)  \right)
\longrightarrow H^{2}\left(  X,\mathcal{O}\left(  m-d\right)  \right)
\overset{f}{\longrightarrow}H^{2}\left(  X,\mathcal{O}\left(  m\right)
\right)  \rightarrow0
\]
obtained from the previous one remains exact, where $f$ is acting as the
multiplication operator. In particular, $H^{1}\left(  Z,\mathcal{O}_{Z}\left(
m\right)  \right)  =H^{2}\left(  X,\mathcal{O}\left(  m-d\right)  \right)  $
for all $m>-3$, that is,%
\[
H^{1}\left(  Z,\mathcal{O}_{Z}\left(  m\right)  \right)  =\bigoplus
\limits_{\left\vert \alpha\right\vert =m-d,\alpha<0}k\mathbf{x}^{\alpha}\text{
for }-3<m\leq d-3,
\]
and $H^{1}\left(  Z,\mathcal{O}_{Z}\left(  m\right)  \right)  =0$ for $m>d-3$.
In these cases the formula for $\dim_{k}H^{1}\left(  Z,\mathcal{O}_{Z}\left(
m\right)  \right)  $ holds. In particular, $\dim_{k}H^{1}\left(
Y,\mathcal{O}_{Y}\right)  =\frac{1}{2}\left(  d-1\right)  \left(  d-2\right)
$.

For $m\leq-3$ we have the following exact sequence
\[
0\rightarrow H^{1}\left(  Z,\mathcal{O}_{Z}\left(  m\right)  \right)
\longrightarrow\oplus_{\left\vert \alpha\right\vert =m-d,\alpha<0}%
k\mathbf{x}^{\alpha}\overset{f}{\longrightarrow}\oplus_{\left\vert
\alpha\right\vert =m,\alpha<0}k\mathbf{x}^{\alpha}\rightarrow0,
\]
where $f$ is acting as the multiplication operators over the monomials
$\mathbf{x}^{\alpha}$. If $f=x_{1}^{d}$ then for every $a=\sum_{\left\vert
\alpha\right\vert =m-d,\alpha<0}a_{\alpha}\mathbf{x}^{\alpha}$ we have
$f\left(  a\right)  =\sum_{\left\vert \alpha\right\vert =m-d,\alpha
<0}a_{\alpha}\mathbf{x}^{\alpha}x_{1}^{d}$. Thus $a\in\ker\left(  f\right)  $
iff $a_{\alpha}=0$ for all $\alpha$, $\alpha_{1}+d<0$. Therefore $H^{1}\left(
Z,\mathcal{O}_{Z}\left(  m\right)  \right)  =\bigoplus\limits_{\left\vert
\alpha\right\vert =m-d,\alpha<0,\alpha_{1}\geq-d}k\mathbf{x}^{\alpha}$ and
\begin{align*}
\dim_{k}H^{1}\left(  Z,\mathcal{O}_{Z}\left(  m\right)  \right)   &  =\left(
-m-1\right)  +\left(  -m\right)  +\left(  -m+1\right)  +\cdots+\left(
-m+d-2\right) \\
&  =-md-1+\left(  d-1\right)  \left(  d-2\right)  /2=d\left(  d-2m-3\right)
/2.
\end{align*}
Now assume that $f=\left(  x_{0}-x_{2}\right)  x_{1}^{d-1}$, $d\geq2$. As
above we have
\begin{align*}
f\left(  a\right)   &  =\sum_{\left\vert \alpha\right\vert =m-d,\alpha
<0}a_{\alpha}x_{0}^{\alpha_{0}+1}x_{1}^{\alpha_{1}+d-1}x_{2}^{\alpha_{2}%
}-a_{\alpha}x_{0}^{\alpha_{0}}x_{1}^{\alpha_{1}+d-1}x_{2}^{\alpha_{2}+1}\\
&  =\sum_{\left\vert \beta\right\vert =m,\beta<0,\beta_{1}-d+1<0}\left(
a_{\beta_{0}-1,\beta_{1}-d+1,\beta_{2}}-a_{\beta_{0},\beta_{1}-d+1,\beta
_{2}-1}\right)  \mathbf{x}^{\beta}.
\end{align*}
In particular, $a\in\ker\left(  f\right)  $ iff $a_{\alpha}=a_{\alpha
_{0}+1,\alpha_{1},\alpha_{2}-1}$ whenever $\alpha_{1}<-d+1$, $\alpha_{0}<-1$.
The latter means that
\begin{align*}
a_{m+1,-d,-1}  &  =a_{m+2,-d,-2}=\cdots=a_{-1,-d,m+1}:=a_{-d},\\
a_{m+2,-d-1,-1}  &  =a_{m+3,-d-1,-2}=\cdots=a_{-1,-d-1,m+2}:=a_{-d-1},\\
&  \vdots\\
a_{-2,m-d+3,-1}  &  =a_{-1,m-d+3,-2}:=a_{m-d+3},\\
a_{-1,m-d+2,-1}  &  :=a_{m-d+2}.
\end{align*}
Thus
\[
a=\sum\limits_{\left\vert \alpha\right\vert =m-d,\alpha<0,\alpha_{1}\geq
-d+1}a_{\alpha}\mathbf{x}^{\alpha}+\sum\limits_{j=2}^{-m}a_{m-d+j}e_{j},
\]
where $e_{j}=\sum\limits_{\alpha_{0}+\alpha_{2}=-j}x_{0}^{\alpha_{0}}%
x_{1}^{m-d+j}x_{2}^{\alpha_{2}}\in H^{2}\left(  X,\mathcal{O}\left(
m-d\right)  \right)  $, $2\leq j\leq-m$. Obviously $\mathbf{x}^{\alpha}\in
\ker\left(  f\right)  $ whenever $\left\vert \alpha\right\vert =m-d$,
$\alpha<0$, $\alpha_{1}\geq-d+1$. Moreover,
\begin{align*}
f\left(  e_{j}\right)   &  =f\left(  x_{0}^{-j+1}x_{1}^{m-d+j}x_{2}^{-1}%
+x_{0}^{-j+2}x_{1}^{m-d+j}x_{2}^{-2}+\cdots+x_{0}^{-2}x_{1}^{m-d+j}%
x_{2}^{-j+2}+x_{0}^{-1}x_{1}^{m-d+j}x_{2}^{-j+1}\right) \\
&  =x_{0}^{-j+2}x_{1}^{m+j-1}x_{2}^{-1}+x_{0}^{-j+3}x_{1}^{m+j-1}x_{2}%
^{-2}+\cdots+x_{0}^{-2+1}x_{1}^{m+j-1}x_{2}^{-j+2}+0-\\
&  0-x_{0}^{-j+2}x_{1}^{m+j-1}x_{2}^{-2+1}-x_{0}^{-j+3}x_{1}^{m+j-1}%
x_{2}^{-3+1}-\cdots-x_{0}^{-2+1}x_{1}^{m+j-1}x_{2}^{-j+2}\\
&  =0,
\end{align*}
that is, $\left\{  e_{j}\right\}  \subseteq\ker\left(  f\right)  $.
Consequently,
\[
H^{1}\left(  Z,\mathcal{O}_{Z}\left(  m\right)  \right)  =\ker\left(
f\right)  =\bigoplus\limits_{\left\vert \alpha\right\vert =m-d,\alpha
<0,\alpha_{1}\geq-d+1}k\mathbf{x}^{\alpha}\oplus\bigoplus\limits_{j=2}%
^{-m}ke_{j}.
\]
In particular,%
\begin{align*}
\dim_{k}H^{1}\left(  Z,\mathcal{O}_{Z}\left(  m\right)  \right)   &  =\left(
-m\right)  +\left(  -m+1\right)  +\cdots+\left(  -m+d-2\right)  +\left(
-m-1\right) \\
&  =-md-1+\left(  d-1\right)  \left(  d-2\right)  /2=d\left(  d-2m-3\right)
/2.
\end{align*}
Whence $\dim_{k}H^{1}\left(  Z,\mathcal{O}_{Z}\left(  m\right)  \right)
=d\left(  d-2m-3\right)  /2$, $m\leq-3$ in the both cases of $f$.
\end{proof}

\begin{proposition}
Let $\left(  Y,\mathcal{O}_{2,Y}\right)  $ be the projective $2$-scheme in
$\mathbb{P}_{k,2}^{2}$ given by $\left(  x_{0}-x_{2}\right)  x_{1}^{d-1}=0$,
$d\geq2$. Then $H^{0}\left(  Y,\mathcal{O}_{2,Y}\right)  =k$, $H^{i}\left(
Y,\mathcal{O}_{2,Y}\right)  =0$, $i\geq2$, and
\[
H^{1}\left(  Y,\mathcal{O}_{2,Y}\right)  =\prod\limits_{m=0}^{\infty}k^{r_{m}%
}\quad\text{with\quad}r_{0}=\left(  d-1\right)  \left(  d-2\right)
/2\text{,\quad}r_{m}=\left(  d-1\right)  \left(  d+4m-4\right)  /2,m\geq1.
\]
In particular, $H^{1}\left(  Y,\mathcal{O}_{2,Y}\right)  =\prod\limits_{m=1}%
^{\infty}k^{4m-2}$ for $d=2$.
\end{proposition}

\begin{proof}
Based on Lemma \ref{lemdimCG1}, we derive that $H^{1}\left(  Y,\mathcal{O}%
_{2,Y}\right)  =H^{1}\left(  Y,\mathcal{O}_{Y}\right)  \oplus\prod
\limits_{m=1}^{\infty}H^{1}\left(  L_{1},\mathcal{O}_{L_{1}}\left(
-2m\right)  \right)  $ and $r_{0}=\dim H^{1}\left(  Y,\mathcal{O}_{Y}\right)
=\left(  d-1\right)  \left(  d-2\right)  /2$. But $L_{1}=Z\left(  x_{1}%
^{d-1}\right)  $ and Lemma \ref{lemdimCG1} is applicable for this case too.
For the group $H^{1}\left(  L_{1},\mathcal{O}_{L_{1}}\left(  -2\right)
\right)  $ we have $-3<-2\leq d-1-3$ and%
\[
r_{1}=\dim_{k}H^{1}\left(  L_{1},\mathcal{O}_{L_{1}}\left(  -2\right)
\right)  =\frac{1}{2}\left(  2+d-1-1\right)  \left(  2+d-1-2\right)  =\frac
{1}{2}d\left(  d-1\right)  .
\]
If $m\geq2$ then $-2m\leq-3$ and%
\[
r_{m}=\dim_{k}H^{1}\left(  L_{1},\mathcal{O}_{L_{1}}\left(  -2m\right)
\right)  =\frac{1}{2}\left(  d-1\right)  \left(  d-1+4m-3\right)  =\frac{1}%
{2}\left(  d-1\right)  \left(  d+4m-4\right)  .
\]
Notice that for $m=1$ the latter formula for $\dim_{k}H^{1}\left(
L_{1},\mathcal{O}_{L_{1}}\left(  -2m\right)  \right)  $ equals to $\frac{1}%
{2}\left(  d-1\right)  d$, which is $r_{1}$ obtained above. Hence
$H^{1}\left(  Y,\mathcal{O}_{2,Y}\right)  =\prod\limits_{m=0}^{\infty}%
k^{r_{m}}$, and in the case of $d=2$ we obtain that $r_{0}=0$ and $r_{m}%
=4m-2$, $m\geq1$.
\end{proof}

\begin{remark}
A similar formula can be obtained for the geometric $3$-quantization of $Y$
based on Lemma \ref{lem3Q1}.
\end{remark}

\subsection{The geometric $2$-quantization of the hypersurface in
$\mathbb{P}_{k,2}^{3}$}

Let us consider the hypersurface $Y$ in $\mathbb{P}_{k}^{3}$ given by the
graded ideal $\left(  x_{0}x_{1}-x_{2}x_{3}\right)  $ in $S=k\left[
\mathbf{x}\right]  =k\left[  x_{0},x_{1},x_{2},x_{3}\right]  $. Thus
$Y=Z\left(  f\right)  $ is the saddle being the set of zeros of the
homogeneous polynomial $f=x_{0}x_{1}-x_{2}x_{3}\in S$. Let us describe its
$2$-quantization in $\mathbb{P}_{k,2}^{3}$. In this case, $\mathbf{y}=\left(
y_{ij}\right)  $ with $y_{ij}=\left[  x_{i},x_{j}\right]  $, $i<j$. Put
\begin{align*}
I_{0}  &  =\left(  x_{0}x_{1}-x_{2}x_{3}\right)  ,\\
I_{1}  &  =\left(  x_{1},x_{3}\right)  y_{02}\oplus\left(  x_{1},x_{2}\right)
y_{03}\oplus\left(  x_{0},x_{3}\right)  y_{12}\oplus\left(  x_{0}%
,x_{2}\right)  y_{13}\oplus\left(  x_{0},x_{1},x_{2}x_{3}\right)  y_{01}%
\oplus\left(  x_{0}x_{1},x_{2},x_{3}\right)  y_{23},\\
I_{m}  &  =\left(  x_{1},x_{3}\right)  y_{02}^{m}\oplus\left(  x_{1}%
,x_{2}\right)  y_{03}^{m}\oplus\left(  x_{0},x_{3}\right)  y_{12}^{m}%
\oplus\left(  x_{0},x_{2}\right)  y_{13}^{m}\oplus\bigoplus\left\{
S\mathbf{y}^{\alpha}:\mathbf{y}^{\alpha}\neq y_{02}^{m},y_{03}^{m},y_{12}%
^{m},y_{13}^{m},\left\vert \alpha\right\vert =m\right\}
\end{align*}
for all $m\geq2$. Note that $I_{m}$ is an ideal of the algebra $S\otimes
R_{2}^{2m}\left\langle \mathbf{y}\right\rangle $ for every $m$. Based on
Remark \ref{remSpec1}, we have the following list of differential operators%
\begin{align*}
\Delta_{0}  &  =0,\quad\Delta_{1}=-\partial_{0}y_{01},\quad\Delta
_{2}=-\partial_{0}y_{02}-\partial_{1}y_{12},\quad\Delta_{3}=-\partial
_{0}y_{03}-\partial_{1}y_{13}-\partial_{2}y_{23},\\
\nabla_{0}  &  =-\partial_{1}y_{01}-\partial_{2}y_{02}-\partial_{3}%
y_{03},\quad\nabla_{1}=-\partial_{2}y_{12}-\partial_{3}y_{13},\quad\nabla
_{2}=-\partial_{3}y_{23},\quad\nabla_{3}=0
\end{align*}
on the algebra $S_{2}$. Since
\begin{align*}
\left(  \partial_{0}\left(  I_{0}\right)  +\partial_{1}\left(  I_{0}\right)
\right)  y_{01}  &  \subseteq\left(  x_{0},x_{1},x_{2}x_{3}\right)
y_{01},\left(  \partial_{0}\left(  I_{0}\right)  +\partial_{2}\left(
I_{0}\right)  \right)  y_{02}\subseteq\left(  x_{1},x_{3}\right)  y_{02},\\
\left(  \partial_{0}\left(  I_{0}\right)  +\partial_{3}\left(  I_{0}\right)
\right)  y_{03}  &  \subseteq\left(  x_{1},x_{2}\right)  y_{03},\left(
\partial_{1}\left(  I_{0}\right)  +\partial_{2}\left(  I_{0}\right)  \right)
y_{12}\subseteq\left(  x_{0},x_{3}\right)  y_{12},\\
\left(  \partial_{1}\left(  I_{0}\right)  +\partial_{3}\left(  I_{0}\right)
\right)  y_{13}  &  \subseteq\left(  x_{0},x_{2}\right)  y_{13},\left(
\partial_{2}\left(  I_{0}\right)  +\partial_{3}\left(  I_{0}\right)  \right)
y_{23}\subseteq\left(  x_{0}x_{1},x_{2},x_{3}\right)  y_{23},
\end{align*}
it follows that $\Delta_{i}\left(  I_{0}\right)  +\nabla_{j}\left(
I_{0}\right)  \subseteq I_{1}$ for all $i,j$. Further, the operators
$\partial_{i}$, $\partial_{j}$, $i<j$ leave invariant the ideal $\left(
x_{s},x_{t}\right)  $, $s<t$ whenever $\left\{  i,j\right\}  \cap\left\{
s,t\right\}  =\varnothing$. Therefore $\Delta_{i}\left(  I_{m}\right)
+\nabla_{j}\left(  I_{m}\right)  \subseteq I_{m+1}$ for all $m\geq1$. Thus
$\left\{  I_{m}\right\}  $ is a differential chain in $S_{2}$.

Now consider the following lines $L_{02}=Z\left(  x_{1},x_{3}\right)  $,
$L_{03}=Z\left(  x_{1},x_{2}\right)  $, $L_{12}=Z\left(  x_{0},x_{3}\right)
$, $L_{13}=Z\left(  x_{0},x_{2}\right)  $, and the couples $P_{01}=Z\left(
x_{0},x_{1},x_{2}x_{3}\right)  $, $P_{23}=Z\left(  x_{0}x_{1},x_{2}%
,x_{3}\right)  $ of points on the surface $Y$. Put $Y_{1}=L_{02}\sqcup
L_{03}\sqcup L_{12}\sqcup L_{13}\sqcup P_{01}\sqcup P_{23}\subseteq
\bigsqcup\limits^{6}\mathbb{P}_{k}^{3}$ with the canonical morphism $\rho
_{1}:Y_{1}\rightarrow Y$. For $m\geq2$ we have the schemes $Y_{m}=L_{02}\sqcup
L_{03}\sqcup L_{12}\sqcup L_{13}\sqcup\varnothing\subseteq\bigsqcup
\limits^{r_{m}}\mathbb{P}_{k}^{3}$ and the morphisms $\rho_{m}:Y_{m}%
\rightarrow Y$, where $r_{m}=\dbinom{m+3}{m}$. Based on the arguments from
Subsection \ref{subsecNS}, we have the projective $2$-scheme $\left(
Y,\mathcal{O}_{q,Y}\right)  $ in $\mathbb{P}_{k,2}^{3}$ with its structure
sheaf
\begin{align*}
\mathcal{O}_{2,Y}  &  =\mathcal{O}_{Y}\oplus\mathcal{O}_{L_{02}}\left(
-2\right)  \oplus\mathcal{O}_{L_{03}}\left(  -2\right)  \oplus\mathcal{O}%
_{L_{12}}\left(  -2\right)  \oplus\mathcal{O}_{L_{13}}\left(  -2\right)
\oplus\mathcal{O}_{P_{01}}\left(  -2\right)  \oplus\mathcal{O}_{P_{23}}\left(
-2\right) \\
&  \oplus\prod\limits_{m\geq2}\mathcal{O}_{L_{02}}\left(  -2m\right)
\oplus\mathcal{O}_{L_{03}}\left(  -2m\right)  \oplus\mathcal{O}_{L_{12}%
}\left(  -2m\right)  \oplus\mathcal{O}_{L_{13}}\left(  -2m\right)  ,
\end{align*}
where $\mathcal{O}_{P_{01}}=k^{2}$ and $\mathcal{O}_{P_{23}}=k^{2}$. For the
sections $f=f\left(  \mathbf{x}\right)  $ and $g=g\left(  \mathbf{x}\right)  $
on $Y$ we have the following multiplication formula
\begin{align*}
f\cdot g  &  =\left(  fg\right)  \left(  \mathbf{x}\right)  -\left(
\partial_{1}f\partial_{0}g\right)  \left(  \mathbf{x}\right)  y_{01}-\left(
\partial_{3}f\partial_{2}g\right)  \left(  \mathbf{x}\right)  y_{23}\\
&  +\sum_{m=1}^{\infty}\dfrac{\left(  -1\right)  ^{m}}{m!}\left(  \left(
\partial_{2}^{m}f\partial_{0}^{m}g\right)  \left(  \mathbf{x}\right)
y_{02}^{m}+\left(  \partial_{3}^{m}f\partial_{0}^{m}g\right)  \left(
\mathbf{x}\right)  y_{03}^{m}+\left(  \partial_{2}^{m}f\partial_{1}%
^{m}g\right)  \left(  \mathbf{x}\right)  y_{12}^{m}+\left(  \partial_{3}%
^{m}f\partial_{1}^{m}g\right)  \left(  \mathbf{x}\right)  y_{13}^{m}\right)  .
\end{align*}
In particular, take sections $g=\dfrac{x_{0}}{x_{2}}$ and $f=\dfrac{x_{1}%
}{x_{2}}$ from $\mathcal{O}_{q,Y}\left(  Y\cap U_{2}\right)  $. Then%
\begin{align*}
f\cdot g  &  =\dfrac{x_{0}x_{1}}{x_{2}^{2}}-\left(  \partial_{1}f\partial
_{0}g\right)  |_{P_{01}\cap U_{2}}y_{01}-\left(  \partial_{2}f\partial
_{0}g\right)  |_{L_{02}\cap U_{2}}y_{02}\\
&  =g\cdot f-\dfrac{1}{x_{2}^{2}}|_{P_{01}\cap U_{2}}y_{01}+\left(
\dfrac{x_{1}}{x_{2}^{3}}\right)  |_{L_{02}\cap U_{2}}y_{02}=g\cdot f-y_{01}.
\end{align*}
Thus $\left[  g,f\right]  =y_{01}$. Finally, $f=-4+\sum_{i=1}^{4}\left(
1-t_{i}\right)  ^{-1}+t_{5}+t_{6}$ is the related formal power series. Namely,
$f=\sum_{m\geq1}t_{1}^{m}+t_{2}^{m}+t_{3}^{m}+t_{4}^{m}+t_{5}+t_{6}$ and%
\[
f\left(  \mathcal{O}_{L_{02}}\left(  -2\right)  ,\mathcal{O}_{L_{03}}\left(
-2\right)  ,\mathcal{O}_{L_{12}}\left(  -2\right)  ,\mathcal{O}_{L_{13}%
}\left(  -2\right)  ,\mathcal{O}_{P_{01}}\left(  -2\right)  ,\mathcal{O}%
_{P_{23}}\left(  -2\right)  \right)  =\mathcal{O}_{2,Y}%
\]
based on the formal functional calculus for the sheaves on $Y$.

\subsection{The geometric quantizations in $\mathbb{P}_{k,2}^{n}$}

As we have seen above some projective schemes admit the geometric
quantizations (see (\ref{FI2})), whose structure sheaves are infinite direct
product of the coherent sheaves. In this case, we say that a projective scheme
$Y$ in $\mathbb{P}_{k}^{n}$ admits an infinite geometric quantization in
$\mathbb{P}_{k,q}^{n}$. That is not the case for all projective schemes. The
tail of the structure sheaf may vanish being a coherent sheaf. For example,
let us consider the geometric $2$-quantization of the parabola $Y=Z\left(
x_{0}x_{1}-x_{2}^{2}\right)  $ in $\mathbb{P}_{k,2}^{2}$. In this case,
\[
\Delta_{0}=0,\quad\Delta_{1}=-\partial_{0}y_{01},\quad\Delta_{2}=-\partial
_{0}y_{02}-\partial_{1}y_{12},\quad\nabla_{0}=-\partial_{1}y_{01}-\partial
_{2}y_{02},\quad\nabla_{1}=-\partial_{2}y_{12},\quad\nabla_{2}=0,
\]
and a chain $\left\{  I_{m}\right\}  $ of graded ideals $I_{m}=\bigoplus
{}_{\left\vert \alpha\right\vert =m}I_{m,\alpha}\mathbf{y}^{\alpha}$ is a
differential chain iff all (building blocks) operators $\partial_{0}y_{01}$,
$\partial_{1}y_{01}$, $\partial_{0}y_{02}$, $\partial_{2}y_{02}$,
$\partial_{1}y_{12}$, $\partial_{2}y_{12}$ leave invariant the chain. For the
parabola we obtain that $I_{0}=\left(  x_{0}x_{1}-x_{2}^{2}\right)  $,
$I_{1}\supseteq\left(  x_{0},x_{1},x_{2}^{2}\right)  y_{01}\oplus\left(
x_{1},x_{2}\right)  y_{02}\oplus\left(  x_{0},x_{2}\right)  y_{12}$, but
$I_{m}$ can not be chosen to be a proper ideal for $m\geq2$. By fixing
$I_{1}=Sy_{01}\oplus\left(  x_{1},x_{2}\right)  y_{02}\oplus\left(
x_{0},x_{2}\right)  y_{12}$ and $I_{m}=S\otimes R_{2}^{2m}\left\langle
\mathbf{y}\right\rangle $, $m\geq2$, we obtain that $Y_{1}=p_{02}\sqcup
p_{12}\sqcup\varnothing\subseteq\bigsqcup\limits^{3}\mathbb{P}_{k}^{2}$ with
$p_{02}=\left(  1:0:0\right)  $, $p_{12}=\left(  0:1:0\right)  $, and
\[
\mathcal{O}_{2,Y}=\mathcal{O}_{Y}\oplus k\left(  -2\right)  y_{02}\oplus
k\left(  -2\right)  y_{12}=f\left(  k\left(  -2\right)  \right)
\]
is the coherent structure sheaf with its formal power series $f\left(
t\right)  =1-2t$. The following assertion clarifies this phenomenon.

\begin{proposition}
\label{propIQ1}Let $Y$ be a projective scheme in $\mathbb{P}_{k}^{n}$ with its
graded ideal $I_{0}\subseteq S$. Then $Y$ admits an infinite geometric
quantization in $\mathbb{P}_{k,2}^{n}$ if and only if there exist a couple
$\left(  i,j\right)  $ of indices $0\leq i<j\leq n$ and a proper graded ideal
$J\subseteq S$ such that $\partial_{i}$ and $\partial_{j}$ leave invariant $J$
and $I_{0}+\partial_{i}\left(  I_{0}\right)  +\partial_{j}\left(
I_{0}\right)  \subseteq J$.
\end{proposition}

\begin{proof}
First assume that for some $0\leq i<j\leq n$ and a proper graded ideal
$J\subseteq S$ , the operators $\partial_{i}$ and $\partial_{j}$ leave
invariant $J$ and $I_{0}+\partial_{i}\left(  I_{0}\right)  +\partial
_{j}\left(  I_{0}\right)  \subseteq J$. We define the ideals $I_{m}$ by
induction on $m$. Namely, put $I_{1}=\bigoplus\limits_{s<t}I_{1,st}y_{st}$,
where $I_{1,ij}=J$ and $I_{1,st}=\left(  I_{0}+\partial_{s}\left(
I_{0}\right)  +\partial_{t}\left(  I_{0}\right)  \right)  $ is the graded
ideal generated by $I_{0}+\partial_{s}\left(  I_{0}\right)  +\partial
_{t}\left(  I_{0}\right)  $. Suppose that the ideals $I_{t}=\bigoplus
\limits_{\left\vert \alpha\right\vert =t}I_{t,\alpha}\mathbf{y}^{\alpha}$ in
$S\otimes R_{2}^{2t}\left\langle \mathbf{y}\right\rangle $ have been defined
for all $t<m$. Take $\alpha$ with $\left\vert \alpha\right\vert =m$. Then
$\mathbf{y}^{\alpha}$ can be decomposed as $\mathbf{y}^{\beta_{st}}y_{st}$
($q=2$) for some couple $\left(  s,t\right)  $ with $s<t$, and a tuple
$\beta_{st}$ with $\left\vert \beta_{st}\right\vert =m-1$. Put $I_{m,\alpha
}=\sum_{s<t}\left(  I_{m-1,\beta_{st}}+\partial_{s}\left(  I_{m-1,\beta_{st}%
}\right)  +\partial_{t}\left(  I_{m-1,\beta_{st}}\right)  \right)  $, and
$I_{m}=\bigoplus\limits_{\left\vert \alpha\right\vert =m}I_{m,\alpha
}\mathbf{y}^{\alpha}$. Since $\partial_{i}\left(  J\right)  +\partial
_{j}\left(  J\right)  \subseteq J$, it follows that $I_{m,\alpha}=J$ whenever
$\mathbf{y}^{\alpha}=y_{ij}^{m}$. One can easily see that $\left\{
I_{m}\right\}  $ is a chain in $S_{2}$. Based on Remark \ref{remSpec1}, we
have the following differential operators $\Delta_{t}=-\sum_{s<t}\partial
_{s}y_{st}$ and $\nabla_{s}=-\sum_{s<t}\partial_{t}y_{st}$. Fix a couple
$\left(  s,t\right)  $ with $s<t$. Then%
\begin{align*}
\Delta_{t}\left(  I_{m-1,\beta}\mathbf{y}^{\beta}\right)  +\nabla_{s}\left(
I_{m-1,\beta}\mathbf{y}^{\beta}\right)   &  \subseteq\sum_{r<t}\partial
_{r}\left(  I_{m-1,\beta}\right)  \mathbf{y}^{\beta}y_{rt}+\sum_{s<r}%
\partial_{r}\left(  I_{m-1,\beta}\right)  \mathbf{y}^{\beta}y_{sr}\\
&  =\left(  \partial_{s}\left(  I_{m-1,\beta}\right)  +\partial_{t}\left(
I_{m-1,\beta}\right)  \right)  \mathbf{y}^{\beta}y_{st}+\cdots\\
&  \subseteq I_{m,\alpha}\mathbf{y}^{\alpha}+\cdots,
\end{align*}
that is, $\Delta_{t}\left(  I_{m-1}\right)  +\nabla_{s}\left(  I_{m-1}\right)
\subseteq I_{m}$. Hence $\left\{  I_{m}\right\}  $ is a differential chain in
$S_{2}$. Moreover, $N_{m}=S\otimes R_{2}^{2m}\left\langle \mathbf{y}%
\right\rangle /I_{m}=\left(  S/J\right)  y_{ij}^{m}\oplus\cdots$ and
$\mathcal{O}_{2,Y}=\mathcal{O}_{Y}\oplus\prod\limits_{m\geq1}\mathcal{O}%
_{Z}\left(  -2m\right)  y_{ij}^{m}\oplus\cdots$, where $Z=\operatorname{Proj}%
S/J$. Thus $Y$ admits an infinite geometric quantization in $\mathbb{P}%
_{k,2}^{n}$.

Now assume that $Y$ admits an infinite geometric quantization in
$\mathbb{P}_{k,2}^{n}$. Based on Theorem \ref{theoremMain}, we conclude that
there is a differential chain $\left\{  I_{m}\right\}  $ in $S_{2}$ such that
$I_{0}$ is the ideal of the projective scheme $Y$, and for every $m$ there
exists $\alpha$ such that the related graded ideal $I_{m,\alpha}\subseteq S$
of $I_{m}=\oplus_{\left\vert \alpha\right\vert =m}I_{m,\alpha}\subseteq
S\otimes R_{2}^{2m}\left\langle \mathbf{y}\right\rangle $ is proper. Note that
the proper means that $\operatorname{Proj}\left(  S/I_{m,\alpha}\right)
\neq\varnothing$. Prove that there exists a couple $\left(  i,j\right)  $,
$i<j$ such that every $I_{m,\alpha}$ is proper whenever $\mathbf{y}^{\alpha
}=y_{ij}^{m}$, $\left\vert \alpha\right\vert =m$. Indeed, in the contrary
case, there exists $N$ such that $I_{m,\alpha}=S$ for all $\mathbf{y}^{\alpha
}=y_{ij}^{m}$, $i<j$ and $m\geq N$. Pick $\mathbf{y}^{\alpha}$ with
$\left\vert \alpha\right\vert \geq Nn\left(  n+1\right)  /2$. There exists a
couple $\left(  i,j\right)  $, $i<j$ such that $m=\alpha_{ij}\geq N$. But
$I_{m,\alpha_{ij}}=S$ for the monomial $y_{ij}^{\alpha_{ij}}$, and
$Sy_{ij}^{\alpha_{ij}}y_{st}\subseteq I_{m+1,\beta}y_{ij}^{\alpha_{ij}}y_{st}$
for all $\left(  s,t\right)  $, $s<t$, that is, $I_{m+1,\beta}=S$ whenever
$\mathbf{y}^{\beta}=y_{ij}^{\alpha_{ij}}y_{st}$. It follows that $I_{m,\beta
}=S$ for all large $\beta$, a contradiction. Thus for a couple $\left(
i,j\right)  $ with $i<j$ all ideals $I_{m,\alpha}$, $\mathbf{y}^{\alpha
}=y_{ij}^{m}$ are proper. Put $J_{m}=I_{m,\alpha}$. Since $\left\{
I_{m}\right\}  $ is a chain, it follows that $J_{1}\subseteq J_{2}%
\subseteq\cdots$ is an increasing family of ideals. Since $\left\{
I_{m}\right\}  $ is a differential chain, it follows that $\partial_{i}\left(
J_{m}\right)  +\partial_{j}\left(  J_{m}\right)  \subseteq J_{m+1}$ for all
$m$. But $S$ is noetherian, therefore $J_{m}=J_{m+1}=\ldots=J$ for large $m$.
In particular, $\partial_{i}\left(  J\right)  +\partial_{j}\left(  J\right)
\subseteq J$. Moreover, $I_{0}y_{ij}^{m}\subseteq Jy_{ij}^{m}$, that is,
$I_{0}\subseteq J$ too.
\end{proof}

Based on Proposition \ref{propIQ1}, we can explain why the parabola considered
above does not admit an infinite geometric quantization in $\mathbb{P}%
_{k,2}^{2}$. Indeed, in this case $I_{0}=\left(  x_{0}x_{1}-x_{2}^{2}\right)
$ and $I_{01}=\left(  x_{0},x_{1},x_{2}^{2}\right)  $, $I_{02}=\left(
x_{1},x_{2}\right)  $ and $I_{12}=\left(  x_{0},x_{2}\right)  $, where
$I_{ij}=I_{0}+\partial_{i}\left(  I_{0}\right)  +\partial_{j}\left(
I_{0}\right)  $, $i<j$. Since $x_{0},x_{1}\in I_{01}$, $x_{2}\in I_{02}$ and
$x_{2}\in I_{12}$, it follows that for every couple $i<j$ there is no proper
ideal of $S$ containing $I_{ij}$ and being invariant with respect to
$\partial_{i}$ and $\partial_{j}$. By the same reasoning, the projective curve
$Y=Z\left(  x_{0}^{3}-x_{1}x_{2}^{2}\right)  $ does not admit an infinite
geometric quantization in $\mathbb{P}_{k,2}^{2}$, whereas all curves defined
by $f\left(  x_{0},x_{1}\right)  x_{2}^{d}$ do admit, where $f$ is a
homogeneous polynomial. Indeed, in this case we have $I_{01}\subseteq
J=\left(  x_{2}\right)  $ and $\partial_{0}\left(  J\right)  +\partial
_{1}\left(  J\right)  \subseteq J$.

\begin{corollary}
A projective curve in $\mathbb{P}_{k}^{2}$ defined by a homogeneous polynomial
$f\left(  x_{0},x_{1},x_{2}\right)  $ admits an infinite geometric
quantization in $\mathbb{P}_{k,2}^{2}$ if and only if $f$ is divisible by some
$x_{i}$.
\end{corollary}

\begin{proof}
Put $Y=Z\left(  f\right)  $, $I_{0}=\left(  f\right)  $, and $d=\deg\left(
f\right)  $. First assume that $f$ is divisible by some $x_{i}$, say $i=2$. As
above we have $I_{0}+\partial_{0}\left(  I_{0}\right)  +\partial_{1}\left(
I_{0}\right)  \subseteq J=\left(  x_{2}\right)  $ and $\partial_{0}\left(
J\right)  +\partial_{1}\left(  J\right)  \subseteq J$. Based on Proposition
\ref{propIQ1}, we conclude that $Y$ admits an infinite geometric quantization
in $\mathbb{P}_{k,2}^{2}$. Conversely, to be certain assume that there exists
a proper graded ideal $J\subseteq S$ such that $\partial_{0}$ and
$\partial_{1}$ leave invariant $J$ and $I_{0}+\partial_{0}\left(
I_{0}\right)  +\partial_{1}\left(  I_{0}\right)  \subseteq J$ (see Proposition
\ref{propIQ1}). Note that $f=\sum_{t=0}^{d}f_{t}\left(  x_{0},x_{1}\right)
x_{2}^{t}$ for some homogeneous $f_{t}\in S^{d-t}$, $0\leq t\leq d$. If
$f_{0}\neq0$ then $\partial_{0}^{\lambda}\partial_{1}^{\mu}f_{0}=c$ for a
certain couple $\lambda,\mu\geq0$ with $\lambda+\mu=d$, where $c\in k-\left\{
0\right\}  $. But $\partial_{0}^{\lambda}\partial_{1}^{\mu}f_{0}=\partial
_{0}^{\lambda}\partial_{1}^{\mu}f\in J$ and $J$ is a proper ideal, a
contradiction. Hence $f_{0}=0$ and $f\in\left(  x_{2}\right)  $.
\end{proof}

In particular, the union of two lines $x_{0}=x_{1}$ and $x_{2}=x_{1}$ in
$\mathbb{P}_{k}^{2}$ does not admit an infinite geometric quantization in
$\mathbb{P}_{k,2}^{2}$.

\subsection{Projective Lie-spaces}

Finally, we define the projective Lie space $\mathbb{P}_{\mathfrak{lie}%
,k,2}^{n}$ as an NC-nilpotent, projective $2$-scheme in $\mathbb{P}_{k,2}^{n}%
$. It turns out to be a non-geometrical quantization of $\mathbb{P}_{k}^{n}$.
By Theorem \ref{theoremMain}, we need to define the related differential chain
in $S_{q}$ for $q=2$. Put $I_{0}=I_{1}=0$ and consider the sum $\bigoplus
\limits_{\alpha\in\Xi_{m}}S\mathbf{y}^{\alpha}$, $m\geq2$, where $\Xi_{m}$ is
the set of those $\alpha$ such that $\left\vert \alpha\right\vert =m$ and
$\mathbf{y}^{\alpha}$ contains one of the factors $y_{ij}y_{il}$,
$y_{ij}y_{sj}$ or $y_{ij}y_{jt}$. Thus $\alpha\in\Xi_{m}$ iff $\alpha
_{ij}\alpha_{il}+\alpha_{ij}\alpha_{sj}+\alpha_{ij}\alpha_{jt}\geq1$ for some
$i<l$ and $s<j<t$. If $\alpha_{ij}=\alpha_{il}=1$ for some $j\neq l$, or
$\alpha_{ij}\geq2$, then $\alpha\in\Xi_{m}$ automatically. If $\alpha
_{ij}=\alpha_{sj}=1$ for some $i\neq s$, then $\alpha\in\Xi_{m}$ too. Finally,
if $\alpha_{ij}=\alpha_{jt}=1$ for some $j<t$, then $\alpha\in\Xi_{m}$. Define
$I_{m}$ to be an $S$-submodule in $\bigoplus\limits_{\alpha\in\Xi_{m}%
}S\mathbf{y}^{\alpha}$ of those sums $f=\sum_{\alpha}f_{\alpha}\mathbf{y}%
^{\alpha}$ such that $f_{\alpha}=-f_{\beta}$ whenever $\mathbf{y}^{\alpha
}=\mathbf{y}^{\gamma}y_{ij}y_{sl}$ and $\mathbf{y}^{\beta}=\mathbf{y}^{\gamma
}y_{is}y_{jl}$ for some tuple $\gamma$ and $i<j,s<l$. Thus $I_{m}=S\otimes
R_{q}^{2m}\left\langle \mathbf{y}\right\rangle $ for all $m>\left[  \left(
n+1\right)  /2\right]  $, and $N_{m}=S\otimes\Lambda^{2m}\left\langle
\mathbf{y}\right\rangle $, $0\leq m\leq\left[  \left(  n+1\right)  /2\right]
$, where $\Lambda^{2m}\left\langle \mathbf{y}\right\rangle $ is the quotient
of $R_{q}^{2m}\left\langle \mathbf{y}\right\rangle $ modulo relations
$y_{ij}y_{il}=y_{ij}y_{sj}=y_{ij}y_{jt}=0$ and $y_{ij}y_{sl}+y_{is}y_{jl}=0$.
Actually, $\Lambda^{2m}\left\langle \mathbf{y}\right\rangle $ is the space of
even differential forms in $d\mathbf{x}=\left(  dx_{0},\ldots,dx_{n}\right)  $
with $y_{ij}=2\left[  dx_{i},dx_{j}\right]  $, $i<j$ (see \cite{FS}). Note
that $T_{m}=N_{m}$ for all $m$ (see Subsection \ref{subsecNS}), and
$\widetilde{N_{m}}=\mathcal{O}\otimes\Lambda^{2m}\left\langle \mathbf{y}%
\right\rangle $ for all $m\geq1$. Based on Lemma \ref{lemExactSh}, we obtain
that
\[
\mathcal{O}_{\mathfrak{lie},2}=\mathcal{O}\oplus\mathcal{O}\left(  -2\right)
\otimes\Lambda^{2}\left\langle \mathbf{y}\right\rangle \oplus\cdots
\oplus\mathcal{O}\left(  -2m\right)  \otimes\Lambda^{2m}\left\langle
\mathbf{y}\right\rangle \oplus\cdots=\bigoplus\limits_{0\leq m\leq\left[
\left(  n+1\right)  /2\right]  }\mathcal{O}\left(  -2m\right)  \otimes
\Lambda^{2m}\left\langle \mathbf{y}\right\rangle ,
\]
which is the coherent sheaf of the projective Lie-space $\mathbb{P}%
_{\mathfrak{lie},k,2}^{n}$. In the case of $q=3$, we have the radical
variables $\mathbf{y}_{\left(  1\right)  }=\left(  y_{ij}:i<j\right)  $ and
$\mathbf{y}_{\left(  2\right)  }=\left(  y_{ijk}:i\geq j,j<k\right)  $, where
$y_{ijk}=\left[  x_{i},\left[  x_{j},x_{k}\right]  \right]  $. First consider
the sum $\bigoplus\limits_{\alpha\in\Xi_{m}}S\mathbf{y}^{\alpha}$, where
$\Xi_{m}$ is the set of those $\alpha$ such that $\left\langle \alpha
\right\rangle =m$ and $\mathbf{y}^{\alpha}$ contains one of the following
factors $y_{ij}y_{slu}$ or $y_{ijp}y_{slu}$. As above consider the submodule
$I_{m}\subseteq\bigoplus\limits_{\alpha\in\Xi_{m}}S\mathbf{y}^{\alpha}$ of
those $f=\sum_{\alpha}f_{\alpha}\mathbf{y}^{\alpha}$ such that $f_{\alpha
}=-f_{\beta}$ whenever $\mathbf{y}^{\alpha}=\mathbf{y}^{\gamma}y_{ij}%
y_{sl}y_{up}\mathbf{y}^{\theta}$ and $\mathbf{y}^{\beta}=\mathbf{y}^{\gamma
}y_{is}y_{jl}y_{up}\mathbf{y}^{\theta}$ for some $\gamma$, $i<j,s<l$, $u<p$
and a tuple $\theta$. Note that $E_{m}=S\otimes\Lambda^{m}\left\langle
\mathbf{y}\right\rangle $, where $\Lambda^{m}\left\langle \mathbf{y}%
\right\rangle $ is the quotient of $R_{3}^{m}\left\langle \mathbf{y}%
\right\rangle $ modulo relations $y_{ij}y_{slu}=y_{ijp}y_{slu}=0$ and
$y_{ij}y_{sl}y_{up}+y_{is}y_{jl}y_{up}=0$ (see \cite{EM} and \cite{DIZV}). It
turns out that $\mathcal{O}_{\mathfrak{lie},2}=\bigoplus\limits_{m}%
\mathcal{O}\left(  -m\right)  \otimes\Lambda^{m}\left\langle \mathbf{y}%
\right\rangle $ is the coherent structure sheaf of the projective Lie spaces
$\mathbb{P}_{\mathfrak{lie},k,3}^{n}$. In the case of any $q\geq4$ we do not
know whether the quotient algebra $\Lambda_{q}\left\langle \mathbf{y}%
\right\rangle $ of $R_{q}\left\langle \mathbf{y}\right\rangle $ is generated
by homogeneous relations (see \cite{DIZV}), that is, $\Lambda_{q}\left\langle
\mathbf{y}\right\rangle $ admits a grading $\Lambda_{q}\left\langle
\mathbf{y}\right\rangle =\oplus_{m\geq2}\Lambda_{q}^{m}\left\langle
\mathbf{y}\right\rangle $. For such $q$ we conclude that $\mathcal{O}%
_{\mathfrak{lie},q}=\bigoplus\limits_{m}\mathcal{O}\left(  -m\right)
\otimes\Lambda_{q}^{m}\left\langle \mathbf{y}\right\rangle $.

\section{Appendix\label{SecApp}}

In this section we provide the paper with supplementary material just to avoid
discomfort during reading the paper. They are mostly standard but sometimes it
is impossible to assign direct references.

\subsection{The scheme $\operatorname{Proj}S$\label{SubsecPp}}

Let $S=\oplus_{d\geq0}S^{d}$ be a positively graded commutative ring with its
ideal $S_{+}=\oplus_{d>0}S^{d}$, $X=\operatorname{Proj}S$ the related scheme
with its structure sheaf $\mathcal{O}_{X}$, and let $M=\oplus_{d\in\mathbb{Z}%
}M^{d}$ be a graded $S$-module. Recall that $M$ defines the quasi-coherent
$\mathcal{O}_{X}$-module $\widetilde{M}$ whose stalks are given by the
localizations $M_{\left(  \mathfrak{p}\right)  }$, $\mathfrak{p}\in X$, where
$M_{\left(  \mathfrak{p}\right)  }$ is the submodule of degree zero ratios in
the graded $S_{\mathfrak{p}}$-module $M_{\mathfrak{p}}$. If $s\in S_{+}$ is a
homogenous element then $M_{\left(  \mathfrak{p}\right)  }=M_{\left(
s\right)  }\otimes_{S_{\left(  s\right)  }}S_{\left(  \mathfrak{p}\right)  }$
whenever $\mathfrak{p}\in D_{+}\left(  s\right)  $, where $D_{+}\left(
s\right)  \subseteq X$ is the principal open subset associated with $s$. Since
$D_{+}\left(  s\right)  =\operatorname{Spec}S_{\left(  s\right)  }$ up to a
natural scheme isomorphism and $\left(  M_{\left(  s\right)  }\right)
_{\mathfrak{p}}=M_{\left(  s\right)  }\otimes_{S_{\left(  s\right)  }}\left(
S_{\left(  s\right)  }\right)  _{\mathfrak{p}}=M_{\left(  s\right)  }%
\otimes_{S_{\left(  s\right)  }}S_{\left(  \mathfrak{p}\right)  }=M_{\left(
\mathfrak{p}\right)  }$ for all $\mathfrak{p}\in D_{+}\left(  s\right)  $, it
follows that $\widetilde{M}|_{D_{+}\left(  s\right)  }=\widetilde{M_{\left(
s\right)  }}$ up to a sheaf isomorphism, where $\widetilde{M_{\left(
s\right)  }}$ is the quasi-coherent sheaf on the affine scheme
$\operatorname{Spec}S_{\left(  s\right)  }$ given by the $S_{\left(  s\right)
}$-module $M_{\left(  s\right)  }$.

Assume that $S$ is generated by $S^{1}$ as an $S^{0}$-algebra, that is,
$S=S^{0}\left[  S^{1}\right]  $. For any $l\in\mathbb{Z}$ the twisted module
$M\left(  l\right)  $ is a graded $S$-module with $M\left(  l\right)
^{d}=M^{d+l}$, $d\in\mathbb{Z}$. In particular, $S\left(  l\right)
=\oplus_{d}S\left(  l\right)  ^{d}=\oplus_{d\geq-l}S^{d+l}$ is the ring $S$
shifted to the left $l$ units. Being a graded $S$-module, $S\left(  l\right)
$ defines the quasi-coherent $\mathcal{O}_{X}$-module $\mathcal{O}_{X}\left(
l\right)  $ to be $\widetilde{S\left(  l\right)  }$, which is a locally free
$\mathcal{O}_{X}$-module. If $M$ and $N$ are graded $S$-modules then so is
$M\otimes_{S}N$ with $\left(  M\otimes_{S}N\right)  ^{d}=\oplus_{i+j=d}%
M^{i}\otimes_{S}N^{j}$, $d\in\mathbb{Z}$. Moreover, for every $s\in S^{1}$ we
obtain that $M_{\left(  s\right)  }\otimes_{S_{\left(  s\right)  }}N_{\left(
s\right)  }=\left(  M\otimes_{S}N\right)  _{\left(  s\right)  }$ up to a
canonical module isomorphism. It follows that $\widetilde{M}\otimes
_{\mathcal{O}_{X}}\widetilde{N}=\left(  M\otimes_{S}N\right)  ^{\sim}$ up to a
sheaf isomorphism of $\mathcal{O}_{X}$-modules. For every $\mathcal{O}_{X}%
$-module $\mathcal{F}$ we also define the twisted sheaf $\mathcal{F}\left(
l\right)  =\mathcal{O}_{X}\left(  l\right)  \mathcal{\otimes}_{\mathcal{O}%
_{X}}\mathcal{F}$. Since $S\left(  d\right)  \otimes_{S}M\left(  l\right)
=M\left(  d+l\right)  $ for all $d,l\in\mathbb{Z}$, it follows that%
\[
\widetilde{M\left(  l\right)  }=\mathcal{O}_{X}\left(  l\right)
\mathcal{\otimes}_{\mathcal{O}_{X}}\widetilde{M}=\widetilde{M}\left(
l\right)  \quad\text{and\quad}\mathcal{O}_{X}\left(  d\right)
\mathcal{\otimes}_{\mathcal{O}_{X}}\widetilde{M}\left(  l\right)
=\widetilde{M}\left(  d+l\right)  .
\]
In particular, $\mathcal{O}_{X}\left(  d\right)  \mathcal{\otimes
}_{\mathcal{O}_{X}}\mathcal{O}_{X}\left(  l\right)  =\mathcal{O}_{X}\left(
d+l\right)  $ and $\mathcal{O}_{X}\left(  d\right)  \mathcal{\otimes
}_{\mathcal{O}_{X}}\mathcal{F}\left(  l\right)  \mathcal{=F}\left(
d+l\right)  $. It allows us to define the action of $s\in S^{d}$ as a global
section of $\mathcal{O}_{X}\left(  d\right)  $ over the global section
$t\in\Gamma\left(  X,\mathcal{F}\left(  l\right)  \right)  $ by $s\otimes
t\in\mathcal{O}_{X}\left(  d\right)  \mathcal{\otimes}_{\mathcal{O}_{X}%
}\mathcal{F}\left(  l\right)  =\mathcal{F}\left(  d+l\right)  $. Namely,
$\Gamma_{\ast}\left(  \mathcal{F}\right)  =\oplus_{l}\Gamma\left(
X,\mathcal{F}\left(  l\right)  \right)  $ turns out to be a graded $S$-module
associated to $\mathcal{F}$. It is well known \cite[2.5.15]{Harts} that every
quasi-coherent $\mathcal{O}_{X}$-module $\mathcal{F}$ is reduced to
$\Gamma_{\ast}\left(  \mathcal{F}\right)  ^{\sim}$ whenever $S$ is finitely
generated by $S^{1}$ as an $S^{0}$-algebra.

\subsection{Morphisms and the direct image formula}

Recall that the direct image $f_{\ast}\mathcal{F}$ of a sheaf $\mathcal{F}$ of
abelian groups on $X$ along a continuous mapping $f:X\rightarrow Y$ of
topological spaces is given by $\left(  f_{\ast}\mathcal{F}\right)  \left(
U\right)  =\mathcal{F}\left(  f^{-1}\left(  U\right)  \right)  $ for any open
subset $U\subseteq Y$. It is a sheaf of abelian groups on $Y$.

Let $\varphi:S\rightarrow T$ be a graded (of degree zero) homomorphism of
positively graded rings, $X=\operatorname{Proj}S$, $Y=\operatorname{Proj}T$,
and let $U=\left\{  \mathfrak{q}\in Y:\varphi\left(  S_{+}\right)
\nsubseteq\mathfrak{q}\right\}  $ be a subset of $Y$. Since $U=\cup\left\{
D_{+}\left(  \varphi\left(  s\right)  \right)  :s\in S_{+}\right\}  $, it
follows that $U\subseteq Y$ is an open subset of $Y$ and $f:U\rightarrow X$,
$f\left(  \mathfrak{q}\right)  =\varphi^{-1}\left(  \mathfrak{q}\right)  $ is
a continuous mapping. Actually, $f^{-1}\left(  D_{+}\left(  s\right)  \right)
=D_{+}\left(  \varphi\left(  s\right)  \right)  $ for all homogenous $s\in
S_{+}$. For each $\mathfrak{q}\in U$ we have the local ring map $\varphi
_{\mathfrak{q}}:S_{\left(  f\left(  \mathfrak{q}\right)  \right)  }\rightarrow
T_{\left(  \mathfrak{q}\right)  }$, $\varphi_{\mathfrak{q}}\left(  a/b\right)
=\varphi\left(  a\right)  /\varphi\left(  b\right)  $ (or the local map
$\varphi_{\mathfrak{q}}:\mathcal{O}_{X,f\left(  \mathfrak{q}\right)
}\rightarrow\mathcal{O}_{Y,\mathfrak{q}}$ over stalks). They in turn define
the local sheaf morphism $f^{\times}:\mathcal{O}_{X}\rightarrow f_{\ast
}\mathcal{O}_{U}$. Thus $\left(  f,f^{\times}\right)  :U\rightarrow X$ is a
scheme morphism, which is glued by means of the morphisms $f_{s}:D_{+}\left(
\varphi\left(  s\right)  \right)  \rightarrow D_{+}\left(  s\right)  $ of the
affine schemes corresponding to the ring maps $\varphi_{\left(  s\right)
}:S_{\left(  s\right)  }\rightarrow T_{\left(  \varphi\left(  s\right)
\right)  }$. If $N$ is a graded $T$-module then%
\[
f_{\ast}\left(  \widetilde{N}|U\right)  |D_{+}\left(  s\right)  =f_{\ast
}\left(  \widetilde{N}|D_{+}\left(  \varphi\left(  s\right)  \right)  \right)
=f_{s\ast}\left(  \widetilde{N_{\left(  \varphi\left(  s\right)  \right)  }%
}\right)  =\left(  _{S_{\left(  s\right)  }}N_{\left(  \varphi\left(
s\right)  \right)  }\right)  ^{\sim}=\widetilde{N_{\left(  s\right)  }%
}=\left(  _{S}N\right)  ^{\sim}|D_{+}\left(  s\right)
\]
for all homogenous $s\in S_{+}$, where $_{S}N$ indicates to the graded
$S$-module structure on $N$ along $\varphi$, and $_{S_{\left(  s\right)  }%
}N_{\left(  \varphi\left(  s\right)  \right)  }$ is the $S_{\left(  s\right)
}$-module structure on $N_{\left(  \varphi\left(  s\right)  \right)  }$ along
$\varphi_{\left(  s\right)  }$. Thus $f_{\ast}\left(  \widetilde{N}|U\right)
=\left(  _{S}N\right)  ^{\sim}$. One can also prove its twisted version.
Namely,
\begin{align*}
f_{\ast}\left(  \widetilde{N}\left(  l\right)  |U\right)   &  =\left(
_{S}N\left(  l\right)  \right)  ^{\sim}=\left(  _{S}T\left(  l\right)
\otimes_{T}N\right)  ^{\sim}=\left(  _{S}S\left(  l\right)  \otimes
_{S}T\otimes_{T}N\right)  ^{\sim}\\
&  =\left(  _{S}S\left(  l\right)  \otimes_{S}N\right)  ^{\sim}=\left(
_{S}N\right)  ^{\sim}\left(  l\right)  =f_{\ast}\left(  \widetilde{N}%
|U\right)  \left(  l\right)  ,
\end{align*}
that is, the direct image formula
\begin{equation}
f_{\ast}\left(  \widetilde{N}\left(  l\right)  |U\right)  =\left(
_{S}N\right)  ^{\sim}\left(  l\right)  =f_{\ast}\left(  \widetilde{N}%
|U\right)  \left(  l\right)  \label{fdi}%
\end{equation}
holds for all $l\in\mathbb{Z}$. In particular, $f_{\ast}\left(  \mathcal{O}%
_{Y}\left(  l\right)  |U\right)  =f_{\ast}\left(  \widetilde{T\left(
l\right)  }|U\right)  =\left(  _{S}T\right)  \left(  l\right)  ^{\sim}=\left(
f_{\ast}\mathcal{O}_{U}\right)  \left(  l\right)  $.

\subsection{The disjoint unions\label{subsecDU}}

Let $X_{i}$, $i=1,2$ be topological spaces. Their disjoint union $X_{1}\sqcup
X_{2}$ is a topological space equipped with the finest topology such that both
inclusions $X_{i}\rightarrow X_{1}\sqcup X_{2}$ are continuous. The topology
base of open subsets in $X_{1}\sqcup X_{2}$ consists of those $U_{1}\sqcup
U_{2}$ with open $U_{i}\subseteq X_{i}$. In particular, $X_{i}$ is identified
with the related clopen subset of $X_{1}\sqcup X_{2}$ up to a homeomorphism.
If $\mathcal{F}$ is a sheaf of abelian groups on $X_{1}\sqcup X_{2}$ then it
defines the sheaves $\mathcal{F}^{i}=\mathcal{F}|X_{i}$ on $X_{i}$, $i=1,2$.

The disjoint union of affine schemes responds to the direct sum of the
commutative rings. Note that if $R_{i}$ is a commutative ring with its
multiplicative subset $\mathcal{S}_{i}\subseteq R_{i}$, $i=1,2$, then
$\mathcal{S=S}_{1}\times S_{2}$ is a multiplicative subset of $R=R_{1}\oplus
R_{2}$ and $\mathcal{S}^{-1}R=\mathcal{S}_{1}^{-1}R_{1}\oplus\mathcal{S}%
_{2}^{-1}R_{2}$ up to a canonical ring isomorphism. There are idempotents
$e_{1}=\left(  1,0\right)  $ and $e_{2}=\left(  0,1\right)  $ in $R$ such that
$e_{1}e_{2}=0$ and $R_{i}$ is identified with the ideal $Re_{i}$, $i=1,2$. The
canonical projections $\pi_{i}:R\rightarrow R_{i}$ define the closed
immersions $X_{i}\rightarrow X$, where $X=\operatorname{Spec}\left(  R\right)
$ and $X_{i}=\operatorname{Spec}\left(  R_{i}\right)  $. If $\mathfrak{p}\in
X$ then $e_{i}\in\mathfrak{p}$ for a certain $i$, that is, $R_{i}%
\subseteq\mathfrak{p}$. If $R_{2}\subseteq\mathfrak{p}$ then $\mathfrak{p}%
_{1}=R_{1}\cap\mathfrak{p}$ is a prime of $R_{1}$ and
$\mathfrak{p=\mathfrak{p}}_{1}+R_{2}=\mathfrak{p}_{1}\oplus R_{2}$, that is,
$\mathfrak{p=\pi}_{1}^{-1}\left(  \mathfrak{p}_{1}\right)  $. In this case,
$\mathcal{S=}R-\mathfrak{p}=\left(  R_{1}-\mathfrak{p}_{1}\right)  \times
R_{2}$ and $R_{\mathfrak{p}}=\mathcal{S}^{-1}R=\left(  R_{1}\right)
_{\mathfrak{p}_{1}}\oplus\left\{  0\right\}  =\left(  R_{1}\right)
_{\mathfrak{p}_{1}}$. If $s=\left(  s_{1},s_{2}\right)  \notin%
\operatorname{nil}\left(  R\right)  $ and $D\left(  s\right)  $ is the related
principal open subset with $\mathfrak{p\in}D\left(  s\right)  $, then
$\mathfrak{p}_{1}\in D\left(  s_{1}\right)  $. It means that $X=X_{1}\sqcup
X_{2}$, $D\left(  s\right)  =D\left(  s_{1}\right)  \sqcup D\left(
s_{2}\right)  $ and $X_{i}\rightarrow X$ are open immersions with
$\mathcal{O}_{X}|X_{i}=\mathcal{O}_{X_{i}}$, $i=1,2$.

If $R_{i}$ are positively graded rings with $X_{i}=\operatorname{Proj}\left(
R_{i}\right)  $, then so is $R$ with $R^{d}=R_{1}^{d}\oplus R_{2}^{d}$,
$d\geq0$, and let $X=\operatorname{Proj}\left(  R\right)  $. If $s=\left(
s_{1},s_{2}\right)  \in R^{d}$ for $d>0$, then $R_{s}^{d}=\left(
R_{1}\right)  _{s_{1}}^{d}\oplus\left(  R_{2}\right)  _{s_{2}}^{d}$ for all
$d\in\mathbb{Z}$. In particular, we derive that $R_{\left(  s\right)
}=\left(  R_{1}\right)  _{\left(  s_{1}\right)  }\oplus\left(  R_{2}\right)
_{\left(  s_{2}\right)  }$. If $\mathfrak{p}\in\operatorname{Proj}\left(
R\right)  $ then as above $R_{i}\subseteq\mathfrak{p}$ for some $i$. If
$R_{2}\subseteq\mathfrak{p}$ then $\mathfrak{p}_{1}=R_{1}\cap\mathfrak{p}$ is
a graded ideal of $R$ and $\mathfrak{p=p}_{1}\oplus R_{2}$, which means that
$\mathfrak{p\in}\operatorname{Proj}\left(  R_{1}\right)  $. If $s=\left(
s_{1},s_{2}\right)  \in R_{+}$ with $\mathfrak{p\in}D_{+}\left(  s\right)  $,
then $D_{+}\left(  s\right)  =D_{+}\left(  s_{1}\right)  \sqcup D_{+}\left(
s_{2}\right)  $. It means that $X=X_{1}\sqcup X_{2}$ and
\[
R_{\left(  \mathfrak{p}\right)  }=\left(  R_{\left(  s\right)  }\right)
_{\mathfrak{p}}=\left(  \left(  R_{1}\right)  _{\left(  s_{1}\right)  }%
\oplus\left(  R_{2}\right)  _{\left(  s_{2}\right)  }\right)  _{\mathfrak{p}%
}=\left(  \left(  R_{1}\right)  _{\left(  s_{1}\right)  }\right)
_{\mathfrak{p}_{1}}=\left(  R_{1}\right)  _{\left(  \mathfrak{p}_{1}\right)
},
\]
that is, $\mathcal{O}_{X}|X_{i}=\mathcal{O}_{X_{i}}$, $i=1,2$. We skip the
discussion on disjoint unions of the general schemes.

Now assume that $X=X_{1}=X_{2}$. In this case, one can also choose the family
$\left\{  U\sqcup U:U\subseteq X\right\}  $ as a topology base in $X\sqcup X$.
There is a canonical mapping $\sigma:X\sqcup X\rightarrow X$ being the
identity map on each copy of $X$ in the disjoint union. Obviously $\sigma$ is
a continuous mapping of the related topological spaces. For $x\in X$ and its
open neighborhood $U\subseteq X$ we see that $\sigma^{-1}\left\{  x\right\}
=\left\{  x_{1},x_{2}\right\}  $ are the related copies of $x$, and
$\sigma^{-1}\left(  U\right)  =U\sqcup U=U_{1}\cup U_{2}$ where $U_{1}%
=U\sqcup\varnothing$ and $U_{2}=\varnothing\sqcup U$ are open neighborhoods of
$x_{1}$ and $x_{2}$ in $X\sqcup X$, respectively. In particular, $X\sqcup
X=X_{1}\cup X_{2}$, where $X_{i}$ are the related copies of $X$ in $X\sqcup X$.

\begin{lemma}
\label{lemDU1}If $\mathcal{F}$ is a sheaf of abealian groups on $X\sqcup X$,
then $\sigma_{\ast}\mathcal{F=F}^{1}\mathcal{\oplus F}^{2}$ up to a sheaf
isomorphism. In particular, if $0\rightarrow\mathcal{F}^{\prime}%
\overset{\varphi}{\longrightarrow}\mathcal{F}\overset{\psi}{\longrightarrow
}\mathcal{F}^{\prime\prime}\rightarrow0$ is an exact sequence of sheaves on
$X\sqcup X$ then the sequence $0\rightarrow\sigma_{\ast}\mathcal{F}^{\prime
}\overset{\sigma_{\ast}\varphi}{\longrightarrow}\sigma_{\ast}\mathcal{F}%
\overset{\sigma_{\ast}\psi}{\longrightarrow}\sigma_{\ast}\mathcal{F}%
^{\prime\prime}\rightarrow0$ remains exact.
\end{lemma}

\begin{proof}
Fix a point $x\in X$ and a stalk $\left\langle U,s\right\rangle \in\left(
\sigma_{\ast}\mathcal{F}\right)  _{x}$, where $s\in\mathcal{F}\left(
\sigma^{-1}\left(  U\right)  \right)  $ and $U$ is an open neighborhood of $x$
with $\sigma^{-1}\left(  U\right)  =U_{1}\cup U_{2}$. Put $s_{i}=s|U_{i}%
\in\mathcal{F}\left(  U_{i}\right)  $ with the stalks $\left\langle
U_{i},s_{i}\right\rangle \in\mathcal{F}_{x_{i}}^{i}$, $i=1,2$. Define the
homomorphism $\alpha:\left(  \sigma_{\ast}\mathcal{F}\right)  _{x}%
\rightarrow\mathcal{F}_{x_{1}}^{1}\oplus\mathcal{F}_{x_{2}}^{2}$,
$\alpha\left\langle U,s\right\rangle =\left(  \left\langle U_{1}%
,s_{1}\right\rangle ,\left\langle U_{2},s_{2}\right\rangle \right)  $, which
is obviously injective. Take sections $t_{i}\in\mathcal{F}^{i}\left(
V_{i}\right)  $, $x_{i}\in V_{i}\subseteq X_{i}$, where $V_{i}$ are open
subsets in $X$. Put $W=V_{1}\cap V_{2}$ to be an open neighborhood of $x$ in
$X$. Then $W\sqcup W=W_{1}\cup W_{2}$ and $\left\langle V_{i},t_{i}%
\right\rangle =\left\langle W_{i},t_{i}^{\prime}\right\rangle $ in
$\mathcal{F}_{x_{i}}^{i}$, where $t_{i}^{\prime}=t_{i}|_{W_{i}}$. Since
$W_{1}\cap W_{2}=\varnothing$, there is a unique section $t^{\prime}%
\in\mathcal{F}\left(  W\sqcup W\right)  =\left(  \sigma_{\ast}\mathcal{F}%
\right)  \left(  W\right)  $ obtained from $t_{1}^{\prime}$ and $t_{2}%
^{\prime}$. Moreover, $\alpha\left\langle W,t^{\prime}\right\rangle =\left(
\left\langle W_{1},t_{1}^{\prime}\right\rangle ,\left\langle W_{2}%
,t_{2}^{\prime}\right\rangle \right)  $, which means that $\alpha$ implements
an isomorphism of the stalks, that is, $\left(  \sigma_{\ast}\mathcal{F}%
\right)  _{x}=\mathcal{F}_{x_{1}}^{1}\oplus\mathcal{F}_{x_{2}}^{2}$ for every
$x$. Hence $\sigma_{\ast}\mathcal{F=F}^{1}\mathcal{\oplus F}^{2}$ up to a
sheaf isomorphism.

In particular, the sequence $0\rightarrow\left(  \sigma_{\ast}\mathcal{F}%
^{\prime}\right)  _{x}\longrightarrow\left(  \sigma_{\ast}\mathcal{F}\right)
_{x}\longrightarrow\left(  \sigma_{\ast}\mathcal{F}^{\prime\prime}\right)
_{x}\rightarrow0$ of stalks is reduced to the sum of exact sequences
$0\rightarrow\mathcal{F}_{x_{i}}^{\prime i}\longrightarrow\mathcal{F}_{x_{i}%
}^{i}\longrightarrow\mathcal{F}_{x_{i}}^{\prime\prime i}\rightarrow0$, $i=1,2$
of stalks of the original sequence. Therefore it is exact.
\end{proof}

Now assume that $X=\operatorname{Proj}\left(  S\right)  $ that responds to a
positively graded ring $S$. Consider the graded ring $S\oplus S$ with the
diagonal (graded) homomorphism $\epsilon:S\rightarrow S\oplus S$,
$\epsilon\left(  a\right)  =\left(  a,a\right)  $. In particular, $S\oplus S$
turns out to be a graded $S$-module along $\epsilon$. As we have seen above
$\operatorname{Proj}\left(  S\oplus S\right)  =X\sqcup X$, that is, $X\sqcup
X$ is the scheme associated to $S\oplus S$, and $\mathcal{O}_{X\sqcup
X}|X=\mathcal{O}_{X}$.

As above we consider the open subset $U=\left\{  \mathfrak{p}\in
\operatorname{Proj}\left(  S\oplus S\right)  :S_{+}\nsubseteq\epsilon
^{-1}\left(  \mathfrak{p}\right)  \right\}  $ in $\operatorname{Proj}\left(
S\oplus S\right)  $. In this case $U=\operatorname{Proj}\left(  S\oplus
S\right)  $, and the continuous mapping $\operatorname{Proj}\left(  S\oplus
S\right)  \rightarrow\operatorname{Proj}\left(  S\right)  $, $\mathfrak{p}%
\mapsto\epsilon^{-1}\left(  \mathfrak{p}\right)  $ of the scheme morphism is
reduced to the canonical mapping $\sigma:X\sqcup X\rightarrow X$ of the
disjoint union considered above. Thus we have the scheme morphism $\left(
\sigma,\sigma^{\times}\right)  :X\sqcup X\rightarrow X$ with the related sheaf
morphism $\sigma^{\times}:\mathcal{O}_{X}\rightarrow\sigma_{\ast}%
\mathcal{O}_{X\sqcup X}$.

\begin{lemma}
\label{lemDU2}There is a natural identification $\sigma_{\ast}\mathcal{O}%
_{X\sqcup X}=\mathcal{O}_{X}\oplus\mathcal{O}_{X}$ up to a sheaf isomorphism,
and the morphism $\sigma^{\times}$ is reduced to the diagonal morphism
$\mathcal{O}_{X}\rightarrow\mathcal{O}_{X}\oplus\mathcal{O}_{X}$.
\end{lemma}

\begin{proof}
First note that $S\oplus S$ being a graded $S$-module along $\epsilon$ admits
the localization $\left(  S\oplus S\right)  _{s}$ which in turn is reduced to
the graded ring $\left(  S\oplus S\right)  _{\epsilon\left(  s\right)  }$ for
a homogeneous $s\in S_{+}$. Thus $S_{s}\oplus S_{s}=\left(  S\oplus S\right)
_{\epsilon\left(  s\right)  }$ as the rings, and $S_{\left(  s\right)  }\oplus
S_{\left(  s\right)  }=\left(  S\oplus S\right)  _{\left(  \epsilon\left(
s\right)  \right)  }$. The $S$-module $S\oplus S$ defines the quasi-coherent
$\mathcal{O}_{X}$-module $\left(  S\oplus S\right)  ^{\sim}$, which is just
$\mathcal{O}_{X}\oplus\mathcal{O}_{X}$. Moreover, there is a sheaf morphism
$\widetilde{\epsilon}:\mathcal{O}_{X}\rightarrow\mathcal{O}_{X}\oplus
\mathcal{O}_{X}$ on $X$ determined by $\epsilon$. Note that $\epsilon
_{s}\left(  S_{\left(  s\right)  }\right)  \subseteq\left(  S\oplus S\right)
_{\left(  \epsilon\left(  s\right)  \right)  }$ and put $\epsilon_{\left(
s\right)  }=\epsilon_{s}|S_{\left(  s\right)  }:S_{\left(  s\right)
}\rightarrow S_{\left(  s\right)  }\oplus S_{\left(  s\right)  }$. If
$D_{+}\left(  s\right)  $ is the principal open subset in $\operatorname{Proj}%
\left(  S\right)  $ defined by $s$, then $\left(  \sigma_{\ast}\mathcal{O}%
_{X\sqcup X}\right)  \left(  D_{+}\left(  s\right)  \right)  =\mathcal{O}%
_{X}\left(  D_{+}\left(  s\right)  \right)  \oplus\mathcal{O}_{X}\left(
D_{+}\left(  s\right)  \right)  $ by Lemma \ref{lemDU1}, and the homomorphism
$\sigma^{\times}\left(  D_{+}\left(  s\right)  \right)  :\mathcal{O}%
_{X}\left(  D_{+}\left(  s\right)  \right)  \rightarrow\left(  \sigma_{\ast
}\mathcal{O}_{X\sqcup X}\right)  \left(  D_{+}\left(  s\right)  \right)  $ is
reduced to the canonical homomorphism $\epsilon_{\left(  s\right)  }$. Whence
$\sigma^{\times}=\widetilde{\epsilon}$ and $\sigma_{\ast}\mathcal{O}_{X\sqcup
X}=\mathcal{O}_{X}\oplus\mathcal{O}_{X}$ up to a sheaf isomorphism.
\end{proof}

Recall that in the case of the graded polynomial algebra $S=k\left[
x_{0},\ldots,x_{n}\right]  $ over a field $k$, the scheme $\operatorname{Proj}%
\left(  S\right)  $ is a projective scheme over $k$, which is just the
projective space $\mathbb{P}_{k}^{n}$. In particular, there is a morphism
$\left(  \sigma,\sigma^{\times}\right)  :\mathbb{P}_{k}^{n}\sqcup
\mathbb{P}_{k}^{n}\rightarrow\mathbb{P}_{k}^{n}$ obtained from the diagonal
homomorphism $\epsilon$.

\subsection{The projective schemes over $\operatorname{Spec}A$%
\label{SubsecPSA}}

Let $A$ be a commutative ring, $S=A\left[  x_{0},\ldots,x_{n}\right]  $ a
graded $A$-algebra, and let $X=\operatorname{Proj}\left(  S\right)
=\mathbb{P}_{A}^{n}$ be the projective space over $A$ with the canonical
morphism $X\rightarrow\operatorname{Spec}A$. The latter is just the projection
$\mathbb{P}_{A}^{n}=\mathbb{P}_{\mathbb{Z}}^{n}\times_{\mathbb{Z}%
}\operatorname{Spec}A\rightarrow\operatorname{Spec}A$. Recall that a scheme
morphism $f:Y\rightarrow\operatorname{Spec}A$ is called a projective morphism
if it factors into a closed immersion $i:Y\hookrightarrow\mathbb{P}_{A}^{n}$
followed by the projection $\mathbb{P}_{A}^{n}\rightarrow\operatorname{Spec}A$
for a certain $n$. We also that $Y$ is a projective scheme over $A$.

If $i:Y\hookrightarrow X$ is a closed immersion of schemes then $i$ is
quasi-compact and separated \cite[2.4.6]{Harts}, which in turn implies that
$i_{\ast}\mathcal{O}_{Y}$ is a quasi-coherent $\mathcal{O}_{X}$-module. In
particular, the ideal sheaf $\mathcal{I}_{Y}$ of $Y$ being the kernel of the
sheaf morphism $i^{\times}:\mathcal{O}_{X}\rightarrow i_{\ast}\mathcal{O}_{Y}$
turns out to be a quasi-coherent $\mathcal{O}_{X}$-module. If $X$ is
noetherian and $U=\operatorname{Spec}\left(  A\right)  \subseteq X$ is an open
affine subset, then $I=\mathcal{I}_{Y}\left(  U\right)  $ is an ideal of the
noetherian ring $A=\mathcal{O}_{X}\left(  U\right)  $ and $\mathcal{I}%
_{Y}|U=\widetilde{I}$ turns out to be a coherent sheaf. Conversely, if
$\mathcal{I}$ is a quasi-coherent sheaf of ideals on $X$ then $\mathcal{O}%
_{X}/\mathcal{I}$ is a quasi-coherent $\mathcal{O}_{X}$-module, whose support
$Y$ is a subspace of $X$ and $\left(  Y,\mathcal{O}_{X}/\mathcal{I}\right)  $
is a scheme. As above if $U=\operatorname{Spec}\left(  A\right)  $ is an open
affine subset in $X$ then $\mathcal{I}|U=\widetilde{I}$ and $\left(
\mathcal{O}_{X}/\mathcal{I}\right)  |U=\left(  A/I\right)  ^{\sim}$. It
follows that $Y\cap U=\operatorname{Spec}\left(  A/I\right)  $ is a closed
subspace of $U$. Thus $\left(  Y,\mathcal{O}_{X}/\mathcal{I}\right)  $ is the
unique closed subscheme of $X$ with the ideal sheaf $\mathcal{I}$, and
$\mathcal{I}_{Y}=\mathcal{I}$.

Now let $Y$ be a projective scheme over $A$ with a closed immersion
$i:Y\hookrightarrow\mathbb{P}_{A}^{n}$. For every $l\in\mathbb{Z}$ the twisted
sheaf $\mathcal{O}_{X}\left(  l\right)  $ is a locally free $\mathcal{O}_{X}%
$-module (of rank $1$), therefore $0\rightarrow\mathcal{O}_{X}\left(
l\right)  \otimes_{\mathcal{O}_{X}}\mathcal{I}_{Y}\rightarrow\mathcal{O}%
_{X}\left(  l\right)  \rightarrow\mathcal{O}_{X}\left(  l\right)
\otimes_{\mathcal{O}_{X}}i_{\ast}\mathcal{O}_{Y}\rightarrow0$ remains exact.
Since the functor $\Gamma$ is left exact, we deduce that $\Gamma_{\ast}\left(
\mathcal{I}_{Y}\right)  $ is a homogenous ideal of $\Gamma_{\ast}\left(
\mathcal{O}_{X}\right)  $. Put $I=\Gamma_{\ast}\left(  \mathcal{I}_{Y}\right)
$. Taking into account that $\mathcal{I}_{Y}$ is a quasi-coherent
$\mathcal{O}_{X}$-module and $S$ is an algebra finite over $A$, we conclude
that $\mathcal{I}_{Y}=\widetilde{I}$ (see Subsection \ref{SubsecPp}). Recall
that $\Gamma_{\ast}\left(  \mathcal{O}_{X}\right)  =S$ (see \cite[2.5.13]%
{Harts}). Then $I\subseteq S$ is a homogenous ideal of $S$. From the other
hand side, $I$ defines the closed subscheme $\operatorname{Proj}\left(
S/I\right)  \hookrightarrow X$ (obtained from graded ring map $S\rightarrow
S/I$) whose ideal sheaf is $\widetilde{I}$. Hence $Y$ is a closed subscheme of
$X$ given by the ideal $I$, that is, $Y=\operatorname{Proj}\left(  S/I\right)
$ up to a scheme isomorphism. In this case, we can assume that $I\subseteq
S_{+}$, $\left(  S/I\right)  ^{0}=A$ and $S/I=A\left[  x_{0}^{\prime}%
,\ldots,x_{n}^{\prime}\right]  $ is an algebra finite over $A$.

\end{document}